\documentclass[9pt]{amsart}
\usepackage{amsmath}
\usepackage{amsxtra}
\usepackage{amscd}
\usepackage{amsthm}
\usepackage{amsfonts}
\usepackage{amssymb}
\usepackage{eucal}
\usepackage[all]{xy}
\usepackage{graphicx}
\usepackage[usenames]{color}

\usepackage[hidelinks]{hyperref}
\usepackage[normalem]{ulem}

\newtheorem{cor}[subsubsection]{Corollary}
\newtheorem{lem}[subsubsection]{Lemma}
\newtheorem{prop}[subsubsection]{Proposition}

\newtheorem{thmconstr}[subsubsection]{Theorem-Construction}

\newtheorem{conj}[subsubsection]{Conjecture}
\newtheorem{thm}[subsubsection]{Theorem}

\theoremstyle{remark}
\newtheorem{rem}[subsubsection]{Remark}

\theoremstyle{definition}

\theoremstyle{remark}

\newcommand{\thmref}[1]{Theorem~\ref{#1}}

\newcommand{\secref}[1]{Sect.~\ref{#1}}
\newcommand{\lemref}[1]{Lemma~\ref{#1}}
\newcommand{\propref}[1]{Proposition~\ref{#1}}
\newcommand{\corref}[1]{Corollary~\ref{#1}}
\newcommand{\conjref}[1]{Conjecture~\ref{#1}}

\numberwithin{equation}{section}

\newcommand{\nc}{\newcommand}
\nc{\renc}{\renewcommand}
\nc{\ssec}{\subsection}
\nc{\sssec}{\subsubsection}
\nc{\on}{\operatorname}

\nc{\ips}{{\iota_P^{(S)}}}
\nc{\ipms}{{\iota_{P^-}^{(S)}}}
\nc{\sfpps}{{\sfp_P^{(S)}}}
\nc{\sfppms}{{\sfp_{P^-}^{(S)}}}

\nc\ol{\overline}
\nc\ul{\underline}
\nc\wt{\widetilde}
\nc\tboxtimes{\wt{\boxtimes}}
\nc\tstar{\wt{\star}}
\nc{\alp}{\alpha}

\nc{\ZZ}{{\mathbb Z}}
\nc{\NN}{{\mathbb N}}
\nc{\OO}{{\mathbb O}}
\renc{\SS}{{\mathbb S}}
\nc{\DD}{{\mathbb D}}
\nc{\GG}{{\mathbb G}}

\nc{\Fq}{{\mathbb F}_q}
\nc{\Fqb}{\ol{\mathbb F}_q}
\nc{\Ql}{{\mathbb Q}_\ell}
\nc{\Qlb}{{\ol{\mathbb Q}_\ell}}
\nc{\id}{\text{id}}
\nc\X{\mathcal X}

\nc{\red}{\on{red}}
\nc{\Ho}{\on{Ho}}
\nc{\Hom}{\on{Hom}}
\nc{\coef}{\on{coef}}
\nc{\Lie}{\on{Lie}}
\nc{\Loc}{\on{Loc}}
\nc{\Pic}{\on{Pic}}
\nc{\Bun}{\on{Bun}}
\nc{\IC}{\on{IC}}
\nc{\Aut}{\on{Aut}}
\nc{\rk}{\on{rk}}
\nc{\Sh}{\on{Sh}}
\nc{\Perv}{\on{Perv}}
\nc{\pos}{{\on{pos}}}
\nc{\Conv}{\on{Conv}}
\nc{\Sph}{\on{Sph}}
\nc{\Sym}{\on{Sym}}
\nc{\BunBb}{\overline{\Bun}_B}
\nc{\BunNb}{\overline{\Bun}_N}
\nc{\BunTb}{\overline{\Bun}_T}
\nc{\BunBbm}{\overline{\Bun}_{B^-}}
\nc{\BunBbel}{\overline{\Bun}_{B,el}}
\nc{\BunBbmel}{\overline{\Bun}_{B^-,el}}
\nc{\Buno}{\overset{o}{\Bun}}
\nc{\BunPb}{{\overline{\Bun}_P}}
\nc{\BunBM}{\Bun_{B(M)}}
\nc{\BunBMb}{\overline{\Bun}_{B(M)}}
\nc{\BunPbw}{{\widetilde{\Bun}_P}}
\nc{\BunBP}{\widetilde{\Bun}_{B,P}}
\nc{\GUb}{\overline{G/U}}
\nc{\GUPb}{\overline{G/U(P)}}

\nc{\Hhom}{\underline{\on{Hom}}}
\nc\syminfty{\on{Sym}^{\infty}}
\nc\lal{\ol{\lambda}}
\nc\xl{\ol{x}}
\nc\thl{\ol{\theta}}
\nc\nul{\ol{\nu}}
\nc\mul{\ol{\mu}}
\nc\Sum\Sigma
\nc{\oX}{\overset{o}{X}{}}
\nc{\hl}{\overset{\leftarrow}h{}}
\nc{\hr}{\overset{\rightarrow}h{}}
\nc{\M}{{\mathcal M}}
\nc{\N}{{\mathcal N}}
\nc{\F}{{\mathcal F}}
\nc{\D}{{\mathcal D}}
\nc{\Q}{{\mathcal Q}}
\nc{\Y}{{\mathcal Y}}
\nc{\G}{{\mathcal G}}
\nc{\E}{{\mathcal E}}
\nc{\CalC}{{\mathcal C}}
\nc\Dh{\widehat{\D}}

\nc{\C}{{\mathcal C}}
\nc{\K}{{\mathcal K}}
\renewcommand{\H}{{\mathcal H}}

\nc{\T}{{\mathcal T}}
\nc{\V}{{\mathcal V}}
\renc{\P}{{\mathcal P}}
\nc{\A}{{\mathcal A}}
\nc{\B}{{\mathcal B}}
\nc{\U}{{\mathcal U}}

\nc{\Gr}{{\on{Gr}}}

\nc{\frn}{{\check{\mathfrak u}(P)}}

\nc{\fC}{\mathfrak C}
\nc{\fO}{\mathfrak O}
\nc{\fT}{\mathfrak T}
\nc{\p}{\mathfrak p}
\nc{\q}{\mathfrak q}
\nc\f{{\mathfrak f}}

\nc{\qo}{{\mathfrak q}}
\nc{\po}{{\mathfrak p}}
\nc{\s}{{\mathfrak s}}
\nc\w{\text{w}}

\renewcommand{\mod}{{\on{-mod}}}

\nc\mathi\iota
\nc\Spec{\on{Spec}}
\nc\Proj{\on{Proj}}
\nc\Mod{\on{Mod}}
\nc{\tw}{\widetilde{\mathfrak t}}
\nc{\pw}{\widetilde{\mathfrak p}}
\nc{\qw}{\widetilde{\mathfrak q}}
\nc{\jw}{\widetilde j}

\nc{\grb}{\overline{\Gr}}
\nc{\I}{\mathcal I}

\nc{\lambdach}{{\check\lambda}}
\nc{\Lambdach}{{\check\Lambda}{}}
\nc{\much}{{\check\mu}}
\nc{\omegach}{{\check\omega}}
\nc{\nuch}{{\check\nu}}
\nc{\etach}{{\check\eta}}
\nc{\alphach}{{\check\alpha}}
\nc{\boblvtach}{{\check\boblvta}}
\nc{\rhoch}{{\check\rho}}
\nc{\ch}{{\check h}}

\nc{\Hb}{\overline{\H}}

\nc{\BA}{{\mathbb{A}}}
\nc{\BC}{{\mathbb{C}}}
\nc{\BE}{{\mathbb{E}}}
\nc{\BF}{{\mathbb{F}}}
\nc{\BG}{{\mathbb{G}}}
\nc{\BM}{{\mathbb{M}}}
\nc{\BO}{{\mathbb{O}}}
\nc{\BD}{{\mathbb{D}}}
\nc{\BN}{{\mathbb{N}}}
\nc{\BP}{{\mathbb{P}}}
\nc{\BQ}{{\mathbb{Q}}}
\nc{\BR}{{\mathbb{R}}}
\nc{\BZ}{{\mathbb{Z}}}
\nc{\BS}{{\mathbb{S}}}
\nc{\Deep}{{\bf{deep}}}
\nc{\deep}{deep}

\nc{\CA}{{\mathcal{A}}}
\nc{\CB}{{\mathcal{B}}}

\nc{\CE}{{\mathcal{E}}}
\nc{\CF}{{\mathcal{F}}}
\nc{\CH}{{\mathcal{H}}}

\nc{\CL}{{\mathcal{L}}}
\nc{\CC}{{\mathcal{C}}}
\nc{\CG}{{\mathcal{G}}}
\nc{\CalD}{{\mathcal{D}}}
\nc{\CM}{{\mathcal{M}}}
\nc{\CN}{{\mathcal{N}}}
\nc{\CK}{{\mathcal{K}}}
\nc{\CO}{{\mathcal{O}}}
\nc{\CP}{{\mathcal{P}}}
\nc{\CQ}{{\mathcal{Q}}}
\nc{\CR}{{\mathcal{R}}}
\nc{\CS}{{\mathcal{S}}}
\nc{\CT}{{\mathcal{T}}}
\nc{\CU}{{\mathcal{U}}}
\nc{\CV}{{\mathcal{V}}}
\nc{\CW}{{\mathcal{W}}}
\nc{\CX}{{\mathcal{X}}}
\nc{\CY}{{\mathcal{Y}}}
\nc{\CZ}{{\mathcal{Z}}}
\nc{\CI}{{\mathcal{I}}}

\nc{\csM}{{\check{\mathcal A}}{}}
\nc{\oM}{{\overset{\circ}{\mathcal M}}{}}
\nc{\obM}{{\overset{\circ}{\mathbf M}}{}}
\nc{\oCA}{{\overset{\circ}{\mathcal A}}{}}
\nc{\obA}{{\overset{\circ}{\mathbf A}}{}}
\nc{\ooM}{{\overset{\circ}{M}}{}}
\nc{\osM}{{\overset{\circ}{\mathsf M}}{}}
\nc{\vM}{{\overset{\bullet}{\mathcal M}}{}}
\nc{\nM}{{\underset{\bullet}{\mathcal M}}{}}
\nc{\oD}{{\overset{\circ}{\mathcal D}}{}}
\nc{\obD}{{\overset{\circ}{\mathbf D}}{}}
\nc{\oA}{{\overset{\circ}{\mathbb A}}{}}
\nc{\op}{{\overset{\bullet}{\mathbf p}}{}}
\nc{\cp}{{\overset{\circ}{\mathbf p}}{}}
\nc{\oU}{{\overset{\bullet}{\mathcal U}}{}}
\nc{\oZ}{{\overset{\circ}{\mathcal Z}}{}}
\nc{\ofZ}{{\overset{\circ}{\mathfrak Z}}{}}
\nc{\oF}{{\overset{\circ}{\fF}}}

\nc{\fa}{{\mathfrak{a}}}
\nc{\fb}{{\mathfrak{b}}}
\nc{\fd}{{\mathfrak{d}}}
\nc{\ff}{{\mathfrak{f}}}
\nc{\fg}{{\mathfrak{g}}}
\nc{\fgl}{{\mathfrak{gl}}}
\nc{\fh}{{\mathfrak{h}}}
\nc{\fj}{{\mathfrak{j}}}
\nc{\fl}{{\mathfrak{l}}}
\nc{\fm}{{\mathfrak{m}}}
\nc{\fn}{{\mathfrak{n}}}
\nc{\fu}{{\mathfrak{u}}}
\nc{\fp}{{\mathfrak{p}}}
\nc{\fr}{{\mathfrak{r}}}
\nc{\fs}{{\mathfrak{s}}}
\nc{\ft}{{\mathfrak{t}}}
\nc{\fz}{{\mathfrak{z}}}
\nc{\fsl}{{\mathfrak{sl}}}
\nc{\hsl}{{\widehat{\mathfrak{sl}}}}
\nc{\hgl}{{\widehat{\mathfrak{gl}}}}
\nc{\hg}{{\widehat{\mathfrak{g}}}}
\nc{\chg}{{\widehat{\mathfrak{g}}}{}^\vee}
\nc{\hn}{{\widehat{\mathfrak{n}}}}
\nc{\chn}{{\widehat{\mathfrak{n}}}{}^\vee}

\nc{\fA}{{\mathfrak{A}}}
\nc{\fB}{{\mathfrak{B}}}
\nc{\fD}{{\mathfrak{D}}}
\nc{\fE}{{\mathfrak{E}}}
\nc{\fF}{{\mathfrak{F}}}
\nc{\fG}{{\mathfrak{G}}}
\nc{\fK}{{\mathfrak{K}}}
\nc{\fL}{{\mathfrak{L}}}
\nc{\fM}{{\mathfrak{M}}}
\nc{\fN}{{\mathfrak{N}}}
\nc{\fP}{{\mathfrak{P}}}
\nc{\fU}{{\mathfrak{U}}}
\nc{\fV}{{\mathfrak{V}}}
\nc{\fZ}{{\mathfrak{Z}}}

\nc{\ba}{{\mathbf{a}}}
\nc{\bb}{{\mathbf{b}}}
\nc{\bc}{{\mathbf{c}}}
\nc{\bd}{{\mathbf{d}}}
\nc{\bbf}{{\mathbf{f}}}
\nc{\be}{{\mathbf{e}}}
\nc{\bi}{{\mathbf{i}}}
\nc{\bj}{{\mathbf{j}}}
\nc{\bm}{{\mathbf{m}}}
\nc{\bn}{{\mathbf{n}}}
\nc{\bo}{{\mathbf{o}}}
\nc{\bp}{{\mathbf{p}}}
\nc{\bq}{{\mathbf{q}}}
\nc{\bu}{{\mathbf{u}}}
\nc{\bv}{{\mathbf{v}}}
\nc{\bx}{{\mathbf{x}}}
\nc{\bs}{{\mathbf{s}}}
\nc{\by}{{\mathbf{y}}}
\nc{\bw}{{\mathbf{w}}}
\nc{\bA}{{\mathbf{A}}}
\nc{\bK}{{\mathbf{K}}}
\nc{\bB}{{\mathbf{B}}}
\nc{\bC}{{\mathbf{C}}}
\nc{\bG}{{\mathbf{G}}}
\nc{\bD}{{\mathbf{D}}}
\nc{\bE}{{\mathbf{E}}}
\nc{\bH}{{\mathbf{H}}}
\nc{\bL}{{\mathbf{L}}}
\nc{\bM}{{\mathbf{M}}}
\nc{\bN}{{\mathbf{N}}}
\nc{\bO}{{\mathbf{O}}}
\nc{\bQ}{{\mathbf{Q}}}
\nc{\bV}{{\mathbf{V}}}
\nc{\bW}{{\mathbf{W}}}
\nc{\bX}{{\mathbf{X}}}
\nc{\bZ}{{\mathbf{Z}}}
\nc{\bS}{{\mathbf{S}}}

\nc{\sA}{{\mathsf{A}}}
\nc{\sB}{{\mathsf{B}}}
\nc{\sC}{{\mathsf{C}}}
\nc{\sD}{{\mathsf{D}}}
\nc{\sF}{{\mathsf{F}}}
\nc{\sG}{{\mathsf{G}}}
\nc{\sH}{{\mathsf{H}}}
\nc{\sK}{{\mathsf{K}}}
\nc{\sM}{{\mathsf{M}}}
\nc{\sO}{{\mathsf{O}}}
\nc{\sW}{{\mathsf{W}}}
\nc{\sQ}{{\mathsf{Q}}}
\nc{\sP}{{\mathsf{P}}}
\nc{\sR}{{\mathsf{R}}}
\nc{\sZ}{{\mathsf{Z}}}
\nc{\sfp}{{\mathsf{p}}}
\nc{\sfq}{{\mathsf{q}}}
\nc{\sfj}{{\mathsf{j}}}
\nc{\sr}{{\mathsf{r}}}
\nc{\bk}{{\mathsf{k}}}
\nc{\sg}{{\mathsf{g}}}
\nc{\sff}{{\mathsf{f}}}
\nc{\sfb}{{\mathsf{b}}}
\nc{\sfc}{{\mathsf{c}}}
\nc{\ssfe}{{\mathsf{e}}}
\nc{\sd}{{\mathsf{d}}}

\nc{\BK}{{\bar{K}}}

\nc{\tA}{{\widetilde{\mathbf{A}}}}
\nc{\tB}{{\widetilde{\mathcal{B}}}}
\nc{\tg}{{\widetilde{\mathfrak{g}}}}
\nc{\tG}{{\widetilde{G}}}
\nc{\TM}{{\widetilde{\mathbb{M}}}{}}
\nc{\tO}{{\widetilde{\mathsf{O}}}{}}
\nc{\tU}{{\widetilde{\mathfrak{U}}}{}}
\nc{\TZ}{{\tilde{Z}}}
\nc{\tx}{{\tilde{x}}}
\nc{\tbv}{{\tilde{\bv}}}
\nc{\tfP}{{\widetilde{\mathfrak{P}}}{}}
\nc{\tz}{{\tilde{\zeta}}}
\nc{\tmu}{{\tilde{\mu}}}

\nc{\urho}{\underline{\rho}}
\nc{\uB}{\underline{B}}
\nc{\uC}{{\underline{\mathbb{C}}}}
\nc{\ui}{\underline{i}}
\nc{\uj}{\underline{j}}
\nc{\ofP}{{\overline{\mathfrak{P}}}}
\nc{\oB}{{\overline{\mathcal{B}}}}
\nc{\og}{{\overline{\mathfrak{g}}}}
\nc{\oI}{{\overline{I}}}

\nc{\eps}{\varepsilon}
\nc{\hrho}{{\hat{\rho}}}

\nc{\one}{{\mathbf{1}}}
\nc{\two}{{\mathbf{t}}}

\nc{\Rep}{{\mathop{\operatorname{\rm Rep}}}}
\nc{\Tot}{{\mathop{\operatorname{\rm Tot}}}}
\nc{\Ker}{{\mathop{\operatorname{\rm Ker}}}}
\nc{\im}{{\mathop{\operatorname{\rm Im}}}}
\nc{\Hilb}{{\mathop{\operatorname{\rm Hilb}}}}
\nc{\End}{{\mathop{\operatorname{\rm End}}}}
\nc{\Ext}{{\mathop{\operatorname{\rm Ext}}}}
\nc{\CHom}{{\mathop{\operatorname{{\mathcal{H}}\it om}}}}
\nc{\CEnd}{{\mathop{\operatorname{{\mathcal{E}}\it nd}}}}
\nc{\GL}{{\mathop{\operatorname{\rm GL}}}}
\nc{\gr}{{\mathop{\operatorname{\rm gr}}}}
\nc{\HN}{{\mathop{\operatorname{\rm HN}}}}
\nc{\Id}{{\mathop{\operatorname{\rm Id}}}}
\nc{\de}{{\mathop{\operatorname{\rm def}}}}
\nc{\length}{{\mathop{\operatorname{\rm length}}}}
\nc{\supp}{{\mathop{\operatorname{\rm supp}}}}

\nc{\Cliff}{{\mathsf{Cliff}}}
\nc{\Fl}{\on{Fl}}
\nc{\Fib}{{\mathsf{Fib}}}
\nc{\Coh}{{\on{Coh}}}
\nc{\QCoh}{{\on{QCoh}}}
\nc{\IndCoh}{{\on{IndCoh}}}
\nc{\FCoh}{{\mathsf{FCoh}}}

\nc{\reg}{{\text{\rm reg}}}

\nc{\cplus}{{\mathbf{C}_+}}
\nc{\cminus}{{\mathbf{C}_-}}
\nc{\cthree}{{\mathbf{C}_\bullet}}
\nc{\Qbar}{{\bar{Q}}}
\nc\Eis{\on{Eis}}
\nc\Eisb{\ol\Eis{}}
\nc\Eisr{\on{Eis}^{rat}{}}
\nc\wh{\widehat}
\nc{\Def}{\on{Def_{\check{\fb}}(E)}}
\nc{\barZ}{\overline{Z}{}}
\nc{\barbarZ}{\overline{\barZ}{}}
\nc{\barpi}{\overline\pi}
\nc{\barbarpi}{\overline\barpi}
\nc{\barpip}{\overline\pi{}^+}
\nc{\barpim}{\overline\pi{}^-}

\nc{\fq}{\mathfrak q}

\nc{\fqb}{\ol{\sfq}{}}
\nc{\fpb}{\ol{\sfp}{}}
\nc{\fpr}{{\sfp^{rat}}{}}
\nc{\fqr}{{\sfq^{rat}}{}}

\nc{\hattimes}{\wh\otimes}

\nc{\bh}{{\bar{h}}}
\nc{\bOmega}{{\overline{\Omega(\check \fn)}}}

\nc{\seq}[1]{\stackrel{#1}{\sim}}

\nc{\cT}{{\check{T}}}
\nc{\cG}{{\check{G}}}
\nc{\cM}{{\check{M}}}
\nc{\cB}{{\check{B}}}

\nc{\ct}{{\check{\mathfrak t}}}
\nc{\cg}{{\check{\fg}}}
\nc{\cb}{{\check{\fb}}}
\nc{\cn}{{\check{\fn}}}

\nc{\cLambda}{{\check\Lambda}}

\nc{\cla}{{\check\lambda}}
\nc{\cmu}{{\check\mu}}
\nc{\cnu}{{\check\nu}}
\nc{\ceta}{{\check\eta}}

\nc{\DefbE}{{\on{Def}_{\cB}(E_\cT)}}

\nc{\imathb}{{\ol{\imath}}}
\nc{\rlr}{\overset{\longrightarrow}{\underset{\longrightarrow}\longleftarrow}}

\nc{\oBun}{\overset{\circ}\Bun}
\nc{\LocSys}{\on{LocSys}}
\nc{\BunBbb}{\ol{\ol{Bun}}_B}
\nc{\BunBr}{\Bun_B^{rat}}
\nc{\BunBrsg}{\Bun_B^{rat,\on{s.g.}}}
\nc{\BunBrp}{\Bun_B^{rat,polar}}
\nc{\BunBrpbg}{\Bun_B^{rat,polar,\on{b.g.}}}
\nc{\BunBrpsg}{\Bun_B^{rat,polar,\on{s.g.}}}
\nc{\BunTrp}{\Bun_T^{rat,polar}}
\nc{\BunTrpbg}{\Bun_T^{rat,polar,\on{b.g.}}}
\nc{\BunTrpsg}{\Bun_T^{rat,polar,\on{s.g.}}}
\nc{\BunNr}{\Bun_N^{rat}}
\nc{\BunNre}{\Bun_N^{enh,rat}}
\nc{\BunTr}{\Bun_T^{rat}}
\nc{\Vect}{\on{Vect}}
\nc{\Whit}{\on{Whit}}
\nc{\CTb}{\ol{\on{CT}}}
\nc{\Ran}{\on{Ran}}
\nc{\CTr}{\on{CT}^{rat}{}}
\nc\jmathr{\jmath^{rat}{}}
\nc{\ux}{\underline{x}}
\nc{\clambda}{{\check\lambda}}
\nc{\calpha}{{\check\alpha}}
\nc{\bind}{{\mathbf{ind}}}
\nc{\bcoind}{{\mathbf{coind}}}
\nc{\bRes}{{\mathbf{Res}}}
\nc{\boblv}{{\mathbf{oblv}}}
\nc{\btriv}{{\mathbf{triv}}}
\nc{\free}{{\mathbf{free}}}
\nc{\ox}{{\overline{x}}}
\nc{\cLa}{\check{\Lambda}}
\nc{\StinftyCat}{\on{DGCat}}
\nc{\inftyCat}{\infty\on{-Cat}}
\nc{\inftygroup}{\infty\on{-Grpd}}
\nc{\Dmod}{\on{D-mod}}
\nc{\CMaps}{{\mathcal Maps}}
\nc{\Maps}{\on{Maps}}
\nc{\affSch}{\on{Sch}^{\on{aff}}}
\nc{\Sch}{\on{Sch}}
\nc{\dr}{{\on{dR}}}
\nc{\oCF}{\overset{\circ}\CF}
\nc{\oCY}{\overset{\circ}\CY}
\nc{\opi}{\overset{\circ}\pi}
\nc{\leqG}{\underset{G}\leq}
\nc{\leqM}{\underset{M}\leq}
\nc{\leqGad}{\underset{G_{ad}}\leq}
\nc{\leqMad}{\underset{M_{ad}}\leq}
\nc{\Tr}{\on{Tr}}
\nc{\Frob}{\on{Frob}}
\nc{\DGCat}{\on{DGCat}}
\nc{\tDGCat}{2\on{-DGCat}}
\nc{\ev}{\on{ev}}
\nc{\mmod}{\on{-}\mathbf{mod}}
\nc{\commod}{\on{-}\mathbf{comod}}
\nc{\sotimes}{\overset{!}\otimes}
\nc{\Shv}{\on{Shv}}
\nc{\bShv}{\mathbf{Shv}}
\nc{\Spc}{\on{Spc}}
\nc{\LS}{\on{LS}}
\nc{\Res}{\on{Res}}
\nc{\bDelta}{{\mathbf{\Delta}}}
\nc{\bMaps}{{\mathbf{Maps}}}
\nc{\cD}{\mathcal D}
\nc{\ocD}{\overset{\circ}\cD}
\nc{\ppart}{(\!(u)\!)}
\nc{\qqart}{[\![u]\!]}
\nc{\oCU}{\overset{\circ}{\CU}}
\nc{\Ind}{\on{Ind}}
\nc{\coInd}{\on{coInd}}
\nc{\binv}{\mathbf{inv}}
\nc{\bcoinv}{\mathbf{coinv}}
\nc{\inclmon}{\mathbf{incl}}
\nc{\bmon}{\mathbf{mon}}
\nc{\AGCat}{\on{AGCat}}
\nc{\tCat}{2\on{-Cat}}
\nc{\Cat}{\on{Cat}}
\nc{\Morita}{\on{Morita}}
\nc{\tAGCat}{2\on{-AGCat}}
\nc{\mmmod}{\on{-}\mathfrak{mod}}
\nc{\oblv}{\on{oblv}}
\nc{\triv}{\on{triv}}
\nc{\inv}{\on{inv}}
\nc{\coinv}{\on{coinv}}
\nc{\ind}{\on{ind}}
\nc{\coind}{\on{coind}}
\renc{\Frob}{{\on{Frob}}}
\nc{\bemb}{\mathbf{emb}}
\nc{\Av}{\on{Av}}
\nc{\uHom}{\underline{\Hom}}
\nc{\sFunct}{\mathsf{Funct}}
\nc{\sfunct}{\mathsf{funct}}
\nc{\bjacq}{\mathbf{jacq}}
\nc{\BMaps}{{\mathbb M}aps}

\begin{document}


\title[Applications of (higher) categorical trace II: Deligne-Lusztig theory]
{Applications of (higher) categorical trace II: \\
Deligne-Lusztig theory}

\dedicatory{To George Lusztig}

\author{D.~Gaitsgory, N.~Rozenblyum and Y.~Varshavsky}

\begin{abstract}
We use the formalism of the (2-category) $\AGCat$, developed in \cite{GRV}, and the operation
of higher categorical trace to (re)derive a number of results in the Deligne-Lusztig theory.
\end{abstract}

\date{\today}

\maketitle

\tableofcontents

\section*{Introduction}

\ssec{What is this paper about?}

Let $G$ be a connected reductive group over the ground field $k$. For the applications in this paper we take $k=\ol\BF_q$,
the algebraic closure of a finite field $\BF_q$.

\sssec{}

In \cite{GRV} we introduced a (symmetric monoidal) $(\infty,2)$-category $\AGCat$, which is a version
of the (usual) $(\infty,2)$-category $\DGCat$ of DG categories (with coefficients $\ol\BQ_\ell$) that is
better compatible with the theory of $\ell$-adic sheaves.

\medskip

Namely, the assignment
$$(X\in \Sch)\mapsto (\ul\Shv(X)\in \AGCat)$$
is \emph{strictly} symmetric monoidal, i.e., the functor
\begin{equation} \label{e:cat Kunneth}
\ul\Shv(X_1)\otimes \ul\Shv(X_2)\to \ul\Shv(X_1\times X_2)
\end{equation}
is an equivalence (by design). This is in contrast with
$$(X\in \Sch)\mapsto (\Shv(X)\in \DGCat),$$
for which the functor
$$\Shv(X_1)\otimes \Shv(X_2)\to \Shv(X_1\times X_2)$$
is only fully faithful (and it never an equivalence unless one of the schemes is a union of points).

\sssec{}

The category $\AGCat$ is a formal device (its definition has no algebro-geometric or categorical content).
Yet, it is useful.

\medskip

In this paper, we will apply $\AGCat$ to develop a theory of
categorical representations of $G$ and to give conceptual proofs of a number of results about character sheaves and
Deligne-Lusztig functors. We establish the independence of the Deligne-Lusztig functor of the parabolic \emph{at the level characters}
(which seems to have been known in full generality).

\sssec{}

Here is how $\AGCat$ is used in this paper:

\medskip

Given an algebraic group $G$, we define a 2-category $G\mmod$ of \emph{categorical representations of $G$},
where by a categorical representation we will mean one on an object of $\AGCat$ (rather than a more naive notion,
on which the action would be on an object of $\DGCat$).

\medskip

A typical example of an object of $G\mmod$ is $\ul\Shv(\CY)$, where $\CY$ is a scheme (or, more generally, an arbitrary prestack)
acted on by $G$.

\ssec{Categorical representations: general theory} \label{ss:gen th}

\sssec{}

The fact that \eqref{e:cat Kunneth} is an equivalence implies that $G\mmod$ is a very well-behaved object.
Here are the three main tools that we will use to develop the theory:

\medskip

\begin{itemize}

\item The fact that for an endomorphism $\phi$ of $G$, the ($2$-categorical) trace $\Tr(\phi,G\mmod)\in \AGCat$
identifies with $\ul\Shv(G/\on{Ad}_\phi(G))$, where $\on{Ad}_\phi(G)$ refers the action of $G$ on itself by
$\phi$-twisted conjugacy. The two main cases of interest are when $\phi=\on{id}$ and $\phi=\Frob$;

\medskip

\item The induction functor
$$\bind_P:M\mmod\to G\mmod,$$
where $G$ is reductive, $P\subset G$ is a parabolic, and $M$ is the Levi quotient of $P$;

\medskip

\item For a pair $(P_1,P_2)$ of \emph{associated parabolics}, a natural transformation
\begin{equation} \label{e:intertw Intro}
\bind_{P_1}\overset{I_{P_1,P_2}}\to \bind_{P_2},
\end{equation}
which we call \emph{the intertwining functor}.

\end{itemize}

\medskip

These three pieces of structure combine to produce some interesting, albeit well-known, constructions
in geometric representation theory.

\sssec{}

For $\phi=\on{id}$, we have
$$\Tr(\on{Id},G\mmod)\simeq \ul\Shv(G/\on{Ad}(G));$$
there is no surprise here.

\medskip

Now, we can consider
$$\Tr(\on{Id},\bind_P):\Tr(\on{Id},M\mmod)\to \Tr(\on{Id},G\mmod).$$

Thus, we obtain a functor
\begin{equation} \label{e:Tr Id ind}
\ul\Shv(M/\on{Ad}(M))\simeq \Tr(\on{Id},M\mmod)\overset{\Tr(\on{Id},\bind_P)}\longrightarrow  \Tr(\on{Id},G\mmod)\simeq \ul\Shv(G/\on{Ad}(G)).
\end{equation}

We show that \eqref{e:Tr Id ind} identifies with the \emph{Grothendieck-Springer} functor
$$\on{GrothSpr}_P:\ul\Shv(M/\on{Ad}(M)) \to \ul\Shv(G/\on{Ad}(G)),$$
i.e., the functor given by pull-push along the diagram
\begin{equation} \label{e:Groth diag intro}
\CD
P/\on{Ad}(P) @>>> G/\on{Ad}(G) \\
@VVV \\
M/\on{Ad}(M).
\endCD
\end{equation}

\sssec{}

For $\phi=\Frob$, by Lang's theorem, we can identify
$$G/\on{Ad}_\Frob(G)\simeq \on{pt}/G(\BF_q),$$
and hence
$$\ul\Shv(G/\on{Ad}_\Frob(G))\simeq \ul\Shv(\on{pt}/G(\BF_q)).$$

\medskip

Note now that the functor
\begin{equation} \label{e:finite grp Intro}
\bi(\Shv(\on{pt}/G(\BF_q)))\to \ul\Shv(\on{pt}/G(\BF_q))
\end{equation}
is an equivalence (see \lemref{l:finite grp is restr}), where $\bi$ is the tautological embeddings $\DGCat\to \DGCat$
(see \cite[Sect. 2.4.1]{GRV}).

\medskip

Hence,
\begin{equation} \label{e:Tr Frob Intro}
\Tr(\Frob,G\mmod)\simeq \ul\Shv(G/\on{Ad}_\Frob(G))\simeq \ul\Rep(G(\BF_q)),
\end{equation}
where
$$\ul\Rep(G(\BF_q)):=\bi(\Rep(G(\BF_q))).$$

\sssec{}

We consider the ``double trace" isomorphism
$$\Tr(\on{id},\Rep(G(\BF_q))\simeq \Tr(\on{id},\Tr(\Frob,G\mmod)),$$
which expands to
\begin{multline} \label{e:double Trace Intro}
\sFunct(G(\BF_q)/\on{Ad}(G(\F_q)),\ol\BQ_\ell)\simeq
\Tr(\on{id},\Rep(G(\BF_q))\simeq
\Tr(\on{id},\Tr(\Frob,G\mmod))\simeq \\
\simeq \Tr(\Frob,\Tr(\on{id},G\mmod))\simeq \Tr(\Frob,\ul\Shv(G/\on{Ad}(G))\simeq
\sFunct((G/\on{Ad}(G))(\BF_q),\ol\BQ_\ell),
\end{multline}
and we show that it equals the tautological identification.

\sssec{}

For a Levi subgroup $M$ defined over $\BF_q$ and a parabolic
$$M\hookrightarrow P\twoheadrightarrow M$$
(but $P$ is not necessarily defined over $\BF_q$), the natural transformation
$I_{P,\Frob(P)}$ gives rise to a natural transformation
$$\bind_P\circ \Frob_M\to \Frob_G\circ \bind_P,$$
and hence to a functor
$$\Tr(\Frob,\bind_P):\Tr(\Frob_M,M\mmod)\to \Tr(\Frob_G,G\mmod).$$

Thus we obtain a functor
\begin{equation} \label{e:Tr Frob ind Intro}
\ul\Rep(M(\BF_q))\simeq \Tr(\Frob_M,M\mmod) \overset{\Tr(\Frob,\bind_P)}\longrightarrow
\Tr(\Frob_G,G\mmod) \simeq \ul\Rep(G(\BF_q)).
\end{equation}

We identify \eqref{e:Tr Frob ind Intro} with the functor, induced by the Deligne-Lusztig functor introduced in \cite{DL,Lu1} (see also \cite{DM}) 
$$\on{DL}_P:\Rep(M(\BF_q))\to \Rep(G(\BF_q)).$$

\sssec{}

However, this is as far as we could take the theory with the above tools.

\ssec{Categorical representations: the spectrally finite subcategory} \label{ss:spec fin Intro}

Now, the theory gets richer once we replace the entire $G\mmod$
by its full subcategory that we denote by
$$G\mmod^{\on{spec.fin}}\subset G\mmod;$$
we refer the reader to \secref{ss:restr Intro} for the definition and to
\secref{ss:restr} for complete details.

\sssec{}

We show that the (full) subcategory
$$\Tr(\on{id},G\mmod^{\on{spec.fin}})\subset \Tr(\on{id},G\mmod)$$
identifies with the subcategory of \emph{character sheaves}
$$\ul\Shv^{\on{ch}}(G/\on{Ad}_\Frob(G))\subset \ul\Shv(G/\on{Ad}_\Frob(G)).$$

This essentially a paraphrase of a result of D.~Ben-Zvi and D.~Nadler.

\sssec{}

We also show that the (full) subcategory
$$\Tr(\Frob,G\mmod^{\on{spec.fin}})\subset \Tr(\Frob,G\mmod)$$
is \emph{entire category} $\ul\Rep(G(\BF_q))$; this a paraphrase of a result of A.~Eteve.

\sssec{}

We use this to show that the ``double trace" isomorphism \eqref{e:double Trace Intro} restricts to an isomorphism
\begin{multline*}
\sFunct(G(\BF_q)/\on{Ad}(G(\F_q)),\ol\BQ_\ell)\simeq \Tr(\on{id},\Rep(G(\BF_q))\simeq
\Tr(\on{id},\Tr(\Frob,G\mmod^{\on{spec.fin}}))\simeq \\
\simeq \Tr(\Frob,\Tr(\on{id},G\mmod^{\on{spec.fin}})),
\end{multline*}
leading to the conclusion that
\begin{equation} \label{e:Tr isom Intro}
\Tr(\Frob,\Shv^{\on{ch}}(G/\on{Ad}(G)))\simeq \sFunct((G/\on{Ad}(G))(\BF_q),\ol\BQ_\ell).
\end{equation}

Note that in the above formula, we are considering the object $\Shv^{\on{ch}}(G/\on{Ad}(G))\in \DGCat$ rather than
$\ul\Shv^{\on{ch}}(G/\on{Ad}(G))\in \AGCat$. One can view \eqref{e:Tr isom Intro} as a counterpart for character
sheaves of the Trace Conjecture from \cite{AGKRRV1}.

\begin{rem}

The passage
$$\ul\Shv^{\on{ch}}(G/\on{Ad}(G))\rightsquigarrow \Shv^{\on{ch}}(G/\on{Ad}(G))$$ above is part of the general
phenomenon: as we shall see in \thmref{t:restr main}, the passage
$$G\mmod \rightsquigarrow G\mmod^{\on{spec.fin}}$$
gets rid of the extra ``algebro-geometric" directions inherent to objects of $\AGCat$, see \secref{sss:restr Intro}.

\end{rem}

\sssec{}

The functor $\bind_P$ restricts to a functor
$$\bind^{\on{spec.fin}}_P:M\mmod^{\on{spec.fin}}\to G\mmod^{\on{spec.fin}},$$
and the natural transformation $I_{P_1,P_2}$ restricts to a natural transformation
restricts to a natural transformation
\begin{equation} \label{e:intertw restr Intro}
\bind^{\on{spec.fin}}_{P_1}\overset{I^{\on{spec.fin}}_{P_1,P_2}}\to \bind^{\on{spec.fin}}_{P_2}.
\end{equation}

Here is, however, the fact that really distinguishes the spectrally finite situation from the
general one: the natural transformation \eqref{e:intertw restr Intro} is an isomorphism.

\medskip

This implies, among the rest, that $I^{\on{spec.fin}}_{P_1,P_2}$ induces an isomorphism
between the Grothendieck-Springer functors
$$\on{GrothSpr}^{\on{ch}}_{P_1}:\ul\Shv^{\on{ch}}(M/\on{Ad}(M)) \to \ul\Shv^{\on{ch}}(G/\on{Ad}(G))$$
and
$$\on{GrothSpr}^{\on{ch}}_{P_2}:\ul\Shv^{\on{ch}}(M/\on{Ad}(M)) \to \ul\Shv^{\on{ch}}(G/\on{Ad}(G)),$$
to be denoted
\begin{equation} \label{e:funct eq Intro}
\on{GrothSpr}^{\on{ch}}_{P_1} \overset{i_{P_1,P_2}}\simeq \on{GrothSpr}^{\on{ch}}_{P_2}.
\end{equation}

We will refer to \eqref{e:funct eq Intro} as the \emph{functional equation} for the  Grothendieck-Springer functors.
We perceive it as an analog of the functional equation for geometric Eisenstein series in the geometric
Langlands theory, see \cite{BG}.

\begin{rem}

As was shown by S.~Gunningham, an analog of the functional equation \eqref{e:funct eq Intro} would be \emph{false}
if we did not restrict ourselves to character sheaves.

\end{rem}

\sssec{}

A naive attempt to extend the functional equation to the Deligne-Lustig functor fails (for some obvious reasons,
see Remark \ref{r:DL parab}). And indeed, it cannot hold, since it is known that $\on{DL}_P$, viewed as
a functor, \emph{does} depend on the choice of $P$.

\medskip

However, it was conjectured in \cite{Lu1} and established in many cases in \cite{DL,Lus,BM} (see a summary in \cite[0.1]{Lu3} or \cite[Chapter 9]{DM}) that the map
\begin{equation} \label{e:ch DL}
\on{ch}(\on{DL}_P): \sFunct((M/\on{Ad}(M))(\BF_q))\to \sFunct((G/\on{Ad}(G))(\BF_q),\ol\BQ_\ell)
\end{equation}
is \emph{independent} of the choice of $P$.

\medskip

This independence is one of the two main new results of this paper; it appears as \thmref{t:indep}. We outline
the proof in the next subsection.

\ssec{Independence of Deligne-Lusztig characters of the parabolic} \label{ss:indep Intro}

\sssec{}

The first step in the proof is the double trace isomorphism, but applied to \emph{to the functor} $\bind^{\on{spec.fin}}_P$.

\medskip

First, using the intertwining functor $I_{P,\Frob(P)}$ we construct an isomorphism
$$\Frob_{G/\on{Ad}(G)}\circ \on{GrothSpr}^{\on{ch}}_P\simeq \on{GrothSpr}^{\on{ch}}_P\circ \Frob_{M/\on{Ad}(M)},$$
to be denoted $\beta_P$. This identification induces a map
$$\Tr(\Frob_{M/\on{Ad}(M)},\ul\Shv^{\on{ch}}(M/\on{Ad}(M)))\to \Tr(\Frob_{G/\on{Ad}(G)},\ul\Shv^{\on{ch}}(G/\on{Ad}(G))),$$
to be denoted $\Tr(\beta_P,\on{GrothSpr}_P)$.

\medskip

The double trace isomorphism says that the map \eqref{e:ch DL} equals the map
\begin{multline}  \label{e:Frob Groth Intro}
\sFunct((M/\on{Ad}(M))(\BF_q)) \overset{\text{\eqref{e:Tr isom Intro}}}\simeq \Tr(\Frob_{M/\on{Ad}(M)},\ul\Shv^{\on{ch}}(M/\on{Ad}(M)))
\overset{\Tr(\beta_P,\on{GrothSpr}^{\on{ch}}_P)}\longrightarrow \\
\simeq  \Tr(\Frob_{G/\on{Ad}(G)},\ul\Shv^{\on{ch}}(G/\on{Ad}(G)))  \overset{\text{\eqref{e:Tr isom Intro}}}\simeq \sFunct((G/\on{Ad}(G))(\BF_q)).
\end{multline}

This assertion is the second main new result of this paper; we state it as \corref{c:ind and Frob restr}.

\begin{rem}

The equality of \eqref{e:ch DL} and \eqref{e:Frob Groth Intro} for $q\gg 0$ is the main result of \cite{Lu3}; it is also conjectured in 
{\it loc.cit.} that the assumption that $q\gg 0$ can be removed.

\medskip

A particular case of the equality of \eqref{e:ch DL} and \eqref{e:Frob Groth Intro} is the following statement established in
\cite{Lus}:

\medskip

Let $M=T$ be a torus, and let $\chi$ be a 1-dimensional character sheaf on $T$ equipped with a Weil structure.
On the one hand, we can form the Grothendieck-Springer sheaf $\on{GrothSpr}_B(\chi)$. We equip it with a Weil
structure, by first defining on the regular semi-simple locus of $G$ (using the $W$-action and the Weil structure on $\chi$),
and then applying the Goresky-MacPherson extension.

\medskip

On the one hand, we consider the function
$$\sfunct(\on{GrothSpr}_B(\chi))\in \sFunct((G/\on{Ad}(G))(\BF_q)),$$
obtained by taking pointwise traces of Frobenius on $\on{GrothSpr}_B(\chi)$ with respect
to the above Weil structure.

\medskip

On the other hand, we consider the function
$$\on{ch}(\on{DL}_B(\chi^0))\in \sFunct(G(\BF_q)/\on{Ad}(G(\BF_q)),\ol\BQ_\ell).$$
where $\chi^0$ is the character of $T(\BF_q)$, obtained by taking pointwise traces of Frobenius on $\chi$
(i.e., $\chi^0=\sfunct(\chi)$).

\medskip

Then the equality of \eqref{e:ch DL} and \eqref{e:Frob Groth Intro} in particular implies that
\begin{equation} \label{e:char Groth Intro}
\sfunct((\on{GrothSpr}_B(\chi))=\on{ch}(\on{DL}_B(\chi^0))
\end{equation}
as elements of
$$\sFunct((G/\on{Ad}(G))(\BF_q),\ol\BQ_\ell)=\sFunct(G(\BF_q)/\on{Ad}(G(\BF_q)),\ol\BQ_\ell).$$

Note that the usual proof of \eqref{e:char Groth Intro} (see, e.g., \cite[Theorem 2.3.1]{Laum}, which follows \cite{Lus}) uses
some non-trivial computations involving Jordan decomposition.
By contrast, our proof is a diagram chase (once we accept the formula for local terms in the Grothendieck-Lefschetz fixed point
theorem, as generalized in \cite{GaVa}).

\end{rem}

\sssec{}

Thus, in order to show that the map \eqref{e:ch DL} is independent of $P$, it suffices to show that the map \eqref{e:Frob Groth Intro} is.

\medskip

Let $P_1$ and $P_2$ be a pair of parabolics containing $M$. We already know that thanks to the functional equation
\eqref{e:funct eq Intro}, the functors $\on{GrothSpr}^{\on{ch}}_{P_1}$ and $\on{GrothSpr}^{\on{ch}}_{P_2}$ are
isomorphic.

\medskip

Therefore, in order to show that the maps \eqref{e:ch DL} for $P_1$ and $P_2$ are equal, it suffices to show that
$i_{P_1,P_2}$ intertwines the isomorphisms $\beta_{P_1}$ and $\beta_{P_2}$, i.e., that the diagram
$$
\CD
\Frob_{G/\on{Ad}(G)}\circ \on{GrothSpr}^{\on{ch}}_{P_1} @>{\beta_{P_1}}>> \on{GrothSpr}^{\on{ch}}_{P_1}\circ \Frob_{M/\on{Ad}(M)} \\
@V{i_{P_1,P_2}}VV @VV{i_{P_1,P_2}}V \\
\Frob_{G/\on{Ad}(G)}\circ \on{GrothSpr}^{\on{ch}}_{P_2} @>{\beta_{P_2}}>> \on{GrothSpr}^{\on{ch}}_{P_2}\circ \Frob_{M/\on{Ad}(M)}
\endCD
$$
commutes.

\medskip \label{sss:trans Intro}

Unwinding the construction of $\beta_P$, we obtain that we have prove the following statement, which does not involve the Frobenius map:

\medskip

For a triple of parabolics $P_1,P_2,P_3$ containing $M$, we have
\begin{equation} \label{e:trans Intro}
i_{P_2,P_3}\circ i_{P_1,P_2}\simeq i_{P_1,P_3}
\end{equation}
as natural transformations
$$\on{GrothSpr}^{\on{ch}}_{P_1} \to \on{GrothSpr}^{\on{ch}}_{P_3}.$$

\medskip

The fact that \eqref{e:trans Intro} holds is our \thmref{t:intertwiner trans}. Let us comment on how it is proved.

\sssec{}

First, there is the easy case, namely, when the relative positions of $(P_1,P_2)$ and $(P_2,P_3)$ add up.
In this case, we have
\begin{equation} \label{e:add up Intro}
I_{P_2,P_3}\circ I_{P_1,P_2}\simeq I_{P_1,P_3}
\end{equation}
as natural transformations
$$\bind_{P_1}\to \bind_{P_3},$$
and \eqref{e:trans Intro} is obtained by taking $\Tr(\on{Id},-)$ on \eqref{e:add up Intro}.

\sssec{}

Thus, by induction on the distance between $P_1$ and $P_2$, we reduce the verification of \eqref{e:trans Intro} to the case when
$P_3=P_1$. I.e., we need to show that
\begin{equation} \label{e:involut Intro}
i_{P_2,P_1}\circ i_{P_1,P_2}\simeq \on{id}_{\on{GrothSpr}^{\on{ch}}_{P_1}}.
\end{equation}

Note that it is \emph{not at all} true that
\begin{equation} \label{e:square Intertw}
I_{P_2,P_1}\circ I_{P_1,P_2}
\end{equation}
is the identity endomorphism of $\bind_{P_1}$. In fact, this endomorphism is closely related to the Serre functor,
see \secref{sss:Serre intro}.

\medskip

Yet, it so happens that the non-triviality of \eqref{e:square Intertw} gets swallowed by the operation $\Tr(\on{Id},-)$.
We prove this in \thmref{t:involut}.

\sssec{}

We find it curious that manipulations involving a 3-category (in our case $\tAGCat$) can be used to prove a non-trivial
equality of ($\ell$-adic) numbers (which is what the independence of \eqref{e:ch DL} of the parabolic is). The proof essentially boils
down to checking commutativity of huge diagrams, but after packaging them so as to make it humanly possible.

\medskip

At its most complex, one needs to show that two certain 3-morphisms (in a 3-category) are equal. These 3-morphisms arise
by taking units and counits of adjunctions of some 1-morphisms (these are 2-morphisms), then taking \emph{their}
adjoints, and considering 3-morphisms that are units and counits of these adjunctions. An example of this
is \corref{c:2-dual}, which stands behind the equality of \eqref{e:ch DL} and
\eqref{e:Frob Groth Intro}.

\medskip

A related strategy, but one that occurs at one categorical level lower (i.e., one works with the 2-category $\DGCat$
rather than with the 3-category $\tAGCat$) was realized in \cite{KP} for the proof of the Riemann-Roch theorem.

\ssec{What else is done in this paper?} \label{ss:else}

We will now explain how the 3-categorical perspective allows one to reprove a number
of results about character sheaves.

\sssec{} \label{sss:ind 2-adj Intro}

Consider again the functor
$$\bind_P:M\mmod\to G\mmod.$$

It admits a right adjoint, called the Jacquet functor
$$M\mmod\leftarrow G\mmod:\bjacq_P.$$

Moreover, this adjunction has the following property: the unit and counit natural transformations
\begin{equation} \label{e:unit and counit Intro}
\on{unit}:\on{Id}_{M\mmod}\to \bjacq_P\circ \bind_P \text{ and } \on{counit}:\bind_P\circ \bjacq_P\to \on{Id}_{G\mmod}
\end{equation}
themselves admit right adjoints.

\medskip

Passing to these right adjoints, we obtain the natural transformations
$$\on{unit}^R:\bjacq_P\circ \bind_P\to \on{Id}_{M\mmod} \text{ and } \on{counit}^R:\on{Id}_{G\mmod}\to \bind_P\circ \bjacq_P,$$
which make $\bjacq_P$ into a \emph{left} adjoint of $\bind_P$.

\medskip

Thus, we obtain that the $(\bind_P,\bjacq_P)$-adjunction is \emph{ambidexterous}.

\sssec{}

However, an even more interesting thing happens when we pass to the spec.fin subcategries. In this case, the natural transformations
\begin{equation} \label{e:unit R Intro}
\on{unit}^R:\bjacq^{\on{spec.fin}}_P\circ \bind^{\on{spec.fin}}_P\to \on{Id}_{M\mmod^{\on{spec.fin}}}
\end{equation}
and
\begin{equation} \label{e:counit R Intro}
\on{counit}^R:\on{Id}_{G\mmod^{\on{spec.fin}}}\to \bind^{\on{spec.fin}}_P\circ \bjacq^{\on{spec.fin}}_P
\end{equation}
\emph{themselves} admit right adjoints.

\sssec{}

The existence of these right adjoints has a consequence for the Grothendieck-Springer functor.
Let
$$\fr_P:\ul\Shv(G/\on{Ad}(G))\to \ul\Shv(M/\on{Ad}(M))$$
be the right adjoint of $\on{GrothSpr}_P$. Explicitly, $\fr_P$ is obtained by pull-push along the diagram transpose to \eqref{e:Groth diag intro}.

\medskip

Denote
$$\fr_P^{\on{ch}}:=\fr_P|_{\ul\Shv^{\on{ch}}(G/\on{Ad}(G))}.$$

\medskip

The fact that the natural transformations
\eqref{e:unit R Intro} and \eqref{e:counit R Intro} admit right adjoints implies that
$$\Tr(\on{id},\bjacq^{\on{spec.fin}}_P)\simeq \fr_P^{\on{ch}}$$
is a \emph{left} adjoint of
$$\Tr(\on{id},\bind^{\on{spec.fin}}_P)=\on{GrothSpr}_P^{\on{ch}}.$$

I.e., the $(\on{GrothSpr}^{\on{ch}}_P,\fr^{\on{ch}}_P)$-adjunction is also ambidexterous.

\medskip

Using the fact that the functor $\on{GrothSpr}^{\on{ch}}_P$ commutes with Verdier duality, the above implies that
the functor $\fr^{\on{ch}}_P$ also commutes with Verdier duality. The latter fact was, however, known thanks to
\cite{CY,GIT}.

\sssec{}

Let us return to the natural transformations \eqref{e:unit R Intro} and \eqref{e:counit R Intro}. By passing to \emph{their}
right adjoints, we obtain \emph{another} adjunction
\begin{equation} \label{e:new adj intro}
\bind^{\on{spec.fin}}_P \leftrightarrow \bjacq^{\on{spec.fin}}_P.
\end{equation}

This adjunction differs from the original $(\bind^{\on{spec.fin}}_P,\bjacq^{\on{spec.fin}}_P)$-adjunction by
an automorphism of $\bind^{\on{spec.fin}}_P$, which we call the \emph{Serre functor}. Effectively, the Serre
functor is an endomorphism of
$$\ul\Shv(G/N(P))^{\on{spec.fin}}\in (G\times M)\mmod;$$
we denote it by $\on{Se}_{G/N(P)}$.

\sssec{} \label{sss:Serre intro}

Our \thmref{t:2-Cat Serre} identifies $\on{Se}_{G/N(P)}$ explicitly. Namely, it says that
\begin{equation} \label{e:Serre Intro}
\on{Se}_{G/N(P)} \simeq I^{\on{spec.fin}}_{P^-,P}\circ I^{\on{spec.fin}}_{P,P^-},
\end{equation}
where $P^-$ is the opposite parabolic.

\sssec{}

We use \eqref{e:Serre Intro} to give an alternative prove of the fact (also established in \cite{CY,GIT}) that
the Harish-Chandra transform
$$\Shv(G/\on{Ad}(G)) \overset{\on{HC}_P}\to \Shv((N(P)\backslash G/N(P))/\on{Ad}(M)),$$
followed by the intertwining functor
$$\on{Id}\otimes I^!_{P,P^-}: \Shv((N(P)\backslash G/N(P))/\on{Ad}(M))\to  \Shv((N(P)\backslash G/N(P^-))/\on{Ad}(M))$$
also commutes with Verdier duality, when restricted to
$$\Shv^{\on{ch}}(G/\on{Ad}(G))\subset \Shv(G/\on{Ad}(G)),$$
see \corref{c:HC is ambi bis}.

\medskip

In the above formulas, $\on{Id}\otimes\on{HC}_P$ is given by !-pull and *-push along the diagram
$$
\CD
G/\on{Ad}(P) @>>> (N(P)\backslash G/N(P))/\on{Ad}(M) \\
@VVV \\
G/\on{Ad}(G),
\endCD
$$
and $I^!_{P,P^-}$ is given by *-pull and !-push along the diagram
$$
\CD
(N(P)\backslash G)/\on{Ad}(M)  @>>> (N(P)\backslash G/N(P^-))/\on{Ad}(M)  \\
@VVV \\
(N(P)\backslash G/N(P))/\on{Ad}(M).
\endCD
$$

\ssec{Monodromicity and restricted categorical representations} \label{ss:restr Intro}

As was mentioned above, the more interesting results of this paper require that we work with the
subcategory
$$G\mmod^{\on{spec.fin}}\subset G\mmod.$$

We will now explain its definition.

\sssec{}

Let us first take $G=T$ to be the torus. We shall say that an an object $\ul\bC\in T\mmod$
is \emph{monodromic} if it is generated by its $(T,\chi)$-invariants, where $\chi$ runs over
all possible character sheaves on $T$.

\medskip

We let
$$T\mmod^{\on{spec.fin}}\subset T\mmod$$
consist of monodromic objects.

\sssec{}

For a general reductive $G$, we let $G\mmod^{\on{spec.fin}}$ to be the full subcategory consisting of objects
for which
$$\binv_N(\ul\bC)\in T\mmod$$
is monodromic. In the above formula, $\binv_N$ is the functor of invariants with respect to $N\subset B$.

\sssec{Example}

Consder
$$\ul\Shv(G/B)\in G\mmod.$$

This object belongs to $G\mmod^{\on{spec.fin}}$.

\medskip

Consider now
$$\ul\Shv(G/N)\in G\mmod.$$

This object does \emph{not} belong to $G\mmod^{\on{spec.fin}}$.

\sssec{}

Of course, one needs to prove something in order to make the above definition useful. The key fact is that the functor
$$\binv_N:G\mmod\to \AGCat,$$
which naturally upgrades to a functor
\begin{equation} \label{e:Hecke Intro}
\binv_N^{\on{enh}}:G\mmod\to \ul\CH\mmod(\AGCat), \quad \ul\CH:=\ul\Shv(N\backslash G/N),
\end{equation}
with the latter functor being an equivalence, according to a theorem of \cite{BGO} (reproved in the present paper as \thmref{t:G mod via Hecke}).

\medskip

Thus, the knowledge of $\binv_N(\ul\bC)$ (equipped with an action of $\ul\CH$) fully recovers $\ul\bC$.

\sssec{}

Let
$$\ul\CH^{\on{mon}}\subset \ul\CH$$
be the full subcategory that consists of $T$-monodromic objects (with respect to the action of $T$ on $N\backslash G/N$ from
the left/right, or equivalently, both sides). The above embedding admits a right adjoint. Moreover, the object $\ul\CH^{\on{mon}}\in \AGCat$
carries a unique structure associative algebra for which the above right adjoint
$$\ul\CH\to \ul\CH^{\on{mon}}$$
is a homomorphism.

\medskip

We can identify
$$G\mmod^{\on{spec.fin}}\simeq \ul\CH^{\on{mon}}\mmod(\AGCat).$$

\sssec{} \label{sss:restr Intro}

One of the key points of the spec.fin theory is that it effectively takes place in the usual world of DG categories,
rather than $\AGCat$:

\medskip

Let $$\bi:\DGCat\to \AGCat$$ be the tautological embedding (see \cite[Sect. 2.4.1]{GRV}). The point is
that $\ul\CH^{\on{mon}}$ (unlike $\ul\CH$!) lies in the essential image of this functor, i.e.,
$$\ul\CH^{\on{mon}}=\bi(\CH^{\on{mon}}),\quad \CH^{\on{mon}}:=\Shv(N\backslash G/N)^{\on{mon}}.$$

\medskip

Set
$$G\mmod^{\on{restr}}:=\CH^{\on{mon}}\mod(\DGCat)\in \tDGCat.$$

We obtain that
\begin{equation} \label{e:ext DGCat AGCat Intro}
G\mmod^{\on{spec.fin}}\simeq \AGCat\underset{\DGCat}\otimes G\mmod^{\on{restr}}.
\end{equation}

I.e., $G\mmod^{\on{spec.fin}}$ is the extension of scalars of a 2-category that is a module over
the usual $\DGCat$.

\medskip

This implies among other things that
\begin{equation} \label{e:ch shv restr}
\ul\Shv^{\on{ch}}(G/\on{Ad}(G))\simeq \bi(\Shv^{\on{ch}}(G/\on{Ad}(G))).
\end{equation}

\sssec{}

The equivalence \eqref{e:ext DGCat AGCat Intro} means that the entire restricted/spectrally finite theory
can be rewritten without ever mentioning $\AGCat$: one can just work with the 2-category
$$\CH^{\on{mon}}\mod(\DGCat),$$
and one can transcribe all the proofs in its terms.

\medskip

The advantage of using $\AGCat$ is that it gives a clearer geometric picture, which suggests
natural proofs of many results.

\medskip

A vague analogy is:
to think about representations of a p-adic group as such vs as a collection of modules for
Hecke algebras.

\ssec{The hidden motivation for writing this paper}

One of the key motivations for writing this paper is that it is supposed to be a
prelude to the case
when instead of the usual reductive group $G$, one considers its loop group $\fL(G)$.

\sssec{}

One can define the object
$$\fL(G)\mmod\in \tDGCat$$
in a way completely parallel to $G\mmod$.

\medskip

We have a well-defined object
$$\ul\Shv(\fL(G)/\on{Ad}(\fL(G))\in \AGCat$$
and an identification
$$\Tr(\on{Id},\fL(G)\mmod)\simeq \ul\Shv(\fL(G)/\on{Ad}(\fL(G)).$$

\sssec{}

Note, however, that Lang's theorem fails for $\fL(G)$. The quotient
$$\fL(F)/\on{Ad}_\Frob(\fL(G))=:\on{Isoc}_G$$
is the stack of isocrystals, extensively studied recently by X.~Zhu in \cite{Zhu}.
One can think of $\on{Isoc}_G$ as an algebro-geometric incarnation of the Fargues-Scholze $\Bun_G(\on{FFCurve})$.

\medskip

As in \eqref{e:Tr Frob Intro}, we have
$$\Tr(\Frob,\fL(G)\mmod)\simeq \ul\Shv(\fL(F)/\on{Ad}_\Frob(\fL(G)))=\ul\Shv(\on{Isoc}_G).$$

\sssec{}

Here are some basic questions that one wants to ask in the loop case:

\medskip

\begin{enumerate}

\item Is it true that the loop analog of \eqref{e:finite grp Intro}, i.e., the functor
$$\bi(\Shv(\on{Isoc}_G))\to \ul\Shv(\on{Isoc}_G),$$
is an equivalence?

\medskip

\item What is the subcategory $\fL(G)\mmod^{\on{spec.fin}}\subset \fL(G)\mmod$, so that
$$\fL(G)\mmod^{\on{spec.fin}}\simeq \AGCat\underset{\DGCat}\otimes \fL(G)\mmod^{\on{restr}}$$
for $\fL(G)\mmod^{\on{restr}}\in \tDGCat$?

\medskip

\item Is it true that an analog of Eteve's theorem holds, i.e., the (fully faithful) functor
$$\Tr(\Frob,\fL(G)\mmod^{\on{spec.fin}})\to \Tr(\Frob,\fL(G)\mmod)\simeq \ul\Shv(\on{Isoc}_G)$$
is an equivalence?

\medskip

\item What is the definition of
$$\ul\Shv^{\on{ch}}(\fL(G)/\on{Ad}(\fL(G))\subset \ul\Shv(\fL(G)/\on{Ad}(\fL(G))?$$

\medskip

\item Is it true that
$$\Tr(\on{Id}, \fL(G)\mmod^{\on{spec.fin}})=\ul\Shv^{\on{ch}}(\fL(G)/\on{Ad}(\fL(G))$$
as subcategories of
$$\Tr(\on{Id}, \fL(G)\mmod)\simeq \ul\Shv(\fL(G)/\on{Ad}(\fL(G)),$$
respectively.

\end{enumerate}

\sssec{}

Here are some tentative answers:

\medskip

First, it seems that Question (1) can be settled affirmatively using the results of \cite{Zhu}.

\medskip

Question (2) is far more complicated than in the finite-dimensional case, because we do not have
anything resembling the equivalence of \eqref{e:Hecke Intro}.

\medskip

Rather, a promising approach seems to characterize $\fL(G)\mmod^{\on{spec.fin}}$ using a (geometric analog)
of the Fintzen-Yu classification (or its refinement, recently found by \cite{DVY}).

\medskip

As to Question (3), it does seem plausible that the method of proof of Eteve's theorem (see \thmref{t:Frob restr})
is applicable and gives the desired result.

\medskip

Finally, for Questions (4) and (5) the best we can do is cheat and \emph{define}
$$\ul\Shv^{\on{ch}}(\fL(G)/\on{Ad}(\fL(G)):=\Tr(\on{Id}, \fL(G)\mmod^{\on{spec.fin}}).$$

\ssec{Structure of the paper}

We now briefly describe the contents of this paper section-by-section.

\sssec{}

In \secref{s:Trace} we explain how the formalism of trace plays out in the 2- and 3-categorical settings:

\medskip

In the case of 1-categories, given a symmetric monoidal category $\bO$, a dualizable object $\bo\in \bO$ and an
endomorphism $F$ of $\bO$, one attaches to it an endomorphism
$\Tr(F,\bo)$ of the unit object $\one_\bO$.

\medskip

This gets more interesting for 2-categories: given two pairs $(\bo_1,F_1)$ and $(\bo_2,F_2)$ as above and a 1-morphism
$t:\bo_1\to \bo_2$ that \emph{weakly intertwines $F_1$ and $F_2$}, i.e, we have a natural transformation
$$t\circ F_1\overset{\alpha}\to F_2\circ t,$$
we obtain a map
$$\Tr(\alpha,t):\Tr(F_1,\bo_1)\to \Tr(F_2,\bo_2),$$
\emph{provided that $t$ is adjointable}\footnote{In this paper ``adjointable" means ``admits a right adjoint".}, a notion that makes
sense for a 1-morphism in a 2-category.

\medskip

And this gets even richer (and more mind-twisting) for 3-categories.

\medskip

This section summarizes the results about trace manipulations that we will apply in later sections in
representation-theoretic situations.

\sssec{}

In \secref{s:Cat Rep} we introduce the notion of categorical representation of an algebraic group. These
will form an object, denoted $G\mmod$ in the ambient 3-category that we denote $\tAGCat$
(morally, 2-categories, tensored over $\AGCat$).

\medskip

We study basic concepts such as dualizability and adjointability of 1-morphisms.

\medskip

One of the observations that will play a key role in this paper is that the 1-morphism
$$\bind_P:M\mmod\to G\mmod$$
is 2-adjointable, see \secref{sss:ind 2-adj Intro}.

\sssec{}

In \secref{s:Tr Cat Rep} we perform the basic trace calculations in $G\mmod$ and for
1-morphisms
$$G_1\mmod\to G_2\mmod.$$

In particular, we prove the calculations described in \secref{ss:gen th}.

\sssec{}

In \secref{s:mon} we introduce the notion of monodromic action an algebraic group
on an object of $\AGCat$. The actual case of interest is when the group in question
is a torus, but we do this for any algebraic group, just because this is a natural generality.

\medskip

We start by introducing the subcategory
$$\Shv(G)^{\on{mon}}\subset \Shv(G),$$
which is the subcategory generated by 1-dimensional character sheaves. Having this definition,
one can define the monodromic subcategory
$$\Shv(\CY)^{\on{mon}}\subset \Shv(\CY)$$
for any prestack $\CY$ acted on by $G$. This, in turn, leads to the definition of
$\ul\Shv(G)^{\on{mon}}$
which is a \emph{Hopf} algebra object in $\AGCat$. That said, the object $\ul\Shv(G)^{\on{mon}}$
has the property that the canonical map
$$\bi(\Shv(G)^{\on{mon}})\to \ul\Shv(G)^{\on{mon}}$$
is an isomorphism.

\medskip

By definition,
\begin{equation} \label{e:mon Intro}
G\mmod^{\on{mon}}\subset G\mmod
\end{equation}
is a full subcategory, consisting of objects, on which the co-action of $\ul\Shv(G)$ factors through
$\ul\Shv(G)^{\on{mon}}$. We show that the embedding \eqref{e:mon Intro} has very favorable
properties (for example, it is 2-adjointable).

\medskip

We also introduce a related notion:
$$\Shv(G)^{\on{q-mon}}\subset \Shv(G),$$
and its descendants
$$\Shv(\CY)^{\on{q-mon}}\subset \Shv(\CY),\,\, \ul\Shv(G)^{\on{q-mon}}\subset \ul\Shv(G),\,\, G\mmod^{\on{q-mon}}\subset G\mmod$$
(when $G$ is a torus, there is no difference between the two).

\medskip

The reason for discussing the quasi-monodromic versions is the following. They provide an answer to the following
question: do there exists objects in $G\mmod$, whose underlying object of $\AGCat$ lies in the essential image
of $\bi:\DGCat\to \AGCat$? The answer is that such objects exactly correspond to quasi-monodromic actions.

\sssec{}

In \secref{s:restr}, we introduce the subcategories
$$G\mmod^{\on{restr}}\subset G\mmod^{\on{spec.fin}}\subset G\mmod$$
and establish the properties described in Sects. \ref{ss:spec fin Intro} and \ref{ss:restr Intro} above.

\sssec{}

In \secref{s:char shv} we review the theory of character sheaves, and then recast from the perspective
of $G\mmod$.

\sssec{}

In \secref{s:trans} we prove the transitivity property of the functional equation, mentioned in
\secref{sss:trans Intro}.

\sssec{}

Finally, in \secref{s:DL} we assemble all the ingredients and prove the independence of
Deligne-Lusztig characters of the parabolic, as outlined in \secref{ss:indep Intro}.

\sssec{}

In \secref{s:shv on gr} we supply the proofs for some technical statements from \secref{s:mon}.

\ssec{Conventions and notation}

\sssec{}

Our conventions and notations follow those adopted in \cite{GRV}.

\sssec{Category theory} \label{sss:cat}

We will be working in the context of higher category theory. When we say ``category"
we always mean an $\infty$-category.

\medskip

We let $\Spc$ denote the category of \emph{spaces}, a.k.a. \emph{groupoids}.

\medskip

Given a category $\bC$, and $\bc_1,\bc_2\in \bC$, we let
$$\Maps_\bC(\bc_1,\bc_2)\in \Spc$$
denote the corresponding space of morphisms.

\medskip

For a category $\bC$, we let $\bC^{\on{grpd}}$ denote the space of its objects
(i.e., the groupoid obtained from $\bC$ by discarding non-invertible morphisms).

\sssec{}

We will work also with $(\infty,2)$ and $(\infty,3)$-categories. We adopt the inductive definition:

\medskip

Having defined the (symmetric) monoidal $(\infty,1)$-category $n\on{-Cat}$
of $(\infty,n)$-categories, we define the (symmetric) monoidal $(\infty,1)$-category $(n+1)\on{-Cat}$
to consist of $(\infty,1)$-categories enriched over $n\on{-Cat}$. One then shows that
$n\on{-Cat}$ itself upgrades to an object of $(n+1)\on{-Cat}$.

\medskip

Given $\bC\in n\on{-Cat}$ and $n'\leq n$, we let
$$\bC^{n'\on{-Cat}}\in n'\on{-Cat}$$
be the object obtained by discarding non-invertible $n''$-morphisms for $n'<n''\leq n$.

\medskip

Following the same conventions as in \secref{sss:cat}, in most cases, we will simply say
``$n$-category" when we mean an $(\infty,n)$-category. 

\sssec{}

For an $(\infty,2)$-category $\fO$ and $\bC_1,\bC_2\in \fO$, we let
$$\bMaps_{\fO}(\bC_1,\bC_2)$$
the resulting $(\infty,1)$-\emph{category} of morphisms, so that
$$\bMaps_{\fO}(\bC_1,\bC_2)^{\on{grpd}}\simeq \Maps_{\fO}(\bC_1,\bC_2).$$

\medskip

For an $(\infty,3)$-category $\BE$ and $\fO_1,\fO_2\in \BE$, we let
$$\BMaps_{\BE}(\fO_1,\fO_2)$$
the resulting $(\infty,2)$-\emph{category} of morphisms, so that
$$\BMaps_{\BE}(\fO_1,\fO_2)^{1\on{-Cat}}\simeq \bMaps_{\BE}(\fO_1,\fO_2).$$

\sssec{An identity crisis}

When writing this paper we faced a certain dilemma regarding the usage of symbols $\on{Id}$ vs $\on{id}$.
To the extent possible, we have tried to adhere to the following conventions:

\medskip

For an object $\bc$ in an $(\infty,1)$-category $\bC$, we denote by $\on{id}_\bc\in \Maps_\bC(\bc,\bc)$
its identity endomorphism. And we denote by $\on{Id}_\bC$ the identity endofunctor of $\bC$.

\medskip

More generally, for a $(\infty,2)$-category $\fO$, and an object $\bC\in \fO$ we denote by $\on{Id}_\bC$ the identity
1-morphism from $\bC$ to itself. For a 1-morphism $\phi:\bC_1\to \bC_2$ in $\fO$, we let
$\on{id}_\phi$ denote the identity 2-endomorphism of $\phi$.

\medskip

Now, this runs into a clash when we are dealing with an ambient $(\infty,3)$-category $\BE$. For a pair of
objects $\fO_1,\fO_2\in \BE$ and a 1-morphism $\Phi:\fO_1\to \fO_2$ we will sometimes denote the identity
endomorphism of $\Phi$ by $\on{id}_\Phi$ and sometimes by $\on{Id}_\Phi$, depending on the context:

\medskip

We will use the former, when we want to \emph{de-emphasize} the 3-categorical aspect of $\BE$
(i.e., when for the local discussion, non-invertible 3-morphisms do not play a big role). And we will use the latter
when $\fO_1$ and $\fO_2$ are fixed, and our focus is the $(\infty,2)$-category $\BMaps_\BE(\fO_1,\fO_2)$,
in which $\Phi$ is an object.

\sssec{Higher algebra}

We will do \emph{higher algebra} in vector spaces over the field of coefficients, which
in this paper we take to be $\ol\BQ_\ell$.

\medskip

We let $\DGCat$ denote the $(\infty,2)$-category of $\ol\BQ_\ell$-linear DG categories,
defined as in \cite[Chapter 1, Sect. 10]{GR}.

\medskip

We view $\DGCat$ as equipped with the symmetric monoidal structure, given by the
Lurie tensor product. The unit object for this structure is $\Vect$,
the category of (chain complexes) $\ol\BQ_\ell$-vector spaces.

\medskip

For a DG category $\bC$ and $\bc_1,\bc_2\in \bC$, we let
$$\CHom_\bC(\bc_1,\bc_2)\in \Vect$$
denote the corresponding object. We have
$$\Maps_{\bC}(\bc_1,\bc_2)=\tau^{\leq 0}(\CHom_\bC(\bc_1,\bc_2)),$$
where we view the right-hand side as a connective spectrum.

\sssec{Algebraic geometry}

Throughout this paper we will be working with the category of (classical, separated) schemes of finite type
over a ground field $k$, assumed algebraically closed. We denote this category by $\Sch$.
We will not need derived algebraic geometry over $k$ for this paper.

\medskip

Whenever we talk about the Frobenius endomorphism, we will assume that $k=\ol\BF_q$, but the scheme
$X$ in question is defined over $\BF_q$, so that the geometric Frobenius $\Frob_X$ on $X$ is defined.

\medskip

By a prestack we shall mean a (classical) prestack locally of finite type over $k$
(see \cite[Chapter 2, Sect. 1.3.6]{GR} for what this means). We denote the corresponding category
by $\on{PreStk}$.

\sssec{}

We will work with the sheaf theory of $\ell$-adic sheaves
$$X\mapsto \Shv(X),\quad X\in \Sch$$
as defined in \cite[Sect. 1.1]{AGKRRV1}.

\medskip

The above assignment extends to prestacks by the procedure of right Kan extension.

\sssec{Groups}

We let $G$ be a linear algebraic group over $k$. For most of the paper, we will assume that $G$
is is connected and reductive.

\ssec{Acknowledgements} We dedicate this paper to G.~Lusztig for creating this area of mathematics.

\medskip

We wish to thank R.~Bezrukavnikov, C.~Chan, G.~Dhillon, A.~Eteve, K.~Lin, S.~Raskin, W.~Reeves and P.~Scholze
for some very helpful discussions.

\medskip

The research of N.R. is supported by an NSERC Discovery Grant (RGPIN-2025-0696).
The research of Y.V. was partially supported by the ISF grants 2091/21 and 2889/25.

\section{Higher categorical trace} \label{s:Trace}

The usual operation of trace takes a (dualizable!) object $\bo$ in a symmetric monoidal category $\bO$,
and attaches to an endomorphism $F$ of $\bO$ an endomorphism $\Tr(F,\bo)$ of $\one_\bO$. Now,
when $\bO$ is a symmetric monoidal 2-category $\bO$, the trace operation interacts in a non-trivial
way with 1-morphisms in $\bO$. And this interaction becomes significantly richer when $\bC$ is a 3-category.

\medskip

We will describe these interactions in the present section.

\ssec{Trace in 1-categories}

\sssec{} \label{sss:usual trace}

Let $\bO$ be a symmetric monoidal category. Let $\bL(\bO)$ denote the (symmetric monoidal) groupoid
consisting of pairs $(\bO,F)$, where $\bo$ is a \emph{dualizable} object of $\bO$
and $F$ is its endomorphism.

\medskip

The trace construction is a (symmetric monoidal) functor
\begin{equation} \label{e:trace functor}
\Tr:\bL(\bO)\to \End(\one_\bO)=:\Omega(\bO).
\end{equation}

\sssec{}

Explicitly, \eqref{e:trace functor} sends $F\in \End(\bo)$ to the composition
\begin{equation} \label{e:trace functor expl}
\one_\bO\overset{\on{u}_\bo}\to \bo\otimes \bo^\vee \overset{F\otimes \on{id}}\longrightarrow
\bo\otimes \bo^\vee \overset{\on{ev}_\bo}\to \bo.
\end{equation}

\sssec{}

Here are the basic properties of the trace construction:

\begin{itemize}

\item Trace is invariant under duality, i.e.,
\begin{equation} \label{e:Tr of dual}
\Tr(F,\bo)\simeq \Tr(F^\vee,\bo^\vee),
\end{equation}
where $F^\vee$ denotes the dual morphism. Indeed,
$$\Tr(F,\bo)=\on{ev}_\bo\circ (F\otimes \on{id})(\on{u}_\bo)\simeq
\on{ev}_\bo\circ (\on{id}\otimes F^\vee)(\on{u}_\bo)= \Tr(F^\vee,\bo^\vee);$$

\item Trace is cyclic, i.e., for $\bo_1\overset{F_{1,2}}\to \bo_2$ and $\bo_2\overset{F_{2,1}}\to \bo_1$, we have
$$\Tr(\bo_1,F_{2,1}\circ F_{1,2})\simeq \Tr(\bo_2,F_{1,2}\circ F_{2,1}).$$

Indeed, both sides are obtained by applying $\on{ev}_{\bo_1}\otimes \on{ev}_{\bo_2}$ to
$$\sigma_{1,3}\left((F_{1,2}\otimes \on{id})(\on{u}_{\bo_1})\otimes (F_{2,1}\otimes \on{id})(\on{u}_{\bo_2})\right)\in
\Maps(\one_\bO,\bo_1\otimes \bo_1^\vee\otimes \bo_2\otimes \bo_2^\vee),$$
where $\sigma_{1,3}$ is the transposition of factors.

\end{itemize}

\ssec{Trace in 2-categories}

\sssec{}

Let now $\bO$ be a 2-category. In this case, we will view $\Omega(\bO)$ as a \emph{category}
(rather than groupoid).

\medskip

Let $\on{Arr}(\bO)$ be the category whose objects are
\begin{equation} \label{e:obj of arr}
\bo_1\overset{t}\to \bo_2,
\end{equation}
and whose morphisms are commutative diagrams
\begin{equation} \label{e:2-morph}
\xy
(0,0)*+{\bo_1}="A";
(20,0)*+{\bo_2}="C";
(0,-20)*+{\bo'_1}="B";
(20,-20)*+{\bo'_2.}="D";
{\ar@{->}_{F_1} "A";"B"};
{\ar@{->}^{t} "A";"C"};
{\ar@{->}^{t'} "B";"D"};
{\ar@{->}^{F_2} "C";"D"};
{\ar@{=>}^\alpha "B";"C"};
\endxy
\end{equation}

If $\bO$ is symmetric monoidal, then $\on{Arr}(\bO)$ also acquires a symmetric monoidal structure.

\sssec{}

Note that in a 2-category, we can talk about a 1-morphism being adjointable
(here and in the sequel, unless specified otherwise, by ``adjointable" we will mean
``admits a right adjoint").

\begin{lem} \label{l:dualizability}
An object \eqref{e:obj of arr} is dualizable if and only if $\bo_1,\bo_2$ are dualizable
as objects of $\bO^{1\on{-Cat}}$ and $t$ is adjointable. Moreover, in this case, the dual of \eqref{e:obj of arr}
is given by
$$\bo^\vee_1\overset{(t^R)^\vee}\to \bo_2^\vee.$$
\end{lem}

\begin{proof}

We will prove the (slightly less tautological) ``only" if direction.
The functors
$$\on{Arr}(\bO)\to \bO^{1\on{-Cat}}$$
that send \eqref{e:obj of arr} to $\bo_1$ and $\bo_2$, respectively, are symmetric monoidal.
Hence, if \eqref{e:obj of arr} is dualizable, then so are $\bo_1$ and $\bo_2$, and the dual has
the shape
$$\bo^\vee_1\overset{s}\to \bo^\vee_2,$$
for some 1-morphism $s$. We will show that $s^\vee:\bo_2\to \bo_1$ is a right adjoint of $t$.

\medskip

Indeed, consider the unit of the duality between $\bo_1\overset{t}\to \bo_2$ and $\bo^\vee_1\overset{s}\to \bo^\vee_2$;
it is represented by a diagram
$$
\xy
(0,0)*+{\one_\bO}="A";
(20,0)*+{\one_\bO}="C";
(0,-20)*+{\bo_1\otimes \bo_1^\vee}="B";
(20,-20)*+{\bo_2\otimes \bo_2^\vee.}="D";
{\ar@{->}_{\on{u}_{\bo_1}} "A";"B"};
{\ar@{->}^{\on{Id}} "A";"C"};
{\ar@{->}^{t\otimes s} "B";"D"};
{\ar@{->}^{\on{u}_{\bo_2}} "C";"D"};
{\ar@{=>}^\alpha "B";"C"};
\endxy
$$

The 2-morphism $\alpha$ is thus a map
$$(t\otimes s)\circ \on{u}_{\bo_1}\to \on{u}_{\bo_2}.$$

Rewriting
$$(t\otimes s)\circ \on{u}_{\bo_1}\simeq ((t\circ s^\vee)\otimes \on{Id}_{\bo^\vee_2})\circ \on{u}_{\bo_2},$$
we obtain a map
$$((t\circ s^\vee)\otimes \on{Id}_{\bo^\vee_2})\circ \on{u}_{\bo_2}\to \on{u}_{\bo_2},$$
or, equivalently, a map
$$t\circ s^\vee \to \on{Id}_{\bo_2}.$$

Similarly, by considering the counit of the adjunction,
one obtains a map
$$\on{Id}_{\bo_1}\to s^\vee\otimes t,$$
and one readily checks that these two maps satisfy the adjunction axioms.

\end{proof}

\sssec{}

Note that
$$\Omega(\one_{\on{Arr}(\bO)})\simeq \on{Funct}([0,1],\Omega(\bO))^{\on{grpd}}.$$

In particular, applying the trace construction from \secref{sss:usual trace} to $\on{Arr}(\bO)$ and using
\lemref{l:dualizability}, we obtain the following construction:

\medskip

Let us be given a diagram
\begin{equation} \label{e:basic diag}
\xy
(0,0)*+{\bo_1}="A";
(20,0)*+{\bo_2}="C";
(0,-20)*+{\bo_1}="B";
(20,-20)*+{\bo_2,}="D";
{\ar@{->}_{F_1} "A";"B"};
{\ar@{->}^{t} "A";"C"};
{\ar@{->}^{t} "B";"D"};
{\ar@{->}^{F_2} "C";"D"};
{\ar@{=>}^\alpha "B";"C"};
\endxy
\end{equation}
where $\bo_1$ and $\bo_2$ are dualizable as objects of $\bO^{1\on{-Cat}}$, and $t$ is adjointable.
Then the 2-morphism $\alpha$ induces a map
\begin{equation} \label{e:2-categ trace}
\Tr(\alpha,t):\Tr(F_1,\bo_1) \to \Tr(F_2,\bo_2)
\end{equation}
in the \emph{category} $\Omega(\bO)$.

\sssec{}

Explicitly, $\Tr(\alpha,t)$ is the 2-morphism equal to the composition
\begin{equation}\label{e:2-trace}
\xy
(0,0)*+{\one_\bO}="A";
(0,-20)*+{\one_\bO}="B";
(30,0)*+{\bo_1\otimes \bo_1^\vee}="C";
(30,-20)*+{\bo_2\otimes \bo_2^\vee}="D";
(60,0)*+{\bo_1\otimes \bo_1^\vee}="E";
(60,-20)*+{\bo_2\otimes \bo_2^\vee}="F";
(90,0)*+{\one_\bO}="G";
(90,-20)*+{\one_\bO}="H";
{\ar@{->}_{\on{Id}} "A";"B"};
{\ar@{->}^{\on{Id}} "G";"H"};
{\ar@{->}^{t\otimes (t^R)^\vee} "C";"D"};
{\ar@{->}^{t\otimes (t^R)^\vee} "E";"F"};
{\ar@{->}^{F_1\otimes \on{Id}} "C";"E"};
{\ar@{->}^{F_2\otimes \on{Id}} "D";"F"};
{\ar@{->}^{\on{u}_{\bo_1}} "A";"C"};
{\ar@{->}^{\on{u}_{\bo_2}} "B";"D"};
{\ar@{->}^{\on{ev}_{\bo_1}} "E";"G"};
{\ar@{->}^{\on{ev}_{\bo_2}} "F";"H"};
{\ar@{=>} "C";"B"};
{\ar@{=>}^{\alpha\otimes \on{Id}} "E";"D"};
{\ar@{=>} "G";"F"};
\endxy
\end{equation}

\sssec{}

Equivalently,  $\Tr(\alpha,t)$ is the following composition
\begin{multline} \label{e:2-cat Tr}
\Tr(F_1,\bo_1)\overset{\text{unit of the adj}}\longrightarrow \Tr(F_1\circ t^R\circ t,\bo_1)
\overset{\text{cyclicity of trace}}\longrightarrow \Tr(t\circ F_1\circ t^R,\bo_2) \overset{\alpha}\to \\
\to \Tr(F_2\circ t \circ t^R,\bo_2) \overset{\text{counit of the adj}}\longrightarrow \Tr(F_2,\bo_2),
\end{multline}
where the arrows are induced by the (obvious) map
$$\Maps_{\End(\bo)}(F',F'') \to \Maps_{\End(\one_\bO)}(\Tr(F',\bO),\Tr(F'',\bO)),$$
for a dualizable object $\bo\in \bO$ and $F',F''\in \End(\bo)$.

\sssec{}

A particular case of the above construction is when $\bo_1=\bo=\bo_2$, $t=\on{Id}$, and $\alpha$ amounts to a 2-morphism
$$F_1\overset{\alpha}\Rightarrow F_2.$$

In this case $\alpha$ induces a map
$$\Tr(F_1,\bo)\overset{\Tr(\alpha,\on{Id})}\longrightarrow \Tr(F_2,\bo).$$

\sssec{} \label{sss:trace and duality}

Given a diagram \eqref{e:basic diag}, consider the diagram
\begin{equation} \label{e:basic diag dual}
\xy
(0,0)*+{\bo^\vee_1}="A";
(20,0)*+{\bo^\vee_2}="C";
(0,-20)*+{\bo^\vee_1}="B";
(20,-20)*+{\bo^\vee_2,}="D";
{\ar@{->}_{F^\vee_1} "A";"B"};
{\ar@{->}^{(t^R)^\vee} "A";"C"};
{\ar@{->}^{(t^R)^\vee} "B";"D"};
{\ar@{->}^{F^\vee_2} "C";"D"};
{\ar@{=>}^{\alpha'} "B";"C"};
\endxy
\end{equation}
where 
$$(t^R)^\vee\circ F_1^\vee \overset{\alpha'}\to F_2^\vee\circ (t^R)^\vee$$
is obtained by duality from 
$$F_1\circ t^R\overset{\alpha^{\on{BC}}}\to t^R\circ F_2,$$
where $\alpha^{\on{BC}}$ is obtained\footnote{Here and in the sequel, ``BC" stands for Beck-Chevalley.}
from $\alpha$ by adjunction. 

\medskip

Note that $(t^R)^\vee$ is adjointable, namely, its right adjoint is $t^\vee$. So the map
$$\Tr(F^\vee_1,\bo^\vee_1) \overset{\Tr(\alpha',(t^R)^\vee)}\longrightarrow \Tr(F^\vee_2,\bo^\vee_2)$$
is defined.

\medskip

Unwinding the constructions, we obtain that the diagram
$$
\CD
\Tr(F_1,\bo_1) @>{\Tr(\alpha,t)}>> \Tr(F_2,\bo_2)  \\
@V{\text{\eqref{e:Tr of dual}}}V{\sim}V @V{\text{\eqref{e:Tr of dual}}}VV \\
\Tr(F^\vee_1,\bo^\vee_1) @>{\Tr(\alpha',(t^R)^\vee)}>> \Tr(F^\vee_2,\bo^\vee_2)
\endCD
$$
commutes. 

\sssec{} \label{sss:trace and comp}

More generally, for an index category $I$, we can consider the (symmetric monoidal) 1-category
$$\on{Funct}_{\on{right-lax}}(I,\bO)$$
and a symmetric functor
\begin{equation} \label{e:trace and comp}
\bL(\on{Funct}_{\on{right-lax}}(I,\bO))\overset{\Tr}\to \Omega(\one_{\on{Funct}_{\on{right-lax}}(I,\bO)})\simeq
\on{Funct}(I,\Omega(\bO))^{\on{grpd}}.
\end{equation}

Furthermore, the functors \eqref{e:trace and comp} are functorial with respect to $I$, in the sense that
we obtain a natural transformation between two functors
$$\on{Cat}\to \on{Grpd}^{\on{SymMon}}.$$

\sssec{} \label{sss:L for 2}

Let $\bL(\bO)$ be the symmetric monoidal category, whose underlying groupoid is $\bL(\bO^{1-\on{Cat}})$, and
where 1-morphisms are given by diagrams \eqref{e:basic diag} with $t$ adjointable.

\medskip

The functoriality of \eqref{e:trace and comp} with respect to the index category implies that the trace construction
extends to a symmetric monoidal functor
\begin{equation} \label{e:2-cat Tr funct}
\bL(\bO)\to \Omega(\bO).
\end{equation}

\ssec{The notion of 2-dualizability} \label{ss:2-dualizable}

\sssec{}

Let $\bO$ be a symmetric monoidal 2-category. Recall that an object $\bo\in \bO$ is said to be 2-dualizable if
the following conditions hold:

\begin{itemize}

\item $\bo$ is dualizable as an object of $\bO^{1\on{-Cat}}$;

\item The unit and counit 1-morphisms
$$\on{u}_\bo:\one_\bO\to \bo\otimes \bo^\vee \text{ and } \on{ev}_\bo:\bo\otimes \bo^\vee \to \one_\bO$$
are adjointable.

\end{itemize}

\sssec{Example}

Let $\bO=\DGCat$, so that $\bo$ is a DG category $\bC$. Assume that $\bC$ is compactly generated
(in particular, it is dualizable as an object of $\DGCat^{1\on{-Cat}}$).

\medskip

Then $\bC$ is 2-dualizable if and
only if it is \emph{smooth} and \emph{proper}.

\begin{rem}

Passing to the right adjoints of the unit and counit of the duality between $\bo$ and $\bo^\vee$, we obtain a
\emph{different} datum of duality. The discrepancy between the two is an automorphism of $\bo$, called
\emph{Serre functor}, and denoted $\on{Se}_\bo$.

\end{rem}

\sssec{}

Let $\bO$ be as above, and let $\bo\in \bO$ be a 2-dualizable object. Let $F$ be an adjointable 1-endomorphsim of
$\bo$. We have:

\begin{lem} \label{l:Trace dualizable}
Under the above circumstances, the object $\Tr(F,\bo)\in \Omega(\bO)$ is dualizable with dual
$$\Tr(F^R,\bo)\simeq \Tr((F^R)^\vee,\bo^\vee).$$
\end{lem}

\begin{rem}

One can show that if $\bo_1$ and $\bo_2$ are 2-dualizable, and $t:\bo_1\to \bo_2$ admits a right adjoint,
then $t^R$ is also admits a right adjoint. In fact,
$$(t^R)^R\simeq \on{Se}_{\bo_1}\circ t\circ \on{Se}^{-1}_{\bo_2}.$$

\end{rem}

\begin{proof}[Proof of \lemref{l:Trace dualizable}]

We rewrite $\Tr(F,\bo)$ as $\Tr(F^\vee,\bo^\vee)$. Now the assertion of the lemma
follows the fact that the functor
$$\Tr:\bL(\bO)\to \Omega(\bO)$$
is symmetric monoidal and
compatible with compositions of adjointable 1-morphisms (see \eqref{e:2-cat Tr funct}), so that the duality data for $\bo$ and $\bo^\vee$
induces one on $\Tr(F^R,\bo)$ and $\Tr(F^\vee,\bo^\vee)$ via
the diagrams
$$
\xy
(0,0)*+{\one_\bO}="A";
(20,0)*+{\bo\otimes \bo^\vee}="C";
(0,-20)*+{\one_\bO}="B";
(20,-20)*+{\bo\otimes \bo^\vee}="D";
{\ar@{->}_{\on{Id}} "A";"B"};
{\ar@{->}^{\on{u}_\bo} "A";"C"};
{\ar@{->}^{\on{u}_\bo} "B";"D"};
{\ar@{->}^{F^R\otimes F^\vee} "C";"D"};
{\ar@{=>}^{\alpha_1} "B";"C"};
\endxy
$$
and
$$
\xy
(20,0)*+{\one_\bO}="A";
(0,0)*+{\bo\otimes \bo^\vee}="C";
(20,-20)*+{\one_\bO,}="B";
(0,-20)*+{\bo\otimes \bo^\vee}="D";
{\ar@{->}^{\on{Id}} "A";"B"};
{\ar@{<-}_{\on{ev}_\bo} "A";"C"};
{\ar@{<-}^{\on{ev}_\bo} "B";"D"};
{\ar@{->}_{F^R\otimes F^\vee} "C";"D"};
{\ar@{<=}_{\alpha_2} "A";"D"};
\endxy
$$
respectively, where:

\begin{itemize}

\item $\alpha_1$ is the 2-morphism
\begin{multline*}
\on{u}_\bo \to ((F^R\circ F)\otimes \on{Id})(\on{u}_\bo)\simeq
(F^R\otimes \on{Id})\circ (F\otimes \on{Id})(\on{u}_\bo)\simeq \\
\simeq (F^R\otimes \on{Id})\circ (\on{Id}\otimes F^\vee)(\on{u}_\bo)\simeq
(F^R\otimes F^\vee)(\on{u}_\bo);
\end{multline*}

\item $\alpha_2$ is the 2-morphism
$$\on{ev}_\bo\circ (F^R\otimes F^\vee)\simeq \on{ev}_\bo\circ ((F\circ F^R)\otimes \on{Id})\to \on{ev}_\bo.$$

\end{itemize}

\end{proof}

\sssec{} \label{sss:setting for iterated trace}

Let $\bo$ be again a 2-dualizable object, and let $F_1$ and $F_2$ be two adjointable endomorphisms of $\bo$
that (lax) commute with each other, i.e., we are given a $2$-morphism
$$F_2\circ F_1 \overset{\beta}{\to} F_1\circ F_2.$$
By adjunction, this also gives a $2$-morphism
$$F_1 \circ F_2^R\overset{\beta^{\on{BC}}}{\to} F_2^R \circ F_1,$$
where the superscript $\on{BC}$ stands for Beck-Chevalley, see \secref{sss:trace and duality}. 

\medskip

The construction \eqref{e:2-categ trace} defines maps
\begin{equation} \label{e:map on traces}
\Tr(\beta,F_2):\Tr(F_1,\bo)\to \Tr(F_1,\bo) \text{ and }
\Tr(\beta^{\on{BC}},F_1):\Tr(F_2^R,\bo)\to \Tr(F_2^R,\bo)
\end{equation}
in the category $\Omega(\bO)$.

\medskip

By \lemref{l:Trace dualizable}, the objects
$$\Tr(F_1,\bo) \text{ and } \Tr(F_2^R,\bo)$$
of $\Omega(\bO)$ are dualizable, so it makes sense to consider the objects
\begin{equation} \label{e:two traces}
\Tr(\Tr(\beta,F_2),\Tr(F_1,\bo)) \text{ and } \Tr(\Tr(\beta^{\on{BC}},F_1),\Tr(F_2^R,\bo))
\end{equation}
in
$$\Omega^2(\bO):=\Omega(\Omega(\bO)).$$

We have the following fundamental result:
\begin{thm}[\cite{BN1}] \label{t:2 traces}
The objects \eqref{e:two traces} are canonically isomorphic.
\end{thm}

\begin{proof}

Consider the diagram defining the morphism $\Tr(\beta, F_2)$:
\begin{equation}\label{e:pre iterated trace}
\xy
(0,0)*+{\one_\bO}="A";
(0,-20)*+{\one_\bO}="B";
(30,0)*+{\bo\otimes \bo^\vee}="C";
(30,-20)*+{\bo\otimes \bo^\vee}="D";
(60,0)*+{\bo\otimes \bo^\vee}="E";
(60,-20)*+{\bo\otimes \bo^\vee}="F";
(90,0)*+{\one_\bO}="G";
(90,-20)*+{\one_\bO}="H";
{\ar@{->}_{\on{Id}} "A";"B"};
{\ar@{->}^{\on{Id}} "G";"H"};
{\ar@{->}^{F_2\otimes (F_2^R)^\vee} "C";"D"};
{\ar@{->}^{F_2\otimes (F_2^R)^\vee} "E";"F"};
{\ar@{->}^{F_1\otimes \on{Id}} "C";"E"};
{\ar@{->}^{F_1\otimes \on{Id}} "D";"F"};
{\ar@{->}^{\on{u}_{\bo}} "A";"C"};
{\ar@{->}^{\on{u}_{\bo}} "B";"D"};
{\ar@{->}^{\on{ev}_{\bo}} "E";"G"};
{\ar@{->}^{\on{ev}_{\bo}} "F";"H"};
{\ar@{=>} "C";"B"};
{\ar@{=>}^{\beta\otimes \on{Id}} "E";"D"};
{\ar@{=>} "G";"F"};
\endxy
\end{equation}
The outer square of this diagram is
\begin{equation}\label{e:outer pre iterated trace}
\xy
(0,0)*+{\one_{\bO}}="A";
(30,0)*+{\one_{\bO}}="C";
(0,-20)*+{\one_{\bO}}="B";
(30,-20)*+{\one_{\bO}.}="D";
{\ar@{->}_{\on{Id}} "A";"B"};
{\ar@{->}^{\Tr(F_1, \bo)} "A";"C"};
{\ar@{->}^{\Tr(F_1, \bo)} "B";"D"};
{\ar@{->}^{\on{Id}} "C";"D"};
{\ar@{<=}^{\Tr(\beta, F_2)} "B";"C"};
\endxy
\end{equation}
Now, we perform the following maneuver on diagram \eqref{e:pre iterated trace}: first pass to right adjoints along the vertical arrows and then
we flip the resulting diagram around the horizontal axis.  This produces the following diagram:

\begin{equation}\label{e:rotated pre iterated trace}
\xy
(0,0)*+{\one_\bO}="A";
(0,-20)*+{\one_\bO}="B";
(30,0)*+{\bo\otimes \bo^\vee}="C";
(30,-20)*+{\bo\otimes \bo^\vee}="D";
(60,0)*+{\bo\otimes \bo^\vee}="E";
(60,-20)*+{\bo\otimes \bo^\vee}="F";
(90,0)*+{\one_\bO}="G";
(90,-20)*+{\one_\bO}="H";
{\ar@{->}_{\on{Id}} "A";"B"};
{\ar@{->}^{\on{Id}} "G";"H"};
{\ar@{->}^{F_2^R\otimes F_2^\vee} "C";"D"};
{\ar@{->}^{F_2^R\otimes F_2^\vee} "E";"F"};
{\ar@{->}^{F_1\otimes \on{Id}} "C";"E"};
{\ar@{->}^{F_1\otimes \on{Id}} "D";"F"};
{\ar@{->}^{\on{u}_{\bo}} "A";"C"};
{\ar@{->}^{\on{u}_{\bo}} "B";"D"};
{\ar@{->}^{\on{ev}_{\bo}} "E";"G"};
{\ar@{->}^{\on{ev}_{\bo}} "F";"H"};
{\ar@{<=} "C";"B"};
{\ar@{<=}^{\beta^{\on{BC}}\otimes \on{Id}} "E";"D"};
{\ar@{<=} "G";"F"};
\endxy
\end{equation}
where the leftmost and rightmost 2-morphisms are the ones in \lemref{l:Trace dualizable} defining the duality between $\Tr(F_2^R, \bo)$ and $\Tr(F_2^{\vee}, \bo^{\vee})$.

\medskip

By the hypothesis on $\bo$ and $F_1, F_2$, we can read the diagram \eqref{e:rotated pre iterated trace} from left to right as a sequence of 3 composable
1-morphisms in $\bL(\bO)$.

\medskip

On the one hand, applying the trace functor \eqref{e:2-cat Tr funct} to this sequence, we obtain
\begin{multline}\label{e:outer pre-trace}
\one_{\Omega(\bO)}  \overset{\on{u}_{\Tr((F_2)^R, \bo)}}\longrightarrow \Tr((F_2)^R, \bo) \otimes \Tr(F_2^\vee, \bo^{\vee})
\overset{\Tr(\beta^{\on{BC}}, F_1) \otimes \on{Id}}\longrightarrow \\
\to \Tr((F_2)^R, \bo) \otimes \Tr(F_2^\vee, \bo^{\vee}) \overset{\on{ev}_{\Tr((F_2)^R, \bo)}}\longrightarrow \one_{\Omega(\bO)},
\end{multline}
%
whose composite is by definition
$$ \Tr(\Tr(\beta^{\on{BC}}, F_1), \Tr(F_2^R, \bo)) \in \Omega^2(\bO).$$

\medskip

On the other hand, the outer square of \eqref{e:rotated pre iterated trace} is obtained from \eqref{e:outer pre iterated trace} by passing to right adjoints along the vertical morphisms and flipping the resulting diagram around the horizontal axis, which gives
\begin{equation}\label{e:outer rotated pre iterated trace}
\xy
(0,0)*+{\one_{\bO}}="A";
(30,0)*+{\one_{\bO}}="C";
(0,-20)*+{\one_{\bO}}="B";
(30,-20)*+{\one_{\bO}.}="D";
{\ar@{->}_{\on{Id}} "A";"B"};
{\ar@{->}^{\Tr(F_1, \bo)} "A";"C"};
{\ar@{->}^{\Tr(F_1, \bo)} "B";"D"};
{\ar@{->}^{\on{Id}} "C";"D"};
{\ar@{=>}^{\Tr(\beta, F_2)} "B";"C"};
\endxy
\end{equation}
Now, applying the trace functor \eqref{e:2-cat Tr funct} to \eqref{e:outer rotated pre iterated trace}, we obtain
$$ \Tr(\Tr(\beta, F_2), \Tr(F_1, \bo)) \in \Omega^2(\bO),$$
giving the desired isomorphism.

\end{proof}

\sssec{}

Suppose now that in the setting of \secref{sss:setting for iterated trace}, $F_2=\on{Id}$ and $\beta$ is the identity endomorphism of $F_1$.
Note that in this case $\beta^{\on{BC}}$ is also the identity endomorphism of $F_1$.

\medskip

In this case \thmref{t:2 traces} implies:

\begin{cor}  \label{c:2 traces}
Let $F$ be an adjointable endomorphism of a 2-dualizable object.
We have a canonical isomorphism
\begin{equation} \label{e:2 traces}
\Tr(\on{Id},\Tr(F,\bo)) \simeq \Tr(\Tr(\on{id}, F), \Tr(\on{Id}, \bo)).
\end{equation}
\end{cor}

\sssec{Example}

Let $\bC$ be a smooth and proper category, and let $F$ be its endofunctor that preserves
compactness. From \lemref{l:Trace dualizable},
we obtain that $\Tr(F,\bC)\in \Vect$ is dualizable, i.e., finite-dimensional.

\medskip

From \eqref{t:2 traces} we obtain
$$\dim(\Tr(F,\bC))=\Tr(F,\Tr(\on{Id},\bC))$$
as elements of $\ol\BQ_\ell$.

\ssec{The category of arrows for a 3-category}

We now increase the level of complexity and start tackling the 3-categorical situation.

\sssec{} \label{sss:Arr for 2}

Let now $\bO$ be a 3-category. In this case, we will perceive $\Omega(\bO)$ as a 2-category.

\medskip

Let $\on{Arr}(\bO)$ be the following 2-category. We let
$$(\on{Arr}(\bO))^{1-\on{Cat}}:=\on{Arr}(\bO^{2-\on{Cat}}).$$

For a pair of objects
$$\bo_1\overset{t}\to \bo_2 \text{ and } \bo'_1\overset{t'}\to \bo'_2$$
and a pair of morphisms between them
$$
\xy
(0,0)*+{\bo_1}="A";
(20,0)*+{\bo_2}="C";
(0,-20)*+{\bo'_1}="B";
(20,-20)*+{\bo'_2}="D";
{\ar@{->}_{F_1} "A";"B"};
{\ar@{->}^{t} "A";"C"};
{\ar@{->}^{t'} "B";"D"};
{\ar@{->}^{F_2} "C";"D"};
{\ar@{=>}^{\alpha} "B";"C"};
\endxy
$$
and
$$
\xy
(0,0)*+{\bo_1}="A";
(20,0)*+{\bo_2}="C";
(0,-20)*+{\bo'_1}="B";
(20,-20)*+{\bo'_2,}="D";
{\ar@{->}_{\wt{F}_1} "A";"B"};
{\ar@{->}^{t} "A";"C"};
{\ar@{->}^{t'} "B";"D"};
{\ar@{->}^{\wt{F}_2} "C";"D"};
{\ar@{=>}^{\wt\alpha} "B";"C"};
\endxy
$$
a 2-morphism
$$(F_1,F_2,\alpha) \Rightarrow (\wt{F}_1,\wt{F}_2,\wt\alpha)$$
is the datum of
$$(\beta_1:F_1 \Rightarrow \wt{F}_1,\,\,\beta_2:F_2 \Rightarrow \wt{F}_2,\,\, \gamma),$$
where $\gamma$ is a 3-morphism\footnote{Note the direction of the arrow $\gamma$!}
$$\wt\alpha\circ t'(\beta_1) \Rrightarrow \beta_2(t) \circ \alpha,$$
i.e., a 2-morphism in the diagram

$$
\xy
(0,0)*+{t'\circ F_1}="A";
(0,-20)*+{t'\circ \wt{F}_1}="C";
(20,0)*+{F_2\circ t}="B";
(20,-20)*+{\wt{F}_2\circ t.}="D";
{\ar@{->}^{\alpha} "A";"B"};
{\ar@{->}_{t'(\beta_1)} "A";"C"};
{\ar@{->}^{\beta_2(t)} "B";"D"};
{\ar@{->}_{\wt\alpha} "C";"D"};
{\ar@{<=}_\gamma "B";"C"};
\endxy
$$

\sssec{}

We observe:

\begin{lem} \label{l:adj in Arr}
A 1-morphism
\begin{equation} \label{e:adj sq}
\xy
(0,0)*+{\bo_1}="A";
(20,0)*+{\bo_2}="B";
(0,-20)*+{\bo'_1}="C";
(20,-20)*+{\bo'_2}="D";
{\ar@{->}^{F_1} "A";"B"};
{\ar@{->}^{t} "A";"C"};
{\ar@{->}^{t'} "B";"D"};
{\ar@{->}_{F_2} "C";"D"};
{\ar@{=>}^{\alpha} "B";"C"};
\endxy
\end{equation}
in $\on{Arr}(\bO)$ is adjointable if and only if the following conditions hold:

\begin{itemize}

\item $F_1$ and $F_2$ are adjointable in $\bO^{1\on{-Cat}}$;

\item The 2-morphism $\alpha$ is adjointable.

\end{itemize}

\noindent In this case, the right adjoint of \eqref{e:adj sq} is given by
\begin{equation} \label{e:adj sq adj}
\xy
(0,0)*+{\bo'_1}="A";
(30,0)*+{\bo'_2}="C";
(0,-30)*+{\bo_1}="B";
(30,-30)*+{\bo_2,}="D";
{\ar@{->}_{F^R_1} "A";"B"};
{\ar@{->}^{t'} "A";"C"};
{\ar@{->}_{t} "B";"D"};
{\ar@{->}^{F^R_2} "C";"D"};
{\ar@{=>}^{(\alpha^R)^{\on{BC}}} "B";"C"};
\endxy
\end{equation}
where $(-)^{\on{BC}}$ is as in \secref{sss:trace and duality}. 
\end{lem}

\begin{proof}

We will prove the more difficult direction, i.e., that if the conditions of the lemma hold, then then
\eqref{e:adj sq} is adjointable. We will construct the unit of the adjunction. The counit is constructed similarly.

\medskip

Consider the composite 1-morphisms in $\on{Arr}(\bO)$:
\begin{equation} \label{e:comp diag}
\xy
(0,0)*+{\bo_1}="A";
(20,0)*+{\bo_2}="C";
(0,-20)*+{\bo'_1}="B";
(20,-20)*+{\bo'_2}="D";
(0,-40)*+{\bo_1}="E";
(20,-40)*+{\bo_2,}="F";
{\ar@{->}_{F_1} "A";"B"};
{\ar@{->}^{t} "A";"C"};
{\ar@{->}^{t'} "B";"D"};
{\ar@{->}^{F_2} "C";"D"};
{\ar@{=>}^{\alpha} "B";"C"};
{\ar@{->}^{t} "E";"F"};
{\ar@{->}_{F^R_1} "B";"E"};
{\ar@{->}^{F^R_2}  "D";"F"};
{\ar@{=>}^{\wt\alpha} "E";"D"};
\endxy
\end{equation}
as a 1-morphism in $\on{Arr}(\bO)$.

\medskip

We need to construct a 2-morphism to \eqref{e:comp diag} from
$$
\xy
(0,0)*+{\bo_1}="A";
(20,0)*+{\bo_2}="C";
(0,-20)*+{\bo_1}="B";
(20,-20)*+{\bo_2}="D";
{\ar@{->}_{\on{Id}} "A";"B"};
{\ar@{->}^{t} "A";"C"};
{\ar@{->}^{t} "B";"D"};
{\ar@{->}^{\on{Id}} "C";"D"};
{\ar@{=>}^{\on{id}} "B";"C"};
\endxy
$$

The 2-morphisms $\on{u}_{F_i}$
$$\on{Id}\to F^R_1 \circ F_1 \text{ and } \on{Id}\to F^R_2 \circ F_2$$
are the units of the $(F_i,F_i^R)$-adjunctions.

\medskip

It remains to supply a 3-morphism
from the composition
\begin{multline} \label{e:comp 2-morph}
t \overset{\on{u}_{F_1}}\longrightarrow t\circ F_1^R \circ F_1
\overset{\on{u}_{F_2}}\longrightarrow F_2^R \circ F_2 \circ t \circ F_1^R \circ F_1
\overset{\alpha^R}\longrightarrow \\
\to F_2^R \circ t'\circ F_1 \circ F_1^R \circ F_1 \overset{\on{ev}_{F_1}}\longrightarrow
F_2^R \circ t'\circ F_1 \overset{\alpha}\to F_2^R \circ F_2 \circ t
\end{multline}
to
$$t \overset{\on{u}_{F_2}}\longrightarrow F_2^R \circ F_2 \circ t.$$

\medskip

We rewrite \eqref{e:comp 2-morph} as
\begin{multline*}
t \overset{\on{u}_{F_2}}\longrightarrow F_2^R \circ F_2 \circ t
\overset{\on{u}_{F_1}}\longrightarrow F_2^R \circ F_2 \circ t \circ F_1^R \circ F_1
\overset{\alpha^R}\longrightarrow \\
\to F_2^R \circ t'\circ F_1 \circ F_1^R \circ F_1 \overset{\on{ev}_{F_1}}\longrightarrow
F_2^R \circ t'\circ F_1 \overset{\alpha}\to F_2^R \circ F_2 \circ t.
\end{multline*}

\medskip

Hence, it suffices to construct a map from
\begin{equation} \label{e:f2 t}
F_2 \circ t
\overset{\on{u}_{F_1}}\longrightarrow F_2 \circ t \circ F_1^R \circ F_1
\overset{\alpha^R}\longrightarrow
t'\circ F_1 \circ F_1^R \circ F_1 \overset{\on{ev}_{F_1}}\longrightarrow
t'\circ F_1 \overset{\alpha}\to F_2 \circ t
\end{equation}
to the identity endomorphism of $F_2 \circ t$.

\medskip

We rewrite \eqref{e:f2 t} as
$$F_2 \circ t \overset{\alpha^R}\longrightarrow t'\circ F_1 \overset{\on{u}_{F_1}}\longrightarrow
t'\circ F_1 \circ F_1^R \circ F_1 \overset{\on{ev}_{F_1}}\longrightarrow
t'\circ F_1 \overset{\alpha}\to F_2 \circ t,$$
which, by the adjunction axiom is the same as
$$F_2 \circ t \overset{\alpha^R}\longrightarrow t'\circ F_1 \overset{\alpha}\to F_2 \circ t,$$
which maps to the identity endomorphism via the counit of the $(\alpha,\alpha^R)$-adjunction.

\end{proof}

\ssec{Trace in the 3-categorical setting}

In this subsection we let $\bO$ be a symmetric monoidal 3-category.

\sssec{} \label{sss:3 funct Tr}

Let $\bO$ be a symmetric monoidal 3-category.  In this case $\on{Arr}(\bO)$ also acquires a symmetric monoidal structure
(as a 2-category).

\medskip

Note that we have a canonical equivalence
\begin{equation} \label{e:Arr Omega}
\Omega(\on{Arr}(\bO)) \simeq \on{Arr}(\Omega(\bO))
\end{equation}
as symmetric monoidal 2-categories.

\medskip

Applying \eqref{e:2-cat Tr funct}, we obtain a functor
$$\bL(\on{Arr}(\bO))\overset{\Tr}\to \Omega(\on{Arr}(\bO))\simeq \on{Arr}(\Omega(\bO)),$$
where we view $\bL(\on{Arr}(\bO))$ as a 1-category, see \secref{sss:L for 2}.

\sssec{}

In particular, we obtain the following construction. Let us be given two dualizable objects
$$\bo_1\overset{t}\to \bo_2 \text{ and } \bo'_1\overset{t'}\to \bo_2'$$
of $\on{Arr}(\bO)$, equipped with the following data:

\begin{itemize}

\item An \emph{adjointable} 1-morphism
$$
\xy
(0,0)*+{\bo_1}="A";
(20,0)*+{\bo_2}="C";
(0,-20)*+{\bo'_1}="B";
(20,-20)*+{\bo'_2}="D";
{\ar@{->}_{G_1} "A";"B"};
{\ar@{->}^{t} "A";"C"};
{\ar@{->}^{t'} "B";"D"};
{\ar@{->}^{G_2} "C";"D"};
{\ar@{=>}^\delta "B";"C"};
\endxy
$$
(i.e., by \lemref{l:adj in Arr}, $G_1$ and $G_2$ are adjointable 1-morphisms and $\delta$ is an adjointable 2-morphism);

\medskip

\item An endomorphism
$$
\xy
(0,0)*+{\bo_1}="A";
(20,0)*+{\bo_2}="C";
(0,-20)*+{\bo_1}="B";
(20,-20)*+{\bo_2;}="D";
{\ar@{->}_{F_1} "A";"B"};
{\ar@{->}^{t} "A";"C"};
{\ar@{->}^{t} "B";"D"};
{\ar@{->}^{F_2} "C";"D"};
{\ar@{=>}^\alpha "B";"C"};
\endxy
$$

\item An endomorphism
$$
\xy
(0,0)*+{\bo'_1}="A";
(20,0)*+{\bo'_2}="C";
(0,-20)*+{\bo'_1}="B";
(20,-20)*+{\bo'_2,\;}="D";
{\ar@{->}_{F'_1} "A";"B"};
{\ar@{->}^{t'} "A";"C"};
{\ar@{->}^{t'} "B";"D"};
{\ar@{->}^{F'_2} "C";"D"};
{\ar@{=>}^{\alpha'} "B";"C"};
\endxy
$$

\item A 2-morphism
\begin{equation} \label{e:2-morphism G}
(G_1,G_2,\delta)\circ (F_1,F_2,\alpha)\to (F'_1,F'_2,\alpha')\circ (G_1,G_2,\delta)
\end{equation}
in $\on{Arr}(\bO)$.

\end{itemize}

Then we obtain a map
\begin{equation} \label{e:induced map on Tr}
\Tr((F_1,F_2,\alpha),\bo_1\overset{t}\to \bo_2)\to \Tr((F'_1,F'_2,\alpha'),\bo'_1\overset{t'}\to \bo_2')
\end{equation}
on $\Omega(\on{Arr}(\bO))$.

\medskip

Let us decipher what the above means.

\sssec{}

First, the data of \eqref{e:2-morphism G} amounts to:

\begin{itemize}

\item A 2-morphism $G_1\circ F_1 \overset{\beta_1}\to F'_1\circ G_1$;

\item A 2-morphism $G_2\circ F_2 \overset{\beta_2}\to F'_2\circ G_2$;

\item A 3-morphism $\delta\circ \alpha'\circ \beta_1 \overset{\gamma}\Rightarrow \beta_2\circ \alpha\circ \delta$,
where the two sides are 2-morphisms
$$t'\circ G_1\circ F_1 \to F'_2\circ G_2\circ t.$$

\end{itemize}

\sssec{}

Under the identification \eqref{e:Arr Omega}, the objects
$$\Tr((F_1,F_2,\alpha),\bo_1\overset{t}\to \bo_2) \text{ and } \Tr((F'_1,F'_2,\alpha'),\bo'_1\overset{t'}\to \bo_2')$$
correspond to the objects
$$\Tr(F_1,\bo_1) \overset{\Tr(\alpha,t)}\longrightarrow \Tr(F_2,\bo_2) \text{ and }
\Tr(F'_1,\bo'_1) \overset{\Tr(\alpha',t')}\longrightarrow \Tr(F'_2,\bo'_2)$$
of $\on{Arr}(\Omega(\bO))$, respectively.

\medskip

The map \eqref{e:induced map on Tr} amounts to a diagram
\begin{equation} \label{e:induced map on Tr explained}
\xy
(0,0)*+{\Tr(F_1,\bo_1)}="A";
(40,0)*+{ \Tr(F_2,\bo_2)}="C";
(0,-20)*+{\Tr(F'_1,\bo'_1)}="B";
(40,-20)*+{\Tr(F'_2,\bo'_2),}="D";
{\ar@{->}_{\Tr(\beta_1,G_1)} "A";"B"};
{\ar@{->}^{\Tr(\alpha,t)} "A";"C"};
{\ar@{->}^{\Tr(\alpha',t')} "B";"D"};
{\ar@{->}^{\Tr(\beta_2,G_2)} "C";"D"};
{\ar@{=>}^{\Tr(\gamma,\delta)} "B";"C"};
\endxy
\end{equation}
where the new piece of data as compared to the 2-categorical trace is the 3-morphism that in the above diagram
we denote $\Tr(\gamma,\delta)$.

\sssec{} \label{sss:higher funct}

Consider the particular case, when
$$\bo'_1=\bo_1,\,\,\bo'_2=\bo_2,\,\, G_1=\on{Id},\,\, G_2=\on{Id},\,\, F_1=F'_1,\,\, F_2=F'_2,\,\,
\beta_1=\on{Id},\,\, \beta_2=\on{Id}.$$

In this case:

\begin{itemize}

\item $\delta$ is an \emph{adjointable} 2-morphism $t'\Rightarrow t$;

\medskip

\item $\gamma$ is a 3-morphism $\delta\circ \alpha'\Rightarrow \alpha\circ \delta$, where the two sides are 2-morphisms
$$t'\circ F_1\to F_2\circ t.$$

\end{itemize}

In this case, we obtain a 3-morphism
\begin{equation} \label{e:induced map on Tr explained particular}
\Tr(\alpha',t') \overset{\Tr(\gamma,\delta)}\longrightarrow \Tr(\alpha,t),
\end{equation}
where the two sides are 2-morphisms
$$\Tr(F_1,\bo_1) \to \Tr(F_2,\bo_2).$$

\ssec{Functoriality of the 3-categorical trace}

We continue to assume that $\bO$ is a symmetric monoidal 3-category.

\sssec{}

As in \secref{sss:trace and comp}, instead of $\on{Arr}(\bO)$, for an index category $I$, we can consider the symmetric monoidal 2-category
$$\on{Funct}(I,\bO)_{\on{right-lax}},$$
where
$$(\on{Funct}(I,\bO)_{\on{right-lax}})^{1-\on{Cat}}=\on{Funct}(I,\bO^{2-\on{Cat}})_{\on{right-lax}},$$
and where 2-morphisms are defined as in \secref{sss:Arr for 2}.

\medskip

Parallel to \secref{e:Arr Omega}, we can identify
\begin{equation} \label{e:Arr Omega I}
\Omega(\on{Funct}(I,\bO)_{\on{right-lax}})\simeq \on{Funct}(I,\Omega(\bO))_{\on{right-lax}}.
\end{equation}

\sssec{} \label{sss:trace and comp 3}

Applying \eqref{e:2-cat Tr funct} to $\on{Funct}(I,\bO)_{\on{right-lax}}$, we obtain a symmetric monoidal functor
\begin{equation} \label{e:2-cat Tr funct I}
\bL(\on{Funct}(I,\bO)_{\on{right-lax}}) \overset{\Tr}\to \Omega(\on{Funct}(I,\bO)_{\on{right-lax}})\simeq \on{Funct}(I,\Omega(\bO))_{\on{right-lax}}.
\end{equation}

Moreover, the functors \eqref{e:2-cat Tr funct I} are themselves functorial in $I$ in the sense that they give rise to a natural
transformation between two functors
$$\on{Cat}\to \on{Cat}^{\on{SymMon}}.$$

\sssec{} \label{sss:L for 3}

Let $\bL(\bO)$ be the symmetric monoidal 2-category, whose underlying 1-category is $\bL(\bO^{2-\on{Cat}})$ (see \secref{sss:L for 2}),
and whose 2-morphisms are given by the data in \secref{sss:higher funct}\footnote{Note that in the formulas below we
notationally swap $t$ and $t'$.} with $\delta$ adjointable.

\medskip

I.e., a 2-morphism between the 1-morphisms
$$
\xy
(0,0)*+{\bo_1}="A";
(20,0)*+{\bo_2}="C";
(0,-20)*+{\bo_1}="B";
(20,-20)*+{\bo_2;}="D";
{\ar@{->}_{F_1} "A";"B"};
{\ar@{->}^{t} "A";"C"};
{\ar@{->}^{t} "B";"D"};
{\ar@{->}^{F_2} "C";"D"};
{\ar@{=>}^\alpha "B";"C"};
\endxy
$$
and
$$
\xy
(0,0)*+{\bo_1}="A";
(20,0)*+{\bo_2}="C";
(0,-20)*+{\bo_1}="B";
(20,-20)*+{\bo_2;}="D";
{\ar@{->}_{F_1} "A";"B"};
{\ar@{->}^{t'} "A";"C"};
{\ar@{->}^{t'} "B";"D"};
{\ar@{->}^{F_2} "C";"D"};
{\ar@{=>}^{\alpha'} "B";"C"};
\endxy
$$
is a datum of:

\begin{itemize}

\item An \emph{adjointable} 2-morphism $t\overset{\delta}\to t'$;

\item A 3-morphism
$$\delta \circ \alpha \overset{\gamma}\to \alpha'\circ \delta,$$
where the two sides are 2-morphisms
$$t\circ F_1\to F_2\circ t'.$$

\end{itemize}

\sssec{}

The functoriality in \secref{sss:trace and comp 3} implies that the trace construction extends to a symmetric monoidal functor
between 2-categories
\begin{equation} \label{e:3-cat Tr funct}
\bL(\bO) \to \Omega(\bO).
\end{equation}

\ssec{The notion of 2-adjointable morphism}

Let $\bO$ be a 3-category.  In this subsection, we will introduce the notion of \emph{2-adjointable morphism},
which will play a key role in the main body of the paper.


\sssec{}

We give the following definition:

\medskip

Let
\begin{equation} \label{e:2b 2d}
t:\bo_1\to \bo_2
\end{equation}
be 1-morphism in $\bO$. We will say that $t$ is \emph{2-adjointable} if

\begin{itemize}

\item $t$ is adjointable;

\item The unit and counit 2-morphisms
$$\on{Id}\to t^R\circ t \text{ and } t\circ t^R\to \on{Id}$$
are adjointable.

\end{itemize}

\sssec{}

We claim:

\begin{lem} \label{l:ambi}
If $t$ is 2-adjointable, then the adjunction $(t,t^R)$ is ambidexterous.
\end{lem}

\begin{proof}

The right adjoints of the unit and the counit of the $(t,t^R)$-adjunction
realize $t^R$ as the left adjoint of $t$.

\end{proof}

\sssec{}

Assume now that $\bO$ carries a symmetric monoidal structure. We will now
explore how the notion of 2-adjointability interacts with the trace construction.

\sssec{}  \label{sss:BC 2 morph}

Let us be given an endomorphism of \eqref{e:2b 2d} as an object of $\on{Arr}(\bO)$,
represented by diagram
\begin{equation} \label{e:basic diag again}
\xy
(0,0)*+{\bo_1}="A";
(20,0)*+{\bo_2}="C";
(0,-20)*+{\bo_1}="B";
(20,-20)*+{\bo_2.}="D";
{\ar@{->}_{F_1} "A";"B"};
{\ar@{->}^{t} "A";"C"};
{\ar@{->}^{t} "B";"D"};
{\ar@{->}^{F_2} "C";"D"};
{\ar@{=>}^\alpha "B";"C"};
\endxy
\end{equation}

\medskip

Let $\alpha^{\on{BC}}$ denote the resulting 
2-morphism
$$F_1\circ t^R\to  t^R\circ F_2.$$

\medskip

Assume that $\alpha^{\on{BC}}$ admits a right adjoint.
 Then we obtain an endomorphism of the object
$$t^R:\bo_2\to \bo_1\in \on{Arr}(\bO)$$ represented by
\begin{equation} \label{e:basic diag adj}
\xy
(0,0)*+{\bo_2}="A";
(25,0)*+{\bo_1}="C";
(0,-20)*+{\bo_2}="B";
(25,-20)*+{\bo_1.}="D";
{\ar@{->}_{F_2} "A";"B"};
{\ar@{->}^{t^R} "A";"C"};
{\ar@{->}^{t^R} "B";"D"};
{\ar@{->}^{F_1} "C";"D"};
{\ar@{=>}^{(\alpha^{\on{BC}})^R} "B";"C"};
\endxy
\end{equation}

\medskip

Assume that $t$ is 2-adjointable. Then by \lemref{l:ambi}, $t^R$ is adjointable, and hence 
the diagram \eqref{e:basic diag adj} gives rise to a 2-morphism
\begin{equation} \label{e:Tr adj}
\Tr(F_2,\bo_2) \overset{\Tr((\alpha^{\on{BC}})^R,t^R)}\longrightarrow \Tr(F_1,\bo_1).
\end{equation}

\medskip

We claim:

\begin{prop} \label{p:adj Tr}
Under the above circumstances, the 1-morphism \eqref{e:Tr adj} is a right adjoint of
\begin{equation} \label{e:Tr funct again}
\Tr(\alpha,t):\Tr(F_1,\bo_1) \to \Tr(F_2,\bo_2).
\end{equation}
\end{prop}

\sssec{Example} \label{sss:BC adj}

Here is a typical situation, when the conditions of \propref{p:adj Tr} are satisfied: this happens when $\alpha$ is an isomorphism
and satisfies the \emph{right Beck-Chevalley condition}, i.e., $\alpha^{\on{BC}}$ is also an isomorphism.

\medskip

The latter is, in turn, the case whenever $F_1$ and $F_2$ are both invertible.

\sssec{Proof of \propref{p:adj Tr}}

We will construct the unit of the adjunction; the counit is constructed similarly.

\medskip

By \eqref{e:3-cat Tr funct}, it suffices to construct a 2-morphism in $\bL(\bO)$ from the 1-morphism represented
$$
\xy
(0,0)*+{\bo_1}="A";
(20,0)*+{\bo_1}="C";
(0,-20)*+{\bo_1}="B";
(20,-20)*+{\bo_1}="D";
{\ar@{->}_{F_1} "A";"B"};
{\ar@{->}^{\on{Id}} "A";"C"};
{\ar@{->}^{\on{Id}} "B";"D"};
{\ar@{->}^{F_1} "C";"D"};
{\ar@{=>}^{\on{id}} "B";"C"};
\endxy
$$
to the composition of 1-morphisms
$$
\xy
(0,0)*+{\bo_1}="A";
(25,0)*+{\bo_2}="C";
(50,0)*+{\bo_1}="E";
(0,-20)*+{\bo_1}="B";
(25,-20)*+{\bo_2}="D";
(50,-20)*+{\bo_1}="F";
{\ar@{->}_{F_1} "A";"B"};
{\ar@{->}_{F_1} "E";"F"};
{\ar@{->}^{t} "A";"C"};
{\ar@{->}^{t^R} "C";"E"};
{\ar@{->}^{t} "B";"D"};
{\ar@{->}^{t^R} "D";"F"};
{\ar@{->}_{F_2} "C";"D"};
{\ar@{=>}^\alpha "B";"C"};
{\ar@{=>}^{(\alpha^{\on{BC}})^R} "D";"E"};
\endxy
$$

In terms of \secref{sss:L for 3}, the corresponding data is given by:

\begin{itemize}

\item The unit of the $(t,t^R)$-adjunction $\on{Id}\to t^R\circ t$;

\item The 3-morphism $\gamma$ in the diagram
\begin{equation} \label{e:3 adj}
\xy
(0,0)*+{F_1}="A";
(60,0)*+{F_1}="B";
(0,-20)*+{t^R\circ t\circ F_1}="C";
(30,-20)*+{t^R\circ F_2\circ t}="D";
(60,-20)*+{F_1\circ t^R\circ t,}="E";
{\ar@{->}_{\on{id}} "A";"B"};
{\ar@{->}_{\on{unit}} "A";"C"};
{\ar@{->}_{\on{unit}} "B";"E"};
{\ar@{->}^{\alpha} "C";"D"};
{\ar@{->}^{(\alpha^{\on{BC}})^R} "D";"E"};
{\ar@{=>}^{\gamma} "B";"C"};
\endxy
\end{equation}
constructed below.

\end{itemize}

\medskip

By the $(\alpha^{\on{BC}},(\alpha^{\on{BC}})^R)$-adjunction, the datum of $\gamma$ is equivalent to the datum
of a 3-morphism $\gamma'$ in the diagram
\begin{equation} \label{e:3 adj bis}
\xy
(0,0)*+{F_1}="A";
(30,0)*+{F_1\circ t^R\circ t}="B";
(0,-20)*+{t^R\circ t\circ F_1}="C";
(30,-20)*+{t^R\circ F_2\circ t.}="D";
{\ar@{->}_{\on{unit}} "A";"B"};
{\ar@{->}_{\on{unit}} "A";"C"};
{\ar@{->}^{\alpha} "C";"D"};
{\ar@{->}^{\alpha^{\on{BC}}} "B";"D"};
{\ar@{=>}^{\gamma'} "B";"C"};
\endxy
\end{equation}

However, by the definition of $\alpha^{\on{BC}}$, the diagram of 1-morphisms in \eqref{e:3 adj bis} actually commutes,
and we take $\gamma'$ to be the natural isomorphism.

\qed[\propref{p:adj Tr}]

\sssec{}

Let us return to the setting of \secref{sss:BC adj}: i.e., $t$ is 2-adjointable, $\alpha$ is an isomorphism, and so is $\alpha^{\on{BC}}$. 
Assume now that $t$ is fully faithful, i.e. the unit of the adjunction
$$\on{Id}\to t^R\circ t$$
is an isomorphism.

\medskip

We claim:

\begin{cor} \label{c:Tr ff}
Under the above circumstances, the unit of the adjunction of \propref{p:adj Tr} is an isomorphism;
in particular, the 2-morphism \eqref{e:Tr funct again} is fully faithful.
\end{cor}

\begin{proof}

We need to show that under the conditions of the corollary, the 3-morphism \eqref{e:3 adj}
is an isomorphism. However, this follows from the construction.

\end{proof}

\ssec{Interaction of 2-adjointable morphisms and 2-dualizability}

In this subsection we will again use the notion of 2-adjointable morphism, but for a different purpose.

\medskip

We continue to
assume that $\bO$ is a symmetric monoidal 3-category.

\sssec{}

The following is an extension of \lemref{l:dualizability}:

\begin{prop} \label{p:2-dual}
An object $t:\bo_1\to \bo_2$ of $\on{Arr}(\bO)$ is 2-dualizable if and only if the following conditions
hold:

\begin{itemize}

\item $\bo_1,\bo_2$ are 2-dualizable as objects of $\bO^{2-\on{Cat}}$;

\item $t$ is 2-adjointable.

\end{itemize}
\end{prop}

\sssec{Example}

Let $\bo_1=\bo_2=\one_\bO$, so that we can think of $t$
as an object of $\Omega(\bO)$. Then $t$ is 2-adjointable as a 1-morphism in $\bO$ if and only if
$t$ is 2-dualizable as an object of $\Omega(\bO)$.

\medskip

Indeed, this follows from the fact that the composition and the tensor product operations in $\Omega(\bO)$
coincide.

\begin{proof}[Proof of \propref{p:2-dual}]

We will give a construction in one direction. Namely, we will show that
$$(t:\bo_1\to \bo_2)\in \on{Arr}(\bO)$$ is 2-dualizable
assuming that the conditions of the proposition hold.

\medskip

By \lemref{l:dualizability}, the dual object is given by
$$(t^R)^\vee:\bo_1^\vee\to \bo_2^\vee,$$
and the unit and counit for the duality are given by
\begin{equation} \label{e:unit}
\xy
(0,0)*+{\one_\bO}="A";
(30,0)*+{\one_\bO}="C";
(0,-20)*+{\bo_1\otimes \bo_1^\vee}="B";
(30,-20)*+{\bo_2\otimes \bo_2^\vee}="D";
{\ar@{->}_{\on{u}_{\bo_1}} "A";"B"};
{\ar@{->}^{\on{Id}} "A";"C"};
{\ar@{->}^{t\otimes (t^R)^\vee} "B";"D"};
{\ar@{->}^{\on{u}_{\bo_2}} "C";"D"};
{\ar@{=>}^{\alpha'} "B";"C"};
\endxy
\end{equation}
and
\begin{equation} \label{e:counit}
\xy
(0,0)*+{\bo_1\otimes \bo_1^\vee}="A";
(30,0)*+{\bo_1\otimes \bo_1^\vee}="C";
(0,-20)*+{\one_\bO}="B";
(30,-20)*+{\one_\bO,}="D";
{\ar@{->}_{\on{ev}_{\bo_1}} "A";"B"};
{\ar@{->}^{\on{Id}} "B";"D"};
{\ar@{->}^{t\otimes (t^R)^\vee} "A";"C"};
{\ar@{->}^{\on{ev}_{\bo_2}} "C";"D"};
{\ar@{=>}^{\alpha''} "B";"C"};
\endxy
\end{equation}
where $\alpha'$ and $\alpha''$ are given by
$$(t\otimes (t^R)^\vee)\circ \on{u}_{\bo_1}\simeq
((t\circ t^R)\otimes \on{Id})\circ \on{u}_{\bo_2} \overset{\gamma'\otimes \on{id}}\longrightarrow \on{u}_{\bo_2}$$
and
$$\on{ev}_{\bo_1} \overset{\gamma''\otimes \on{id}}\longrightarrow  \on{ev}_{\bo_1}\circ ((t^R\circ t)\otimes \on{Id})\simeq
\on{ev}_{\bo_2}\circ (t\otimes (t^R)^\vee),$$
where $\gamma'$ and $\gamma''$ are the counit and unit 2-morphisms for the $(t,t^R)$-adjunction, respectively.

\medskip

We need to show that \eqref{e:unit} and \eqref{e:counit} are adjointable 1-morphisms in $\on{Arr}(\bO)$. However,
this follows from \lemref{l:adj in Arr}:

\medskip

Indeed, $\on{ev}_{\bo_i}$ and $\on{u}_{\bo_i}$ are adjointable, by assumptions. Now, the 
2-morphisms $\alpha'$ and $\alpha''$ are obtained from the 2-morphisms $\gamma'$ and $\gamma''$,
respectively, while the latter are adjointabile by assumption.

\end{proof}

\sssec{} \label{sss:2 endos}

Let
$$t:\bo_1\to \bo_2$$
be an object of $\on{Arr}(\bO)$, endowed with an endomorhism $F$.

\medskip

By definition, this datum consists of a diagram
$$
\xy
(0,0)*+{\bo_1}="A";
(20,0)*+{\bo_2}="C";
(0,-20)*+{\bo_1}="B";
(20,-20)*+{\bo_2;}="D";
{\ar@{->}_{F_1} "A";"B"};
{\ar@{->}^{t} "A";"C"};
{\ar@{->}^{t} "B";"D"};
{\ar@{->}^{F_2} "C";"D"};
{\ar@{=>}^{\alpha} "B";"C"};
\endxy
$$

\sssec{} \label{sss:cond for adj}

Assume now that:

\begin{itemize}

\item $(t:\bo_1\to \bo_2)$ is 2-dualizable, i.e., by \propref{p:2-dual}: 

\begin{itemize}

\item $\bo_1,\bo_2$ are 2-dualizable;

\item $t$ is 2-adjointable;

\end{itemize}

\item $F$ is adjointable, i.e., by \lemref{l:adj in Arr}, 

\begin{itemize}

\item $F_1,F_2$ are adjointable;

\item The 2-morphism $\alpha$ is adjointable.

\end{itemize}

\end{itemize}

\medskip

In this case, we can form the diagrams
$$
\xy
(0,0)*+{\Tr(F_1,\bo_1)}="A";
(40,0)*+{\Tr(F_2,\bo_2)}="C";
(0,-20)*+{\Tr(F_1,\bo_1)}="B";
(40,-20)*+{\Tr(F_2,\bo_2)}="D";
{\ar@{->}^{\Tr(\alpha,t)} "A";"C"};
{\ar@{->}_{\on{Id}} "A";"B"};
{\ar@{->}_{\Tr(\alpha,t)} "B";"D"};
{\ar@{->}^{\on{Id}} "C";"D"};
{\ar@{=>}^{\on{id}} "B";"C"};
\endxy
$$
and
$$
\xy
(0,0)*+{\Tr(\on{Id},\bo_1)}="A";
(40,0)*+{\Tr(\on{Id},\bo_2)}="C";
(0,-20)*+{\Tr(\on{Id},\bo_1)}="B";
(40,-20)*+{\Tr(\on{Id},\bo_2).}="D";
{\ar@{->}^{\Tr(\on{id},t)} "A";"C"};
{\ar@{->}_{\Tr(\on{id},F_1)} "A";"B"};
{\ar@{->}_{\Tr(\on{id},t)} "B";"D"};
{\ar@{->}^{\Tr(\on{id},F_2)} "C";"D"};
{\ar@{=>}^\alpha "B";"C"};
\endxy
$$

Furthermore, applying the functor \eqref{e:3-cat Tr funct}, and using \lemref{l:Trace dualizable} and \propref{p:2-dual}, we obtain that:

\medskip

\begin{itemize}

\item The objects $\Tr(F_1,\bo_1),\Tr(\on{Id},\bo_1),\Tr(F_2,\bo_2),\Tr(\on{Id},\bo_2) \in \Omega(\bO)$
are dualizable;

\medskip

\item The 2-morphisms $\Tr(\on{id},t)$ and $\Tr(\alpha,t)$ (i.e., 1-morphisms in $\Omega(\bO)$)
are adjointable.

\end{itemize}

\medskip

Hence, we have well-defined morphisms in $\Omega^2(\bO)$:
\begin{equation} \label{e:2 traces 3-categ'}
\Tr(\on{Id},\Tr(F_1,\bo_1)) \overset{t}\to \Tr(\on{Id},\Tr(F_2,\bo_2))
\end{equation}
and
\begin{equation} \label{e:2 traces 3-categ''}
\Tr(\Tr(\on{id},F_1),\Tr(\on{Id},\bo_1)) \overset{t}\to \Tr(\Tr(\on{id},F_2),\Tr(\on{Id},\bo_2)).
\end{equation}

\sssec{}

Finally, applying \corref{c:2 traces} to the object
$$(\bo_1\overset{t}\to \bo_2)\in \on{Arr}(\bO),$$
equipped with the endomorphisms given by $(F,\alpha)$ and $(\on{Id},\on{id})$, we obtain:

\begin{cor} \label{c:2-dual}
The following diagram in $\Omega^2(\bO)$ commutes:
$$
\CD
\Tr(\on{Id},\Tr(F_1,\bo_1))  @>{t}>>  \Tr(\on{Id},\Tr(F_2,\bo_2)) \\
@V{\text{\corref{c:2 traces}}}V{\sim}V  @V{\sim}V{\text{\corref{c:2 traces}}}V \\
\Tr(\Tr(\on{id},F_1),\Tr(\on{Id},\bo_1)) @>{t}>> \Tr(\Tr(\on{id},F_2),\Tr(\on{Id},\bo_2)).
\endCD
$$
\end{cor}

\section{Categorical representations} \label{s:Cat Rep}

In this section we define the main object of study of this paper--the notion
of categorical representation of a group.

\ssec{The 3-category of 2-DG categories}

In this subsection we will be referring to the terminology and notations from \cite[Sect. A.2]{GRV}.

\medskip

Our goal here
is to define the ambient 3-category, to be denoted $\tAGCat$, in which (the 2-category of) representations of
a given algebraic group will be an object.

\sssec{}

For a regular cardinal $\kappa$, let $\Cat^{\on{pres}}_\kappa$ be as in \cite[Sect. A.2.5]{GRV}. 
Let
\begin{equation} \label{e:DGCat}
\DGCat_\kappa\in \Cat^{\on{pres}}_\kappa
\end{equation}
be as in \cite[Sect. A.2.8]{GRV}.

\medskip

We regard $\Cat^{\on{pres}}_\kappa$ is a symmetric monoidal category with respect to the Lurie tensor product. 
The object \eqref{e:DGCat} has a natural structure of symmetric monoidal algebra in it.

\medskip

In particular, $\DGCat_\kappa$ acquires a structure of symmetric monoidal category. 

\medskip

Since $\DGCat_\kappa$ is closed as a monoidal category, the assignment
$$\bC_1,\bC_2\in \DGCat \mapsto \uHom_{\DGCat_\kappa}(\bC_1,\bC_2)\in \DGCat_\kappa \mapsto \uHom_{\DGCat}(\bC_1,\bC_2)\in \Cat$$
makes $\DGCat_\kappa$ into a (symmetric monoidal) $(\infty,2)$-category.

\sssec{}

We set
$$\tDGCat_\kappa:=\DGCat_\kappa\mod(\Cat^{\on{pres}}_\kappa)\in \Cat_{\on{large}}.$$

Recall that in \cite[Sect. A.2.8]{GRV}, we define 
$$\DGCat:=\underset{\kappa}{\on{colim}}\, \DGCat_\kappa,$$
where the colimit is taken in $\Cat_{\on{large}}$. 

\medskip

Similarly, set 
$$\tDGCat:=\underset{\kappa}{\on{colim}}\, \tDGCat_\kappa\in \Cat_{\on{large}}.$$

This is a (closed) symmetric monoidal category. Its unit is $\DGCat$, viewed as a module over itself.

\medskip

For a pair of objects $\fO_1,\fO_2\in \tDGCat$, we let
$$\uHom_{\tDGCat}(\fO_1,\fO_2)\in \tDGCat$$
the corresponding internal Hom object.

\sssec{Example}

Let $\bA$ be an associative algebra object in $\DGCat$. We will associate to it the object
$$\bA\mod(\DGCat)\in \tDGCat.$$

\medskip

For a pair of algebras $\bA_1,\bA_2$, we have
$$\uHom_{\tDGCat}(\bA_1\mod(\DGCat),\bA_2\mod(\DGCat))\simeq
(\bA_1^{\on{rev}}\otimes \bA_2)\mod(\DGCat).$$

\sssec{}

For an object $\fO\in \tDGCat$, we can consider the underlying large category (which we denote by the same symbol $\fO$).
For a pair of objects $\bC_1,\bC_2\in \fO$ we have a well-defined
\emph{relative internal Hom}
$$\ul\Hom_{\fO,\DGCat}(\bC_1,\bC_2)\in \DGCat,$$
characterized by the property that for $\bC\in \DGCat$
$$\Maps_{\DGCat}(\bC,\ul\Hom_{\fO,\DGCat}(\bC_1,\bC_2)) \simeq \Maps_{\fO}(\bC\otimes \bC_1,\bC_2)$$
as objects in $\on{Spc}$.

\sssec{}

The assignment
$$\bC_1,\bC_2\in \fO \mapsto \uHom_{\fO,\DGCat}(\bC_1,\bC_2)\in \DGCat \mapsto \uHom_{\fO,\DGCat}(\bC_1,\bC_2)\in \Cat$$
upgrades $\fO$ to an $(\infty,2)$-category, i.e., we obtain a forgetful functor
$$\tDGCat\to \tCat.$$

\medskip

Now, the assignment
$$\fO_1,\fO_2\in \tDGCat \mapsto \uHom_{\tDGCat}(\fO_1,\fO_2)\in \tDGCat  \mapsto  \uHom_{\tDGCat}(\fO_1,\fO_2)\in \tCat$$
makes $\tDGCat$ into a (closed, symmetric monoidal) $(\infty,3)$-category.

\sssec{}

Recall now (see \cite[Sect. 2.3.2]{GRV}) the commutative algebra object
$$\AGCat \in \Cat_{\on{large}}.$$
It is equipped with a homomorphism (of commutative algebra objects in $\Cat_{\on{large}}$)
$$\bi:\DGCat\to \AGCat,$$
see \cite[Sect. 2.4.1]{GRV}.

\medskip

By its construction, $\AGCat$ naturally upgrades to a commutative algebra object in $\tDGCat$, and $\bi$
upgrades to a homomorphism of commutative algebras in $\tDGCat$, supplying $\AGCat$ with a unit.

%
%
%
%
%

\sssec{}

Finally, set
$$\tAGCat\simeq \AGCat\mod(\tDGCat).$$


\sssec{}

The structure of (closed, symmetric monoidal) $(\infty,3)$-category on $\tDGCat$ induces one on $\tAGCat$.
Its unit object is $\AGCat$, viewed as a module over itself.

\sssec{Example}

Let $\ul\bA$ be an associative algebra object in $\AGCat$. We will associate to it the object
$$\ul\bA\mod(\AGCat)\in \tAGCat.$$

\medskip

For a pair of algebras $\ul\bA_1,\ul\bA_2$, we have
$$\uHom_{\tAGCat}(\ul\bA_1\mod(\AGCat),\ul\bA_2\mod(\AGCat))\simeq
(\ul\bA_1^{\on{rev}}\otimes \ul\bA_2)\mod(\AGCat).$$

\sssec{} \label{sss:mmod}

For an object $\ul\fO\in \tAGCat$, we can consider the underlying object of $\tDGCat$, which we denote by the same symbol $\ul\fO$.
For a pair of objects $\ul\bC_1,\ul\bC_2\in \ul\fO$ we have a well-defined
\emph{relative internal Hom}
$$\ul\Hom_{\ul\fO,\AGCat}(\ul\bC_1,\ul\bC_2)\in \AGCat,$$
characterized by the property that for $\ul\bC\in \AGCat$
$$\uHom_{\AGCat,\DGCat}(\ul\bC,\ul\Hom_{\ul\fO,\AGCat}(\ul\bC_1,\ul\bC_2)) \simeq \Maps_{\ul\fO,\DGCat}(\ul\bC\otimes \ul\bC_1,\ul\bC_2)$$
as objects of $\DGCat$.

\sssec{Notational conventions}

In the rest of the paper, we will (do our best to) adhere to the following conventions:

\medskip

For a monoidal $(\infty,1)$-category $\bO$ and $A\in \on{AssocAlg}(\bO)$ we will denote by
$$A\mod(\bO)$$
the usual $(\infty,1)$-category of $A$-modules in $\bO$.

\medskip

For a monoidal $(\infty,2)$-category $\fO$ and $\bA\in \on{AssocAlg}(\fO)$ we will denote by
$$\bA\mmod(\fO)$$
the $(\infty,2)$-category of $\bA$-modules in $\fO$. I.e., we use boldface ($\mmod$ vs $\mod$)
to emphasize the level of categoricity.

\medskip

In addition, we will differentiate objects of $\AGCat$ from objects of $\DGCat$ by using the underline
($\ul\bC$ vs $\bC$) for the former.

\sssec{The tensoring up operation}

We have a pair of adjoint functors
\begin{equation} \label{e:DGCat  AGCat adj}
\tDGCat\rightleftarrows \tAGCat,
\end{equation}
given by restriction and tensoring up, respectively.

\medskip

In particular, for $\fE\in \tDGCat$, we can consider the object
\begin{equation} \label{e:ten AGCat}
\AGCat\underset{\DGCat}\otimes \fE\in \tAGCat,
\end{equation}
and further
$$\Res^{\AGCat}_{\DGCat}(\AGCat\underset{\DGCat}\otimes \fE)\in \tDGCat,$$
which, by a slight abuse of notation, we denote by the same symbol \eqref{e:ten AGCat}.

\medskip

The unit of the \eqref{e:DGCat  AGCat adj} gives rise to a functor
\begin{equation} \label{e:to ten AGCat}
(\bi\otimes \on{Id}_{\fE}):\fE\to \AGCat\underset{\DGCat}\otimes \fE\in \tAGCat
\end{equation}
in $\tDGCat$.

\medskip

Note, however, that the functor $\bi:\DGCat\to \AGCat$ is fully faithful and admits
a $\DGCat$-linear right adjoint, denoted $\be$, see \cite[Sect. 2.4.1]{GRV}. This implies
that the functor \eqref{e:to ten AGCat} is fully faithful and also admits a right adjoint,
namely,
$$\fE\leftarrow \AGCat\underset{\DGCat}\otimes \fE:(\be\otimes \on{Id}_{\fE}).$$

Moreover, by \cite[Corollary 4.5.6]{GRV}, the above functor \eqref{e:to ten AGCat}
is \emph{2-adjointable}. This implies that the
$$(\bi\otimes \on{Id}_{\fE})\leftrightarrow (\be\otimes \on{Id}_{\fE})$$
adjunction is actually ambidexterous.

\ssec{The notion of categorical representation}

\sssec{}

Recall that to any $X\in \Sch$ we assign an object
\begin{equation} \label{e:shv under}
\ul\Shv(X)\in \AGCat.
\end{equation}

Since any $X$ naturally upgrades to a cocommutative coalgebra object in $\Sch$,
the object $\ul\Shv(X)$ is naturally a cocommutative coalgebra in $\AGCat$.

\begin{rem}

Note that the latter uses the fact that the functor
$$\ul\Shv(-):\Sch\to \AGCat$$ is \emph{strictly} symmetric
monoidal. This is the key difference between $\ul\Shv(-)$ and
the functor
$$\Shv:\Sch\to \DGCat,$$
which is only right-lax symmetric monoidal.

\end{rem}

\sssec{}

Let now $G$ be an algebraic group, i.e., a group-object in $\Sch$. We obtain that
$\ul\Shv(G)$ carries a structure of \emph{cocommutative} Hopf algebra in $\AGCat$.

\sssec{}

Recall also that for every $X\in \Sch$, the object \eqref{e:shv under} is canonically
self-dual.

\medskip

In particular, by duality, $\ul\Shv(G)$ obtains a structure of \emph{commutative} Hopf algebra.
In order to avoid confusion, we will notationally differentiate the two as follows:

\begin{itemize}

\item We will keep the symbol $\ul\Shv(G)$ for the (commutative) Hopf algebra structure, in which the
\emph{comultiplication} is induced by the functor $\on{mult}^!$;

\medskip

\item We will use the symbol $\ul\Shv(G)_{\on{co}}$ for the (cocommutative) Hopf algebra structure, in which the
\emph{multiplication} is induced by the functor $\on{mult}_*$.

\end{itemize}

Above and in the rest of the paper,
$$\on{mult}:G\times G\to G$$
is the group structure on $G$.

\medskip

Note that the functors $\on{mult}^!$ and $\on{mult}_*$ are each others' duals, as they should be.

\sssec{}

Set
$$G\mmod:=\ul\Shv(G)_{\on{co}}\mmod(\AGCat)\in \tAGCat.$$

Equivalently,
$$G\mmod\simeq \ul\Shv(G)\commod(\AGCat).$$

\medskip

For $\ul\bC\in G\mmod$, we will denote by
$$\on{act}:\ul\Shv(G)_{\on{co}}\otimes \ul\bC\to \ul\bC$$
the corresponding 1-morphism that encodes the module structure.

%

\sssec{Example}

Let $\CY$ be a prestack equipped with an action of $G$. Then pullback with respect to the action map
$G\times \CY\to \CY$ makes $\ul\Shv(\CY)$ an object of $G\mmod$, thought of as $\ul\Shv(G)\commod(\AGCat)$.

\medskip

The action of $\ul\Shv(G)_{\on{co}}$ is given by pushforward along the (schematic) map $G\times \CY\to \CY$.

\medskip

Let $f:\CY_1\to \CY_2$ be a map between prestacks acted on by $G$. The functor
$$f^!:\ul\Shv(\CY_2)\to \ul\Shv(\CY_1)$$
naturally upgrades to a 1-morphism in $G\mmod$.

\medskip

Assume now that $\CY_1$ and $\CY_2$ are \emph{AG tame} (see \cite[Sect. 3.4.7]{GRV}). In this
case we have the 1-morphism
$$f_\blacktriangle:\ul\Shv(\CY_1)\to \ul\Shv(\CY_1)$$
(see \cite[Sect. 3.7.3]{GRV}), which also naturally upgrades to a 1-morphism in $G\mmod$.

\medskip

In the sequel, given a diagram of AG tame prestacks acted on by $G$
$$
\CD
\CY_{1,2} @>{h}>> \CY_2 \\
@V{v}VV \\
\CY_1,
\endCD
$$
we will refer to the 1-morphism $h_\blacktriangle\circ v^!$ in $G\mmod$
as ``pull-push" along this diagram.

\sssec{}

Let us write down explicitly the data involved in specifying an object $\ul\bC\in G\mmod$ (from the comodule perspective).

\medskip

First, $\ul\bC$ is an object of $\AGCat$, i.e., an assignment
$$(X\in \Sch)\rightsquigarrow (\ul\bC(X)\in \DGCat),$$
with the functorialities specified in \cite[Sect. 0.1.4]{GRV}.

\medskip

The action of $G$ is given by a map
$$\ul\bC\overset{\on{coact}}\to \ul\bC\otimes \ul\Shv(G),$$
equipped with a data of associativity.

\medskip

The above coaction functor amounts to a collection of functors
\begin{equation} \label{e:coact X}
\ul\bC(X) \overset{\on{coact}_X}\to (\ul\bC\otimes \ul\Shv(G))(X)\simeq \ul\bC(X\times G),
\end{equation}
(functorial in $X$), and the associativity is encoded by the commutativity of the diagrams
\begin{equation} \label{e:assoc X}
\CD
\ul\bC(X)  @>{\on{coact}_X}>> \ul\bC(X\times G) \\
@V{\on{coact}_X}VV @VV{\on{mult}^!}V \\
\ul\bC(X\times G)  @>{\on{coact}_{X\times G}}>> \ul\bC(X\times G\times G),
\endCD
\end{equation}
and diagrams that involve higher powers of $G$ for the data of higher coherence.

\sssec{}

The structure on $\ul\Shv(G)$ (resp., $\ul\Shv(G)_{\on{co}}$) of
cocommutative (resp., commutative) Hopf algebra induces a symmetric monoidal
structure on $G\mmod$, compatible with the forgetful functor
$$\boblv_G:G\mmod\to \AGCat.$$

The unit object in $G\mmod$ is $\ul\Vect$, equipped with the trivial action.

\sssec{}

The functor $\boblv_G$ admits a left and a right adjoint. The left adjoint is given by
$$\ul\bD\mapsto \ul\Shv(G)_{\on{co}}\otimes \ul\bD\in \ul\Shv(G)_{\on{co}}\mmod(\AGCat)=G\mmod.$$

The right adjoint is given by
$$\ul\bD\mapsto \bD\otimes \ul\Shv(G)\in \ul\Shv(G)\commod(\AGCat)=G\mmod.$$

In particular, we obtain that this adjunction is \emph{ambidexterous}.

\sssec{}

Let $G_1$ and $G_2$ be a pair of groups. Since
$$\ul\Shv(G_1)\otimes \ul\Shv(G_2)\simeq \ul\Shv(G_1\times G_2)$$
as cocommutative Hopf algebras, we obtain that the natural functor
\begin{equation} \label{e:prod groups reps}
G_1\mmod\otimes G_2\mmod\to (G_1\times G_2)\mmod
\end{equation}
is an equivalence in $\tAGCat$.

\sssec{}

One can ask the following question: are there (interesting) examples of objects of $G\mmod$, such that
the underlying object of $\AGCat$ is \emph{restricted}, i.e., lies in the essential image of the functor
$$\bi:\DGCat\to \AGCat?$$

Note that the data of such an object amounts, in particular, to $\bC\in \DGCat$, equipped with a coaction functor
\begin{equation} \label{e:coact naive}
\bC\overset{\on{coact}}\to \bC\otimes \Shv(G),
\end{equation}
which makes the following diagram commute:
\begin{equation} \label{e:assoc X naive}
\CD
\bC  & @>{\on{coact}}>> & \bC\otimes \Shv(G) \\
@V{\on{coact}}VV & & @VV{\on{mult}^!}V \\
\bC\otimes \Shv(G)  @>{\on{coact}\otimes \on{Id}}>> \bC\otimes \Shv(G)\otimes \Shv(G) @>{\on{Id}\otimes (\boxtimes)}>> \bC\otimes \Shv(G\times G).
\endCD
\end{equation}

\medskip

When $G$ is connected, we will (almost) classify such objects in \secref{ss:mon 1}.

\sssec{}

We end this subsection with the following observation:

\begin{lem} \label{l:adj in G mod}
Let $\ul\bC_1\to \ul\bC_2$ be a 1-morphism in $G\mmod$. Then it admits a right/left
adjoint if and only if the induced 1-morphism
$$\boblv_G(\ul\bC_1)\to \boblv_G(\ul\bC_2)$$
does.
\end{lem}

\begin{proof}

We will give a proof for right adjoints; the proof for left adjoints is similar.
Suppose that $F:\ul\bC_1\to \ul\bC_2$ is such that
$$\ul\bC_1\leftarrow \ul\bC_2:F^R$$
(of the underlying objects of $\AGCat$) exists. Then a priori, $F^R$ has a structure of \emph{right-lax} compatibility with
the action of $\ul\Shv(G)_{\on{co}}$. But we claim that this structure is automatically strict.

\medskip

This claim amounts to the following. Let $X$ be a scheme and let $g$ be an $X$-point
of $G$. This point gives rise to an automorphism
$$g\cdot (-):\bC_i(X)\to \bC_i(X), \quad i=1,2.$$
The right-lax compatibility structure says that we have a natural transformation
$$(g\cdot (-))\circ F^R \overset{\alpha_g}\to F^R\circ (g\cdot (-)).$$
We need to show that $\alpha_g$ is invertible.

\medskip

We will write down the inverse of $\alpha_g$ explicitly. Namely, it is given by
\begin{multline*}
F^R\circ (g\cdot (-)) \simeq (g\cdot (-)) \circ (g^{-1}\cdot (-)) \circ F^R \circ (g\cdot (-))
\overset{\alpha_{g^{-1}}}\to \\
\to (g\cdot (-)) \circ F^R\circ (g^{-1}\cdot (-)) \circ (g\cdot (-)) \simeq
(g\cdot (-)) \circ F^R.
\end{multline*}

\end{proof}

\begin{rem}

Note that for the validity of \lemref{l:adj in G mod}, it is essential that $G$ is a group rather than a monoid.

\end{rem}

\ssec{Invariants and coinvariants}

In this subsection we start investigating the basic structures on $G\mmod$, by which we mean that we start probing
it by 1-morphisms from other (simpler) objects of $\tAGCat$.

\sssec{}

Let
$$\btriv_G:\AGCat\to G\mmod$$
be the the functor of trivial representation.

\medskip

This functor admits a right and a left adjoint, denoted $\binv_G$ and $\bcoinv_G$, respectively.
They are explicitly computed as follows:

\begin{itemize}

\item $$\binv_G(\ul\bC)\simeq  \ul\bC\overset{\ul\Shv(G)}\otimes \ul\Vect:=
\on{Tot}(\ul\bC\otimes \ul\Shv(G^\bullet)),$$ where $\ul\bC\otimes \ul\Shv(G^\bullet)$ is the \emph{cosimplicial} object
of $\AGCat$ that encodes the $\ul\Shv(G)$-coaction on $\ul\bC$.

\medskip

Note that since limits in $\AGCat$ are computed value-wise, we have
$$\binv_G(\ul\bC)(X)\simeq \on{Tot}(\ul\bC(X\times G^\bullet)),$$
where the latter totalization is taken in $\DGCat$.

\bigskip

\item $$\bcoinv_G(\ul\bC)\simeq \ul\bC\underset{\ul\Shv(G)_{\on{co}}}\otimes \ul\Vect\simeq
|\ul\bC\otimes \ul\Shv(G^\bullet)_{\on{co}}|,$$ where $\ul\bC\otimes \ul\Shv(G^\bullet)_{\on{co}}$ is the \emph{simplicial} object
of $\AGCat$ that encodes the $\ul\Shv(G)_{\on{co}}$-action on $\ul\bC$.

\medskip

Note that since colimits on $\AGCat$ are also computed value-wise, we have
$$\bcoinv_G(\ul\bC)(X)\simeq |\ul\bC(X\times G^\bullet)|,$$
where the latter geometric realization is taken in $\DGCat$.

\end{itemize}

\sssec{Example}

Let $\CY$ be a prestack with an action of $G$, and consider the corresponding object
$$\ul\Shv(\CY)\in G\mmod.$$

Unwinding, we obtain that
\begin{equation} \label{e:inv as quot}
\binv_G(\ul\Shv(\CY))\simeq \ul\Shv(\CY/G).
\end{equation}

\sssec{}

For a pair of objects $\ul\bC_1,\ul\bC_2\in G\mmod$, we can form the objects
$$\ul\Hom_{\AGCat}(\ul\bC_1,\ul\bC_2) \text{ and }
\ul\bC_1\otimes \ul\bC_2$$
in $\AGCat$, and both naturally upgrade to objects of $G\mmod$ using the Hopf algebra structure.

\medskip

\noindent{\it Notational convention}: In the above formulas, by a slight abuse of notation, we write $\ul\bC$ instead of $\boblv_G(\ul\bC)$,
and we will do so in the rest of the paper unless this is likely to cause a confusion.

\medskip

Unwinding, we obtain:
\begin{equation} \label{e:inv Funct}
\binv_G(\ul\Hom_{\AGCat}(\ul\bC_1,\ul\bC_2))
\simeq \uHom_{G\mmod,\AGCat}(\ul\bC_1,\ul\bC_2)
\end{equation}
(where $\uHom_{G\mmod,\AGCat}(-,-)$ denotes the internal Hom in $G\mmod$ relative to $\AGCat$,
see \cite[Sect. 1.1.2]{GRV}), and
\begin{equation} \label{e:coinv ten}
\bcoinv_G(\ul\bC_1\otimes \ul\bC_2)\simeq
\ul\bC_1\underset{\ul\Shv(G)_{\on{co}}}\otimes \ul\bC_2.
\end{equation}

\sssec{}

Consider the forgetful functor (in $\AGCat$):\footnote{We draw the reader's attention to the distinction between the notations
$\boblv_G$ and $\oblv_G$: the former is a functor between 2-categories (that forgets the categorical $G$-action); the former
is a functor between 1-categories (that forgets the structure of $G$-equivariance, which is a generalization of the notion of usual
$G$-action on an object of a category).}
$$\oblv_G:\binv_G(\ul\bC)\to \ul\bC.$$

\medskip

It is easy to see that it admits a right adjoint (in $\AGCat$), to be denoted $\Av^G_*$.

\medskip

Explicitly, the composition of $\Av^G_*$ with $\oblv_G$ is the comonad on $\ul\bC$, given by
$$\ul\bC \simeq \ul\Vect\otimes \ul\bC\to  \ul\Shv(G)_{\on{co}}\otimes \ul\bC \overset{\on{act}}\to \ul\bC,$$
where the second arrow is the 1-morphism induced by
$$\ul\Vect\to \ul\Shv(G)_{\on{co}}\simeq \ul\Shv(G) \, \Leftrightarrow \, \Vect\to \Shv(G),$$
given by the constant sheaf $\on{const}_G\in \Shv(G)$.

\begin{rem}

One can view the $(\oblv_G,\Av^G_*)$-adjunction as follows:

\medskip

We have an adjunction
\begin{equation} \label{e:triv adj}
\ul\Vect \underset{\on{C}^\cdot(G,-)}{\overset{\on{const}_G}\rightleftarrows} \ul\Shv(G)
\end{equation}
in $G\mmod$, and the adjunction
$$\binv_G(\ul\bC) \underset{\Av^G_*}{\overset{\oblv_G}\rightleftarrows} \ul\bC$$
by taking
$$\ul\Hom_{G\mmod,\AGCat}(-,\ul\bC)$$
from the two sides of \eqref{e:triv adj}.

\end{rem}

\sssec{}

The above 1-morphism
$$\Av^G_*:\ul\bC\to \binv_G(\ul\bC)$$
is $G$-invariant. Hence, it canonically factors as
$$\ul\bC\to \bcoinv_G(\ul\bC) \to \binv_G(\ul\bC).$$

In particular, we obtain a natural transformation
\begin{equation} \label{e:coinv to inv}
\bcoinv_G \to \binv_G
\end{equation}
as functors $G\mmod\to \AGCat$.

\sssec{}

We claim:

\begin{prop} \label{p:coinv to inv}
The natural transformation \eqref{e:coinv to inv} is an isomorphism.
\end{prop}

\begin{proof}

When calculating
$$ \on{Tot}(\ul\bC(X\times G^\bullet)) \text{ and } |\ul\bC(X\times G^\bullet)|,$$
we can replace the cosimplicial limit (resp., simplicial colimit) by the corresponding
semi-cosimplicial (resp., semi-simplicial) one:
\begin{equation} \label{e:tot and geom}
\on{Tot}(\ul\bC(X\times G^\bullet))^{\on{semi}} \text{ and } |\ul\bC(X\times G^\bullet)|^{\on{semi}}.
\end{equation}

Note now that the functor of cohomological shift by $[-2\cdot \bullet \cdot \dim(G)]$ gives rise to an isomorphism
between the semi-cosimplicial category $\ul\bC(X\times G^\bullet)$ in the left-hand side of \eqref{e:tot and geom}
with one obtained from the semi-simplicial category $\ul\bC(X\times G^\bullet)$ in the right-hand side of \eqref{e:tot and geom}
by passing to right adjoints.

\medskip

By \cite[Chapter 1, Proposition 2.5.7]{GR}, this gives rise to \emph{an} equivalence between the two sides of \eqref{e:tot and geom}.
However, unwinding the definitions, it is easy to see that the resulting functor $\leftarrow$ in \eqref{e:tot and geom}
is the value of \eqref{e:coinv to inv} on $X$.

\end{proof}

\sssec{}

As a corollary, we obtain:

\begin{cor} \label{c:inv colimits}
The functor
$$\binv_G:G\mmod\to \AGCat$$
commutes with colimits and with tensors by objects of $\AGCat$.
\end{cor}

\ssec{Dualilzability for categorical representations}

Recall that our ability to perform the trace construction necessitated the dualizability
assumption on objects.

\medskip

In this subsection we will see what this condition means for $G\mmod$
(at the 3-categorical and 2-categorical levels).

\sssec{} \label{sss:G-mod dualizable}

First, we claim that
$$G\mmod\in \tAGCat$$
is dualizable and self-dual:
\begin{equation} \label{e:G mod self-dual}
(G\mmod)^\vee\simeq G\mmod.
\end{equation}

\medskip

Indeed, this is the case for objects of the form
$$A\mod(\bO)\in \bO\mmod,$$
where:

\begin{itemize}

\item $\bO$ is a symmetric monoidal category, so $\bO\mmod$ is a symmetric monoidal 2-category;

\item $A$ is an associative algebra on $\bO$, equipped with an
anti-involution.\footnote{A priori, the dual is $A^{\on{rev}}\mod$, but the anti-involution
provides an identification $A\simeq A^{\on{rev}}$ (here and elsewhere, the superscript ``rev" means
``algebra with reversed multiplication").}

\end{itemize}

\medskip

In our case $\bO:=\AGCat$ and $A:=\ul\Shv(G)_{\on{co}}$, with the anti-involution provided by the antipode.

\sssec{}

Let $H$ be a cocommutative Hopf algebra in a symmetric monoidal category $\bO$, so that
$H\mod(\bO)$ acquires a symmetric monoidal structure compatible with the forgetful functor
$$\boblv_H: H\mod(\bO)\to \bO.$$

\medskip

Note that in this situation, an object of $H\mod(\bO)$ is dualizable if and only if the underlying
object $\bo\in \bO$ is dualizable, and the dual is given by $\bo^\vee$ with the natural \emph{right}
action, which we turn into a left action using the antipode on $H$.

\sssec{}

In particular, we obtain that an object of $G\mmod$ is dualizable if and only if the underlying
object $\ul\bC\in \AGCat$ is dualizable, and the dual is given by $\ul\bC^\vee$ with the
natural action.

\medskip

In particular, we obtain that if $\CY$ is an \emph{AG tame} prestack (see \cite[Sect. 3.4.7]{GRV}) equipped with
an action of $G$, then the object
$$\ul\Shv(\CY)\in G\mmod$$
is dualizable and self-dual (by \cite[Theorem 3.6.2 and Sect. 3.5.11]{GRV}).

\medskip

In particular, this is the case whenever $\CY$ is a \emph{pseudo-scheme} (in particular, an ind-scheme)
(see \cite[Lemma 3.4.5]{GRV}).

\medskip

Crucially for us, this is the case when $\CY$ is an algebraic stack that is locally a quotient of a scheme
by an algebraic group: this follows from \cite[Proposition 4.6.4]{GRV} and \cite[Theorem F.2.8]{AGKRRV1}.

\medskip

In fact, according to \cite[Conjecture F.2.7]{AGKRRV1}, we expect this to be the case for any quasi-compact algebraic stack with a schematic diagonal.

\sssec{}

Recall that an object $M$ of the category of the form $A\mod(\bO)$ (here $A$ is an associative algebra in some symmetric monoidal category $\bO$)
is said to be \emph{totally $\bO$-compact} (a.k.a dualizable as an $A$-module, see \cite[Lemma B.4.4(ii)]{GRV})
if there exists an object $M^\vee\in A^{\on{rev}}\mod(\bO)$ and
an isomorphism
$$\ul\Hom_{A\mod(\bO),\bO}(M,-) \simeq M^\vee\underset{A}\otimes -$$
as functors $A\mod(\bO)\to \bO$, where $\ul\Hom_{A\mod(\bO),\bO}(-,-)$ denotes the functor of \emph{internal Hom in $A\mod(\bO)$ relative to $\bO$},
see \cite[Sect. 1.1.2]{GRV}.

\medskip

In this case, we call $M^\vee$ the dual of $M$ as an $A$-module.

\sssec{}

Take $\bO=\AGCat$ and $A=\ul\Shv(G)_{\on{co}}$. We claim:

\begin{prop} \label{p:two senses of dual}
An object
$$\ul\bC\in G\mmod=\ul\Shv(G)_{\on{co}}\mmod(\AGCat)$$
is totally compact if and only if the object of $\AGCat$ underlying $\ul\bC$
is dualizable. Moreover, in this case, the dual of $\ul\bC$ as an object of
$$(\ul\Shv(G)_{\on{co}})^{\on{rev}}\mmod(\AGCat)\simeq \ul\Shv(G)_{\on{co}}\mmod(\AGCat)\simeq G\mmod$$
identifies with $\ul\bC^\vee$ with the natural action.
\end{prop}

\begin{proof}

One direction is easy: if $\ul\bC\in G\mmod$ is dualizable as a $\ul\Shv(G)_{\on{co}}$-module with dual $\ul\bC'$, then for any
$\ul\bD\in \AGCat$ and
$$\ul\Shv(G)\otimes \ul\bD\in G\mmod,$$
we have
$$\uHom_{G\mmod,\AGCat}(\ul\bC,\ul\Shv(G)\otimes \ul\bD) \simeq \ul\Hom_{\AGCat}(\ul\bC,\ul\bD)$$
and
$$\ul\bC'\underset{\ul\Shv(G)_{\on{co}}}\otimes (\ul\Shv(G)\otimes \ul\bD)\simeq \ul\bC'\otimes \ul\bD,$$
so we obtain an equivalence
$$\ul\Hom_{\AGCat}(\ul\bC,\ul\bD) \simeq \ul\bC'\otimes \ul\bD,$$
functorial in $\ul\bD$. This implies that $\ul\bC$ is dualizable as an object of $\AGCat$ with dual $\ul\bC'$.

\medskip

Vice versa, assume that $\ul\bC$ is dualizable as an object of $\AGCat$. For $\ul\bC'\in G\mmod$, we write
\begin{multline*}
\uHom_{G\mmod,\AGCat}(\ul\bC,\ul\bC') \overset{\text{\eqref{e:inv Funct}}}\simeq
\binv_G(\ul\Hom_{\AGCat}(\ul\bC,\ul\bC')) \overset{\text{\propref{p:coinv to inv}}}\simeq \\
\simeq \bcoinv_G(\ul\Hom_{\AGCat}(\ul\bC,\ul\bC')) \simeq \bcoinv_G(\ul\bC^\vee\otimes \ul\bC')
\overset{\text{\eqref{e:coinv ten}}}\simeq \ul\bC^\vee\underset{\ul\Shv(G)_{\on{co}}}\otimes \ul\bC',
\end{multline*}
i.e., $\ul\bC$ is indeed dualizable as a $\ul\Shv(G)_{\on{co}}$-module with dual $\ul\bC^\vee$.

\end{proof}

\sssec{}

We now claim:

\begin{cor} \label{c:inv dualizable}
Let $\ul\bC\in G\mmod$ be dualizable (in the sense of the symmetric monoidal structure). Then $\binv_G(\ul\bC)\in \AGCat$ is dualizable.
\end{cor}

\begin{proof}

Follows from Propositions \ref{p:coinv to inv} and \ref{p:two senses of dual}, using
the fact that the functor $\bcoinv_G$
sends totally compact objects to dualizable ones (see, e.g., \cite[Lemma B.4.4]{GRV}).

\end{proof}

\ssec{Adjoints of 1-morphisms}

Recall that the trace operation in a 2-category is functorial with respect to 1-morphisms that are
\emph{adjointable}.

\medskip

In this subsection we will see how the adjointability condition plays out for 1-morphisms
$$G_1\mmod\to G_2\mmod.$$

\sssec{} \label{sss:bimodules as morphisms}

Let $G_1$ and $G_2$ be a pair of groups, and recall the equivalence \eqref{e:prod groups reps}.
We obtain that
\begin{multline*}
\uHom_{\tAGCat}(G_1\mmod,G_2\mmod) \simeq (G_1\mmod)^\vee\otimes G_2\mmod \simeq \\
\simeq G_1\mmod\otimes G_2\mmod \simeq (G_1\times G_2)\mmod
\end{multline*}
as objects of $\tAGCat$.

\medskip

In particular, given
$$\bM_{1,2}\in (G_1\times G_2)\mmod,$$
we obtain a functor, to be denoted
$$T_{\bM_{1,2}}:G_1\mmod\to G_2\mmod.$$

\sssec{}

Recall that for three of algebras $A_1,A_2,A_3$ in a symmetric monoidal category $\bO$, under the identifications
$$\ul\Hom_{\bO\mmod}(A_i\mod(\bO),A_j\mod(\bO))\simeq (A_i^{\on{rev}}\otimes A_j)\mmod(\bO),$$
the composition functor
\begin{multline*}
\ul\Hom_{\bO\mmod}(A_1\mod(\bO),A_2\mod(\bO))\times \ul\Hom_{\bO\mmod}(A_2\mod(\bO),A_2\mod(\bO))\to \\
\to \ul\Hom_{\bO\mmod}(A_1\mod(\bO),A_3\mod(\bO))
\end{multline*}
corresponds to
$$(A_1^{\on{rev}}\otimes A_2)\mmod(\bO)\times (A_2^{\on{rev}}\otimes A_3)\mmod(\bO)\to
(A_1^{\on{rev}}\otimes A_3)\mmod(\bO), \quad M_{1,2},M_{2,3} \mapsto M_{2,3}\underset{A_2}\otimes M_{1,2}.$$

\medskip

Hence, for
$$\bM_{1,2}\in (G_1\times G_2)\mmod \text{  and } \bM_{2,3}\in (G_2\times G_3)\mmod,$$
we have
$$T_{\bM_{2,3}}\circ T_{\bM_{1,2}}\simeq T_{\bM_{1,3}},$$
where
$$\bM_{1,3}=\bM_{1,2}\underset{\ul\Shv(G_2)_{\on{co}}}\otimes \bM_{2,3}\simeq
\bcoinv_{\Delta(G_2)}(\bM_{1,2}\otimes \bM_{2,3})\in (G_1\times G_3)\mmod.$$

\medskip

For example, let $\bM_{1,2}=\ul\Shv(\CY_{1,2})$ and $\bM_{2,3}=\ul\Shv(\CY_{2,3})$, where $\CY_{1,2}$ (resp., $\CY_{2,3}$)
is a prestack acted on by $G_1\times G_2$ (resp., $G_2\times G_3$), and let us assume that one of these prestaks is AG tame.

\medskip

Then $\bM_{1,3}\simeq \ul\Shv(\CY_{1,3})$ for
$$\CY_{1,3}=(\CY_{1,2}\times \CY_{2,3})/\Delta(G_2).$$

Indeed, this follows from
\begin{multline} \label{e:nice comp}
\bcoinv_{G_2}(\ul\Shv(\CY_{1,2})\otimes \ul\Shv(\CY_{2,3})) \overset{\text{\propref{p:coinv to inv}}}\simeq
\binv_{G_2}(\ul\Shv(\CY_{1,2})\otimes \ul\Shv(\CY_{2,3})) \overset{\text{\cite[Cor. 3.6.7]{GRV}}}\simeq  \\
\simeq \binv_{G_2}(\ul\Shv(\CY_{1,2}\times \CY_{2,3})) \overset{\text{\eqref{e:inv as quot}}}\simeq
\ul\Shv((\CY_{1,2}\times \CY_{2,3})/\Delta(G_2)).
\end{multline}

\sssec{} \label{sss:when adjntbl}

For $\bM_{1,2}$ as above, the functor $T_{\bM_{1,2}}$  is adjointable if and only if $\bM_{1,2}$ is totally compact as an object of $G_2\mmod$. However, according to
\propref{p:two senses of dual} this is equivalent to $\bM_{1,2}$ being dualizable as an object of $\AGCat$.

\medskip

Moreover, in this case
the adjoint functor is given by $\bM_{1,2}^\vee$, viewed as an object of $(G_1\times G_2)\mmod$.

\sssec{} \label{sss:when adjntbl expl}

Explicitly, the unit of the adjunction is the functor
$$\ul\Shv(G_1)\to \binv_{G_2}(\bM_{1,2}\otimes \bM_{2,1})\simeq \bM_{1,2}\underset{\ul\Shv(G_2)_{\on{co}}}\otimes \bM_{2,1},$$
where the action of $G_1$ on the unit of the duality
$$\ul\Vect\to \bM_{1,2}\otimes \bM_{2,1},$$
which automatically factors via
$$\ul\Vect\to \binv_{G_2}(\bM_{1,2}\otimes \bM_{2,1}).$$

\medskip

The counit is described similarly.

\sssec{} \label{sss:adj geom prel}

In particular, we obtain that if
$$\bM_{1,2}=\ul\Shv(\CY_{1,2}),$$
where $\CY_{1,2}$ is an AG tame prestack acted on by $G_1\times G_2$, then the corresponding functor
$$T_{\CY_{1,2}}:=T_{\bM_{1,2}}$$
is adjointable.

\medskip

Moreover,
$$(T_{\CY_{1,2}})^R\simeq T_{\CY_{2,1}},$$
where $\CY_{2,1}$ is the same as $\CY_{1,2}$ as a prestack,
viewed as acted on by $G_2\times G_1$.

\sssec{} \label{sss:adj geom}

Let us describe explicitly the adjunction data for the 1-morphisms $T_{\CY_{1,2}}$ and $T_{\CY_{2,1}}$ above.
Namely, applying \cite[Sect. 3.7]{GRV} and using \secref{sss:when adjntbl expl}, we obtain the following:

\medskip

The unit of the adjunction is given by pull-push along the diagram
$$
\CD
G_1\times (\CY_{1,2}/G_2) @>{h}>> (\CY_{1,2}\times \CY_{2,1})/\Delta(G_2) \\
@V{v}VV \\
G_1,
\endCD
$$
where:

\begin{itemize}

\item $\ul\Shv(G_1)$ is viewed as an object in $(G_1\times G_1)\mmod$ giving rise to the identity endofunctor of $G_1\mmod$;

\medskip

%

\item We view $G_1\times (\CY_{1,2}/G_2)$ as a prestack acted on by $G_1\times G_1$ by
$$(g'_1,g''_1)\cdot (g_1,y)=(g'_1\cdot g_1\cdot (g''_1)^{-1},g''_1\cdot y);$$

\medskip

\item The map $v$ is the projection on the $G_1$ factor;

\medskip

\item The map $h$ is $(g_1,y)\mapsto (g_1\cdot y,y)$.

\end{itemize}

\medskip

The counit of the adjunction is defined similarly using the diagram
$$
\CD
G_2\times (\CY_{2,1}/G_1) @>>> G_2 \\
@VVV \\
(\CY_{2,1}\times \CY_{1,2})/\Delta(G_1).
\endCD
$$

\sssec{}

Let us return to the data of self-duality of $G\mmod$ as an object of $\tAGCat$. By construction, the unit and counit
are given by $\ul\Shv(G)$, viewed as an object of
$$G\mmod\otimes G\mmod\simeq (G\times G)\mmod.$$

Since $\ul\Shv(G)\in \AGCat$ is dualizable, we obtain:

\begin{cor} \label{c:G-mod 2-dual}
The object $G\mmod\in \tAGCat$ is 2-dualizable.
\end{cor}

\begin{rem}

The description of the adjoints in \secref{sss:adj geom prel} implies that the Serre functor on $G\mmod$
is canonically isomorphic to the identity. I.e., $G\mmod$ is a \emph{Calabi-Yau} object of $\tAGCat$.

\end{rem}

\ssec{Patterns of 2-adjointability for categorical representations}

\sssec{}

Let us now be given a pair of groups $G_1,G_2$ and a prestack $\CY$ acted on by $G_1\times G_2$. Consider the corresponding 1-morphism
\begin{equation} \label{e:stack 2-adj}
T_\CY:\ul\Shv(G_1)\mmod\to \ul\Shv(G_2)\mmod.
\end{equation}

We would like to give sufficient conditions for \eqref{e:stack 2-adj} to be 2-adjointable.

\sssec{} \label{sss:mock-proper}

We will say that a map
$$f:\CY_1\to \CY_2$$
between AG tame prestacks is \emph{mock-proper}, if the corresponding 1-morphism
$$f_\blacktriangle:\ul\Shv(\CY_1)\to \ul\Shv(\CY_2)$$
is adjointable (in $\AGCat$).

\medskip

We will say that an AG tame prestack $\CY$ is \emph{mock-proper} if its map to $\on{pt}$
is mock-proper.

\medskip

It is easy to see that if $f$ is mock-proper, then so is the map
$$f\times \on{id}:\CY_1\times \CZ\to \CY_2\times \CZ$$
for any AG tame prestack $\CZ$ (recall that the class of AG tame prestacks is closed under products, see
\cite[Proposition 3.6.11]{GRV}).

\medskip

Also, it is clear that a composition of mock-proper maps is mock-proper.

\sssec{Examples}

\begin{itemize}

\item A proper schematic map between AG tame prestacks is mock-proper.

\medskip

\item

For any algebraic group $G$, the stack $\on{pt}/G$ is mock proper.

\medskip

\item Combining the previous two examples, we obtain that if $Y$ is a proper scheme and
$G$ is an algebraic group acting on $Y$, then $Y/G$ is mock-proper.

\medskip

\item The stack $\BA^n/\BG_m$ is mock-proper, where $\BG_m$ acts on $\BA^n$ by dilations.

\medskip

\item A cotruncative open embedding (see \cite[Sect. 3]{DrGa2}) is mock-proper.

\end{itemize}

\sssec{} \label{sss:mock-proper schemes}

One can show that if a map $f:\CY_1\to \CY_2$ between schemes is
mock-proper, them it is actually proper.

\medskip

Indeed, it follows from \cite[Theorem B.6.3(a) and Proposition B.4.4(b)]{AGKRRV2} that if $f$ is mock-proper,
then $f_!\to f_*$ is an isomorphism. However, the latter implies that $f$ is proper.

\sssec{}

We will say that a map
$$f:\CY_1\to \CY_2$$
between prestacks is \emph{mock-smooth}, if the corresponding 1-morphism
$$f^!:\ul\Shv(\CY_2)\to \ul\Shv(\CY_1)$$
is adjointable (in $\AGCat$).

\medskip

Remarks pertaining to the notion of mock-properness in \secref{sss:mock-proper}
are equally applicable to the notion of mock-smoothness.

\sssec{Examples}

\begin{itemize}

\item A smooth schematic map is mock-smooth.

\medskip

\item Let $\CY$ be a smooth quasi-compact algebraic stack. Then it is mock-smooth if and
only if it is \emph{safe}\footnote{``Safe" means that any constructible sheaf on $\CY$ is compact,
see \cite[Sect. 9]{DrGa1}.} (see \cite[Sect. 1.3(d)]{GaVa}).

\medskip

\item A truncative locally closed embedding (see \cite[Sect. 3]{DrGa2}) is mock-smooth.

\medskip

\item

In particular, the map
$$\on{pt}/\BG_m\to \BA^n/\BG_m$$
is mock-smooth.

\end{itemize}

\sssec{} \label{sss:mock-smooth schemes}

One can show that a map $f:\CY_1\to \CY_2$ between schemes is mock-smooth if and only if
the constant sheaf of $\CY_1$ is ULA with respect to $f$.

\sssec{}

The following assertion results from the description of the unit and counit of the adjunction in \secref{sss:adj geom}:

\begin{prop} \label{p:stack 2-adj} Assume that:

\smallskip

\noindent{\em(i)} The action map $G_1\times (\CY/G_2)\to (\CY\times \CY)/G_2$ is mock-proper;

\smallskip

\noindent{\em(ii)} The stack $\CY/G_1$ is mock-proper;

\smallskip

\noindent{\em(iii)} The action map $G_2\times (\CY/G_1)\to (\CY\times \CY)/G_1$ is mock-smooth;

\medskip

\noindent{\em(iv)} The stack $\CY/G_2$ is mock-smooth.

\medskip

Then the 1-morphism \eqref{e:stack 2-adj} is 2-adjointable.

\end{prop}

\sssec{}

Combining Propositions \ref{p:stack 2-adj} and \ref{p:2-dual} with \corref{c:G-mod 2-dual}, we obtain:

\begin{cor} \label{c:stack 2-adj}
Let $\CY$ be a prestack equipped with an action of $G_1\times G_2$, satisfying
the assumptions of \propref{p:stack 2-adj}. Then the object of
$\on{Arr}(\tAGCat)$ given by \eqref{e:stack 2-adj} is 2-dualizable.
\end{cor}

\sssec{A non-example}

Let us take $G_1=1$ and $G_2=G$, so we are dealing with a 1-morphism
$$\one_{\AGCat}\mmod\to \ul\Shv(G)\mmod$$
given by a prestack $\CY$.

\medskip

Assume that $\CY=Y$ is a smooth separated scheme and $G$ is connected.
We claim that the conditions of \propref{p:stack 2-adj} are \emph{never} satisfied unless $Y$ is a union of points
and $G$ is unipotent (which is not an interesting case).

\medskip

Indeed, condition (ii) implies that $Y$ is proper (see \secref{sss:mock-proper schemes}).
Condition (iii) implies that the map
$$G\times Y\to Y\times Y$$
is ULA, which implies that $G$ acts transitively on each connected component of $Y$,

\medskip

Hence, (ii) and (iii) together imply that each connected component of $Y$ is of the form
$G/P$, where $P\subset G$ is a parabolic
subgroup. Now, condition (iv) says that $\on{pt}/P$ is mock-smooth, and hence safe.
This implies that $P$ is unipotent, which in turn implies that $P=G$ is unipotent.

\sssec{The key example} \label{sss:prince series}

Let us take $G_2=G$ to be a connected reductive group, and $G_1=M=P/N(P)$ to be the Levi
quotient of a parabolic $P$. Take
$$Y:=G/N(P),$$
with the natural action of $G\times M$.

\medskip

We claim that this $Y$ satisfies the conditions of \propref{p:stack 2-adj}.

\medskip

Condition (i) holds because the map in question is schematic and proper. Indeed, the map
$$M \times Y \to Y\times Y$$
is a closed embedding; this is an expression of the fact that $M$ acts freely on $Y$ with quotient $G/P$.

\medskip

Condition (ii) holds because $Y/M=G/P$ is proper.

\medskip

Condition (iii) holds because the map in question is schematic and smooth. Indeed, the map
$$G\times Y\to Y\times Y$$
is smooth because the action of $G$ on $Y$ is transitive.

\medskip

Finally, condition (iv) holds because $Y/G=\on{pt}/N(P)$ is smooth and safe, hence it is mock-smooth.

\sssec{}

Thus, by \corref{c:stack 2-adj}, we conclude that for $Y=G/N(P)$, the resulting object of $\on{Arr}(\tAGCat)$
is 2-dualizable.

\medskip

We will apply to this example the paradigm of \secref{sss:2 endos} and in particular
\corref{c:2-dual} in order to obtain Deligne-Lusztig representations, and compute their
traces.

\section{Traces of categorical representations}  \label{s:Tr Cat Rep}

In this section we will be concerned with calculating traces (in the ambient 3-category $\tAGCat$) of the objects
of the form $G\mmod$, equipped with endomorphisms (and between morphisms between these traces).

\medskip

Note that the above traces are objects of
$$\Omega(\tAGCat)=\End_{\tAGCat}(\one_{\tAGCat})\simeq \End_{\tAGCat}(\AGCat)\simeq \AGCat.$$

\ssec{The trace on \texorpdfstring{$G\mmod$}{Gmod}}

\sssec{}

Let us be given a prestack $\CX$ acted on by $G\times G$, and thereby giving rise to a 1-morphism
$$F_\CX: G\mmod\to G\mmod,$$
see \secref{sss:bimodules as morphisms}.

\medskip

We claim:

\begin{prop} \label{p:Tr on categ rep}
The object
$$\Tr(F_\CX,G\mmod)\in \AGCat$$
identifies canonically with
$$\ul\Shv(\CX/\Delta(G)),$$
where $G$ acts on $\CX$ via the diagonal embedding $G\to G\times G$.
\end{prop}


\begin{proof}

Let $H$ be a commutative Hopf algebra in a symmetric monoidal category $\bO$. Let us view $H\mod(\bO)$
as an object of the symmetric monoidal (2-)category $\bO\mmod$. Let $M$ be an $H$-bimodule, and let
$$F_M:H\mod(\bO)\to H\mod(\bO)$$
be the corresponding functor, viewed as a (1-)morphism in $\bO\mmod$. Then
\begin{equation} \label{e:Tr for Hopf}
\Tr(F_M,H\mod)\simeq M\underset{H}\otimes \one_\bO,
\end{equation}
where:

\begin{itemize}

\item $H$ acts on $M$ via the diagonal map $H\otimes H$, followed by the antipode on one of the factors;

\item $\one_\bO$ is viewed as a trivial $H$-module.

\end{itemize}

\medskip

Applying this to $\bO=\AGCat$ and $H=\ul\Shv(G)_{\on{co}}$, we obtain that
$$\Tr(F_\CX,G\mmod) \simeq \bcoinv_{\Delta(G)}(\ul\Shv(\CX)).$$

Combining with \propref{p:coinv to inv}, we rewrite the latter expression as
$$\binv_{\Delta(G)}(\ul\Shv(\CX))\overset{\text{\eqref{e:inv as quot}}}\simeq \ul\Shv(\CX/\Delta(G)).$$

\end{proof}

\sssec{}

Let $\phi$ be an endomorphism of $G$ as an algebraic group. Take $\CY:=G_\phi$, which equals $G$ as a scheme,
with right copy of $G$ acting by right multiplication, and the left copy acting by left multiplication precomposed
with $\phi$.

\medskip

We will denote the resulting endofunctor of $G\mmod$ by $F_\phi$. We can view it as follows:

\begin{itemize}

\item When we think of $G\mmod$ as $\ul\Shv(G)_{\on{co}}\mmod$, $F_\phi$ is given by precomposing
the action with the endomorphism of $\ul\Shv(G)_{\on{co}}$ given by $\phi_*$;

\medskip

\item When we think of $G\mmod$ as $\ul\Shv(G)\commod$, $F_\phi$ is given by composing
the coaction with the endomorphism of $\ul\Shv(G)$ given by $\phi^!$.

\end{itemize}

\medskip

Note that the resulting diagonal action of $G$ on $G_\phi$ is the action of $G$ on itself by $\phi$-twisted
conjugacy:
$$\on{Ad}_\phi(g)(g_1):=\phi(g)\cdot g_1\cdot g^{-1}.$$

Hence, in this particular case of \propref{p:Tr on categ rep} says:

\begin{cor} \label{c:Tr endo}
$$\Tr(F_\phi,G\mmod)\simeq \ul\Shv(G/\on{Ad}_\phi(G)).$$
\end{cor}

\sssec{}

Let us specialize to the case when $\phi=\on{id}$. We obtain:

\begin{cor} \label{c:Tr Id}
$$\Tr(\on{Id},G\mmod)\simeq \ul\Shv(G/\on{Ad}(G)).$$
\end{cor}

\sssec{}

Let now the ground field $k$ be $\ol\BF_q$, and assume that $G$ is defined over $\BF_q$, so that
it carries an action of the geometric Frobenius, denoted $\Frob$. We denote the endofunctor
$F_{\Frob}$ of $G\mmod$ simply by $\Frob$.

\medskip

From \corref{c:Tr endo} we obtain:
$$\Tr(\Frob,G\mmod)\simeq \ul\Shv(G/\on{Ad}_{\on{Frob}}(G)).$$

\sssec{}

Assume now that $G$ is connected. In this case, Lang's theorem implies that the $\on{Ad}_{\on{Frob}}(G)$-action is
transitive, and
$$G/\on{Ad}_{\on{Frob}}(G)\simeq \on{pt}/G(\BF_q).$$

We now observe:

\begin{lem} \label{l:finite grp is restr}
For a finite group $\Gamma$, the canonical map
$$\bi(\Rep(\Gamma))\to \ul\Shv(\on{pt}/\Gamma)$$
is an isomorphism.
\end{lem}

\begin{proof}

We have to show that for a scheme $X$, the (fully faithful) functor
$$\Rep(\Gamma)\otimes \Shv(X)\to \Shv((\on{pt}/\Gamma) \times X)$$
is an equivalence.

\medskip

Note, however, that every object $\CF\in \Shv((\on{pt}/\Gamma) \times X)$ splits canonically
as
$$\underset{\rho}\oplus\, (\rho\otimes \CF_\rho),$$
where $\rho$ runs over the set of isomorphism classes of irreducible representations of $\Gamma$.

\medskip

This makes the assertion manifest.\footnote{In \corref{c:restricted stack} we will show that a generalization of \lemref{l:finite grp is restr}
holds for any algebraic stack that has finitely many isomorphism classes of points.}

\end{proof}

\sssec{}

Summarizing, we obtain:

\begin{cor} \label{c:Tr Frob}
Let $G$ be connected. Then
$$\Tr(\Frob,G\mmod)\simeq \bi(\Rep(G(\BF_q)):=\ul\Rep(G(\BF_q)).$$
\end{cor}

\ssec{Morphisms between traces}

\sssec{}

Let us now be given a pair of groups $G_1$ and $G_2$, each equipped with a 1-endomorphism of
$F_{\CX_i}:\ul\Shv(G_i)\mmod$, given by a $(G_i\times G_i)$-prestack $\CX_i$. Let us also be given a 1-morphism
$$T_\CY:\ul\Shv(G_1)\mmod\to \ul\Shv(G_2)\mmod,$$ given by a prestack $\CY$ equipped with an action of $G_1\times G_2$.
Let us assume that $\CY$ is AG tame, which implies that $T_\CY$ is adjointable.

\medskip

Finally, let us be given a 2-morphism in $\tAGCat$
$$T_\CY\circ F_{\CX_1}\overset{\alpha_\CZ}\Rightarrow F_{\CX_2}\circ T_\CY,$$
given by pull-push along to a stack $\CZ$ acted on by $G_1\times G_2$ and equipped with
$(G_1\times G_2)$-equivariant maps
$$
\CD
\CZ @>>> \CY\overset{G_2}\times \CX_2\\
@VVV \\
\CX_1\overset{G_1}\times \CY.
\endCD
$$

\sssec{}

We wish to calculate
$$\Tr(\alpha_\CZ,T_\CY): \Tr(F_{\CX_1},\ul\Shv(G_1)\mmod)\to \Tr(F_{\CX_2},\ul\Shv(G_2)\mmod),$$
which is a 1-morphism in $\AGCat$
$$\ul\Shv(\CX_1/\Delta(G_1))\to \ul\Shv(\CX_2/\Delta(G_2)).$$

\sssec{}

Let
$\wt\CZ_1$ denote the fiber product
$$\CZ\underset{(\CX_1\times \CY)/\Delta(G_1)}\times (\CX_1\times \CY).$$
It is acted on by $G_1\times G_1\times G_2$. Consider the action on it of $G_1\times G_2$,
where $G_1$ acts via its diagonal embedding into $G_1\times G_1$.

\medskip

Define the $(G_1\times G_2)$-stack $\wt\CZ_2$ similarly. We have naturally defined
$(G_1\times G_2)$-equivariant maps
$$\wt\CZ_1 \to \CZ\times \CY \leftarrow \wt\CZ_2,$$
where each factor of $G_1\times G_2$ acts on $\CZ\times \CY$ diagonally.

\medskip

Finally, set
$$\CW:=\wt\CZ_1\underset{\CZ\times \CY}\times \wt\CZ_2.$$

This is a stack, equipped with an action of $G_1\times G_2$. The projection
$$\CW\to \wt\CZ_i\to \CX_i$$
ia $G_i$-equivariant (for the diagonal action of $G_i$ on $\CX_i$).

\sssec{}

We claim:

\begin{prop} \label{p:map of traces stacks}
Under the above circumstances, the 1-morphism $$\ul\Shv(\CX_1/\Delta(G_1))\to \ul\Shv(\CX_2/\Delta(G_2))$$ corresponding to
$\Tr(\alpha_\CZ,T_\CY)$ is given by pull-push along the diagram
$$
\CD
\CW/(G_1\times G_2) @>>> \CX_2/\Delta(G_2) \\
@VVV \\
\CX_1/\Delta(G_1).
\endCD
$$
\end{prop}

\begin{proof}

Follows by applying \eqref{e:2-cat Tr} using the calculation of the right adjoints to $F_{\CX_i}$ given by
\secref{sss:adj geom}.

\end{proof}

\sssec{}

Recall that $G\mmod$ is self-dual as an object of $\tAGCat$. Moreover, the unit and counit for this self-duality are
adjointable.

\medskip

Let $\phi$ be an endomorphism of $G$. By \lemref{l:Trace dualizable}, we obtain that
$$\Tr(F_\phi,G\mmod)\simeq \ul\Shv(G/\on{Ad}_\phi(G))$$
is dualizable, and its dual identifies with
$$\Tr((F_\phi)^R,G\mmod).$$

\sssec{} \label{sss:dual phi}

Note that, by \secref{sss:adj geom prel}, the endomorphism $(F_{\phi})^R$ is represented by $\ul\Shv((G_{\phi})^R)$,
where $(G_\phi)^R$ is a copy of $G$, on which:

\begin{itemize}

\item The left action of $G$ is given by left multiplication;

\item The right action of $G$ is given by right multiplication, precomposed with $\phi$.

\end{itemize}

Hence,
$$\Tr((F_{\phi})^R,G\mmod)\simeq \ul\Shv(G/\on{Ad}'_{\phi}(G)),$$
where
$$\on{Ad}'_{\phi}(g)(g_1)=g\cdot g_1\cdot \phi(g^{-1}).$$

Note, however, that the inversion operation on $G$ identifies
$$G/\on{Ad}'_{\phi}(G)\simeq G/\on{Ad}_{\phi}(G),$$
and hence we obtain:
$$\Tr((F_{\phi})^R,G\mmod)\simeq \ul\Shv(G/\on{Ad}_{\phi}(G)).$$

\sssec{}

Unwinding, from \propref{p:map of traces stacks}, we obtain:

\begin{cor} \label{c:Verdier}
The resulting identification
$$\ul\Shv(G/\on{Ad}_\phi(G))^\vee \simeq \Tr(F_\phi,G\mmod)^\vee\simeq \Tr((F_{\phi})^R,G\mmod) \simeq  \ul\Shv(G/\on{Ad}_{\phi}(G))$$
is the (Verdier duality) identification of \cite[Sect. 3.7]{GRV}.
\end{cor}

\sssec{}

Consider now the 1-morphism
$$\bcoinv_G:G\mmod \to \AGCat.$$
It is given by $\CY=\on{pt}$, viewed as a prestack acted on by $G$.

\medskip

Let $\phi$ be an endomorphism of $G$. We have a tautological natural transformation
$$\bcoinv_G\circ F_\phi\overset{\alpha}\to \bcoinv_G:$$
\begin{equation} \label{e:Tr coinv diag}
\xy
(0,0)*+{G\mmod}="A";
(60,0)*+{\AGCat}="C";
(0,-20)*+{G\mmod}="B";
(60,-20)*+{\AGCat,}="D";
{\ar@{->}_{F_\phi} "A";"B"};
{\ar@{->}^{\bcoinv_G} "A";"C"};
{\ar@{->}_{\bcoinv_G} "B";"D"};
{\ar@{->}^{\on{Id}} "C";"D"};
{\ar@{=>}^\alpha "B";"C"};
\endxy
\end{equation}

\medskip

Consider the resulting functor
\begin{equation} \label{e:Tr coinv}
\ul\Shv(G/\on{Ad}_\phi(G))\simeq \Tr(\phi,G\mmod) \overset{\Tr(\alpha,\binv_G)}\longrightarrow \Tr(\on{Id},\AGCat)\simeq \ul\Vect.
\end{equation}

Unwinding, from \propref{p:map of traces stacks}, we obtain:

\begin{cor} \label{c:Tr coinv}
The functor \eqref{e:Tr coinv} identifies with
$$\on{C}^\cdot_\blacktriangle(G/\on{Ad}_\phi(G),-):\ul\Shv(G/\on{Ad}_\phi(G))\to \ul\Vect.$$
\end{cor}

\ssec{The Grothendieck-Springer and Jacquet functors}

In this subsection we take $G$ to be a connected reductive group.

\sssec{}

Let $P\subset G$ be a parabolic subgroup, and let $M$ be the Levi quotient of $P$. Let us denote by
$$\bind_P:\ul\Shv(M)\to \ul\Shv(G)$$
the 1-morphism given by the $(G\times M)$-scheme $G/N(P)$.

\sssec{}

Consider the corresponding 1-morphism
\begin{equation} \label{e:GrothSpr}
\ul\Shv(M/\on{Ad}(M))\simeq \Tr(\on{Id},\ul\Shv(M)) \overset{\Tr(\on{id},\bind_P)}\longrightarrow \Tr(\on{Id},\ul\Shv(G)) \simeq
\ul\Shv(G/\on{Ad}(G)).
\end{equation}

We claim:

\begin{prop} \label{p:Springer functor}
The 1-morphism \eqref{e:GrothSpr} identifies canonically with the Grothendieck-Springer functor
$$\on{GrothSpr}_P:\ul\Shv(M/\on{Ad}(M))\to \ul\Shv(G/\on{Ad}(G)),$$
given by pull-push along the diagram
$$
\CD
P/\on{Ad}(P) @>{\sfp}>> G/\on{Ad}(G) \\
@V{\sfq}VV \\
M/\on{Ad}(M).
\endCD
$$
\end{prop}

\begin{proof}

We apply \propref{p:map of traces stacks} to the case when $F_{\CX_i}$ is the identity 1-morphism,
and the maps
$$
\CD
\CZ @>>> \CY\overset{G_2}\times G_2 @>{\simeq}>> \CY \\
@VVV \\
 \CY\overset{G_1}\times G_1 \\
 @V{\simeq}VV \\
\CY
\endCD
$$
are the identity maps.

\medskip

The resulting 1-morphism $\ul\Shv(G_1/\on{Ad}(G_1))\to \ul\Shv(G_2/\on{Ad}(G_2))$ is given by pull-push along the diagram
$$
\CD
\left((G_1\times \CY)\underset{\CY\times \CY}\times (G_2\times \CY)\right)/(G_1\times G_2) @>>> G_2/\on{Ad}(G_2) \\
@VVV \\
G_1/\on{Ad}(G_1).
\endCD
$$

Suppose now that the action of $G_1\times G_2$ on $\CY$ is transitive. Fix a point $y\in \CY$, and
let $P:=\on{Stab}_{G_1\times G_2}(y)$. In this case, we can rewrite the fiber product
$$\left((G_1\times \CY)\underset{\CY\times \CY}\times (G_2\times \CY)\right)/(G_1\times G_2)$$
as $P/\on{Ad}(P)$.

\medskip

We apply this to the case $G_1=M$, $G_2=G$ and $\CY=G/N(P)$, and the assertion of the proposition follows.

\end{proof}

\sssec{Trace and inversion} \label{sss:GrSpr and inversion}

Recall that $G\mmod$ is self-dual as an object of $\tAGCat$. Hence, we obtain a canonical
identification
\begin{multline} \label{e:inversion}
\ul\Shv(G/\on{Ad}(G))\simeq
\Tr(\on{Id},G\mmod)\overset{\text{\eqref{e:Tr of dual}}}\simeq \Tr(\on{Id},(G\mmod)^\vee)\overset{\text{\eqref{e:G mod self-dual}}}\simeq \\
\simeq \Tr(\on{Id},G\mmod)\simeq \ul\Shv(G/\on{Ad}(G)).
\end{multline}

It is easy to see that the map \eqref{e:inversion} is induced by the \emph{inversion} on $G$, to be denoted
$\sigma_G$.

\medskip

Note that with respect to the identifications
$$(G\mmod)^\vee\simeq G\mmod \text{ and } (M\mmod)^\vee\simeq M\mmod,$$
we have
$$\bind_P^\vee\simeq \bjacq_P.$$

Hence, we obtain a commutative diagram
$$
\CD
\ul\Shv(M/\on{Ad}(M)) @>{\on{GrothSpr}_P}>> \ul\Shv(G/\on{Ad}(G)) \\
@V{\sim}VV @VV{\sim}V \\
\Tr(\on{Id},M\mmod) @>{\Tr(\on{Id},\bind_P)}>> \Tr(\on{Id},G\mmod)  \\
@V{\sim}VV @VV{\sim}V \\
\Tr(\on{Id},(M\mmod)^\vee) @>{\Tr(\on{Id},((\bind_P)^R)^\vee)}>> \Tr(\on{Id},(G\mmod)^\vee) \\
@V{\sim}VV @VV{\sim}V \\
\Tr(\on{Id},M\mmod) @>{\Tr(\on{Id},\bind_P)}>> \Tr(\on{Id},G\mmod)  \\
@V{\sim}VV @VV{\sim}V \\
\ul\Shv(M/\on{Ad}(M)) @>{\on{GrothSpr}_P}>> \ul\Shv(G/\on{Ad}(G)),
\endCD
$$
in which the composite vertical arrows are given by $\sigma_M$ and $\sigma_G$, respectively.

\medskip

Thus, we obtain an isomorphism
\begin{equation} \label{e:Groth Spr and inversion}
\on{GrothSpr}_G\simeq \sigma_G\circ \on{GrothSpr}_G \circ \sigma_M.
\end{equation}

Note, however, that the pull-push functor in \propref{p:Springer functor} is also equipped
with an isomorphism
\begin{equation} \label{e:Groth Spr and inversion bis}
\sfp_*\circ \sfq^!\simeq \sigma_G\circ (\sfp_*\circ \sfq^!)\circ \sigma_M.
\end{equation}

It is easy to see that the isomorphisms \eqref{e:Groth Spr and inversion} and
\eqref{e:Groth Spr and inversion bis} are compatible via the identification of \propref{p:Springer functor}.

\sssec{The intertwining operator} \label{sss:intertwining functor}

Let $P_1$ and $P_2$ be two parabolic subgroups, such that the projections
$$M_1\twoheadleftarrow P_1\hookleftarrow P_1\cap P_2 \hookrightarrow P_2 \twoheadrightarrow M_2$$
are surjective. Hence, these projections each identify $M_i$, $i=1,2$ with the Levi quotient of $P_1\cap P_2$.
Denote this Levi quotient by $M$, so that we have
$$M_1\simeq M \simeq M_2.$$

We claim that there is a canonical 1-morphism in $\on{Arr}(\tDGCat)$
\begin{equation} \label{e:Springer action}
(M\mmod \overset{\bind_{P_1}}\to G\mmod)\to
(M\mmod \overset{\bind_{P_2}}\to G\mmod).
\end{equation}

Namely, the 1-morphisms
$$\ul\Shv(M)\to \ul\Shv(M) \text{ and } \ul\Shv(G)\to \ul\Shv(G)$$
are the identity ones,
and the 2-morphism
$$I_{P_1,P_2}:\bind_{P_1}\Rightarrow \bind_{P_2}$$ is given by pull-push along the following diagram
of $(M\times G)$-schemes
\begin{equation} \label{e:intertw diagram}
\CD
G/N(P_1\cap P_2) @>{r_2}>> G/N(P_2) \\
@V{r_1}VV \\
G/N(P_1).
\endCD
\end{equation}

Thus, we indeed obtain the desired 1-morphism \eqref{e:Springer action}.

\sssec{}

One can imagine wanting to compare the Grothendieck-Springer functors for different choices of a parabolic with the same Levi
quotient, using the above 2-morphism $I_{P_1,P_2}$.

\medskip

The problem is, however, that
this 2-morphism is \emph{not adjointable}, and so we cannot apply \secref{sss:3 funct Tr}.
In particular, it does \emph{not} induce a natural transformation
$$\on{GrothSpr}_{P_1}\to \on{GrothSpr}_{P_2}.$$

\medskip

And indeed, it is know that no such (non-zero) natural transformation exists:
this is evidenced by \cite[Theorem E]{Gu}.

\medskip

We will show how to remedy this by replacing the objects $G\mmod$
and $M\mmod$ by their relatives, denoted $G\mmod^{\on{spec.fin}}$ and $M\mmod^{\on{spec.fin}}$,
respectively.

\sssec{}

Let us denote by
$$\bjacq_P:G\mmod\to M\mmod$$
the 1-morphism \emph{also} given by the variety $(G\times M)$-scheme $G/N(P)$.

\sssec{}

Consider the corresponding 1-morphism
\begin{equation} \label{e:Jacquet}
\ul\Shv(G/\on{Ad}(G))\simeq \Tr(\on{Id},\ul\Shv(G)) \overset{\Tr(\on{id},\bjacq_P)}\longrightarrow \Tr(\on{Id},\ul\Shv(M)) \simeq
\ul\Shv(G/\on{Ad}(M)).
\end{equation}

We claim:

\begin{prop} \label{p:Jacquet functor}
The 1-morphism \eqref{e:GrothSpr} identifies canonically with the restriction functor
$$\fr_P:\ul\Shv(G/\on{Ad}(G))\to \ul\Shv(M/\on{Ad}(M)),$$ given by pull-push along the diagram
$$
\CD
P/\on{Ad}(P) @>{\sfq}>> M/\on{Ad}(M) \\
@V{\sfp}VV \\
G/\on{Ad}(G).
\endCD
$$
\end{prop}

\begin{proof}

Obtained by reading the proof of \propref{p:Springer functor} backwards.

\end{proof}

\sssec{}

Recall now that according to \secref{sss:prince series}, the 1-morphism $\bind_P$ is 2-adjointable. Moreover,
its right adjoint is given by $\bjacq_P$. Hence, by
\propref{p:adj Tr}, we obtain that the functor \eqref{e:Jacquet} is the right adjoint of \eqref{e:GrothSpr}.

\medskip

However, unwinding, we obtain:
\begin{cor} \label{c:ind restr}
Under the identifications of Propositions \ref{p:Springer functor} and \ref{p:Jacquet functor}, the above
adjunction corresponds to the usual $(\on{GrothSpr}_P,\fr_P)$-adjunction, given by\footnote{In the formula below, we have $\sfq^*\simeq \sfq^!$ because
the morphism $\sfq$ is smooth of relative dimension $0$.}
$$(\sfp_*)^R \simeq (\sfp^!) \text{ and } (\sfq^!)^R\simeq (\sfq^*)^R\simeq \sfq_*.$$
\end{cor}

\ssec{The Deligne-Lusztig functor} \label{ss:DL}

We continue to assume that $G$ is a connected reductive group.

\sssec{}

Let $P$ be a parabolic in $G$, and let $M$ be its Levi quotient. Choose a Levi splitting
$$M\hookrightarrow P,$$
and let us assume that $M$, viewed as a subgroup of $G$, is defined over $\BF_q$.

\medskip

Note that we are \emph{not} assuming that $P$ is defined over $\BF_q$.
Let $P'$ be the preimage of $P$ under the geometric Frobenius endomorphism $\Frob_G$ of $G$.

\medskip

Let us observe that the parabolic subgroups $P$ and $P'$ automatically satisfy
the conditions of \secref{sss:intertwining functor}. Indeed,
$$M\subset P\cap P'.$$

\medskip

Note also that the geometric Frobenius $\Frob_M$ on $M$ can be identified with
$$M=P/N(P)\simeq P\cap P'/N(P\cap P')\simeq P'/N(P') \overset{\Frob_G}\longrightarrow  P/N(P) =M.$$

\sssec{} \label{sss:constr DL}

We now claim that geometric Frobenius on $G$ defines a 1-endomorphism of the object
\begin{equation} \label{e:pre DL}
(M\mmod \overset{\bind_{P}}\to G\mmod)\in \on{Arr}(\tAGCat),
\end{equation}
to be denoted $\Frob_{\bind_P}$.

\medskip

Namely, the 1-endomorphism $\Frob_{\bind_P}$ is the composition of
\begin{equation} \label{e:pre DL 1}
(M\mmod \overset{\bind_{P}}\to G\mmod) \to (M\mmod \overset{\bind_{P'}}\to G\mmod)
\end{equation}
and
\begin{equation} \label{e:pre DL 2}
(M\mmod \overset{\bind_{P'}}\to G\mmod) \to
(M\mmod \overset{\bind_{P}}\to G\mmod),
\end{equation}
where:

\begin{itemize}

\item The 1-morphism \eqref{e:pre DL 1} is given by the action of the geometric Frobenius on $G$, i.e., by the diagram
$$
\xy
(0,0)*+{M\mmod}="A";
(60,0)*+{G\mmod}="C";
(0,-20)*+{M\mmod}="B";
(60,-20)*+{G\mmod,}="D";
{\ar@{->}_{\Frob_M} "A";"B"};
{\ar@{->}^{\bind_P} "A";"C"};
{\ar@{->}_{\bind_{P'}} "B";"D"};
{\ar@{->}^{\Frob_G} "C";"D"};
{\ar@{=>} "B";"C"};
\endxy
$$
where the 2-morphism is given by pushforward along
$$G/N(P')\overset{M,\Frob_M}\times M\to G/N(P), \quad (g,m)\mapsto \Frob(g)\cdot m;$$

\medskip

\item The 1-morphism \eqref{e:pre DL 2}  is given by the 2-morphism $I_{P,P'}$.

\end{itemize}

\sssec{}

In other words, in the datum of $\Frob_{\bind_P}$, the 1-endomorphisms
$$M\mmod\to M\mmod \text{ and }
G\mmod \to G\mmod$$
are given by $\Frob_M$ and $\Frob_G$, respectively.

\medskip

The 2-morphism
\begin{equation} \label{e:2-morph DL}
\bind_{P}\circ \Frob_M\Rightarrow \Frob_G\circ \bind_{P}
\end{equation}
is explicitly described as follows:

\medskip

The left-hand side in \eqref{e:2-morph DL} is given by the $(G\times M)$-scheme
$$(G/N(P)\times M)/M,$$
where the quotient is taken with respect to the action given by
$$m\cdot (g,m_1)=(g\cdot m^{-1},\Frob_M(m)\cdot m_1),$$\
with $G$ and $M$ acting on the left and right naturally.

\medskip

The right-hand side in \eqref{e:2-morph DL} is given by $G/N(P)$, where the $G$-action
is given by left multiplication precomposed with $\Frob_G$, and $M$ acting naturally on the right.

\medskip

The 2-morphism in \eqref{e:2-morph DL} is given by the diagram
$$
\CD
(G/N(P\cap P') \times M)/M @>>> (G/N(P') \times M)/M @>{(g,m)\mapsto \Frob_{G/N(P)}(g)\cdot m}>> G/N(P)\\
@VVV \\
(G/N(P) \times M)/M,
\endCD
$$
where:

\begin{itemize}

\item The action of $G$ is precomposed with $\Frob_G$ on the right-most term, and is the natural
action by left multiplication everywhere else;

\medskip

\item The action of $M$ is by right multiplication.

\end{itemize}

\sssec{}

Taking the trace of the above 1-endomorphism $\Frob_{\bind_P}$ of \eqref{e:pre DL}, we obtain an object of $\Omega(\on{Arr}(\tAGCat))$, i.e., a map
\begin{equation} \label{e:DL functor}
\Tr(\Frob_M,M\mmod)\to \Tr(\Frob_G,G\mmod).
\end{equation}

Applying \corref{c:Tr Frob}, and using the fact that $\bi:\DGCat\to \AGCat$ is fully faithful, we obtain a functor that we denote
$$\on{DL}_P:\Rep(M(\BF_q))\to \Rep(G(\BF_q)).$$

\sssec{}

Let us describe the functor $\on{DL}_P$ explicitly. Applying \propref{p:map of traces stacks}, we obtain that it is given by the correspondence
$$
\CD
\CW_P/(G\times M) @>>> G/\on{Ad}_{\Frob_G}(G)  @>{\sim}>> \on{pt}/G(\BF_q) \\
@VVV \\
M/\on{Ad}_{\Frob_M}(M) \\
@V{\sim}VV \\
\on{pt}/M(\BF_q),
\endCD
$$
where:

\begin{itemize}

\item $\CW_P$ is the variety
$$\{g\in G,\,m\in M,\, x\in G/N(P\cap P') \,|\,  r(x)= \Frob_G(g)\cdot \Frob_{G/N(P)}(r'(x)) \cdot m^{-1}\},$$
where
$$G/N(P) \overset{r}\leftarrow G/N(P\cap P') \overset{r'}\to G/N(P');$$

\item The $G$-action on $\CW_P$ is given by
$$g_1\cdot (g,m,x) = (g_1\cdot g\cdot \Frob(g_1)^{-1}, m, \Frob_G(g_1)\cdot x);$$

\item The $M$-action on $\CW_P$ is given by
$$m_1 \cdot (g,m,x)= (g, m_1\cdot m\cdot \Frob(m_1)^{-1}, x\cdot m_1^{-1}).$$

\end{itemize}

\sssec{}

Note that the fiber product
$$\on{pt} \underset{\on{pt}/G(\BF_q)}\times \CW_P/(G\times M)
\simeq \on{pt}\underset{G}\times (\CW_P/M)$$ is the usual parabolic Deligne-Lusztig variety
\begin{multline*}
\mathcal{DL}_P:=\{m\in M,\, x\in G/N(P\cap P')\, |\, r(x)=\Frob_{G/N(P)}(r'(x)) \cdot m^{-1}\}/M \simeq \\
\simeq \{x\in G/N(P\cap P')\,|\, r(x)=\Frob_{G/N(P)}(r'(x))\}/M(\BF_q),
\end{multline*}
which is equipped with its natural projection to $\on{pt}/M(\BF_q)$.

\medskip

Note also that the projection
$$G/N(P\cap P')\to G/(P\cap P')$$
defines an isomorphism
$$\mathcal{DL}_P\overset{\sim}\to \{y\in G/(P\cap P')\,|\, \ol{r}(y)=\Frob_{G/P}(\ol{r}'(y))\}, \quad G/P \overset{\ol{r}}\leftarrow G/(P\cap P') \overset{\ol{r}'}\to G/P'.$$

\sssec{}

In other words, the functor $\on{DL}_P$ is explicitly given by pull-push along the diagram
$$
\CD
\mathcal{DL}_P/G(\BF_q) @>>> \on{pt}/G(\BF_q) \\
@VVV \\
\on{pt}/M(\BF_q).
\endCD
$$

In other words, it sends $\rho\in \Rep(M(\BF_q))$ to
$$(\on{C}^\cdot(\mathcal{DL}_P,\omega_{\mathcal{DL}_P})\otimes \rho)^{M(\BF_q)}.$$

\begin{rem}

Note that what we call the Deligne-Lusztig functor differs from the usual convention by duality: the latter involves
$$\on{C}_c^\cdot(\mathcal{DL}_P)$$ rather than
$$\on{C}^\cdot(\mathcal{DL}_P,\omega_{\mathcal{DL}_P}).$$

\end{rem}

\sssec{}

We now consider the object \eqref{e:pre DL} as equipped with two 1-endomorphisms: one is $\Frob_{\bind_P}$ constructed
above, and the other is $\on{Id}$. They are equipped with a natural datum of commutation (see \secref{sss:setting for iterated trace}): indeed
any 1-endomorphism is equipped with a datum of commutation against the identity endomorphisms.

\medskip

We would have liked to apply \corref{c:2-dual} and conclude that certain two maps
$$\on{Funct}(M(\BF_q)/\on{Ad}(M(\BF_q)))\to \on{Funct}(G(\BF_q)/\on{Ad}(G(\BF_q)))$$
are equal.

\medskip

However, we cannot do this yet: namely, the  endomorphism $\Frob_{\bind_P}$ is \emph{not} adjointable
(see \secref{sss:cond for adj} for what this means). Namely, the 2-morphism \eqref{e:2-morph DL}
is not adjointable.

\medskip

In the next section we will remedy this by replacing $G\mmod$
and $M\mmod$ by the objects $G\mmod^{\on{spec.fin}}$ and $M\mmod^{\on{spec.fin}}$,
mentioned above.

\begin{rem}

Recall that in the case of the identity endomorphism, the right adjoint $\bjacq_P$ of $\bind_P$ induces a right
adjoint $\fr_P$ of the Grothendieck-Springer functor $\on{GrothSpr}_P$.

\medskip

We would have liked to produce a right adjoint to the Deligne-Lusztig functor, but at this stage we cannot do that:
we cannot apply \propref{p:adj Tr}, because the corresponding 2-morphism $\alpha^{\on{BC}}$ (see \secref{sss:BC 2 morph}; here
$\alpha$ is \eqref{e:2-morph DL}) does \emph{not} admit a right adjoint.

\medskip

This will also be remedied in the next section.

\end{rem}

\begin{rem} \label{r:DL parab}

The failure of the independence of the Deligne-Lusztig functor of $P$ happens\footnote{And such independence is obviously false:
for $M=T$ and the trivial character, different choices of the Borel lead to DL-varieties of different dimensions.}
for an even simpler reason (and cannot be remedied
by passing to the $T$-monodromic subcategories): the reason is that for two parabolics $P_1$ and $P_2$,
the diagram of natural transformations
$$
\CD
\bind_{P_1}\circ \Frob_M @>{I_{P_1,P'_1}}>> \bind_{P'_1}\circ \Frob_M  @>{\sim}>> \Frob_G\circ \bind_{P_1}  \\
@V{I_{P_1,P_2}}VV @VV{I_{P'_1,P'_2}}V  @VV{I_{P_1,P_2}}V \\
\bind_{P_2}\circ \Frob_M @>{I_{P_2,P'_2}}>> \bind_{P'_2}\circ \Frob_M  @>{\sim}>> \Frob_G\circ \bind_{P_2}
\endCD
$$
does \emph{not} commute up to any 3-isomorphism.  In fact, the diagram
\begin{equation} \label{e:bad intertwiners}
\CD
\bind_{P_1} @>{I_{P_1,P'_1}}>> \bind_{P'_1} \\
@V{I_{P_1,P_2}}VV  @VV{I_{P'_1,P'_2}}V \\
\bind_{P_2} @>{I_{P_2,P'_2}}>> \bind_{P'_2}
\endCD
\end{equation}
does \emph{not} commute up to any 3-isomorphism.

\end{rem}

\begin{rem}

That said, one can define a variant of the Deligne-Lusztig functor using \emph{any} 2-morphism
$$\wt{I}_{P,P'}:\bind_P\to \bind_{P'}$$
in $\tAGCat$, i.e., a 1-morphism
$$\wt{I}_{P,P'}:\ul\Shv(G/N(P))\to \ul\Shv(G/N(P'))$$
in $(G\times M)\mmod$.

\medskip

One can replace the diagram \eqref{e:bad intertwiners} by a commutative diagram
\begin{equation} \label{e:bad intertwiners corr}
\CD
\bind^{\on{spec.fin}}_{P'_1} @>{I^{\on{spec.fin}}_{P_1,P'_1}}>> \bind^{\on{spec.fin}}_{P'_1} \\
@V{I^{\on{spec.fin}}_{P_1,P_2}}VV  @VV{I^{\on{spec.fin}}_{P'_1,P'_2}}V \\
\bind^{\on{spec.fin}}_{P_2} @>{\wt{I}^{\on{spec.fin}}_{P_2,P'_2}}>> \bind^{\on{spec.fin}}_{P'_2},
\endCD
\end{equation}
of 1-morphisms \emph{between the spectrally finite subcategories} of the two sides for a suitably defined 1-morphism
$\wt{I}^{\on{spec.fin}}_{P_2,P'_2}$.

\medskip

This will allow us to conclude that
$$\on{DL}_{P_1}\simeq \wt{\on{DL}}_{P_2},$$
where $\wt{\on{DL}}_{P_2}$ is a version of the Deligne-Lusztig functor constructed using $\wt{I}^{\on{spec.fin}}_{P_2,P'_2}$.

\end{rem}

\ssec{The class of an object}

\sssec{}

Let $\phi$ be an endomorphism of $G$. Let $\ul\bC$ be an object of $G\mmod$, equipped with a 1-morphism
\begin{equation} \label{e:cat weak equiv}
\phi_{\ul\bC}:\ul\bC\to F_\phi(\ul\bC).
\end{equation}

In other words, $\phi_{\ul\bC}$ is an endomorphism of $\ul\bC$ as an object of $\AGCat$ that intertwines the action of $G$
with one obtained by precomposition with $\phi$.

\sssec{} \label{sss:class of obj}

We can think of this data as
\begin{equation} \label{e:class of an object}
\xy
(0,0)*+{\AGCat}="A";
(60,0)*+{G\mmod}="C";
(0,-20)*+{\AGCat}="B";
(60,-20)*+{G\mmod,}="D";
{\ar@{->}_{\on{Id}} "A";"B"};
{\ar@{->}^{\ul\bC} "A";"C"};
{\ar@{->}_{\ul\bC} "B";"D"};
{\ar@{->}^{F_\phi} "C";"D"};
{\ar@{=>}^{\phi_{\ul\bC}} "B";"C"};
\endxy
\end{equation}

\medskip

Assume that $\ul\bC$ is dualizable. Then, using \secref{sss:when adjntbl} and \eqref{e:2-categ trace},
we obtain that \eqref{e:class of an object} induces a map
$$\ul\Vect=\Tr(\on{Id},\AGCat) \to \Tr(F_\phi,G\mmod)\simeq \ul\Shv(G/\on{Ad}_\phi(G)).$$

I.e., we obtain an object of $\Shv(G/\on{Ad}_\phi(G))$, which we will denote by
$$\on{cl}(\ul\bC,\phi_{\ul\bC}).$$

\sssec{}

Consider the 1-morphism
$$\boblv_G:G\mmod\to \AGCat.$$

Given a point $g\in G$, we endow $\boblv_G$ with a natural transformation
\begin{equation} \label{e:class of an object fiber}
\xy
(0,0)*+{G\mmod}="A";
(60,0)*+{\AGCat}="C";
(0,-20)*+{G\mmod}="B";
(60,-20)*+{\AGCat,}="D";
{\ar@{->}_{F_\phi} "A";"B"};
{\ar@{->}^{\boblv_G} "A";"C"};
{\ar@{->}_{\boblv_G} "B";"D"};
{\ar@{->}^{\on{Id}} "C";"D"};
{\ar@{=>}^{\alpha_g} "B";"C"};
\endxy
\end{equation}
given by the action of $g$.

\medskip

From \propref{p:map of traces stacks} we obtain that the resulting functor
$$\ul\Shv(G/\on{Ad}_\phi(G))\simeq \Tr(F_\phi,G\mmod) \overset{\Tr(\alpha_g,\boblv_G)}\longrightarrow \Tr(\on{Id},\AGCat)\simeq \ul\Vect$$
is given by $\iota_g^!$, where $\iota_g$ is the map
$$\on{pt} \to G/\on{Ad}_\phi(G)$$
corresponding to $g$.

\sssec{}

Concatenating \eqref{e:class of an object} and \eqref{e:class of an object fiber}, we obtain:

\begin{cor} \label{c:fiber of class}
We have a canonical isomorphism (in $\Vect$):
$$\iota_g^!(\on{cl}(\ul\bC,\phi_{\ul\bC}))\simeq \Tr(g\circ \phi_{\ul\bC}).$$
\end{cor}

\sssec{}

Let us now concatenate \eqref{e:class of an object} with \eqref{e:Tr coinv diag}. We obtain:

\begin{cor} \label{c:sections of class}
There is a canonical isomorphism
$$\on{C}^\cdot_\blacktriangle(G/\on{Ad}_\phi(G),\on{cl}(\ul\bC,\phi_{\ul\bC})) \simeq \Tr(\phi_{\ul\bC},\bcoinv_G(\ul\bC)),$$
where by a slight abuse of notation we denote by the same symbol $\phi_{\ul\bC}$ the endofunctor of $\bcoinv_G(\ul\bC)$
induced by $\phi_{\ul\bC}$.
\end{cor}

\sssec{}

Let $(\ul\bC_1,\phi_{\ul\bC_1})$ and $(\ul\bC_2,\phi_{\ul\bC_2})$ be a pair of objects as above.
It is clear that
\begin{equation} \label{e:product classes}
\on{cl}(\ul\bC_1\otimes \ul\bC_2,\phi_{\ul\bC_1}\otimes \phi_{\ul\bC_1})\simeq
\on{cl}(\ul\bC_1,\phi_{\ul\bC_1})\sotimes \on{cl}(\ul\bC_2,\phi_{\ul\bC_2}).
\end{equation}

In particular, we obtain:

\begin{cor} \label{c:sections of product class}
There is a canonical isomorphism
$$\on{C}^\cdot_\blacktriangle(G/\on{Ad}_\phi(G),\on{cl}(\ul\bC_1,\phi_{\ul\bC_1})\sotimes \on{cl}(\ul\bC_2,\phi_{\ul\bC_2})) \simeq
\Tr(\phi_{\ul\bC_1}\otimes \phi_{\ul\bC_2},\ul\bC_1\underset{\ul\Shv(G)_{\on{co}}}\otimes \ul\bC_2),$$
where by a slight abuse of notation we denote by the same symbol $\phi_{\ul\bC_1}\otimes \phi_{\ul\bC_2}$ the endofunctor of
$\ul\bC_1\underset{\ul\Shv(G)}\otimes \ul\bC_2$ induced by $\phi_{\ul\bC_1}\otimes \phi_{\ul\bC_2}$.
\end{cor}

\sssec{} \label{sss:DL ab}

Here is an application of \corref{c:sections of product class} to the theory of Deligne-Lusztig representations.

\medskip

Let $T_1$ and $T_2$ be two maximal tori in $G$ defined over $\BF_q$. Let $\chi_i$ be a character sheaf on $T_i$, equipped
with a Weil structure. Let $\chi^0_i$ denote the corresponding character of $T_i(\BF_q)$.

\medskip

Let $T_i\subset B_i$ be Borel subgroups, and consider the corresponding Deligne-Lusztig representations
$$\on{DL}_{B_i}(\ol\BQ_\ell^{\chi^0_i}),$$
where $\ol\BQ_\ell^{\chi_i^0}$ denotes the 1-dimensional representation of $T_i(\BF_q)$ corresponding to
the character $\chi_i^0$.

\medskip

We claim:

\begin{thm} \label{t:orth DL}
Suppose that $\chi^{-1}_1\neq \chi^w_2$ for any element $w\in W$. Then
$$(\on{DL}_{B_1}(\ol\BQ_\ell^{\chi^0_1})\otimes \on{DL}_{B_2}(\ol\BQ_\ell^{\chi^0_2}))^{G(\BF_q)}=0.$$
\end{thm}

\ssec{Proof of \thmref{t:orth DL}}

\sssec{} \label{sss:Vect chi}

Let $\ul\Vect^{\chi_i}$ be a copy of $\ul\Vect$, viewed as an object of $T_i\mmod$, with the
action given by $\chi_i$.

\medskip

The Weil structure on $\chi_i$ upgrades the identity endomorphism of  $\ul\Vect$ to a structure
\eqref{e:cat weak equiv} for $\phi=\Frob_{T_i}$; denote it by $\alpha^T_i$.

\medskip

Unwinding, we obtain:
$$\on{cl}(\ul\Vect^{\chi_i},\alpha^T_i)=\ol\BQ_\ell^{\chi_i^0}$$
as objects of
$$\Tr(\Frob_{T_i},T_i\mmod)\simeq \Rep(T_i(\BF_q)).$$

\sssec{}

Consider the object
$$\bind_{B_i}(\ul\Vect^{\chi_i})\in G\mmod.$$

It identifies with
$$\ul\bC_i:=\ul\Shv(G/N)^{T,\chi_i}.$$

The datum of $\alpha^T_i$, combined with the endomorphism $\Frob_{\bind_{B_i}}$,
(see \secref{sss:constr DL}) equips $\ul\bC_i$ with a structure
\eqref{e:cat weak equiv} for $\phi=\Frob_{G}$; denote it by $\alpha^G_i$.

\medskip

By the compatibility of \eqref{e:2-categ trace} with compositions, we have
$$\on{DL}_{B_i}(\on{cl}(\ul\Vect^{\chi_i},\alpha_i^T))\simeq \on{cl}(\ul\bC_i,\alpha_i^G).$$

\medskip

Hence,
$$\on{DL}_{B_i}(\ol\BQ_\ell^{\chi^0_i}) \simeq \on{cl}(\ul\bC_i,\alpha_i^G).$$

\sssec{}

We now apply \corref{c:sections of product class}. We obtain
$$(\on{DL}_{B_1}(\ol\BQ_\ell^{\chi^0_1})\otimes \on{DL}_{B_2}(\ol\BQ_\ell^{\chi^0_2}))^{G(\BF_q)}
\simeq \Tr(\alpha^G_1\otimes \alpha^G_2,\bcoinv_G(\ul\bC_1\otimes \ul\bC_2)).$$

We rewrite
\begin{multline*}
\bcoinv_G(\ul\bC_1\otimes \ul\bC_2)\simeq \binv_G(\ul\bC_1\otimes \ul\bC_2)\simeq
\ul\Shv(G \backslash (\ul\Shv(G/N)^{T,\chi_1}\otimes \ul\Shv(G/N)^{T,\chi_2}))\simeq \\
\simeq \ul\Shv(N\backslash G/N)^{T\times T,(\chi^{-1}_1,\chi_2)}.
\end{multline*}

\sssec{}

Now, the condition that $\chi^{-1}_1\neq \chi^w_2$ for any element $w\in W$ implies that the object
$\ul\Shv(N\backslash G/N)^{T\times T,(\chi^{-1}_1,\chi_2)}$ is zero.

\qed[\thmref{t:orth DL}]

\ssec{Iterated traces}

\sssec{}

Let $\phi_1$ and $\phi_2$ be a pair of \emph{commuting} endomorphisms of $G$. Consider the corresponding
endomorphisms $F_{\phi_1}$, $F_{\phi_2}$ and $(F_{\phi_2})^R$ of $G\mmod$.

\medskip

By \corref{c:Tr endo},
$$\Tr(F_{\phi_1},G\mmod)\simeq \ul\Shv(G/\on{Ad}_{\phi_1}(G)).$$

We also have
$$\Tr((F_{\phi_2})^R,G\mmod)\simeq \ul\Shv(G/\on{Ad}_{\phi_2}(G)),$$
see \secref{sss:dual phi}.

\sssec{}

The fact that $\phi_1$ and $\phi_2$ commute implies that we have a canonical isomorphism
$$F_{\phi_2}\circ F_{\phi_1}\overset{\beta}\to F_{\phi_1}\circ F_{\phi_2}.$$

\medskip

By \propref{p:map of traces stacks}, the map
$$\Tr(F_{\phi_1},G\mmod) \overset{\Tr(\beta,F_{\phi_2})}\longrightarrow \Tr(F_{\phi_1},G\mmod)$$
identifies with
$$\ul\Shv(G/\on{Ad}_{\phi_1}(G)) \overset{\phi_2^!}\to \ul\Shv(G/\on{Ad}_{\phi_1}(G)).$$

\medskip

Similarly,
$$(\Tr(F_{\phi_2})^R,G\mmod) \overset{\Tr(\beta^{\on{BC}},F_{\phi_1})}\longrightarrow \Tr((F_{\phi_1})^R,G\mmod)$$
identifies with
$$\ul\Shv(G/\on{Ad}_{\phi_2}(G)) \overset{\phi_1^!}\to \ul\Shv(G/\on{Ad}_{\phi_2}(G)),$$
where $\beta^{\on{BC}}$ is as in \secref{sss:setting for iterated trace}.

\sssec{}

Now, by \cite[Corollary 5.2.7]{GRV}, we obtain:
$$\Tr(\phi_2^!,\ul\Shv(G/\on{Ad}_{\phi_1}(G)))\simeq \on{C}_\blacktriangle\left((G/\on{Ad}_{\phi_1}(G))^{\phi_2},\omega_{(G/\on{Ad}_{\phi_1}(G))^{\phi_2}}\right)$$
and
$$\Tr(\phi_1^!,\ul\Shv(G/\on{Ad}_{\phi_2}(G)))\simeq \on{C}_\blacktriangle\left((G/\on{Ad}_{\phi_2}(G))^{\phi_1},\omega_{(G/\on{Ad}_{\phi_2}(G))^{\phi_1}}\right).$$

Finally, we note
\begin{equation}  \label{e:double fixed points}
(G/\on{Ad}_{\phi_1}(G))^{\phi_2}=((\on{pt}/G)^{\phi_1})^{\phi_2}=(\on{pt}/G)^{\phi_1,\phi_2}=
((\on{pt}/G)^{\phi_2})^{\phi_1}=(G/\on{Ad}_{\phi_2}(G))^{\phi_1}.
\end{equation}

\sssec{}

We claim:

\begin{lem} \label{l:iterated trace groups}
The diagram
\begin{equation} \label{e:iterated trace groups}
\CD
\Tr(\Tr(\beta,F_{\phi_2}),\Tr(F_{\phi_1},G\mmod))  @>{\text{\thmref{t:2 traces}}}>{\sim}>  \Tr(\Tr(\beta^{\on{BC}},F_{\phi_1}),\Tr((F_{\phi_2})^R,G\mmod)) \\
@V{\sim}VV @VV{\sim}V \\
 \Tr(\phi_2^!,\ul\Shv(G/\on{Ad}_{\phi_1}(G)))  & & \Tr(\phi_1^!,\ul\Shv(G/\on{Ad}_{\phi_2}(G))) \\
 @V{\sim}VV @V{\sim}VV \\
\on{C}_\blacktriangle\left((G/\on{Ad}_{\phi_1}(G))^{\phi_2},\omega_{(G/\on{Ad}_{\phi_1}(G))^{\phi_2}}\right)
@>{\sim}>{\text{\eqref{e:double fixed points}}}>
\on{C}_\blacktriangle\left((G/\on{Ad}_{\phi_2}(G))^{\phi_1},\omega_{(G/\on{Ad}_{\phi_2}(G))^{\phi_1}}\right)
\endCD
\end{equation}
commutes.
\end{lem}

\begin{proof}

Unwinding, we identity the 2-morphism in both \eqref{e:rotated pre iterated trace} and \eqref{e:outer rotated pre iterated trace}
with
$$\on{C}_\blacktriangle\left((\on{pt}/G)^{\phi_1,\phi_2},\omega_{(\on{pt}/G)^{\phi_1,\phi_2}}\right),$$
and these identifications are equal to the composition of the left vertical arrow
in \eqref{e:iterated trace groups} with
$$\on{C}_\blacktriangle\left((G/\on{Ad}_{\phi_1}(G))^{\phi_2},\omega_{(G/\on{Ad}_{\phi_1}(G))^{\phi_2}}\right) \simeq
\on{C}_\blacktriangle\left((\on{pt}/G)^{\phi_1,\phi_2},\omega_{(\on{pt}/G)^{\phi_1,\phi_2}}\right),$$
and the composition of the left vertical arrow in \eqref{e:iterated trace groups} with
$$\on{C}_\blacktriangle\left((G/\on{Ad}_{\phi_2}(G))^{\phi_1},\omega_{(G/\on{Ad}_{\phi_2}(G))^{\phi_1}}\right)\simeq
\on{C}_\blacktriangle\left((\on{pt}/G)^{\phi_1,\phi_2},\omega_{(\on{pt}/G)^{\phi_1,\phi_2}}\right),$$
respectively.

\end{proof}

\sssec{}

Let us see what the commutative diagram \eqref{e:iterated trace groups} gives when $\phi_1=\Frob$ and $\phi_2=\on{Id}$.

\medskip

We identify
$$(G/\on{Ad}_{\Frob}(G))^{\on{id}} \simeq (\on{pt}/G(\BF_q))^{\on{id}}\simeq G(\BF_q)/\on{Ad}(G(\BF_q),$$
and it is easy to see that the identification of \eqref{e:double fixed points} corresponds to the canonical map
$$G(\BF_q)/\on{Ad}(G(\BF_q)) \to (G/\on{Ad}(G))(\BF_q) \simeq (G/\on{Ad}(G))^{\Frob}.$$

\medskip

The left vertical arrow in \eqref{e:iterated trace groups} is the map
\begin{multline*}
\Tr_{\AGCat}(\on{Id},\ul{\Rep}(G(\BF_q))) \simeq \Tr_{\DGCat}(\on{Id},\Rep(G(\BF_q))) \simeq  \\
\simeq \ol\BQ_\ell\underset{\BZ}\otimes K^0(\Rep(G(\BF_q))) \simeq \sFunct(G(\BF_q)/\on{Ad}(G(\BF_q)),\ol\BQ_\ell),
\end{multline*}
where the last isomorphism sends the class of a finite-dimensional representation to its character (this is true for $G(\BF_q)$ replaced
by any finite group $\Gamma$).

\medskip

The right vertical arrow is the \emph{Local Term} map $\on{LT}^{\on{AG}}$
$$\Tr(\Frob^!,\ul\Shv(G/\on{Ad}(G)))\to \sFunct((G/\on{Ad}(G))(\BF_q),\ol\BQ_\ell)$$
of \cite[Sect. 5.2.4]{GRV} (this is true for $G/\on{Ad}(G)$ replaced by an arbitrary AG tame
algebraic stack equipped with an endomorphism).

\sssec{}

Hence, we obtain:

\begin{cor} \label{c:iterated trace Frob}
The diagram
$$
\CD
\Tr(\on{Id},\Tr(\Frob,G\mmod))  @>{\text{\corref{c:2 traces}}}>{\sim}>  \Tr(\Tr(\on{id},\Frob),\Tr(\on{Id},G\mmod))  \\
@V{\text{\corref{c:Tr Frob}}}V{\sim}V @V{\sim}V{\text{\corref{c:Tr Id}}}V \\
\Tr(\on{Id},\ul{\Rep}(G(\BF_q))) & & \Tr(\Frob^!,\ul\Shv(G/\on{Ad}(G))) \\
@V{\rm{taking\,\,characters}}V{\sim}V @V{\sim}V{\on{LT}^{\on{AG}}}V \\
\sFunct(G(\BF_q)/\on{Ad}(G(\BF_q)),\ol\BQ_\ell) @>{\sim}>>  \sFunct((G/\on{Ad}(G))(\BF_q),\ol\BQ_\ell)
\endCD
$$
commutes.
\end{cor}

\section{The condition of (quasi-)monodromicity} \label{s:mon}

As was explained in the Introduction, in order to obtain a theory with interesting
representation-theoretic applications, we will need to replace the 2-category $G\mmod$
by its full subcategory $G\mmod^{\on{spec.fin}}$. Its definition is based on the
notion of \emph{monodromicity}, which we develop in this section.

\medskip

Throughout this section, we will assume that $G$ is connected.

\ssec{The category of quasi-monodromic sheaves on a group}  \label{ss:q-mon}

\sssec{}

Let $X$ be a connected scheme. Let
\begin{equation} \label{e:alm const}
\Shv(X)^{\on{q-const}}\subset \Shv(X)
\end{equation}
be the full subcategory, consisting of objects all of whose cohomologies (with respect to either the perverse or the usual
t-structure) are (possibly infinite) extensions of the (shifted) constant sheaf.

\sssec{}

We will be interested in the case when $X=G$ is a group. 

\medskip

In \secref{sss:q vs ind} we will prove:

\begin{lem} \label{l:q-cnst comp gen}
The category $\Shv(G)^{\on{q-const}}$ is compactly generated.\footnote{Note, however, that the embedding \eqref{e:alm const} does \emph{not}
preserves compactness.}
\end{lem}

\sssec{}

We claim:

\begin{lem} \label{l:alm const}
For a pair of groups $G_1,G_2$, the functor
$$\boxtimes:\Shv(G_1)^{\on{q-const}}\otimes \Shv(G_2)^{\on{q-const}}\to \Shv(G_1\times G_2)$$
is an equivalence onto
$$\Shv(G_1\times G_2)^{\on{q-const}}\subset \Shv(G_1\times G_2).$$
\end{lem}

The lemma will be proved in \secref{ss:alm const}. 

\begin{rem}
One can formulate statements parallel to Lemmas \ref{l:q-cnst comp gen} and \ref{l:alm const}, when 
instead of group(s) we have arbitrary (connected) scheme(s). However, we are not sure whether the
corresponding statements are true.

\medskip

That said, they are true for \emph{simply-connected} schemes; in fact our proofs apply in this
case.

\end{rem}

\sssec{}

Since the !-pullback functor (along a smooth morphism) maps the q-constant categories
to one another, from \lemref{l:alm const} we obtain:

\begin{cor} \label{c:alm const}
The functor
$$\on{mult}^!:\Shv(G)\to \Shv(G\times G)$$
induces a comonoidal structure on $\Shv(G)^{\on{q-const}}$.
\end{cor}

\sssec{}

Let $\chi$ be a multiplicative 1-dimensional local system (a.k.a., abelian character sheaf) on $G$. Let
$$\Shv(G)^{\on{q-}\!\chi}\subset \Shv(G)$$
be the full subcategory equal to the essential image of
$$(-)\otimes \chi:\Shv(G)^{\on{q-const}}\to \Shv(G).$$

\medskip

It follows from \corref{c:alm const} that the functor $\on{mult}^!$ induces a comonoidal
structure on $\Shv(G)^{\on{q-}\!\chi}$.

\sssec{Example}

Here is a typical example of an object in $\Shv(G)^{\on{q-}\!\chi}\commod(\DGCat)$: we take the underlying object of $\DGCat$ to be $\Vect$,
with the coaction functor
$$\Vect\to \Vect\otimes \Shv(G)^{\on{q-}\!\chi}\simeq \Shv(G)^{\on{q-}\!\chi}$$
given by $\chi$.

\medskip

Denote this object by $\Vect^\chi$.

\begin{rem}

It is \emph{not} true that $\Vect^\chi$ generates $\Shv(G)^{\on{q-}\!\chi}\commod(\DGCat)$, i.e.,
the functor
$$\uHom_{\Shv(G)^{\on{q-}\!\chi}\commod,\DGCat}(\Vect^\chi,-):\Shv(G)^{\on{q-}\!\chi}\commod(\DGCat)\to \DGCat$$
is \emph{not} necessarily conservative.

\medskip

A parallel fact is true once we replace $\Shv(G)^{\on{q-}\!\chi}$ by a closely related category $\Shv(G)^{\on{alm-}\!\chi}$ (see \secref{sss:i-q}).

\end{rem}

\sssec{}

Set
$$\Shv(G)^{\on{q-mon}}:=\underset{\chi}\oplus\, \Shv(G)^{\on{q-}\!\chi},$$
where the direct sum is taken over the set of (isomorphism classes of) $\chi$'s.

\medskip

Note that the tautological functor
\begin{equation} \label{e:alm char}
\Shv(G)^{\on{q-mon}}\to \Shv(G)
\end{equation}
is fully faithful:

\medskip

It is clear that $\chi_1\neq \chi_2\, \Rightarrow\, \CHom_{\Shv(G)}(\chi_1,\chi_2)=0$, and then we use the fact that
$\Shv(G)$ is left-complete in its t-structure, see \cite[Theorem 1.1.6]{AGKRRV1}.

\sssec{}

By \corref{c:alm const}, we obtain that the functor $\on{mult}^!$ induces a comonoidal
structure on $\Shv(G)^{\on{q-mon}}$.

\medskip

In addition, the $\sotimes$ operation makes $\Shv(G)^{\on{q-mon}}$
into a symmetric monoidal category. These two structures combine to a commutative Hopf algebra
structure on $\Shv(G)^{\on{q-mon}}$.

\sssec{}

Note that we have an equivalence in $\tDGCat$:
$$\Shv(G)^{\on{q-mon}}\commod(\DGCat)\simeq \underset{\chi}\oplus\, \Shv(G)^{\on{q-}\!\chi}\commod(\DGCat).$$

\sssec{}

Let
$$\Shv(G)^{\on{mult-decomp}}\subset \Shv(G)$$
be the full subcategory, consisting of objects $\CF\in \Shv(G)$, for which
$\on{mult}^!(\CF)$ belongs to the full subcategory
$$\Shv(G)\otimes \Shv(G)\subset \Shv(G\times G).$$

\medskip

By \corref{c:alm const}, we have
\begin{equation} \label{e:alm char decomp}
\Shv(G)^{\on{q-mon}}\subset \Shv(G)^{\on{mult-decomp}}
\end{equation}

In \secref{ss:decomp} we will prove:

\begin{thm} \label{t:decomp}
The inclusion \eqref{e:alm char decomp} is an equivalence.
\end{thm}

\ssec{Quasi-monodromic categorical representations} \label{ss:mon 1}

\sssec{}

Denote
$$G\mmod^{\on{q-}\!\chi}:=\bi(\Shv(G)^{\on{q-}\!\chi})\commod(\AGCat).$$

Note that since $\Shv(G)^{\on{q-}\!\chi}\in \DGCat$ is dualizable, we can identify
$$G\mmod^{\on{q-}\!\chi} \simeq \AGCat\underset{\DGCat}\otimes \left(\Shv(G)^{\on{q-}\!\chi}\commod(\DGCat)\right).$$

By a slight abuse of notation, we will denote by $\bi$ the (fully faithful functor)
$$\Shv(G)^{\on{q-}\!\chi}\commod(\DGCat)\to G\mmod^{\on{q-}\!\chi}.$$

\sssec{Example}

We have
$$\bi(\Vect^\chi)=\ul\Vect^\chi,$$
where the latter is as in \secref{sss:Vect chi}.

\sssec{}

Denote
$$G\mmod^{\on{q-mon}} :=\underset{\chi}\oplus\, G\mmod^{\on{q-}\!\chi}\simeq
\bi(\Shv(G)^{\on{q-mon}})\commod(\AGCat).$$

Note that we can identify
$$G\mmod^{\on{q-mon}} \simeq \AGCat\underset{\DGCat}\otimes (\Shv(G)^{\on{q-mon}}\commod(\DGCat)),$$
and let us denote by $\bi$ the corresponding (fully faithful) functor
$$\Shv(G)^{\on{q-mon}}\commod(\DGCat)\to G\mmod^{\on{q-mon}}.$$

\sssec{}

Note now that by the construction of the comonoidal structure on $\Shv(G)^{\on{q-mon}}$ we have a naturally defined map
of coalgebras in $\AGCat$:
\begin{equation} \label{e:mon to uShv}
\bi(\Shv(G)^{\on{q-mon}})\to \ul\Shv(G).
\end{equation}

Hence, we obtain a functor
\begin{equation} \label{e:G mon}
\bemb^{\on{q-mon}}:G\mmod^{\on{q-mon}}\to G\mmod.
\end{equation}

\begin{rem}
In fact, the map \eqref{e:mon to uShv} is compatible with commutative Hopf algebra structures on the two sides.
This makes \eqref{e:G mon} into a symmetric monoidal functor.
\end{rem}

\sssec{}

We claim:

\begin{thm} \label{t:mon cat rep} \hfill

\smallskip

\noindent{\em(a)} The functor \eqref{e:G mon} is fully faithful.

\smallskip

\noindent{\em(b)} Let $\ul\bC$ be an object of $G\mmod$, whose underlying object
of $\AGCat$ is of the form $\bi(\bC)$, where $\bC$ is a \emph{dualizable} DG category.
Then $\ul\bC$ lies in the essential image of the composition
$$\Shv(G)^{\on{q-mon}}\commod(\DGCat) \overset{\bi}\to G\mmod^{\on{q-mon}}\to G\mmod.$$
\end{thm}

We can restate \thmref{t:mon cat rep}(b) as follows:

\begin{cor} \label{c:mon cat rep}
An object $\ul\bC\in G\mmod$, for which the underlying object of $\AGCat$ is of
the form $\bi(\bC)$, where $\bC$ is a \emph{dualizable} DG category, splits as a direct sum
$$\ul\bC\simeq \underset{\chi}\oplus\, \bi(\bC^\chi),$$
where
$$\bC^\chi\in \Shv(G)^{\on{q-}\!\chi}\mmod(\DGCat).$$
\end{cor}

\sssec{}

We conjecture that the assertion of \thmref{t:mon cat rep}(b) holds without the dualizability
assumption. Namely, we propose:

\begin{conj} \label{c:restr torus}
The essential image of the composition
$$\Shv(G)^{\on{q-mon}}\commod(\DGCat) \overset{\bi}\to G\mmod^{\on{q-mon}}\to G\mmod$$
consists of those objects of $G\mmod$, for which the underlying object of $\AGCat$ lies
in the essential image of the functor $\bi:\DGCat\to \AGCat$.
\end{conj}

\ssec{Proof of \thmref{t:mon cat rep}}

\sssec{Proof of \thmref{t:mon cat rep}(a)}

To prove point (a), it suffices to show that for any $\ul\bC\in \AGCat$, the functor
$$\ul\bC(X)\otimes \Shv(G)^{\on{q-mon}}\to \ul\bC(X\times G)$$
is fully faithful.

\medskip

We factor the above functor as
\begin{equation} \label{e:ff mon}
\ul\bC(X)\otimes \Shv(G)^{\on{q-mon}}\to \ul\bC(X)\otimes \Shv(G)\simeq (\ul\bC\otimes \bi(\Shv(G)))(X)\to
(\ul\bC\otimes \ul\Shv(G))(X)\simeq \ul\bC(X\times G),
\end{equation}
and we note that the third arrow is fully faithful for any $\ul\bC$ (since $\bi(\Shv(G))\to \ul\Shv(G)$ is fully faithful
and admits a right adjoint).

\medskip

Hence, the required assertion follows from the next proposition, which will be proved in \secref{ss:ff mon}:

\begin{prop} \label{p:ff mon}
The functor
$$\bC\otimes \Shv(G)^{\on{q-mon}}\to \bC\otimes \Shv(G)$$
is fully faithful for any $\bC\in \DGCat$.
\end{prop}

\sssec{Proof of \thmref{t:mon cat rep}(b)}

It suffices to show that if $\bC$ is equipped with a map \eqref{e:coact naive} that makes
\eqref{e:assoc X naive} commute, then the essential image of \eqref{e:coact naive} lies
in
$$\bC\otimes \Shv(G)^{\on{q-mon}}\subset \bC\otimes \Shv(G),$$
provided that $\bC$ is dualizable.

\medskip

I.e., it suffices to show that the functor
$$\bC\otimes (\Shv(G)/\Shv(G)^{\on{q-mon}})\to \bC\otimes \left(\Shv(G\times G)/(\Shv(G)\otimes \Shv(G))\right),$$
induced by $\on{mult}^!$, is conservative.

\medskip

The functor
$$\Shv(G)/\Shv(G)^{\on{q-mon}}\to \Shv(G\times G)/(\Shv(G)\otimes \Shv(G))$$
is conservative by \thmref{t:decomp}.

\medskip

Now, for any conservative functor $\bD_1\to \bD_2$, the functor
$$\bC\otimes \bD_1\to \bC\otimes \bD_2$$
stays conservative if $\bC$ is dualizable: indeed, we can rewrite $\bC\otimes \bD$ as
$$\ul\Hom_{\DGCat}(\bC^\vee,\bD).$$

\qed

\ssec{The maximal quasi-monodromic subcategory} \label{ss:q-mon 2}

In this subsection we will show that given an object $\ul\bC\in G\mmod$, there is a well-defined operation that extracts its
maximal quasi-monodromic subobject.

\sssec{}

Consider again the functor $\bemb^{\on{q-mon}}$ of \eqref{e:G mon}. We claim that it admits a right adjoint, to be denoted
\begin{equation} \label{e:extract q-mon}
\ul\bC\mapsto \ul\bC^{\on{q-mon}}.
\end{equation}

Namely, the category $\ul\bC^{\on{q-mon}}(X)$
is the full subcategory of $\ul\bC(X)$ that consists of objects $\bc$ for which
$$\on{coact}_X(\bc)\in \ul\bC(X\times G)$$
belongs to the full subcategory
$$\ul\bC(X)\otimes \Shv(G)^{\on{q-mon}}\subset \ul\bC(X\times G)$$
(see \eqref{e:ff mon} above).

\medskip

The tautological embedding
$$\on{emb}^{\on{q-mon}}:\ul\bC^{\on{q-mon}}\to \ul\bC$$
is the counit of the $(\bemb^{\on{q-mon}},(\bemb^{\on{q-mon}})^R)$-adjunction.

\medskip

The unit of the adjunction is an isomorphism, since $\bemb^{\on{q-mon}}$ is fully faithful.

\sssec{}

Let us describe the functor \eqref{e:extract q-mon} in the geometric situation.

\medskip

For a prestack $\CY$ equipped with an action of $G$, denote by
\begin{equation} \label{e:q-mon shvs}
\Shv(\CY)^{\on{q-mon}}\subset \Shv(\CY)
\end{equation}
the full subcategory consisting of objects $\CF$, for which
$$\on{act}^!(\CF), \quad \on{act}:G\times \CY\to \CY$$
lies in the full subcategory
$$\Shv(G)^{\on{q-mon}}\otimes \Shv(\CY)\subset \Shv(G)\otimes \Shv(\CY) \subset \Shv(G\times \CY).$$

\medskip

It is clear that (for any $\CY$), the object $\Shv(\CY)^{\on{q-mon}}\in \DGCat$ naturally upgrades to an object of
$\Shv(G)^{\on{q-mon}}\commod(\DGCat)$.

\sssec{}

For $\CY$ as above, consider $\ul\Shv(\CY)\in G\mmod$. Consider the corresponding object
$$\ul\Shv(\CY)^{\on{q-mon}}\in G\mmod,$$
see \eqref{e:extract q-mon}.

\medskip

It is clear that for $X\in \Sch$,
$$\ul\Shv(\CY)^{\on{q-mon}}(X)=\Shv(\CY\times X)^{\on{q-mon}}$$
as subcategories of
$$\ul\Shv(\CY)(X)=\Shv(\CY\times X).$$

\sssec{}

We will now describe the sub-object
$$\ul\Shv(\CY)^{\on{q-mon}}\subset \ul\Shv(\CY)$$
in a particular situation.

\medskip

\begin{prop} \label{p:q-mon restr}
Suppose that $\CY$ is algebraic stack on which $G$ acts with finitely many orbits. Then the natural map
$$\bi(\Shv(\CY)^{\on{q-mon}})\to \ul\Shv(\CY)^{\on{q-mon}}$$
is an isomorphism. In particular, the functor
$$\Shv(\CY)^{\on{q-mon}}\otimes \Shv(X)\to \Shv(\CY\times X)^{\on{q-mon}}, \quad X\in \Sch$$
is an equivalence.
\end{prop}

\sssec{}

Note that as a particular case (for $G=1$) we obtain:

\begin{cor} \label{c:restricted stack}
Let $\CY$ be an algebraic stack with finitely many isomorphism classes of points. Then the object
$$\ul\Shv(\CY)\in \AGCat$$
is restricted, i.e., the functor
$$\Shv(\CY)\otimes \Shv(X)\to \Shv(\CY\times X), \quad X\in \Sch$$
is an equivalence.
\end{cor}

\sssec{}

As another particular case, for $\CY=G$ equipped with the left action on itself, we obtain:

\begin{cor} \label{c:q-mon grp retsr}
The object $\ul\Shv(G)^{\on{q-mon}}\in \AGCat$
$$\ul\Shv(G)^{\on{q-mon}}(X):=\Shv(X\times G)^{\on{q-mon}}$$
is restricted, i.e., the functor
$$\Shv(G)^{\on{q-mon}}\otimes \Shv(X)\to \Shv(G\times X)^{\on{q-mon}}, \quad X\in \Sch$$
is an equivalence.
\end{cor}

\ssec{Proof of \propref{p:q-mon restr}} \label{ss:q-mon 3}

\sssec{}

We have to show that for any scheme $X$, the map
\begin{equation} \label{e:emb q-mon again}
\Shv(\CY)^{\on{q-mon}}\otimes \Shv(X)\to \Shv(\CY\times X)^{\on{q-mon}}
\end{equation}
is an equivalence. Since the functor in question is fully faithful, we only have to establish the essential surjectivity.

%

\medskip

By devissage, we can assume that $G$ acts transitively on $\CY$, i.e.,
$$\CY\simeq G/H,$$
where $H$ is an algebraic group that maps to $G$ (not necessarily injectively). We will use this assumption
later in the proof.

\sssec{}

It suffices to show that the right adjoint of \eqref{e:emb q-mon again} is conservative. Consider the right
adjoint $\boxtimes^R$ of the functor
$$\Shv(\CY)\otimes \Shv(X) \overset{\boxtimes}\to \Shv(\CY\times X).$$

We claim that it sends $\Shv(\CY\times X)^{\on{q-mon}}$ to $\Shv(\CY)^{\on{q-mon}}\otimes \Shv(X)$,
thereby providing a right adjoint to \eqref{e:emb q-mon again}.

\medskip

Indeed, this follows from the fact that the diagram
$$
\CD
\Shv(G)^{\on{q-mon}}\otimes \Shv(\CY)\otimes \Shv(X) @<{\on{Id}\otimes \boxtimes^R}<< \Shv(G)^{\on{q-mon}}\otimes \Shv(\CY\times X) \\
@VVV @VVV \\
\Shv(G)\otimes \Shv(\CY)\otimes \Shv(X) @<{\on{Id}\otimes \boxtimes^R}<< \Shv(G)\otimes \Shv(\CY\times X) \\
@VVV @VVV \\
\Shv(G\times \CY)\otimes \Shv(X) @<{\boxtimes^R}<< \Shv(G\times \CY\times X) \\
@A{\on{act}^!\otimes \on{Id}}AA @AA{(\on{act}\times \on{id})^!}A \\
\Shv(\CY)\otimes \Shv(X) @<{\boxtimes^R}<< \Shv(\CY\times X)
\endCD
$$
commutes (the commutation of the bottom square is obtained by passing to the left adjoints along all arrows, and the commutation
of the middle square is \lemref{l:triple} below).

\medskip

Thus, it remains to show that the functor $\boxtimes^R$ is conservative when restricted to $\Shv(\CY\times X)^{\on{q-mon}}$.

\sssec{}

Let us compose the functor $\boxtimes^R$ with the functor
$$(\iota_{1,\CY})^!\otimes \on{Id}:\Shv(\CY)\otimes \Shv(X)\to \Shv(X),$$
where $\iota_{1,\Y}$ is the map $\on{pt}\to G/H\simeq \CY$ corresponding to the unit point.

\medskip

This composition is isomorphic to
$$(\iota_{1,\CY}\times \on{id})^!:\Shv(\CY\times X)\to \Shv(X)$$
(one sees that by passing to left adjoints).

\medskip

Hence, it is sufficient to show that the functor $(\iota_{1,\CY}\times \on{id})^!$ is conservative, when restricted to
the subcategory $\Shv(\CY\times X)^{\on{q-mon}}$.

\sssec{}

Let us now take $\CY=G/H$. We can rewrite the functor $(\iota_{1,\CY}\times \on{id})^!$ as
$$(\iota_{1,G}\times \on{id})^! \circ (\on{p}\times \on{id})^!, \quad p:G\to G/H\simeq \CY.$$

The functor $(p\times \on{id})^!$ is conservative. We claim that it sends $\Shv(\CY\times X)^{\on{q-mon}}$ to
$$\Shv(G)^{\on{q-mon}}\otimes \Shv(X)\subset \Shv(G\times X).$$ 

Indeed, this follows by rewriting $(p\times \on{id})^!$ as
$$(\on{id}_G\times \iota_{1,\CY}\times \on{id})^!\circ (\on{act}\times \on{id})^!.$$

\medskip

Hence, we are reduced to showing that the functor
$$(\iota_{1,G}\times \on{id})^!:\Shv(G\times X)\to \Shv(X)$$
is conservative, when restricted to $\Shv(G)^{\on{q-mon}}\otimes \Shv(X)$.

\sssec{}

Note, however, that the latter functor identifies with
$$\iota_{1,G}^!\otimes \on{id}_{\Shv(X)}.$$

Now the assertion follows from the fact that $\iota_{1,G}^!$ is conservative on $\Shv(G)^{\on{q-mon}}$: indeed, since
$\Shv(X)$ is dualizable, tensoring by it preserves conservativity.

\qed[\propref{p:q-mon restr}]

\sssec{}

We claim:

\begin{lem} \label{l:triple}
For a triple of tame prestacks stacks $\CY_1,\CY_2,\CY_3$, the diagram
$$
\CD
\Shv(\CY_1)\otimes \Shv(\CY_2) \otimes \Shv(\CY_3) @<{\on{Id}\otimes \boxtimes^R}<< \Shv(\CY_1)\otimes \Shv(\CY_2\times \CY_3) \\
@VVV @VVV \\
\Shv(\CY_1\times \CY_2) \otimes \Shv(\CY_3) @<{\boxtimes^R}<< \Shv(\CY_1\times \CY_2 \times \CY_3),
\endCD
$$
which a priori commutes up to a natural transformation, actually commutes.
\end{lem}

\begin{proof}

The tameness assumption readily reduces the assertion to the case when $\CY_i=Y_i$ are schemes. Fix objects
$$\CF_1\in \Shv(Y_1),\,\, \CF_{23}\in \Shv(Y_2\times Y_3),\,\, \CF_{12}\in \Shv(\CY_1\times \CY_2),\,\,\CF_3\in \Shv(\CY_3).$$
It suffices to show that the map
\begin{equation} \label{e:triple}
\langle \CF_{12}\boxtimes \CF_3,\CF_1\boxtimes (\boxtimes\circ \boxtimes^R(\CF_{23}))\rangle_{Y_1\times Y_2\times Y_3}\to
\langle \CF_{12}\boxtimes \CF_3,\boxtimes\circ \boxtimes^R(\CF_1\boxtimes \CF_{23})\rangle_{Y_1\times Y_2\times Y_3}
\end{equation}
is an isomorphism, where
$$\langle -,-\rangle_Y:\Shv(Y)\times \Shv(Y)\to \Vect$$
denotes the Verdier duality pairing.

\medskip

We rewrite the left-hand side of \eqref{e:triple} as
\begin{equation} \label{e:triple LHS}
\langle (p_2)_*(\CF_{12}\sotimes p_1^!(\CF_1))\boxtimes \CF_3,\boxtimes\circ \boxtimes^R(\CF_{23})\rangle_{Y_2\times Y_3}.
\end{equation}

\medskip

Note that for a pair of schemes $Z_1$ and $Z_2$, the functors
$$\boxtimes:\Shv(Z_1)\otimes \Shv(Z_2)\leftrightarrows \Shv(Z_1\times Z_2):\boxtimes^R$$
are mutually \emph{dual} under the Verdier duality identifications
$$\Shv(Z_i)^\vee \simeq \Shv(Z_i).$$
(Indeed, this is a formal consequence of the fact that the functor $\boxtimes$ commutes with Verdier duality on compact objects.)

\medskip

Hence, we can rewrite \eqref{e:triple LHS} as
\begin{equation} \label{e:triple LHS 1}
\langle (p_2)_*(\CF_{12}\sotimes p_1^!(\CF_1))\boxtimes \CF_3,\CF_{23}\rangle_{Y_2\times Y_3},
\end{equation}
and further as
\begin{equation} \label{e:triple LHS 2}
\langle \CF_{12}\boxtimes \CF_3,\CF_1\boxtimes \CF_{23}\rangle_{Y_1\times Y_2\times Y_3}.
\end{equation}

Similarly, we rewrite right-hand side of \eqref{e:triple} as
$$\langle \CF_{12}\boxtimes \CF_3,\CF_1\boxtimes \CF_{23}\rangle_{Y_1\times Y_2\times Y_3},$$
which is the same as \eqref{e:triple LHS 2}.

\medskip

Unwinding, we obtain that this is the same identification as the one induced by the natural transformation.

\end{proof}

\ssec{Quasi-monodromic vs monodromic}

In this subsection we introduce a variant of the ``quasi-constant/quasi-monodromic" condition, which
exhibits better functorial properties.

\medskip

We note, however, that when $G=T$ is a torus (which is the case of primary interest), the two conditions
are the same.

\sssec{} \label{sss:i-q}

Let
\begin{equation} \label{e:ind const}
\Shv(G)^{\on{alm-const}}\subset \Shv(G)
\end{equation}
be the full DG subcategory generated under colimits by the constant sheaf.

\medskip

We have an inclusion
$$\Shv(G)^{\on{alm-const}}\subset \Shv(G)^{\on{q-const}},$$
but it is \emph{not} an equality as long as $G$ has a non-trivial semi-simple part, see \secref{sss:q vs ind}.

\sssec{}

We have the following counterpart of \lemref{l:alm const}:

\begin{lem} \label{l:alm const prod}
For a pair of groups $G_1,G_2$, the functor
$$\boxtimes:\Shv(G_1)^{\on{alm-const}}\otimes \Shv(G_2)^{\on{alm-const}}\to \Shv(G_1\times G_2)$$
is an equivalence onto
$$\Shv(G_1\times G_2)^{\on{alm-const}}\subset \Shv(G_1\times G_2).$$
\end{lem}

\begin{proof}

The functor in question is fully faithful, preserves compacts, and
sends the compact generator of the left-hand side to the
compact generator of the right-hand side.

\end{proof}

\sssec{}

The material in Sects. \ref{ss:q-mon}-\ref{ss:mon 1} and \ref{ss:q-mon 2}-\ref{ss:q-mon 3} readily adapts to the present
context. In particular:

\medskip

\begin{itemize}

\item The functor $\on{mult}^!$ induces comonoidal structures on
$$\Shv(G)^{\on{alm-}\!\chi}:=\chi\otimes \Shv(G)^{\on{alm-const}}
\text{ and } \Shv(G)^{\on{mon}}:=\underset{\chi}\oplus\, \Shv(G)^{\on{alm-}\!\chi};$$

\medskip

\item Consider the (symmetric monoidal) category
\begin{multline*}
G\mmod^{\on{mon}}:=\underset{\chi}\oplus\, G\mmod^{\on{alm-}\!\chi}
\simeq \bi(\Shv(G)^{\on{mon}})\commod(\AGCat) \simeq \\
\simeq \AGCat\underset{\DGCat}\otimes (\Shv(G)^{\on{mon}}\commod(\DGCat)).
\end{multline*}
We have the fully faithful embeddings
$$\Shv(G)^{\on{mon}}\commod(\DGCat) \overset{\bi}\hookrightarrow G\mmod^{\on{mon}} \overset{\bemb^{\on{mon}}}\hookrightarrow G\mmod;$$

\medskip

\item The functor $\bemb^{\on{mon}}$ admits a right adjoint
$$\ul\bC\mapsto \ul\bC^{\on{mon}}$$
and the counit of the adjunction
$$\on{emb}^{\on{mon}}:\ul\bC^{\on{mon}}\to \ul\bC$$
is fully faithful;

\medskip

\item For a prestack $\CY$ with an action of $G$, the subcategory
$$\ul\Shv(\CY)^{\on{mon}}(\on{pt})\subset \ul\Shv(\CY)(\on{pt})=\Shv(\CY)$$
equals
$$\Shv(\CY)^{\on{mon}}\subset \Shv(\CY),$$
i.e., it consists of those $\CF\in \Shv(\CY)$ for which
$$\on{act}^!(\CF)\in \Shv(G)^{\on{mon}}\otimes \Shv(\CY)\subset \Shv(G\times \CY);$$

\medskip

\item A counterpart of \propref{p:q-mon restr} holds: if $\CY$ is an algebraic stack with finitely many $G$-orbits, then
the natural map
\begin{equation} \label{e:mon restr Y}
\bi(\Shv(\CY)^{\on{mon}})\to \ul\Shv(\CY)^{\on{mon}}
\end{equation}
is an isomorphism.

\end{itemize}

\medskip

We now highlight some points of difference between the ``almost" and ``quasi" situations.

\sssec{}

We claim that for a given character sheaf $\chi$, the category $\Shv(G)^{\on{alm-}\!\chi}\commod(\DGCat)$ is
generated by $\Vect^\chi$.

\medskip

In fact, we claim that for any $\bC\in \Shv(G)^{\on{alm-}\!\chi}\commod(\DGCat)$, we have
\begin{equation} \label{e:recover V chi}
\bC\simeq \Vect^\chi\underset{\Shv(\on{pt}/G)}\otimes \ul\Hom_{\Shv(G)^{\on{alm-}\!\chi}\commod(\DGCat),\DGCat}(\Vect^\chi,\bC),
\end{equation}
where we identify
$$\Shv(\on{pt}/G)\simeq \ul\Hom_{\Shv(G)^{\on{alm-}\!\chi}\commod(\DGCat),\DGCat}(\Vect^\chi,\Vect^\chi),$$
(see \cite[Theorem B.4.3]{Ga} for the proof).

\medskip

It follows formally that $G\mmod^{\on{alm-}\!\chi}$ is generated by $\ul\Vect^\chi$ and for $\ul\bC\in G\mmod^{\on{alm-}\!\chi}$ we have
\begin{equation} \label{e:recover V chi alg geom}
\ul\bC\simeq \ul\Vect^\chi\underset{\bi(\Shv(\on{pt}/G))}\otimes \uHom_{G\mmod,\AGCat}(\ul\Vect^\chi,\ul\bC),
\end{equation}

\sssec{} \label{sss:inv gen mon}

For $\bC\in  \Shv(G)^{\on{alm-const}}\commod(\DGCat)$, the right adjoint, $\on{Av}^G_*$ of the forgetful functor
\begin{equation} \label{e:emb inv}
\binv_G(\bC):=\ul\Hom_{\Shv(G)^{\on{alm-const}}\commod(\DGCat),\DGCat}(\Vect,\bC)\to \bC
\end{equation}
is conservative. This follows formally from \eqref{e:recover V chi}.

\medskip

In particular the essential image of \eqref{e:emb inv} generates the target.

\medskip

It follows formally that for $\ul\bC\in G\mmod$, the essential image of
$$\underset{\chi}\oplus\, \binv_{G,\chi}(\ul\bC)(X)\to \ul\bC^{\on{mon}}(X), \quad X\in \Sch$$
generates the target, where
$$\binv_{G,\chi}(\ul\bC):=\binv_G(\ul\bC\otimes \ul\Vect^\chi).$$

\sssec{}

Applying \secref{sss:inv gen mon} to $\bC:=\Shv(\CY)$, where $\CY$ is a prestack with an action of $G$, we obtain that the essential
image of
$$\Shv(\CY/G)\simeq \Shv(\CY)^G:=\binv_G(\Shv(\CY)):=\binv_G(\ul\Shv(\CY))(\on{pt})\to \Shv(\CY)^{\on{alm-const}}$$
generates the target, where
$$\Shv(\CY)^{\on{alm-const}}\subset \Shv(\CY)^{\on{mon}}$$
is the direct summand corresponding to the trivial $\chi$.

\medskip

Similarly, the essential image of
$$\underset{\chi}\oplus\,  \binv_{G,\chi}(\Shv(\CY))\to \Shv(\CY)^{\on{mon}}$$
generates the target, where
$$\binv_{G,\chi}(\Shv(\CY)):=\Shv(\CY)^{G,\chi}.$$

\ssec{The existence of left adjoints}

In this subsection we prove (a seemingly, technical) assertion that says that the $\on{Av}^G_!$-averaging functor
is defined on the $G$-monodromic subcategory. We will need this assertion in the sequel.

\medskip

That said, this subsection can be skipped on the first pass, and returned to when necessary.

\sssec{}

Consider
$$\bC\mapsto \bC:=\boblv_G(\bC) \text{ and } \bC\mapsto \binv_G(\bC)$$
as functors
$$\Shv(G)\commod(\DGCat)\to \DGCat.$$

They are connected by the usual adjunction
\begin{equation}  \label{e:simple inv}
\on{oblv}_G:\binv_G(\bC) \rightleftarrows \bC:\Av^G_*.
\end{equation}

\sssec{}

However, we claim:

\begin{prop} \label{p:simple inv}
The natural transformation $\on{oblv}_G$ in \eqref{e:simple inv} admits a \emph{left} adjoint when restricted
to
$$\Shv(G)^{\on{mon}}\commod(\DGCat)\subset \Shv(G)\commod(\DGCat),$$
Furthermore, we have
$$\Av^G_!\simeq \Av^G_*[\dim(G_{\on{red}})],$$
where $G_{\on{red}}$ denotes the reductive quotient of $G$.
\end{prop}

\sssec{}

Before we prove this proposition, let us see how it applies to $G\mmod$.
Let $\CY$ be a prestack equipped with an action of $G$; let $\pi_\CY$ denote the projection
$\CY\to \CY/G$.

\medskip

Consider the functor
$$(\pi_\CY)_!:\Shv(\CY)\to \Shv(\CY/G).$$

\medskip

For a map $f:\CY_1\to \CY_2$, we have the Beck-Chevalley natural transformations
\begin{equation} \label{e:BC !-Av}
(\pi_{\CY_2})_!\circ f_*\to (f/G)_*\circ (\pi_{\CY_1})_! \text{ and } (\pi_{\CY_2})_!\circ f^!\to (f/G)^!\circ (\pi_{\CY_1})_!
\end{equation}

We obtain:

\begin{cor} \label{c:!-Av mon} \hfill

\smallskip

\noindent{\em(a)}
The natural transformations \eqref{e:BC !-Av} are isomorphisms when restricted to
$$\Shv(\CY_1)^{\on{mon}}\subset \Shv(\CY_1) \text{ and } \Shv(\CY_2)^{\on{mon}}\subset \Shv(\CY_2),$$
respectively.

\smallskip

\noindent{\em(b)} There exists a canonical isomorphism
$$(\pi_\CY)_!\simeq (\pi_\CY)_*[\dim(G_{\on{red}})]$$
as functors $\Shv(\CY)^{\on{mon}}\to  \Shv(\CY)$.
\end{cor}

\sssec{Proof of \propref{p:simple inv}}

It is clear that for the proof of the proposition, we can replace $\Shv(G)^{\on{mon}}\commod(\DGCat)$ by $\Shv(G)^{\on{alm-const}}\commod(\DGCat)$.

\medskip

For $\bC\in \Shv(G)^{\on{alm-const}}\commod(\DGCat)$, by \eqref{e:recover V chi}, we identify
$$\bC\simeq \Vect\underset{\Shv(\on{pt}/G)}\otimes \binv_G(\bC).$$

In terms of this identification $\on{oblv}_G$ corresponds to the functor
$$\bC\simeq \Shv(\on{pt}/G)\underset{\Shv(\on{pt}/G)}\otimes \binv_G(\bC)\to \Vect\underset{\Shv(\on{pt}/G)}\otimes \binv_G(\bC),$$
induced by
\begin{equation} \label{e:pullback to pt}
\Shv(\on{pt}/G) \overset{p^*}\to \Shv(\on{pt})\simeq \Vect.
\end{equation}

\sssec{}

We identify
$$\Shv(\on{pt}/G)\simeq \on{C}_\cdot(G)\mod,$$
where $\on{C}_\cdot(G)$ is an associative algebra with respect to convolution.

\medskip

In terms of this identification, \eqref{e:pullback to pt} corresponds to the forgetful functor
$$\on{C}_\cdot(G)\mod\to \Vect.$$

The left and right adjoints of this functor are given by sending the 1-dimensional vector space to
$$\on{C}_\cdot(G) \text{ and } (\on{C}_\cdot(G))^\vee\simeq \on{C}^\cdot(G),$$
respectively.

\sssec{}

This makes the assertion manifest, since
$$\on{C}^\cdot(G)\simeq \on{C}_\cdot(G)[-\dim(G_{\on{red}})]$$
as $\on{C}_\cdot(G)$-modules.

\qed[\propref{p:simple inv}]

\ssec{The 2-adjointability of the embedding of the monodromic subcategory} \label{ss:2 adj mon}

We now come to the key point of difference between the monodoromic vs quasi-monodromic situations. Namely,
in this subsection we will show that the 1-morphism
$$\bemb^{\on{mon}}:G\mmod^{\on{mon}} \hookrightarrow  G\mmod$$
is 2-adjointable, in contrast with $\bemb^{\on{q-mon}}$, which is not in general 2-adjointable.

\sssec{}  \label{sss:Shv mon self-dual}

Note now that since $\Shv(G)^{\on{mon}}$ is generated by objects that are compact in $\Shv(G)$, for any scheme $X$, the embedding
$$\on{emb}^{\on{mon}}:
\Shv(G)^{\on{mon}}\otimes \Shv(X)\simeq \Shv(G\times X)^{\on{mon}}\to \Shv(G\times X)$$
admits a right adjoint.\footnote{We emphasize again
that this is \emph{not} the case for the corresponding functor $\on{emb}^{\on{q-mon}}:\Shv(G)^{\on{q-mon}}\to \Shv(G)$.}
(We draw the reader's attention to the notational distinction: $\bemb^{\on{mon}}$
denotes the 2-categorical functor, whereas $\on{emb}^{\on{mon}}$ denotes the 1-categorical one.)

\medskip

Moreover, Verdier duality on $G\times X$ induces an identification
$$(\Shv(G\times X)^{\on{mon}})^\vee\simeq \Shv(G\times X)^{\on{mon}}.$$
and under the self-duality and the Verdier self-duality of $\Shv(G\times X)$, the right adjoint of $\on{emb}^{\on{mon}}$
identifies with its dual.

\medskip

The latter interpretation makes it clear that for a map $f:X_1\to X_2$, the natural transformations
$$(\on{Id}\otimes f)^!\circ (\on{emb}^{\on{mon}})^R\to (\on{emb}^{\on{mon}})^R\circ (\on{id}\times f)^! \text{ and }
(\on{Id}\otimes f)_*\circ (\on{emb}^{\on{mon}})^R\to (\on{emb}^{\on{mon}})^R\circ (\on{id}\times f)_*$$
are isomorphisms.

\medskip

Hence, when we view $\on{emb}^{\on{mon}}$ as a map
$$\bi(\Shv(G)^{\on{mon}})\to \ul\Shv(G)$$
in $\AGCat$, it admits a right adjoint (see \cite[Proposition 2.7.2]{GRV}).

\sssec{}

By duality, the structure on $\Shv(G)^{\on{mon}}$ of commutative Hopf algebra in $\DGCat$ induces on
it a structure of cocommutative Hopf algebra. When viewed as such we will denote it by $\Shv(G)^{\on{mon}}_{\on{co}}$.

\medskip

Taking the dual of the canonical map
$$\on{emb}^{\on{mon}}:\bi(\Shv(G)^{\on{mon}})\to \ul\Shv(G)$$
we obtain a map
\begin{equation} \label{e:dual of emb mon}
(\on{emb}^{\on{mon}})^\vee:\ul\Shv(G)_{\on{co}}\to \bi(\Shv(G)^{\on{mon}}_{\on{co}}),
\end{equation}
compatible with the Hopf algebra structures.

\medskip

The functor \eqref{e:dual of emb mon} admits a left adjoint, namely,
$$((\on{emb}^{\on{mon}})^\vee)^L\simeq ((\on{emb}^{\on{mon}})^R)^\vee.$$

When we think of $\ul\Shv(G)_{\on{co}}$ (resp., $\Shv(G)^{\on{mon}}_{\on{co}}$)
as $\ul\Shv(G)$ (resp., $\Shv(G)^{\on{mon}}$), then $((\on{emb}^{\on{mon}})^\vee)^L$ is simply
the original embedding
$$\on{emb}^{\on{mon}}:\bi(\Shv(G)^{\on{mon}})\to \ul\Shv(G).$$

\sssec{}

Let us consider $\ul\Shv(G)_{\on{co}}$ and $\bi(\ul\Shv(G)^{\on{mon}}_{\on{co}})$
as modules over $\ul\Shv(G)_{\on{co}}$, where the action on $\bi(\ul\Shv(G)^{\on{mon}}_{\on{co}})$ is
via $(\on{emb}^{\on{mon}})^\vee$.

\medskip

A priori, the functor $((\on{emb}^{\on{mon}})^\vee)^L$ is \emph{left-lax} compatible with the actions
of $\ul\Shv(G)_{\on{co}}$. However, we claim:

\begin{lem} \label{l:left is strict}
The structure on $((\on{emb}^{\on{mon}})^\vee)^L$ of left-lax compatibility with the $\ul\Shv(G)_{\on{co}}$-actions
is strict.
\end{lem}

\begin{proof}

By definition, we have to show that the natural transformation from
\begin{multline} \label{e:action factor 0}
\ul\Shv(G) \overset{(\on{emb}^{\on{mon}})^R}\longrightarrow \bi(\Shv(G)^{\on{mon}}) \overset{\on{mult}^!}\longrightarrow  \\
\to \bi(\Shv(G)^{\on{mon}})\otimes \bi(\Shv(G)^{\on{mon}}) \overset{\on{Id}\otimes \on{emb}^{\on{mon}}}\longrightarrow
\bi(\Shv(G)^{\on{mon}})\otimes \ul\Shv(G)
\end{multline}
to
\begin{equation} \label{e:action factor}
\ul\Shv(G) \overset{\on{mult}^!}\to \ul\Shv(G\times G) =\ul\Shv(G)\otimes \ul\Shv(G) \overset{(\on{emb}^{\on{mon}})^R\otimes \on{Id}}\longrightarrow \\
\bi(\Shv(G)^{\on{mon}})\otimes \ul\Shv(G)
\end{equation}
is an isomorphism.

\medskip

We rewrite \eqref{e:action factor 0} as
\begin{multline*} \label{e:action factor 1}
\ul\Shv(G) \overset{\on{mult}^!}\to \ul\Shv(G\times G) =\ul\Shv(G)\otimes \ul\Shv(G)
\overset{(\on{emb}^{\on{mon}})^R\otimes (\on{emb}^{\on{mon}})^R}\longrightarrow \\
\to \bi(\Shv(G)^{\on{mon}})\otimes \bi(\Shv(G)^{\on{mon}}) \overset{\on{Id}\otimes \on{emb}^{\on{mon}}} \longrightarrow
\bi(\Shv(G)^{\on{mon}})\otimes \ul\Shv(G),
\end{multline*}
which is the same as the composition of \eqref{e:action factor} with
$$\bi(\Shv(G)^{\on{mon}})\otimes \ul\Shv(G) \overset{\on{Id}\otimes (\on{emb}^{\on{mon}})^R}\longrightarrow
\bi(\Shv(G)^{\on{mon}})\otimes \bi(\Shv(G)^{\on{mon}}) \overset{\on{Id}\otimes \on{emb}^{\on{mon}}} \longrightarrow
\bi(\Shv(G)^{\on{mon}})\otimes \ul\Shv(G).$$

Thus, the statement is equivalent to the assertion that \eqref{e:action factor}
factors through the (fully faithful) embedding
$$\bi(\Shv(G)^{\on{mon}})\otimes \bi(\Shv(G)^{\on{mon}}) \overset{\on{Id}\otimes \on{emb}^{\on{mon}}}\hookrightarrow \bi(\Shv(G)^{\on{mon}})\otimes \ul\Shv(G).$$

By duality, this is equivalent to the statement that the functor
\begin{equation} \label{e:mult by mon}
\bi(\Shv(G)^{\on{mon}})\otimes \ul\Shv(G) \overset{\on{emb}^{\on{mon}}\otimes \on{Id}}\longrightarrow \ul\Shv(G)\otimes \ul\Shv(G)
\overset{\on{mult}_*}\longrightarrow \ul\Shv(G)
\end{equation}
factors via
$$\bi(\Shv(G)^{\on{mon}})\otimes \ul\Shv(G) \overset{\on{Id}\otimes (\on{emb}^{\on{mon}})^R}\twoheadrightarrow \bi(\Shv(G)^{\on{mon}})\otimes \bi(\Shv(G)^{\on{mon}}).$$

The key observation is that the functor \eqref{e:mult by mon} factors via
$$\bi(\Shv(G)^{\on{mon}}) \overset{\on{emb}^{\on{mon}}}\hookrightarrow \ul\Shv(G).$$

Hence, the functor \eqref{e:mult by mon} is isomorphic to the composition
\begin{multline} \label{e:mult by mon 1}
\bi(\Shv(G)^{\on{mon}})\otimes \ul\Shv(G) \overset{\on{emb}^{\on{mon}}\otimes \on{Id}}\longrightarrow \ul\Shv(G)\otimes \ul\Shv(G)
\overset{\on{mult}_*}\longrightarrow \\
\to \ul\Shv(G) \overset{(\on{emb}^{\on{mon}})^R}\longrightarrow \bi(\Shv(G)^{\on{mon}}) \overset{\on{emb}^{\on{mon}}}\hookrightarrow \ul\Shv(G).
\end{multline}

In its turn, \eqref{e:mult by mon 1} is isomorphic to
\begin{multline} \label{e:mult by mon 2}
\bi(\Shv(G)^{\on{mon}})\otimes \ul\Shv(G) \overset{\on{emb}^{\on{mon}}\otimes \on{Id}}\longrightarrow \ul\Shv(G)\otimes \ul\Shv(G)
\overset{(\on{emb}^{\on{mon}})^R\otimes (\on{emb}^{\on{mon}})^R}\longrightarrow \\
\to \bi(\Shv(G)^{\on{mon}}) \otimes \bi(\Shv(G)^{\on{mon}}) \overset{\on{mult}_*}\longrightarrow \bi(\Shv(G)^{\on{mon}})  \overset{\on{emb}^{\on{mon}}}\hookrightarrow \ul\Shv(G),
\end{multline}
which is the same as
\begin{multline*}
\bi(\Shv(G)^{\on{mon}})\otimes \ul\Shv(G)
\overset{\on{Id}\otimes (\on{emb}^{\on{mon}})^R}\longrightarrow
 \bi(\Shv(G)^{\on{mon}}) \otimes \bi(\Shv(G)^{\on{mon}}) \overset{\on{mult}_*}\longrightarrow \\
\to \bi(\Shv(G)^{\on{mon}})  \overset{\on{emb}^{\on{mon}}}\hookrightarrow \ul\Shv(G),
\end{multline*}
making the required factorization of \eqref{e:mult by mon} manifest.

\end{proof}

\sssec{}

Consider the adjunction
\begin{equation} \label{e:emb mon 2adj}
\bemb^{\on{mon}}:G\mmod^{\on{mon}}\leftrightarrows G\mmod:(\bemb^{\on{mon}})^R.
\end{equation}

We claim:

\begin{prop} \label{p:emb mon 2adj}
The  unit and counit of the adjunction in \eqref{e:emb mon 2adj} admit right adjoints (i.e., the functor $\bemb^{\on{mon}}$
is 2-adjointable).
\end{prop}

\begin{proof}

We will think of
$G\mmod$ (resp., $G\mmod^{\on{mon}}$) as
$$\ul\Shv(G)_{\on{co}}\mmod(\AGCat) \text{ and } \bi(\Shv(G)^{\on{mon}}_{\on{co}})\mmod(\AGCat),$$
respectively.

\medskip

In these terms, the functor $\bemb^{\on{mon}}$ is given by restriction along the homomorphism \eqref{e:dual of emb mon},
and its right adjoint is given by
$$\Hom_{\ul\Shv(G)_{\on{co}}}(\bi(\Shv(G)^{\on{mon}}_{\on{co}}),-).$$

\medskip

In terms of these identifications, the unit of the adjunction amounts to the equivalence
$$\ul\bC \simeq \Hom_{\ul\Shv(G)_{\on{co}}}(\bi(\Shv(G)^{\on{mon}}_{\on{co}}),\ul\bC), \quad \ul\bC\in \bi(\Shv(G)^{\on{mon}}_{\on{co}})\mmod(\AGCat),$$
which follows from the fact that
$$\bi(\Shv(G)^{\on{mon}}_{\on{co}})\underset{\ul\Shv(G)_{\on{co}}}\otimes \bi(\Shv(G)^{\on{mon}}_{\on{co}})\to \bi(\Shv(G)^{\on{mon}}_{\on{co}})$$
is an equivalence (which in turn follows from the fact that the functor $(\on{emb}^{\on{mon}})^\vee$ admits a fully faithful (left) adjoint).

\medskip

The counit of the adjunction is the functor
$$\Hom_{\ul\Shv(G)_{\on{co}}}(\bi(\Shv(G)^{\on{mon}}_{\on{co}}),\ul\bC) \to
\Hom_{\ul\Shv(G)_{\on{co}}}(\ul\Shv(G)_{\on{co}},\ul\bC) \simeq \ul\bC, \quad \ul\bC\in \ul\Shv(G)_{\on{co}}\mmod(\AGCat),$$
given by precomposition with $(\on{emb}^{\on{mon}})^\vee$.

\medskip

This functor admits a right adjoint, given by precomposition with $((\on{emb}^{\on{mon}})^\vee)^L$, which is well-defined thanks
to \lemref{l:left is strict}.

\end{proof}

\sssec{}

Concretely, \propref{p:emb mon 2adj} means that the for $\ul\bC\in G\mmod$, the tautological embedding
\begin{equation} \label{e:emb mon down}
\on{emb}^{\on{mon}}:\ul\bC^{\on{mon}}\to \ul\bC
\end{equation}
admits a right adjoint.

\medskip

The idempotent comonad $\on{emb}^{\on{mon}}\circ (\on{emb}^{\on{mon}})^R$ of $\ul\bC$ is given by
the action of the idempotent
$$((\on{emb}^{\on{mon}})^\vee)^L\circ (\on{emb}^{\on{mon}})^\vee(\one_{\Shv(G)})\in \Shv(G)_{\on{co}}.$$

\medskip

From here we obtain:

\begin{lem} \label{l:dual mon}
Suppose that $\ul\bC$ is dualizable. Then so is $\ul\bC^{\on{mon}}$, and we have a canonical
identification
$$(\ul\bC^{\on{mon}})^\vee\simeq (\ul\bC^\vee)^{\on{mon}}$$
so that
$$(\on{emb}_{\ul\bC}^{\on{mon}})^\vee\simeq (\on{emb}^{\on{mon}}_{\ul\bC^\vee})^R.$$
\end{lem}

\sssec{} \label{sss:mon righ adj}

In particular, taking $\ul\bC=\ul\Shv(\CY)$ for a prestack $\CY$ acted on by $G$, we obtain that the (fully faithful)
1-morphism
\begin{equation} \label{e:emb mon stacks}
\on{emb}^{\on{mon}}:\ul\Shv(\CY)^{\on{mon}}\to \ul\Shv(\CY)
\end{equation}
admits a right adjoint.

\medskip

Evaluating on $X=\on{pt}$, we obtain a right adjoint to the fully faithful embedding
$$\on{emb}^{\on{mon}}:\Shv(\CY)^{\on{mon}}\to \Shv(\CY).$$

\begin{rem}

The content of the existence of \eqref{e:emb mon stacks} is that the rights adjoints to
$$\on{emb}^{\on{mon}}:\Shv(\CY\times X)^{\on{mon}}\to \Shv(\CY\times X), \quad X\in \Sch$$
are compatible with the $(-)^!$ and $(-)_*$ functors for maps $X_1\to X_2$.

\medskip

That said, the last commutation is easy to see directly by passing to the left adjoints.

\end{rem}

\ssec{The case of a torus} \label{ss:alm-const T}

In this subsection, we will show that when $G=T$ is a torus, there is no difference between the
``almost" and ``quasi" conditions.

\sssec{}

We claim:

\begin{prop} \label{p:q mon torus}
Let $G=T$ be a torus. Then the inclusion
\begin{equation} \label{e:q const torus}
\Shv(T)^{\on{alm-const}}\hookrightarrow \Shv(T)^{\on{q-const}}
\end{equation}
is an equality.
\end{prop}

As an immediate corollary, we obtain:

\begin{cor} \label{c:q mon torus}
The inclusion
\begin{equation} \label{e:q mon torus}
\Shv(T)^{\on{mon}}\hookrightarrow \Shv(T)^{\on{q-mon}}
\end{equation}
is an equality.
\end{cor}

The rest of this subsection is devoted to the proof of \propref{p:q mon torus}.

\sssec{} \label{sss:q-const T}

Note that the subcategories
$$\Shv(T)^{\on{alm-const}}\hookrightarrow \Shv(T)^{\on{q-const}} \subset \Shv(T)$$
are compatible with the (perverse) t-structure on $\Shv(T)$.

\medskip

Moreover, it is easy to see that the embedding \eqref{e:q mon torus} induces an equivalence
on the corresponding bounded subcategories. Hence, it induces an equivalence on the bounded
below categories (as both categories are right-complete in their t-structures).

\medskip

Recall also (see \cite[Theorem 1.1.6]{AGKRRV1}) that $\Shv(T)$ is left-complete in its t-structure. This
implies that $\Shv(T)^{\on{q-const}}$ is also left-complete in its t-structure. Hence, the embedding
\eqref{e:q const torus} identifies $\Shv(T)^{\on{q-const}}$ with the left completion of $\Shv(T)^{\on{alm-const}}$.

\medskip

Therefore, the assertion of the proposition is equivalent to the statement that $\Shv(T)^{\on{alm-const}}$ is actually
left-complete in its t-structure.

\sssec{} \label{sss:alm const torus expl}

The category $\Shv(T)^{\on{alm-const}}$ is compactly generated by the constant sheaf $\on{const}_T$. Hence,
$$\Shv(T)^{\on{alm-const}}\simeq \on{C}^\cdot(T)\mod.$$

We identify $\on{C}^\cdot(T)\simeq \Sym(V[-1])\mod$, where $V$ is a finite-dimensional $\ol\BQ_\ell$-vector space
(placed in cohomological degree $0$); in fact $V=H^1(T,\ol\BQ_\ell)$.

\medskip

The functor of !-fiber at $1\in T$ identifies (up to a cohomological shift) with the functor
\begin{equation} \label{e:fiber torus}
(-)\underset{\Sym(V[-1])}\otimes \ol\BQ_\ell, \quad \Sym(V[-1])\mod\to \Vect.
\end{equation}

\sssec{}

By Koszul duality, we can identify
$$\Sym(V[-1])\mod \simeq \QCoh(V)_{\{0\}},$$
where the subscript $\{0\}$ indicates set-theoretic support at $0$.

\medskip

In terms of this equivalence, the functor \eqref{e:fiber torus} corresponds to the functor
$$\QCoh(V)_{\{0\}}\hookrightarrow \QCoh(V) \overset{\Gamma(V,-)}\longrightarrow \Vect.$$

\medskip

In particular, the t-structure\footnote{Usual or perverse--the two differ by a cohomological shift on our subcategory.}
on $\Shv(T)^{\on{alm-const}}$ corresponds (possibly, up to an overall cohomological shift)
to the t-structure on $\QCoh(V)_{\{0\}}$,
induced by the natural t-structure on $\QCoh(V)$.

\sssec{}

This makes the assertion manifest: the category $\QCoh(V)$ is left-complete in its t-structure,
and hence so is $\QCoh(V)_{\{0\}}$.

\qed[\propref{p:q mon torus}]

\section{Spectrally finite and restricted categorical representations} \label{s:restr}

In this section we will finally define the subcategory
$$G\mmod^{\on{spec.fin}}\subset G\mmod,$$
and show that as an object of $\tAGCat$, it is tensored up from $\DGCat$.

\medskip

We will also show that the induction functor
$$\bind_P:M\mmod\to G\mmod$$
has an additional adjointability property when restricted to $G\mmod^{\on{spec.fin}}$.

\medskip

In this section we let $G$ be a connected reductive group.

\ssec{The notion of restricted categorical representation} \label{ss:restr}

\sssec{}

We first consider the case when $G=T$ is a torus. In this case we give the following definitions:

\medskip

$$T\mmod^{\on{spec.fin}}:=T\mmod^{\on{mon}};$$
$$T\mmod^{\on{restr}}:=T\mmod^{\on{mon}}\underset{\AGCat}\times \DGCat,$$
where
$$T\mmod^{\on{mon}}\overset{\boblv_G}\to \AGCat \overset{\bi}\leftarrow \DGCat.$$

\medskip

In other words, $T\mmod^{\on{restr}}$ is a full subcategory of $T\mmod^{\on{mon}}$ equal to the essential image
of the forgetful functor
$$\Shv(T)^{\on{mon}}\commod(\DGCat)\overset{\bi}\to \bi(\Shv(T)^{\on{mon}})\commod(\AGCat).$$

\sssec{}

We now consider the case of a general reductive group.

\bigskip

\noindent{\bf Notational convention:} We will denote by $N$ the unipotent radical of (a chosen) Borel subgroup $B\subset G$.
We denote by $T$ the quotient $B/N$; this is the (abstract) Cartan of $G$.

\bigskip

\noindent{\bf Notational convention:} For the rest of this paper (excluding the Appendix), the subscript ``$\on{mon}$'' will refer to the condition
of monodromicity with respect to the Cartan subgroup of $G$ (rather than $G$ itself).

\sssec{}


We set
$$G\mmod^{\on{spec.fin}}:=G\mmod\underset{T\mmod}\times T\mmod^{\on{spec.fin}},$$
where $G\mmod\to T\mmod$ is the functor $\bjacq_B$.

\medskip

We set
$$G\mmod^{\on{restr}}:=G\mmod^{\on{spec.fin}}\underset{T\mmod^{\on{spec.fin}}}\times T\mmod^{\on{restr}}.$$

\sssec{}

Note that we can rewrite
$$G\mmod^{\on{restr}}=G\mmod^{\on{spec.fin}}\underset{\AGCat}\times \DGCat,$$
where
$$G\mmod^{\on{spec.fin}}\overset{\binv_N}\to \AGCat \overset{\bi}\leftarrow \AGCat.$$

\medskip

We claim:

\begin{prop} \label{p:charact restr}
Let $\ul\bC\in G\mmod$ be dualizable. Then its belongs to $G\mmod^{\on{restr}}$ if and only if
$\binv_N(\ul\bC)$ is restricted as an object of $\AGCat$ (i.e., belongs to the essential image of
the functor $\bi:\DGCat\to \AGCat$).
\end{prop}

\begin{proof}

Follows from \thmref{t:mon cat rep}(b) (combined with \corref{c:q mon torus}).

\end{proof}

\begin{rem}

Note that if we assume \conjref{c:restr torus} (for tori), then it would follow that the assertion of \propref{p:charact restr}
holds without the assumption that $\ul\bC\in G\mmod$ be dualizable.

\end{rem}

\sssec{Examples}

The objects $\ul\Shv(G/B)$ and $\ul\Shv(G/N)^{\on{mon}}$
belong to $G\mmod^{\on{restr}}$. If we tensor them
by an object $\ul\bD\in \AGCat$ that is \emph{not} in
$\DGCat$, we obtain objects in $G\mmod^{\on{spec.fin}}$ that
are no longer in $G\mmod^{\on{restr}}$.

\medskip

By contrast $\ul\Shv(G/N)\notin G\mmod^{\on{spec.fin}}$.

\sssec{}

It follows from \lemref{l:dual mon} that if a dualizable $\ul\bC$ belongs to 
$G\mmod^{\on{spec.fin}}$, then so does $\ul\bC^\vee$.

\medskip

It follows formally that the same is true for $G\mmod^{\on{restr}}$.

\sssec{}

Note that $G\mmod^{\on{restr}}$ is naturally tensored over $\DGCat$, and the tautological embedding
$$G\mmod^{\on{restr}}\to G\mmod^{\on{spec.fin}}$$
is $\DGCat$-linear.

\medskip

Hence, we obtain a functor
\begin{equation} \label{e:tensor up restr}
\AGCat\underset{\DGCat}\otimes G\mmod^{\on{restr}}\to G\mmod^{\on{spec.fin}}
\end{equation}
in $\tAGCat$.

\medskip

We are going to prove:

\begin{thm} \label{t:restr main} \hfill

\smallskip

\noindent{\em(a)} The embedding
$$G\mmod^{\on{spec.fin}} \overset{\bemb^{\on{spec.fin}}}\longrightarrow G\mmod$$
is 2-adjointable.

\smallskip

\noindent{\em(a')} The Beck-Chevalley natural transformation
$$\bjacq_B \circ (\bemb_G^{\on{spec.fin}})^R\to (\bemb_T^{\on{spec.fin}})^R\circ \bjacq_B$$
as functors
$$G\mmod\to T\mmod^{\on{spec.fin}}$$
is an isomorphism.

\smallskip

\noindent{\em(b)} The functor \eqref{e:tensor up restr} is an equivalence.

\end{thm}

\ssec{Categorical representations as modules over the Hecke algebra}

The proof of \thmref{t:restr main} is based on interpreting $G\mmod$ as modules
(in $\AGCat$) over the (categorical) Hecke algebra.

\sssec{}

Consider the object
$$\ul\Shv(G/N)\in G\mmod.$$

Denote
$$\ul\CH:=\ul\End_{G\mmod,\AGCat}(\ul\Shv(G/N))\in \on{AssocAlg}(\AGCat).$$

 The usual manipulation shows that as a plain object of $\AGCat$
 $$\ul\CH\simeq \ul\Shv(N\backslash G/N),$$
 and the algebra structure on it is given by convolution.

 \sssec{}

 Consider the functor
\begin{equation} \label{e:inv N on G}
\binv_N:G\mmod\to \AGCat.
\end{equation}

 It is co-represented (in the sense of $\uHom_{G\mmod,\AGCat}(-,-)$) by $\ul\Shv(G/N)$;
 in particular, it admits a left adjoint (see \secref{sss:adj geom prel}). Hence, we can identify
$$\ul\CH\simeq \ul\End_{\Maps_{\tAGCat}(G\mmod,\AGCat),\AGCat}(\binv_N),$$
i.e., $\ul\CH$ is the monad on $\AGCat$ corresponding to \eqref{e:inv N on G}.

\medskip

Thus, the functor $\binv_N$ upgrades to a functor
\begin{equation} \label{e:inv N on G enh}
\binv^{\on{enh}}_N:G\mmod\to \ul\CH\mmod(\AGCat).
\end{equation}

\sssec{}

We are going to prove:

\begin{thm} \label{t:G mod via Hecke}
The functor \eqref{e:inv N on G enh} is an equivalence.
\end{thm}

\begin{rem}

The above \thmref{t:G mod via Hecke} is an adaptation to the context of $\AGCat$ of a theorem
of \cite{BGO}, which was initially proved in the context of D-modules.

\end{rem}

\sssec{}

The rest of this subsection is devoted to the proof of this theorem.

\medskip

Recall
that by \corref{c:inv colimits}, the functor $\binv_G$ commutes with colimits
and tensor products by objects of $\AGCat$.

\medskip

Hence, by the Barr-Beck-Lurie theorem, in order to prove \thmref{t:G mod via Hecke}, it suffices
to show that the functor $\binv_N$  (equivalently, $\binv^{\on{enh}}_N$) is conservative.

\sssec{}

In order to show that $\binv^{\on{enh}}_N$ is conservative, it suffices to prove that the
functor
$$\bjacq_B:G\mmod\to T\mmod$$
is conservative.

\medskip

For the latter, it suffices to show that the endofunctor
$$\bind_B\circ \bjacq_B$$
of $G\mmod$ is conservative.

\sssec{}

Consider the counit of the $(\bind_B, \bjacq_B)$-adjunction. It is given by a map in $(G\times G)\mmod$
\begin{equation} \label{e:conv map G/N}
\ul\Shv(G/N\overset{T}\times N\backslash G)\to \ul\Shv(G),
\end{equation}
where:

\begin{itemize}

\item $\overset{T}\times$ means the quotient by the diagonal action of $T$;

\item The map in \eqref{e:conv map G/N} is given by pull-push along the diagram
$$
\CD
G\overset{B}\times G @>>> G/N\overset{T}\times N\backslash G \\
@VVV \\
G.
\endCD
$$

\end{itemize}

\sssec{}

Note, however, that the vertical map in the above diagram is proper and the horizontal is smooth.
Hence, the functor \eqref{e:conv map G/N} admits a right adjoint.

\medskip

We claim:

\begin{prop} \label{p:Springer monad}
The adjunction
$$\on{counit}:\bind_B\circ \bjacq_B(\ul\bC) \rightleftarrows \ul\bC:\on{counit}^R, \quad \ul\bC\in G\mmod$$
is monadic.
\end{prop}

\sssec{}

Let us assume \propref{p:Springer monad} for a moment, and finish the proof of \thmref{t:G mod via Hecke}:

\medskip

By \propref{p:Springer monad}, we recover $\ul\bC$ as
$$(\on{counit}^R\circ \on{counit})\mod\left(\ul\Shv(G/N\overset{T}\times N\backslash G)\underset{\ul\Shv(G)_{\on{co}}}\otimes \ul\bC\right).$$

Note that the monad $(\on{counit}^R\circ \on{counit})$ acting on
$$\ul\Shv(G/N\overset{T}\times N\backslash G)\underset{\ul\Shv(G)_{\on{co}}}\otimes \ul\bC$$
is given by
$$(\on{counit}^R\circ \on{counit})\otimes \on{Id}_{\ul\bC},$$
where the left factor is the corresponding monad on $\ul\Shv(G/N\overset{T}\times N\backslash G)$.

\medskip

In particular, for a morphism $F:\ul\bC_1\to \ul\bC_2$ in $G\mmod$, the diagram
\begin{equation} \label{e:monad Spr diag}
\CD
\ul\Shv(G/N\overset{T}\times N\backslash G)\underset{\ul\Shv(G)_{\on{co}}}\otimes \ul\bC_1 @>{\on{Id}\otimes F}>>
\ul\Shv(G/N\overset{T}\times N\backslash G)\underset{\ul\Shv(G)_{\on{co}}}\otimes \ul\bC_2 \\
@V{\on{counit}^R\circ \on{counit}}VV @VV{\on{counit}^R\circ \on{counit}}V \\
\ul\Shv(G/N\overset{T}\times N\backslash G)\underset{\ul\Shv(G)_{\on{co}}}\otimes \ul\bC_1 @>{\on{Id}\otimes F}>>
\ul\Shv(G/N\overset{T}\times N\backslash G)\underset{\ul\Shv(G)_{\on{co}}}\otimes \ul\bC_2
\endCD
\end{equation}
commutes.

\medskip

Let $F$ be such that
$$\ul\Shv(G/N\overset{T}\times N\backslash G)\underset{\ul\Shv(G)_{\on{co}}}\otimes \ul\bC_1\overset{\on{Id}\otimes F}\longrightarrow
\ul\Shv(G/N\overset{T}\times N\backslash G)\underset{\ul\Shv(G)_{\on{co}}}\otimes \ul\bC_2$$
is an equivalence. Then from \eqref{e:monad Spr diag} we obtain that $\on{Id}\otimes F$ induces an equivalence
\begin{multline*}
(\on{counit}^R\circ \on{counit})\mod\left(\ul\Shv(G/N\overset{T}\times N\backslash G)\underset{\ul\Shv(G)_{\on{co}}}\otimes \ul\bC_1\right)\to \\
(\on{counit}^R\circ \on{counit})\mod\left(\ul\Shv(G/N\overset{T}\times N\backslash G)\underset{\ul\Shv(G)_{\on{co}}}\otimes \ul\bC_2\right),
\end{multline*}
as required.

\qed[\thmref{t:G mod via Hecke}]

\sssec{Proof of \propref{p:Springer monad}}

By Barr-Beck-Lurie, it suffices to show that the functor $\on{counit}^R$ is conservative. We will
show that the composition
\begin{equation} \label{e:orgnl Spr}
\on{counit}\circ \on{counit}^R
\end{equation}
contains the identity endofunctor of $\ul\bC$ as a direct summand.

\medskip

In fact, we will show that the functor \eqref{e:orgnl Spr}
viewed as an endofunctor of
$$\ul\Shv(G)\in (G\times G)\mmod$$
contains the identity as a direct summand.

\medskip

However, unwinding, we obtain:

\begin{lem} \label{l:Springer}
The functor \eqref{e:orgnl Spr} is given by convolution
with the Springer sheaf
$$\on{Spr}\in \Shv(G/\on{Ad}(G)).$$
\end{lem}

The required assertion follows now from the fact that $\on{Spr}$ contains $\delta_{1,G}$ as a direct summand.

\qed[\propref{p:Springer monad}]

\ssec{The monodromic Hecke algebra}

\sssec{}

Consider the stack $N\backslash G/N$. We view it as acted on by the group $T\times T$. We observe:

\begin{lem} \label{l:mon Hecke}
The inclusions
$$\ul\Shv(N\backslash G/N)^{\on{left-}\!T\on{-mon}} \hookleftarrow
\ul\Shv(N\backslash G/N)^{(T\times T)\on{-mon}}
\hookrightarrow \ul\Shv(N\backslash G/N)^{\on{right-}\!T\on{-mon}}$$
are equalities.
\end{lem}

\begin{proof}

By devissage, it is sufficient to prove that for a single Bruhat cell $BwB$, the inclusions
$$\ul\Shv(N\backslash BwB/N)^{\on{left-}\!T\on{-mon}} \hookleftarrow
\ul\Shv(N\backslash BwB/N)^{(T\times T)\on{-mon}}  \hookrightarrow
\ul\Shv(N\backslash BwB/N)^{\on{right-}\!T\on{-mon}}$$
are equalities.

\medskip

We have
$$\ul\Shv(N\backslash BwB/N) \simeq \ul\Shv(T)$$
as objects of $(T\times T)\mmod$,
with the left action of $T$ given by multiplication, and the right action given by multiplication precomposed with $w$.

\medskip

Hence, for a given character sheaf
$$\ul\Shv(N\backslash BwB/N)^{\on{left-alm-}\!\chi} = \ul\Shv(N\backslash BwB/N)^{\on{right-alm-}\!w(\chi)},$$
making the assertion of the lemma manifest.

\end{proof}

\sssec{}  \label{sss:H mon prel}

In what follows, we will denote the object of $\AGCat$ that appears in \lemref{l:mon Hecke} by
$$\ul\Shv(N\backslash G/N)^{\on{mon}}.$$

\medskip

Consider the (fully faithful) map
$$\on{emb}^{\on{mon}}:\ul\Shv(N\backslash G/N)^{\on{mon}}\hookrightarrow \ul\Shv(N\backslash G/N).$$

By \secref{sss:mon righ adj}, the above 1-morphism admits a right adjoint (in $\AGCat$).
\begin{equation} \label{e:emb mon GNN}
(\on{emb}^{\on{mon}})^R:\ul\Shv(N\backslash G/N)\twoheadrightarrow \ul\Shv(N\backslash G/N)^{\on{mon}}.
\end{equation}

It follows from \lemref{l:mon Hecke} that the algebra structure on $\ul\Shv(N\backslash G/N)$
induces a (uniquely defined) algebra structure on $\ul\Shv(N\backslash G/N)^{\on{mon}}$ so that
\eqref{e:emb mon GNN} is a homomorphism.

\sssec{} \label{sss:H mon}

Denote the corresponding object of $\on{AssocAlg}(\AGCat)$ by $\ul\CH^{\on{mon}}$.

\medskip

From the isomorphism \eqref{e:mon restr Y} we obtain:
$$\ul\CH^{\on{mon}}\simeq \bi(\CH^{\on{mon}}),$$
where
$$\CH^{\on{mon}}:=\Shv(N\backslash G/N)^{\on{mon}}\in \on{AssocAlg}(\DGCat).$$

\sssec{}

For future use we record also the following consequence of \lemref{l:mon Hecke}:

\begin{cor} \label{c:Hecke mon ideal}
The full subcategory
$$\ul\Shv(N\backslash G/N)^{\on{mon}}\overset{\on{emb}^{\on{mon}}}\hookrightarrow
\ul\Shv(N\backslash G/N)$$
is a monoidal ideal.
\end{cor}

\sssec{}

We have a closed embedding
$$T\to N\backslash G/N,$$
which induces a map of algebras
$$\ul\Shv(T)_{\on{co}}\to \ul\CH.$$

\medskip

Note that we have a commutative diagram
$$
\CD
\ul\Shv(T)_{\on{co}} @>>> \ul\CH \\
@VVV @VVV  \\
\ul\Shv(T)^{\on{mon}}_{\on{co}}  @>>> \ul\CH^{\on{mon}}
\endCD
$$
of algebra objects in $\AGCat$.

\medskip

Denote by $\bRes^\CH_T$ the corresponding restriction functors
$$\ul\CH\mmod(\AGCat)\to T\mmod \text{ and } \ul\CH^{\on{mon}}\mmod(\AGCat)\to T\mmod^{\on{mon}}.$$

\sssec{}

Restriction along \eqref{e:emb mon GNN} gives rise to a (fully faithful) functor
$$\bemb^{\on{mon}}_\CH:\ul\CH^{\on{mon}}\mmod(\AGCat) \hookrightarrow \ul\CH\mmod(\AGCat).$$

It admits a right adjoint given by
$$\ul\bC\mapsto \Hom_{\ul\CH}(\ul\CH^{\on{mon}},\ul\bC).$$

We claim:

\begin{prop} \label{p:Hecke mon}  \hfill

\smallskip

\noindent{\em(a)} The functor $\bemb^{\on{mon}}_\CH$ is 2-adjointable, i.e., the unit and counit of the
$(\bemb^{\on{mon}}_\CH,(\bemb^{\on{mon}}_\CH)^R)$-adjunction admit right adjoints;

\smallskip

\noindent{\em(b)} An object $\ul\bC\in \ul\CH\mmod$ belongs to the essential image of $\bemb^{\on{mon}}_\CH$
if and only if $\bRes^\CH_T(\ul\bC)$ belongs to the essential image of
$$\bemb^{\on{mon}}_T:T\mmod^{\on{mon}}\to T\mmod.$$

\smallskip

\noindent{\em(c)} The Beck-Chevalley natural transformation
$$\bRes^\CH_T \circ (\bemb^{\on{mon}}_\CH)^R\to (\bemb^{\on{mon}}_T)^R\circ \bRes^\CH_T,$$
as functors
$$\ul\CH\mmod(\AGCat)\to T\mmod^{\on{mon}}$$
is an isomorphism.

\end{prop}

\begin{proof}

Follows from \secref{ss:2 adj mon} and the fact that the (quotient) objects
$$(\bemb^{\on{left}}_T)^R(\ul\CH) \twoheadleftarrow \ul\CH \twoheadrightarrow (\bemb^{\on{right}}_T)^R(\ul\CH)$$
both identify with $\ul\CH^{\on{mon}}$, which is a reformulation of \lemref{l:mon Hecke}.

\end{proof}

\sssec{Semi-rigidity}

Recall the notion of \emph{semi-rigid} monoidal category, see \cite[Appendix C]{AGKRRV1}.
We claim:

\begin{prop} \label{p:H mon semi-rigid}
The monoidal category $\CH^{\on{mon}}$ is semi-rigid.
\end{prop}

\begin{proof}

We will deduce the assertion from \propref{p:simple inv}. Namely, we will show that the functor
right adjoint to the monoidal operation on $\CH^{\on{mon}}$ is given, up to a cohomological twist,
bu pull-push along the diagram
$$
\CD
N\backslash G\overset{N}\times G/N @>>>  (N\backslash G/N) \times (N\backslash G/N) \\
@V{'\!\on{mult}}VV \\
N\backslash G/N,
\endCD
$$
where the map $'\!\on{mult}$ is induced by the multiplication map.

\medskip

Indeed, by definition, the monoidal operation on $\CH^{\on{mon}}$ is given by pull-push along
$$
\CD
N\backslash G\overset{N}\times G/N @>{'\!\on{mult}}>> N\backslash G/N \\
@VVV \\
(N\backslash G/N) \times (N\backslash G/N),
\endCD
$$
where:

\medskip

\begin{itemize}

\item The vertical map is smooth, so !-pullback with respect to it is isomorphic to the *-pullback, up to a cohomological shift;

\medskip

\item The horizontal map factors as
$$N\backslash G\overset{N}\times G/N \to N\backslash G\overset{B}\times G/N\to N\backslash G/N ,$$
where the second arrow is proper.

\end{itemize}

However, \corref{c:!-Av mon}(b) implies that the right adjoint to the *-pushforward along
$$N\backslash G\overset{N}\times G/N \to N\backslash G\overset{B}\times G/N,$$
\emph{resctricted to the monodromic subcategory},
is given pullback along the same map (up to a cohomological shift).

\end{proof}

 \sssec{}

Combining \propref{p:H mon semi-rigid} with \cite[Propositions C.5.5 and C.5.7]{AGKRRV1}, we obtain:

\begin{cor} \label{c:right adjoint Hecke}
Let $F:\bC_1\to \bC_2$ a 1-morphism in $\CH^{\on{mon}}\mmod(\DGCat)$.
If the functor between the underlying DG categories
admits a right/left adjoint, then so does $F$.
\end{cor}

\ssec{Restricted representations as modules over the monodromic Hecke algebra}

In this subsection we will finally prove \thmref{t:restr main}.

\sssec{}

First, combining \thmref{t:G mod via Hecke} with \propref{p:Hecke mon}(b), we obtain:

\begin{cor} \label{c:restr via Hecke}
There exists a unique isomorphism
\begin{equation} \label{e:restr via Hecke}
G\mmod^{\on{spec.fin}}\simeq \ul\CH^{\on{mon}}\mmod(\AGCat)
\end{equation}
that makes the diagram
$$
\CD
G\mmod^{\on{spec.fin}} @>{\sim}>> \ul\CH^{\on{mon}}\mmod(\AGCat) \\
@V{\bemb^{\on{spec.fin}}}VV @VV{\bemb^{\on{mon}}_\CH}V \\
G\mmod @>{\sim}>> \ul\CH\mmod(\AGCat)
\endCD
$$
commute.
\end{cor}

\sssec{}

The fact that $\bemb^{\on{spec.fin}}$ is 2-adjointable follows now from \propref{p:Hecke mon}(a).

\medskip

Point (a') of \thmref{t:restr main} follows from \propref{p:Hecke mon}(c).

\sssec{}

We now prove point (b) of \thmref{t:restr main}.

\medskip

By \corref{c:restr via Hecke}, we can identify
$$G\mmod^{\on{restr}}\simeq \ul\CH^{\on{mon}}\mmod(\AGCat)\underset{\AGCat}\times \DGCat.$$

Recall the object
$$\CH^{\on{mon}}\in \on{AssocAlg}(\DGCat),$$
see \secref{sss:H mon}.

\medskip

We obtain
$$\ul\CH^{\on{mon}}\mmod(\AGCat)\simeq \AGCat\underset{\DGCat}\otimes \CH^{\on{mon}}\mmod(\DGCat)$$
and
$$G\mmod^{\on{restr}}\simeq \bi(\CH^{\on{mon}})\mmod(\AGCat)\underset{\AGCat}\times \DGCat.$$

Since the functor $\bi:\DGCat\to \AGCat$ is fully faithful, the latter equivalence implies that
$$\bi(\CH^{\on{mon}})\mmod(\AGCat)\underset{\AGCat}\times \DGCat \simeq
\CH^{\on{mon}}\mmod(\DGCat).$$

\qed[\thmref{t:restr main}]

\sssec{}

In what follows we will use a short-hand notation
$$\ul\bC\mapsto \ul\bC^{\on{spec.fin}}$$
for the functor $(\bemb^{\on{spec.fin}})^R$.

\sssec{}

Note that in the course of the proof we have also established:

\begin{cor}  \label{c:restr via Hecke bis}
There exists a unique equivalence
\begin{equation} \label{e:restr via Hecke bis}
G\mmod^{\on{restr}}\simeq \CH^{\on{mon}}\mmod(\DGCat)
\end{equation}
that makes the diagram
$$
\CD
G\mmod^{\on{restr}} @>{\sim}>> \ul\CH^{\on{mon}}\mmod(\DGCat) \\
@V{\bi}VV @VV{\bi}V \\
G\mmod^{\on{spec.fin}} @>{\sim}>> \ul\CH^{\on{mon}}\mmod(\AGCat)
\endCD
$$
commute.
\end{cor}

\sssec{}

As a corollary of Corollaries \ref{c:restr via Hecke} and \ref{c:restr via Hecke bis}, we obtain:

\begin{cor} \label{c:restr prod} \hfill

\smallskip

\noindent{\em(a)}
For a pair of reductive groups $G_1$ and $G_2$, the (fully faithful) functor
$$G_1\mmod^{\on{spec.fin}}\otimes G_2\mmod^{\on{spec.fin}}\to
G_1\mmod \otimes G_2\mmod \simeq (G_1\times G_2)\mmod$$
is an equivalence onto $(G_1\times G_2)\mmod^{\on{spec.fin}}$.

\smallskip

\noindent{\em(b)} There exists a canonical equivalence
$$G_1\mmod^{\on{restr}}\otimes G_2\mmod^{\on{restr}} \simeq
(G_1\times G_2)\mmod^{\on{restr}}$$
compatible with the equivalence of point (a) via the functor $\bi$.
\end{cor}

\begin{proof}

Follows from the fact that the functor
$$\CH^{\on{mon}}_{G_1}\otimes \CH^{\on{mon}}_{G_2}\to
\CH^{\on{mon}}_{G_1\times G_2}$$
is an equivalence.

\end{proof}

\ssec{Inherited properties of restricted representations} \label{ss:inherited}

In \secref{s:Cat Rep} we established some general properties of the objects $G\mmod\in \tDGCat$
and 1-morphisms between them. In this subsection we will see that many of these properties pass
on to the spectrally finite subcategories.

\sssec{}

First, since $G\mmod^{\on{spec.fin}}$ (resp., $G\mmod^{\on{restr}}$) is of the form $A\mmod(\AGCat)$
(resp., $A\mmod(\DGCat)$), with $A=\CH^{\on{mon}}$, it is a dualizable object in $\tAGCat$ (resp., $\tDGCat$).

\medskip

Moreover, the inversion involution on $G$ makes $G\mmod^{\on{spec.fin}}$ (resp., $G\mmod^{\on{restr}}$)
self-dual, see \secref{sss:G-mod dualizable}.

\sssec{}

The unit of the self-duality on $G\mmod^{\on{spec.fin}}$ is represented by the object
$$\ul\Shv(G)^{\on{left-}\!G\on{-spec.fin}} \simeq \ul\Shv(G)^{(G\times G)\on{-spec.fin}} \simeq
\ul\Shv(G)^{\on{right-}\!G\on{-spec.fin}},$$
where the equivalences in the above formula follow from  \lemref{l:mon Hecke}.

\medskip

From there we obtain that the $\on{unit}_{G\mmod^{\on{spec.fin}}}$ is isomorphic to
\begin{multline*}
((\bemb^{\on{spec.fin}})^R\otimes \on{Id})(\on{unit}_{G\mmod}) \simeq \\
\simeq ((\bemb^{\on{spec.fin}})^R\otimes (\bemb^{\on{spec.fin}})^R)(\on{unit}_{G\mmod})
\simeq(\on{Id}\otimes (\bemb^{\on{spec.fin}})^R)(\on{unit}_{G\mmod}).
\end{multline*}

Hence, we obtain:

\begin{cor}
With respect to the self-dualities
$$G\mmod^\vee \simeq G\mmod  \text{ and }
(G\mmod^{\on{spec.fin}})^\vee \simeq G\mmod^{\on{spec.fin}}$$
we have
$$(\bemb^{\on{spec.fin}})^\vee\simeq (\bemb^{\on{spec.fin}})^R.$$
\end{cor}

\sssec{}

Let $G_1$ and $G_2$ be a pair of reductive groups, and let
$$F:G_1\mmod\to G_2\mmod$$
be a 1-morphism in $\tAGCat$.

\medskip

Denote
$$F^{\on{spec.fin}}:=(\bemb^{\on{spec.fin}}_{G_2})^R\circ F\circ \bemb^{\on{spec.fin}}_{G_1}, \quad
G_1\mmod^{\on{spec.fin}}\to G_2\mmod^{\on{spec.fin}}.$$

We claim:

\begin{lem} \label{l:adj is inherited}
Suppose that $F$ is adjointable. Then so is $F^{\on{spec.fin}}$, and
$$(F^{\on{spec.fin}})^R\simeq (F^R)^{\on{spec.fin}}.$$
\end{lem}

\begin{proof}

Follows from the fact that
$$((\bemb^{\on{spec.fin}}_{G})^R)^R\simeq \bemb^{\on{spec.fin}}_{G},$$
which in turn follows from the fact that $\bemb^{\on{spec.fin}}_{G}$ is 2-adjointable
(see \lemref{l:ambi}).

\end{proof}

\sssec{}

Combining with \corref{c:G-mod 2-dual}, we obtain:

\begin{cor}
The object $G\mmod^{\on{spec.fin}}\in \tAGCat$ is 2-dualiazable.
\end{cor}

And further:

\begin{cor}
The object $G\mmod^{\on{restr}}\in \tDGCat$ is 2-dualiazable.
\end{cor}

\sssec{}

For $F$ as above, we shall say that it is \emph{spec.fin-adapted} if it sends
$$G_1\mmod^{\on{spec.fin}} \subset G_1\mmod$$
to
$$G_2\mmod^{\on{spec.fin}} \subset G_2\mmod.$$

This is equivalent to the requirement that the natural
transformation
$$\bemb^{\on{spec.fin}}_{G_2}\circ
(\bemb^{\on{spec.fin}}_{G_2})^R\circ F\circ \bemb^{\on{spec.fin}}_{G_1}\circ (\bemb^{\on{spec.fin}}_{G_1})^R
\to
F\circ \bemb^{\on{spec.fin}}_{G_1}\circ (\bemb^{\on{spec.fin}}_{G_1})^R$$
as functors
$$G_1\mmod\to G_2\mmod$$
is an isomorphism.

\medskip

Still equivalently, the above condition is equivalent to the requirement that the diagram
\begin{equation} \label{e:functor and spec fin}
\CD
G_1\mmod @>{F}>> G_2\mmod \\
@A{\bemb^{\on{spec.fin}}_{G_1}}AA @AA{\bemb^{\on{spec.fin}}_{G_2}}A \\
G_1\mmod^{\on{spec.fin}} @>{F^{\on{spec.fin}}}>> G_2\mmod^{\on{spec.fin}},
\endCD
\end{equation}
which a priori commutes up to a natural transformation, actually commutes.

\sssec{} \label{sss:when spec fin adapted}

The above condition can be rewritten as follows:

\medskip

Let $$\bM_{1,2}\in (G_1\times G_2)\mmod$$
be the object that represents $F$. Then $F$ is \emph{spec.fin-adapted} if and only if
$(\bemb^{\on{spec.fin}}_{G_1})^R(\bM_{1,2})$, viewed as an object of
$(G_1\times G_2)\mmod$ (or, equivalently, just $G_2\mmod$) belongs to
$(G_1\times G_2)\mmod^{\on{spec.fin}}$ (resp., $G_2\mmod^{\on{spec.fin}}$).
By definition this means that the object
\begin{equation} \label{e:crit mon adapt}
(\bjacq_{N_1\times N_2}(\bM_{1,2}))^{T_1\on{-mon}}\in T_2\mmod
\end{equation}
belongs to $T_2\mmod^{\on{mon}}$.

\sssec{}

We shall say that $F$ is \emph{restricted-adapted} if it is \emph{spec.fin-adapted}
and the functor $F^{\on{spec.fin}}$ sends
$$G_1\mmod^{\on{restr}} \to G_2\mmod^{\on{restr}}.$$

This is equivalent to the requirement that the object $(\bemb^{\on{spec.fin}}_{G_1})^R(\bM_{1,2})$
above belongs to the subcategory $(G_1\times G_2)\mmod^{\on{restr}}$ (equivalently, to
$G_2\mmod^{\on{restr}}$, when viewed as an object of just $G_2\mmod$).

\medskip

This is further equivalent to the requirement that the object
$$(\bjacq_{N_1\times N_2}(\bM_{1,2}))^{T_1\on{-mon}}\in \AGCat$$
belongs to the essential image of $\DGCat$.

\medskip

Under the above circumstances, we will denote by  $F^{\on{restr}}$ the resulting
1-morphism
$$G_1\mmod^{\on{restr}} \to G_2\mmod^{\on{restr}}.$$

\sssec{}

We claim:

\begin{lem} \label{l:2-adj spec fin}
Let $F:G_1\mmod\to G_2\mmod$ be 2-adjointable. If it is
spec.fin-adapted, then the 1-morphism
$$F^{\on{spec.fin}}:G_1\mmod^{\on{spec.fin}}\to G_2\mmod^{\on{spec.fin}}$$
is also 2-adjointable.
\end{lem}

\begin{proof}

We need to show that the unit and counit of the adjunction admit right adjoints.
For the counit, this follows without the condition of being spec.fin-adapted:

\medskip

The counit is the composition
\begin{multline*}
(\bemb^{\on{spec.fin}}_{G_2})^R\circ F\circ
\bemb^{\on{spec.fin}}_{G_1}\circ (\bemb^{\on{spec.fin}}_{G_1})^R\circ F^R \circ
\bemb^{\on{spec.fin}}_{G_2} \to  \\
\to (\bemb^{\on{spec.fin}}_{G_2})^R\circ F\circ F^R \circ \bemb^{\on{spec.fin}}_{G_2} \to
(\bemb^{\on{spec.fin}}_{G_2})^R\circ \bemb^{\on{spec.fin}}_{G_2}\simeq \on{Id},
\end{multline*}
where:

\medskip

\begin{itemize}

\item The first arrow is the counit for the
$(\bemb^{\on{spec.fin}}_{G_1},(\bemb^{\on{spec.fin}}_{G_1})^R)$-adjunction, and hence
admits a right adjoint by \thmref{t:restr main}(a);

\medskip

\item The second arrow is the counit of the $(F,F^R)$-adjunction, and hence admits a
right adjoint by the hypothesis on $F$;

\medskip

\item The last isomorphism is due to the fact that $\bemb^{\on{spec.fin}}_{G_2}$
is fully faithful.

\end{itemize}

\medskip

We now consider the unit. It is the composition
$$\on{Id} \to (\bemb^{\on{spec.fin}}_{G_1})^R\circ \bemb^{\on{spec.fin}}_{G_1} \to
(\bemb^{\on{spec.fin}}_{G_1})^R\circ F^R\circ F\circ \bemb^{\on{spec.fin}}_{G_1} \simeq
(F^{\on{spec.fin}})^R\circ F^{\on{spec.fin}},$$
where:

\medskip

\begin{itemize}

\item The first arrow is the unit for the
$(\bemb^{\on{spec.fin}}_{G_1},(\bemb^{\on{spec.fin}}_{G_1})^R)$-adjunction, and hence
admits a right adjoint by \thmref{t:restr main}(a); in fact, it is an isomorphism since
$\bemb^{\on{spec.fin}}_{G_1}$ is fully-faithful;

\medskip

\item The second arrow is the unit of the $(F,F^R)$-adjunction, and hence admits a
right adjoint by the hypothesis on $F$;

\medskip

\item The last isomorphism is due to the fact that $F$ is spec.fin-adapted.

\end{itemize}

\end{proof}

\begin{cor} \label{c:2-adj spec restr}
Let $F:G_1\mmod\to G_2\mmod$ be 2-adjointable. Then if it is
restricted-adapted, then the 1-morphism
$$F^{\on{restr}}:G_1\mmod^{\on{restr}}\to G_2\mmod^{\on{restr}}$$
is also 2-adjointable.
\end{cor}

\ssec{The induction functor on restricted representations} \label{ss:ind on restr}

In this subsection we will encounter a key feature that distinguishes $G\mmod^{\on{spec.fin}}$
from $G\mmod$: the intertwining functor $I_{P_1,P_2}$ becomes an \emph{isomorphism}
when restricted to the spectrally finite subcategory.

\sssec{} \label{sss:ind is comp}

We apply the material of \secref{ss:inherited} to $G_1=M$ and $G_2=G$,
where $M$ is a Levi subgroup in $G$ and $F=\bind_P$.

\medskip

We claim that $\bind_P$ is spec.fin-adapted. Indeed, the corresponding object
\eqref{e:crit mon adapt} is
$$\ul\Shv(N\backslash G/N)^{\on{left-}\!T\on{-mon}},$$
and the assertion follows from \lemref{l:mon Hecke} (see \secref{sss:when spec fin adapted}).

\medskip

Moreover, from \propref{p:q-mon restr} it follows that $\bind_P$ is restricted adapted.

\sssec{}

From \lemref{l:2-adj spec fin} and \corref{c:2-adj spec restr}, combined with \secref{sss:prince series}
we obtain:

\begin{cor}  \label{c:restr adj 2-adj}
The 1-morphisms
$$\bind_P^{\on{spec.fin}}:M\mmod^{\on{spec.fin}}\to G\mmod^{\on{spec.fin}}$$
and
$$\bind_P^{\on{restr}}:M\mmod^{\on{restr}}\to G\mmod^{\on{restr}}$$
are 2-adjointable.
\end{cor}

\sssec{}

Recall the 2-morphism
$$I_{P_1,P_2}:\bind_{P_1}\to \bind_{P_2},$$
see \secref{sss:intertwining functor}.

\medskip

This 2-morphism automatically induces a 2-morphism
$$I^{\on{spec.fin}}_{P_1,P_2}:\bind^{\on{spec.fin}}_{P_1}\to \bind^{\on{spec.fin}}_{P_2},$$
between 1-morphisms
$$M\mmod^{\on{spec.fin}}\to G\mmod^{\on{spec.fin}}$$
and a 2-morphism
$$I^{\on{restr}}_{P_1,P_2}:\bind^{\on{restr}}_{P_1}\to \bind^{\on{restr}}_{P_2},$$
between 1-morphisms
$$M\mmod^{\on{spec.fin}}\to G\mmod^{\on{spec.fin}}.$$

\sssec{}

Here is a key assertion that makes the spectrally finite theory work better than
the entire one:

\begin{prop} \label{p:intertwiner spec finite}
The 2-morphism $I^{\on{spec.fin}}_{P_1,P_2}$ is an \emph{isomorphism};
in particular, it admits a right adjoint.
\end{prop}

As an immediate corollary we obtain:

\begin{cor} \label{c:intertwiner restr}
The 2-morphism $I^{\on{restr}}_{P_1,P_2}$ is an isomorphism;
in particular, it admits a right adjoint.
\end{cor}

\sssec{Proof of \propref{p:intertwiner spec finite}}

Choose a maximal unipotent $N_M\subset M$ so that the preimage $N_i$ under $P_i\to M$ is a maximal
unipotent in $G$. Since the functor
$$\binv_{N_M}:M\mmod\to \AGCat$$
is conservative (by \thmref{t:G mod via Hecke}), it is enough to show that the 2-morphism $I^{\on{restr}}_{P_1,P_2}$
induces an equivalence (in $\AGCat$) after taking $N_M$-invariants.

\medskip

Note that the resulting functor is given by pull-push along
$$
\CD
G/N_1\cap N_2 @>{r'_2}>> G/N_2 \\
@V{r'_1}VV \\
G/N_1.
\endCD
$$

Thus, we need to show that for a scheme $Y$, the functor
$$(\on{id} \times r'_2)_* \circ (\on{id} \times r'_1)^!$$
induces an equivalence
$$\Shv(Y \times (G/N_1))^{T\on{-mon}}\to \Shv(Y\times (G/N_2))^{T\on{-mon}}.$$


\medskip

The relative position of $N_1$ and $N_2$ is a well-defined element $w(N_1,N_2)\in W$. It is easy to see that
for a triple of maximal unipotents $N_1,N_2,N_3$ with
$$\ell(w(N_1,N_2))+\ell(w(N_2,N_3))=\ell(w(N_1,N_3)),$$
we have
$$I_{B_1,B_3}\simeq I_{B_2,B_3}\circ I_{B_1,B_2},$$
where $B_i$ is the Borel containing $N_i$.

\medskip

Hence, we can assume that the relative position of $N_1$ and $N_2$ is a simple reflection $s_i$. Let $P$ be
the subminimal parabolic that contains both $N_1$ and $N_2$. The diagram
$$G/N_1 \to G/P\leftarrow G/N_2$$
\'etale-locally with respect to $G/P$ is isomorphic to a product situation with $G$ of semisimple rank $1$,
and $P_1$ and $P_2$ are a Borel subgroup and its opposite.

\sssec{} \label{sss:Radon}

Thus, we can assume that $G$ has semi-simple rank $1$. Further, it is easy to reduce the assertion to the
case when $G=SL_2$. We will analyze the latter case explicitly.

\medskip

We identify $G/N$ with $(V-0)$, where $V$ is a 2-dimensional vector space with a trivialized determinant
line; in particular we obtain an identification $V\simeq V^*$.

\medskip

Under these identofications, $I_{B,B^-}$ corresponds to the Radon transform
$$\on{Rad}_V:(\on{id}\times r_{V^*})_*\circ (\on{id}\times r_V)^! :\Shv(Y\times (V-0)) \to \Shv(Y\times (V^*-0)),$$
where
$$(V-0) \overset{r_V}\longleftarrow \sH \overset{r_{V^*}}\longrightarrow (V^*-0),\quad \sH=\{(v,v^*)\,|\, \langle v,v^*\rangle=1\}.$$

Now, it is known that $\on{Rad}^*_V$ is an equivalence \emph{on the $\BG_m$-monodromic subcategories} with inverse
$$\wt{\on{Rad}}_{V^*}:(\on{id}\times r_{V})_!\circ (\on{id}\times r_{V^*})^* :\Shv(Y\times (V^*-0)) \to \Shv(Y\times (V-0)).$$

\qed[\propref{p:intertwiner spec finite}]

\ssec{Jacquet functor on restricted representations} \label{ss:jack on restr}

In this subsection we will encounter another new feature of the induction
functor in the spectrally finite situation.

\sssec{}

Recall that the fact that the 1-morphism $\bind_P$ is 2-adjointable implies that the adjunction
$$(\bind_P,\bjacq_P)$$
is ambidexterous.

\medskip

In particular, the 1-morphism $\bjacq_P$ is adjointable. But the above does not say (and it is \emph{not} true) that
$\bjacq_P$ is 2-adjointable.

\sssec{}

Note that by the same logic as in the case of $\bind_P$, the 1-morphism $\bjacq_P$ is spec.fin-adapted (and also
restricted-adapted).

\medskip

In this subsection will prove:

\begin{thm} \label{t:jacq mon}
The 1-morphism
$$\bjacq_P^{\on{spec.fin}}:G\mmod^{\on{spec.fin}}\to M\mmod^{\on{spec.fin}}$$
is 2-adjointable.
\end{thm}

As an immediate corollary we obtain:

\begin{cor} \label{c:jacq mon}
The 1-morphism
$$\bjacq_P^{\on{restr}}:G\mmod^{\on{restr}}\to M\mmod^{\on{restr}}$$
is 2-adjointable.
\end{cor}

\sssec{}

Note that by combining \thmref{t:jacq mon} with \lemref{l:ambi}, we obtain a \emph{new} adjunction between $\bind^{\on{spec.fin}}_P$ (as a left adjoint)
and $\bjacq^{\on{spec.fin}}_P$ (as a right adjoint).

\medskip

This adjunction is obtained from the initial $(\bind^{\on{spec.fin}}_P,\bjacq^{\on{spec.fin}}_P)$-adjunction by replacing the
unit/counit by their respective \emph{double right adjoints}.

\sssec{} \label{sss:vacillating}

More generally, let
$$F:\bo_1\to \bo_2$$
be an adjointable 1-morphism in a 3-category, such that its unit and counit admits double right adjoints.

\medskip

By the above, we obtain a \emph{new} adjunction
$$F\leftrightarrow F^R.$$

We will call such $F$ a \emph{vacillating}\footnote{The name was suggested by S.~Raskin.} functor.

\sssec{}

The rest of the subsection is devoted to the proof of \thmref{t:jacq mon}.

\medskip

By definition, we need to show that the 2-morphisms
\begin{equation} \label{e:adj init}
\on{Id}\to \bjacq_P^{\on{spec.fin}}\circ \bind_P^{\on{spec.fin}} \text{ and }
\bind_P^{\on{spec.fin}} \circ \bjacq_P^{\on{spec.fin}}\to \on{Id}
\end{equation}
that give rise to the initial $(\bind^{\on{spec.fin}}_P,\bjacq^{\on{spec.fin}}_P)$-adjunction admit double right adjoints.

\medskip

Using \corref{c:restr via Hecke}, we interpret the 2-morphisms \eqref{e:adj init} as the morphisms
\begin{equation} \label{e:adj init 1}
\ul\Shv(N_M\backslash M/N_M)^{T\on{-mon}}\overset{i_*}\to \ul\Shv(N\backslash G/N)^{T\on{-mon}}
\end{equation}
(in $(\ul\CH_M^{\on{mod}}\otimes \ul\CH_M^{\on{mod}})\mmod(\AGCat)$), and
\begin{equation} \label{e:adj init 2}
\ul\Shv(N\backslash G/N(P)\overset{M}\times N(P)\backslash G/N)^{T\on{-mon}} \overset{\on{CH}^{\on{ext}}}\longrightarrow
\ul\Shv(N\backslash G/N)^{T\on{-mon}}
\end{equation}
(in $(\ul\CH_G^{\on{mod}}\otimes \ul\CH_G^{\on{mod}})\mmod(\AGCat)$), respectively, where
$$i:N_M\backslash M/N_M \to N_M\backslash (N(P)\backslash G/N(P))/N_M \simeq N\backslash G/N$$
is the natural map, and $\on{CH}^{\on{ext}}$ is pull-push along
$$
\CD
N\backslash G\overset{P}\times G/N @>>> N\backslash G/N \\
@VVV \\
N\backslash G/N(P)\overset{M}\times N(P)\backslash G/N.
\endCD
$$

\sssec{}

Note now that by \propref{p:q-mon restr}, the 1-morphisms \eqref{e:adj init 1} and \eqref{e:adj init 2} are obtained by applying the functor $\bi:\DGCat\to \AGCat$
to the 1-morphisms
\begin{equation} \label{e:adj init 1 bis}
\Shv(N_M\backslash M/N_M)^{T\on{-mon}}\overset{i_*}\to \Shv(N\backslash G/N)^{T\on{-mon}}
\end{equation}
in $(\CH_M^{\on{mod}}\otimes \CH_M^{\on{mod}})\mmod(\DGCat)$, and
\begin{equation} \label{e:adj init 2 bis}
\Shv(N\backslash G/N(P)\overset{M}\times N(P)\backslash G/N)^{T\on{-mon}} \overset{\on{CH}^{\on{ext}}}\longrightarrow
\Shv(N\backslash G/N)^{T\on{-mon}}
\end{equation}
in $(\CH_G^{\on{mod}}\otimes \CH_G^{\on{mod}})\mmod(\DGCat)$,
respectively.

\medskip

So, it is enough to show that the functors \eqref{e:adj init 1 bis} and \eqref{e:adj init 2 bis} admit double right adjoints
in
$$ (\CH_M^{\on{mod}}\otimes \CH_M^{\on{mod}})\mmod(\DGCat) \text{ and }  (\CH_G^{\on{mod}}\otimes \CH_G^{\on{mod}})\mmod(\DGCat),$$
respectively.

\sssec{}

However, applying \corref{c:right adjoint Hecke}, we obtain that it suffices that the 1-morphisms \eqref{e:adj init 1} and \eqref{e:adj init 2} admit double
right adjoints in $\DGCat$.

\sssec{}

The right adjoints to the maps in \eqref{e:adj init 1 bis} and \eqref{e:adj init 2 bis} are given by
\begin{equation} \label{e:adj init bis}
i^! \text{ and } \on{HC}^{\on{ext}}
\end{equation}
respectively,
where $\on{HC}^{\on{ext}}$ is given by pull-push along the diagram
$$
\CD
N\backslash G\overset{P}\times G/N @>>> N\backslash G/N(P)\overset{M}\times N(P)\backslash G/N \\
@VVV \\
N\backslash G/N.
\endCD
$$

\medskip

In order for these functors to admit right adjoints in $\DGCat$, it is necessary and sufficient that they
preserves compactness.

\medskip

Note, however, that all the stacks involved are \emph{safe}, so compactness is equivalent to constructibility.
Now, the fact that the functors \eqref{e:adj init bis} preserve constructibility is obvious.

\qed[\thmref{t:jacq mon}]

\ssec{The Serre structure on the induction functor}

Building on the previous subsection, we will show that the induction functor admits an additional
\emph{structure} in the spectrally finite situation.

\sssec{}

Let us return to the context of \secref{sss:vacillating}. Thus, we start with an initial adjunction
$$F\leftrightarrow F^R$$
and by passing to double right adjoints of the unit and counit, we obtain a \emph{new} adjunction
$$F\overset{\on{new}}\leftrightarrow F^R.$$

Thus, the functor $F$ admits a canonical automorphism, called the Serre functor,
to be denoted $\on{Se}_F$ that intertwines the two adjunctions.

\medskip

In other words,
\begin{equation} \label{e:Serre 1 abs}
\on{unit}^{\on{new}}\simeq (\on{id}_{F^R}(\on{Se}_F))\circ \on{unit}
\end{equation}
or
\begin{equation} \label{e:Serre 2 abs}
\on{counit}\simeq \on{counit}^{\on{new}}\circ (\on{Se}_F(\on{id}_{F^R})),
\end{equation}
where we denote by $\on{id}_{F^R}(\on{Se}_F)$ and $\on{Se}_F(\on{id}_{F^R})$ the
corresponding automorphisms of $F^R\circ F$ and $F\circ F^R$, respectively.

\sssec{}

We apply this to the vacillating functor $\bind^{\on{spec.fin}}_P$. Denote the corresponding Serre automorphism of
$$\ul\Shv(G/N(P))^{\on{spec.fin}}\in (G\times M)\mmod$$ by $\on{Se}_{G/N(P)}$.

\medskip

By definition, $\on{Se}_{G/N(P)}$ is uniquely characterized by \emph{either} of the following two isomorphisms
\begin{equation} \label{e:Serre 1}
\on{unit}^{\on{new}}\simeq (\on{Id}\otimes \on{Se}_{G/N(P)})\circ \on{unit}
\end{equation}
or
\begin{equation} \label{e:Serre 2}
\on{counit}\simeq \on{counit}^{\on{new}}\circ (\on{Id}\otimes \on{Se}_{G/N(P)}).
\end{equation}

\sssec{}

In this subsection we will prove the following result:

\begin{thm} \label{t:2-Cat Serre}
The functor $\on{Se}_{G/N(P)}$ equals the composition
$$I^{\on{spec.fin}}_{P^-,P}\circ I^{\on{spec.fin}}_{P,P^-}[-2\dim(N(P))]$$
viewed as an endomorphism of $\ul\Shv(G/N(P))^{\on{spec.fin}}\in (G\times M)\mmod$.
\end{thm}

The rest of this section is devoted to the proof of \thmref{t:2-Cat Serre}.

\sssec{}

We will prove that $\on{Se}_{G/N(P)}$ satisfies \eqref{e:Serre 1}, while \eqref{e:Serre 2} will be a
(non-obvious) consequence, to be used in \secref{sss:counit Serre}.

\sssec{}

Recall that the unit of the initial $(\bind_P,\bjacq_P)$-adjunction is given by
the functor
$$\iota_P:\ul\Shv(M)\to \ul\Shv(N(P)\backslash G/N(P)),$$
equal to pull-push along
$$
\CD
N(P)\backslash P/N(P) @>>> N(P)\backslash G/N(P) \\
@VVV \\
M.
\endCD
$$

Set
$$\iota^{\on{spec.fin}}_P:=\iota_P|_{\ul\Shv(M)^{\on{spec.fin}}}.$$

Thus, we need to construct an isomorphism
$$(\on{Id}\otimes (I^{\on{spec.fin}}_{P^-,P}\circ I^{\on{spec.fin}}_{P,P^-}))\circ \iota^{\on{spec.fin}}_P[-2\dim(N(P))] \simeq ((\iota^{\on{spec.fin}}_P)^R)^R$$
as maps
$$\ul\Shv(M)^{\on{spec.fin}}\to \ul\Shv(N(P)\backslash G/N(P))^{\on{spec.fin}}$$
in $(M\times M)\mmod$, where the symbol $\on{id}\otimes I_{?,?}$ refers to the action of the corresponding
intertwining functor on the right.

\medskip

This is equivalent to an isomorphism
\begin{equation} \label{e:Serre unit}
((\on{Id}\otimes (I^{\on{spec.fin}}_{P^-,P})^{-1})\circ \iota^{\on{spec.fin}}_P)^R\simeq ((\on{Id}\otimes I^{\on{spec.fin}}_{P,P^-})\circ \iota^{\on{spec.fin}}_P)^L[2\dim(N(P))]
\end{equation}
as maps
$$\ul\Shv(N(P)\backslash G/N(P^-))^{\on{spec.fin}}\to \ul\Shv(M)^{\on{spec.fin}}$$
in $(M\times M)\mmod$.

\sssec{} \label{sss:j p}

Consider the \emph{open} embedding
\begin{equation} \label{e:big cell}
M\hookrightarrow N(P)\backslash G/N(P^-).
\end{equation}

Let $\jmath_{P,P^-}$ denote the corresponding functor of restriction
$$\ul\Shv(N(P)\backslash G/N(P^-))\to \ul\Shv(M).$$

It admits a right adjoint, given by *-extension along \eqref{e:big cell}.

\sssec{}

Unwinding, we obtain:
\begin{equation} \label{e:restr and intertw}
(\on{Id}\otimes I_{P,P^-})\circ \iota_P\simeq (\jmath_{P,P^-})^R
\end{equation}

Hence,
$$(\on{Id}\otimes I^{\on{spec.fin}}_{P,P^-})\circ \iota^{\on{spec.fin}}_P\simeq (\jmath_{P,P^-}^{\on{spec.fin}})^R,$$
where
$$\jmath_{P,P^-}^{\on{spec.fin}}:=\jmath_{P,P^-}|_{\ul\Shv(N(P)\backslash G/N(P^-))^{\on{spec.fin}}}.$$

\sssec{}

Consider the functor
$$\on{Id}\otimes (I^{\on{spec.fin}}_{P^-,P})^{-1}:\ul\Shv(N(P)\backslash G/N(P))^{\on{spec.fin}}\to \ul\Shv(N(P)\backslash G/N(P^-))^{\on{spec.fin}}.$$

It is the value-wise left adjoint of
$$(\on{Id}\otimes I^{\on{spec.fin}}_{P^-,P}):\Shv(Y\times (N(P)\backslash G/N(P^-)))^{\on{spec.fin}} \to
\Shv(Y\times (N(P)\backslash G/N(P)))^{\on{spec.fin}}.$$

Hence, it is given by \emph{*-pull followed by !-push} along the diagram
$$
\CD
Y\times (N(P)\backslash G) @>>> Y\times (N(P)\backslash G/N(P^-)) \\
@VVV \\
Y\times (N(P)\backslash G/N(P)).
\endCD
$$

Hence, the functor
$$(\on{Id}\otimes (I^{\on{spec.fin}}_{P^-,P})^{-1})\circ \iota^{\on{spec.fin}}_P$$
is given value-wise by
$$((\jmath_{P,P^-})^{\on{spec.fin}})^L[-2\dim(N(P))],$$
thanks to the fact that \eqref{e:big cell} is an open embedding.

\sssec{}

This makes \eqref{e:Serre unit} manifest, since both sides identify with
$$\jmath_{P,P^-}^{\on{spec.fin}}[2\dim(N(P))].$$

\qed[\thmref{t:2-Cat Serre}]

\begin{rem}

A somewhat parallel statement (originally proved in \cite{BBM} and then reproved in \cite{CGY}) says 
that the usual Serre functor on the category $\Shv(N\backslash G/B)$ is equal to
$$I_{B^-,B}\circ I_{B,B^-}[-2\dim(N)],$$
where we view the intertwining functor as acting on the right.

\medskip 

We were not able to formally deduce the above description of the
Serre functor for $\Shv(N\backslash G/B)$ from \thmref{t:2-Cat Serre}.

\medskip

We would be curious to know whether such a formal deduction exists. 

\end{rem}

\section{Applications to character sheaves} \label{s:char shv}

In this subsection we will recast the theory of character sheaves from the point
of view of $G\mmod^{\on{spec.fin}}$ by taking the trace.

\medskip

We will encounter some (non-obvious) features of character sheaves
(e.g., the functional equation for the Grothendieck-Springer functor, and the ambidexterity
of the $(\on{GrothSpr}_P,\fr_P)$-adjunction) that follow from the 3-categorical
formalism.

\ssec{Character sheaves--a reminder} \label{ss:char shv}

\sssec{}

For a scheme $Y$, consider the pair of categories
\begin{equation} \label{e:setting for HCh}
\Shv(Y\times (G/\on{Ad}(G))) \text{ and } \Shv(Y\times (N\backslash G/N)/\on{Ad}(T)),
\end{equation}
where $\on{Ad}(T)$ refers to the diagonal action of $T$ on $N\backslash G/N$.

\medskip

We regard the two sides of \eqref{e:setting for HCh} as monoidal categories
with respect to \emph{convolution}.

\sssec{} \label{sss:HCh av}

We have a monoidal functor
$$\on{HC}: \Shv(Y\times (G/\on{Ad}(G))) \to \Shv(Y\times (N\backslash G/N)/\on{Ad}(T)),$$
given by pull-push along the diagram
\begin{equation} \label{e:HCh diag}
\CD
Y \times G/\on{Ad}(B) @>>> Y\times (N\backslash G/N)/\on{Ad}(T) \\
@VVV \\
Y\times G/\on{Ad}(G).
\endCD
\end{equation}

Since the horizontal morphism in \eqref{e:HCh diag} is smooth and the vertical one is proper, the functor
$\on{HC}$ admits a left adjoint, to be denoted $\on{CH}$, which commutes with !-pullbacks and *-pushforwards
along $Y$.

\medskip

The functor $\on{CH}$ is left-lax monoidal, but is strictly compatible with the convolution
action of the monoidal category $\Shv(Y\times (G/\on{Ad}(G)))$ on either side.

\medskip

In particular, we obtain an adjunction between the corresponding objects in $\AGCat$:
$$\on{CH}:\ul\Shv(Y\times (N\backslash G/N)/\on{Ad}(T))\rightleftarrows \ul\Shv(Y\times (G/\on{Ad}(G))):\on{HC}.$$

\sssec{}

We note:

\begin{lem} \label{p:AB}
The endofunctor $\on{CH}\circ \on{HC}$ of $\Shv(Y\times (G/\on{Ad}(G)))$ is isomorphic to the convolution with the
Springer sheaf.
\end{lem}

\begin{proof}

This is obtained from \lemref{l:Springer} by passing to $G$-invariants on both sides.

\end{proof}

\sssec{}

Since the Springer sheaf contains $\delta_{1,G}$ as a direct summand, we obtain:

\begin{cor} \label{c:A is cons}
The functor $\on{HC}$ is conservative.
\end{cor}

\sssec{}

Let
$$\Shv^{\on{ch}}(Y\times (G/\on{Ad}(G))) \subset \Shv(Y\times (G/\on{Ad}(G)))$$
be the full subcategory consisting of objects whose image along $\on{HC}$ lands in
$$\Shv(Y\times (N\backslash G/N)/\on{Ad}(T))^{\on{mon}}\subset \Shv(Y\times (N\backslash G/N)/\on{Ad}(T)).$$
Since $\on{HC}$ is monoidal, from \corref{c:Hecke mon ideal}
we obtain that $\Shv^{\on{ch}}(Y\times (G/\on{Ad}(G)))$ is a monoidal ideal
inside $\Shv(Y\times (G/\on{Ad}(G)))$.

\medskip

It is clear that for $f:Y\to Y'$, the *-pushforward and !-pullback functors along $f \times \on{id}$
map the categories $\Shv^{\on{ch}}(Y\times (G/\on{Ad}(G)))$ and $\Shv^{\on{ch}}(Y'\times (G/\on{Ad}(G)))$
to one another.

\sssec{}

Thus, we obtain a well-defined object
$$\ul\Shv^{\on{ch}}(G/\on{Ad}(G))\in \AGCat,$$
equipped with a fully-faithful 1-morphism\footnote{The notion of fully faithfulness for a 1-morphism is intrinsic to any 2-category:
it means that the induced functor on $\bMaps(-,-)$ is fully faithful. In the case of $\AGCat$, this is equivalent to the fact that 
the given 1-morphism evaluates to a fully faithful functor in $\DGCat$ on every scheme.} 
\begin{equation} \label{e:emb ch}
\on{emb}^{\on{ch}}:\ul\Shv^{\on{ch}}(G/\on{Ad}(G))\to \ul\Shv(G/\on{Ad}(G)).
\end{equation}

\sssec{}

We observe:

\begin{lem} \label{l:HCh pres char}
The functor $\on{CH}$ maps $\Shv(Y\times (N\backslash G/N)/\on{Ad}(T))^{\on{mon}}$ to
$\Shv^{\on{ch}}(Y\times (G/\on{Ad}(G)))$.
\end{lem}

\begin{proof}

We have to show that the endofunctor
$\on{HC}\circ \on{CH}$ of $\Shv(Y\times (N\backslash G/N)/\on{Ad}(T))$ preserves the subcategory
$\Shv(Y\times (N\backslash G/N)/\on{Ad}(T))^{\on{mon}}$.

\medskip

We note that by proper and smooth base change, the functor $\on{HC}\circ \on{CH}$ admits a filtration
indexed by elements $w\in W$, and each subquotient $(\on{HC}\circ \on{CH})_w$ intertwines
the initial $T$-action on $\Shv(Y\times (N\backslash G/N)/\on{Ad}(T))$ with one twisted by the
action of $w$ on $T$. In particular, each $(\on{HC}\circ \on{CH})_w$ preserves the monodromic
subcategory. Hence, so does $\on{HC}\circ \on{CH}$.

\end{proof}

\sssec{}

Thus, we obtain that the adjoint functors $(\on{CH},\on{HC})$ induce adjoint functors
\begin{equation} \label{e:HCh for char}
\Shv(Y\times (N\backslash G/N)/\on{Ad}(T))^{\on{mon}}\rightleftarrows
\Shv^{\on{ch}}(Y\times (G/\on{Ad}(G))).
\end{equation}

Since $\on{HC}$ is conservative (on all of $\Shv(Y\times (G/\on{Ad}(G))$),
we obtain:

\begin{cor} \label{c:HCh gen}
The essential image of the functor
$$\on{CH}:\Shv(Y\times (N\backslash G/N)/\on{Ad}(T))^{\on{mon}}\to
\Shv^{\on{ch}}(Y\times (G/\on{Ad}(G)))$$ generates the target category.
\end{cor}

\sssec{}

We further obtain:

\begin{prop} \label{p:emb ch R}
The inclusion functor
$$\on{emb}^{\on{ch}}:\Shv^{\on{ch}}(Y\times (G/\on{Ad}(G))) \hookrightarrow \Shv(Y\times (G/\on{Ad}(G)))$$
admits a continuous right adjoint, which makes the diagrams
\begin{equation} \label{e:right adj to ch}
\CD
\Shv^{\on{ch}}(Y\times (G/\on{Ad}(G))) @<{(\on{emb}^{\on{ch}})^R}<< \Shv(Y\times (G/\on{Ad}(G))) \\
@V{\on{HC}}VV @VV{\on{HC}}V \\
\Shv(Y\times (N\backslash G/N)/\on{Ad}(T))^{\on{mon}}
@<{(\on{emb}^{\on{mon}})^R}<< \Shv(Y\times (N\backslash G/N)/\on{Ad}(T))
\endCD
\end{equation}
and
\begin{equation} \label{e:right adj to ch bis}
\CD
\Shv^{\on{ch}}(Y\times (G/\on{Ad}(G))) @<{(\on{emb}^{\on{ch}})^R}<< \Shv(Y\times (G/\on{Ad}(G))) \\
@A{\on{CH}}AA @AA{\on{CH}}A \\
\Shv(Y\times (N\backslash G/N)/\on{Ad}(T))^{\on{mon}}
@<{(\on{emb}^{\on{mon}})^R}<< \Shv(Y\times (N\backslash G/N)/\on{Ad}(T))
\endCD
\end{equation}
commute. Furthermore:

\medskip

\noindent{\em(a)}
For a morphism $f:Y\to Y'$, the natural transformations
\begin{equation} \label{e:adj to ch functors}
(f\times \on{id})_* \circ (\on{emb}^{\on{ch}})^R \to (\on{emb}^{\on{ch}})^R \circ (f\times \on{id})_* \text{ and }
(f\times \on{id})^! \circ (\on{emb}^{\on{ch}})^R \to (\on{emb}^{\on{ch}})^R \circ (f\times \on{id})^!
\end{equation}
are isomorphisms.

\medskip

\noindent{\em(b)} The right adjoint $(\on{emb}^{\on{ch}})^R$ is strictly monoidal, and is given by
convolution with the (central idempotent) object
$$\delta_{1,G}^{\on{ch}}:=(\on{emb}^{\on{ch}}_{\on{pt}})^R(\delta_{1,G})\in
\Shv^{\on{ch}}(G/\on{Ad}(G)) \subset \Shv(G/\on{Ad}(G)).$$

\end{prop}

\begin{proof}
The commutation of diagram \eqref{e:right adj to ch} follows by passing to right adjoints from the tautologically
commuting diagram
$$
\CD
\Shv^{\on{ch}}(Y\times (G/\on{Ad}(G))) @>{\on{emb}^{\on{ch}}}>> \Shv(Y\times (G/\on{Ad}(G))) \\
@A{\on{CH}}AA @AA{\on{CH}}A \\
\Shv(Y\times (N\backslash G/N)/\on{Ad}(T))^{\on{mon}} @>{\on{emb}^{\on{mon}}}>>
\Shv(Y\times (N\backslash G/N)/\on{Ad}(T)).
\endCD
$$

For the commutation of \eqref{e:right adj to ch bis}, we need to show that the diagram
$$
\CD
\Shv^{\on{ch}}(Y\times (G/\on{Ad}(G))) @>{\on{emb}^{\on{ch}}}>> \Shv(Y\times (G/\on{Ad}(G))) \\
@V{\wt{\on{HC}}}VV @VV{\wt{\on{HC}}}V \\
\Shv(Y\times (N\backslash G/N)/\on{Ad}(T))^{\on{mon}} @>{\on{emb}^{\on{mon}}}>>
\Shv(Y\times (N\backslash G/N)/\on{Ad}(T))
\endCD
$$
commutes, where $\wt{\on{HC}}$ is the \emph{left} adjoint of $\on{CH}$, i.e., $\wt{\on{HC}}$ is given (up to a cohomological shift)
by *-pull followed by !-push along \eqref{e:HCh diag}.

\medskip

Thus, we need to show that $\wt{\on{HC}}_Y$ sends
$\Shv^{\on{ch}}(Y\times (G/\on{Ad}(G)))$ to $\Shv(Y\times (N\backslash G/N)/\on{Ad}(T))^{\on{mon}}$.
Note, however, that the compact generators of $\Shv^{\on{ch}}(Y\times (G/\on{Ad}(G)))$ are constructible,
and $\wt{\on{HC}}$ is Verdier conjugate of the functor $\on{HC}$. Now the required assertion follows from
the fact that Verdier duality defines anti-involutions on
$$(\Shv(Y\times (N\backslash G/N)/\on{Ad}(T))^{\on{mon}})^c \text{ and } (\Shv^{\on{ch}}(Y\times (G/\on{Ad}(G))))^c.$$

For the former, this follows from the definitions. For the latter this follows from \corref{c:HCh gen}, since the
functor $\on{CH}$ commutes with Verdier duality, up to cohomological shift.

\medskip

To show that the maps in \eqref{e:adj to ch functors} are isomorphisms, it suffices to show that the maps
$$\on{HC}\circ (f\times \on{id})_* \circ (\on{emb}^{\on{ch}})^R\to
\on{HC}\circ (\on{emb}^{\on{ch}})^R \circ (f\times \on{id})_*$$
and
$$\on{HC}\circ (f\times \on{id})^! \circ (\on{emb}^{\on{ch}})^R
 \to \on{HC}\circ (\on{emb}^{\on{ch}})^R \circ (f\times \on{id})^!$$
are isomorphisms (indeed, the functors $\on{HC}$, $\on{HC}$ are conservative by \corref{c:A is cons}).

\medskip

Note that we have a commutative diagram
$$
\CD
\on{HC}\circ  (f\times \on{id})_* \circ (\on{emb}^{\on{ch}})^R @>{\sim}>>
(f\times \on{id})_* \circ \on{HC}\circ  (\on{emb}^{\on{ch}})^R \\
@VVV @VV{\sim}V \\
\on{HC}\circ (\on{emb}^{\on{ch}})^R \circ (f\times \on{id})_*  & &
(f\times \on{id})_* \circ (\on{emb}^{\on{mon}})^R \circ \on{HC} \\
@V{\sim}VV @VV{\sim}V \\
(\on{emb}^{\on{mon}})^R \circ \on{HC} \circ (f\times \on{id})_* @>{\sim}>>
(\on{emb}^{\on{mon}})^R \circ (f\times \on{id})_* \circ \on{HC}.
\endCD
$$
Hence, the upper left vertical arrow is also an isomorphism.

\medskip

The case of $f^!$ is treated similarly.

\medskip

Point (b) follows from the corresponding property of $\on{emb}^{\on{mon}}$ (see \secref{sss:H mon prel}),
using the diagram \eqref{e:right adj to ch}.

\end{proof}

\begin{cor}
The 1-morphism \eqref{e:emb ch} admits a right adjoint
$$(\on{emb}^{\on{ch}})^R:\ul\Shv(G/\on{Ad}(G))\to \ul\Shv^{\on{ch}}(G/\on{Ad}(G)).$$
\end{cor}

\begin{rem} \label{r:Verdier ch}

It is easy to see that \propref{p:emb ch R} implies that the Verdier self-duality of
$\ul\Shv(G/\on{Ad}(G))$ induces a Verdier self-duality on $\ul\Shv^{\on{ch}}(G/\on{Ad}(G))$
so that
$$(\on{emb}^{\on{ch}})^\vee \simeq (\on{emb}^{\on{ch}})^R.$$

\end{rem}

\sssec{}

Denote
$$\Shv^{\on{ch}}(G/\on{Ad}(G)):=\ul\Shv^{\on{ch}}(G/\on{Ad}(G))(\on{pt})\in \DGCat.$$

We claim:

\begin{prop} \label{p:ch is restr}
The canonical map
\begin{equation} \label{e:ch is restr}
\bi(\Shv^{\on{ch}}(G/\on{Ad}(G))) \to \ul\Shv^{\on{ch}}(G/\on{Ad}(G))
\end{equation}
is an isomorphism.
\end{prop}

\begin{proof}

Both sides embed fully faithfully in
$$\bi(\Shv(G/\on{Ad}(G)))  \text{ and } \ul\Shv(G/\on{Ad}(G)),$$
respectively. Sunce
$$\bi(\Shv(G/\on{Ad}(G))) \to \ul\Shv(G/\on{Ad}(G))$$
is fully faithful\footnote{It is not an equivalence!}, functor \eqref{e:ch is restr} is fully faithful.
Hence, it remains to show that it is essentially surjective.

\medskip

We have a commutative diagram
$$
\CD
\Shv(Y)\otimes \Shv(G/\on{Ad}(G))) @>>> \Shv(Y\times G/\on{Ad}(G)) \\
@A{\on{Id}\otimes \on{CH}}AA @AA{\on{CH}}A \\
\Shv(Y)\otimes \Shv((N\backslash G/N)/\on{Ad}(T))^{\on{mon}}@>>>
\Shv(Y\times (N\backslash G/N)/\on{Ad}(T))^{\on{mon}}.
\endCD
$$

Hence, the assertion follows from \corref{c:HCh gen}, since the object
$\ul\Shv((N\backslash G/N)/\on{Ad}(T))^{\on{mon}}\in \AGCat$ is restricted
(by \propref{p:q-mon restr}).

\end{proof}

\ssec{Another interpretation of character sheaves}

\sssec{}

Consider the object
$$\ul\Shv(G)^{\on{spec.fin}}:=(\bemb^{\on{spec.fin}})^R(\ul\Shv(G))\in (G\times G)\mmod^{\on{spec.fin}}\subset (G\times G)\mmod^.$$

In the above formula, we apply $(\bemb^{\on{spec.fin}})^R$ to $\ul\Shv(G)$ either as an object of $G\mmod$ along the action on one of the sides,
or as an object of $(G\times G)\mmod$. The result is the same by \lemref{l:mon Hecke}.

\sssec{}

The (fully faithful) embedding
\begin{equation} \label{e:emb spec fin sheaves}
\ul\Shv(G)^{\on{spec.fin}}\overset{\on{emb}^{\on{spec.fin}}}\longrightarrow \ul\Shv(G)
\end{equation}
admits a right adjoint (by \thmref{t:restr main}(a)), and thus induces a fully faithful embedding
\begin{equation} \label{e:emb char sheaves via spec. fin}
\binv_{\Delta(G)}(\ul\Shv(G)^{\on{spec.fin}})\to \binv_{\Delta(G)}(\ul\Shv(G))\simeq \ul\Shv(G/\on{Ad}(G)).
\end{equation}

We claim:

\begin{prop} \label{p:emb char sheaves via spec. fin}
The essential image of \eqref{e:emb char sheaves via spec. fin} equals
$$\ul\Shv^{\on{ch}}(G/\on{Ad}(G))\subset \ul\Shv(G/\on{Ad}(G)).$$
\end{prop}

The rest of this subsection is devoted to the proof of \propref{p:emb char sheaves via spec. fin}. In the process,
we will describe the subobject \eqref{e:emb spec fin sheaves} more explicitly, see formula \eqref{e:emb spec fin sheaves ident}.

\sssec{}

Note that by \thmref{t:G mod via Hecke}, we have
$$\ul\Shv(G)\simeq \ul\Shv(G/N)\underset{\ul\CH}\otimes \ul\Shv(N\backslash G)$$
as objects of $(G\times G)\mmod$.

\medskip

Under this identification, we have
$$\ul\Shv(G)^{\on{spec.fin}}\simeq \ul\Shv(G/N)^{T\on{-mon}}\underset{\ul\CH}\otimes \ul\Shv(N\backslash G)^{T\on{-mon}},$$
which we can also rewrite as
$$\ul\Shv(G/N)^{T\on{-mon}}\underset{\ul\CH^{\on{mon}}}\otimes \ul\Shv(N\backslash G)^{T\on{-mon}},$$
since the $\ul\CH$-action on $\ul\Shv(G/N)^{T\on{-mon}}$ factors through the colocalization
$$\ul\CH\twoheadrightarrow \ul\CH^{\on{mon}}.$$

\sssec{}

In particular, we obtain that $\ul\Shv(G)^{\on{spec.fin}}$ is the subobject of $\ul\Shv(G)$ generated by the essential image of
\begin{multline*}
\ul\Shv(G/N)^{T\on{-mon}}\otimes \ul\Shv(N\backslash G)^{T\on{-mon}}\to
\ul\Shv(G/N)^{T\on{-mon}}\underset{\ul\CH}\otimes \ul\Shv(N\backslash G)^{T\on{-mon}}\hookrightarrow \\
\to \ul\Shv(G/N)\underset{\ul\CH}\otimes \ul\Shv(N\backslash G)\to \ul\Shv(G).
 \end{multline*}

We rewrite the latter functor as
\begin{multline*}
\ul\Shv(G/N)^{T\on{-mon}}\otimes \ul\Shv(N\backslash G)^{T\on{-mon}} \to
\ul\Shv(G/N)^{T\on{-mon}}\underset{\ul\Shv(T)}\otimes \ul\Shv(N\backslash G)^{T\on{-mon}} \to \\
\to \ul\Shv(G/N)\underset{\ul\Shv(T)}\otimes \ul\Shv(N\backslash G)\to
\ul\Shv(G/N)\underset{\ul\CH}\otimes \ul\Shv(G/N)\simeq \ul\Shv(G)
\end{multline*}

\sssec{}

Note that the essential image of the functor
$$\ul\Shv(G/N)^{T\on{-mon}}\otimes \ul\Shv(G/N)^{T\on{-mon}} \to \ul\Shv(G/N)^{T\on{-mon}}\underset{\ul\Shv(T)}\otimes \ul\Shv(N\backslash G)^{T\on{-mon}}$$
generates the target (when evaluated on every scheme).

\medskip

Hence, we can further rewrite $\ul\Shv(G)^{\on{spec.fin}}\subset \ul\Shv(G)$ as the essential image of
\begin{equation} \label{e:emb spec fin sheaves ident}
\ul\Shv(G/N)^{T\on{-mon}}\underset{\ul\Shv(T)}\otimes \ul\Shv(N\backslash G)^{T\on{-mon}}
\to \ul\Shv(G/N)\underset{\ul\Shv(T)}\otimes \ul\Shv(N\backslash G)\to \ul\Shv(G).
\end{equation}

This is the promised explicit description of the subcategory \eqref{e:emb spec fin sheaves}.

\sssec{}

Applying $\binv_{\Delta(G)}\simeq \bcoinv_{\Delta(G)}$ to \eqref{e:emb spec fin sheaves ident}
we obtain that the essential image of \eqref{e:emb char sheaves via spec. fin} is the essential image of
$$\ul\Shv((N\backslash G/N)/\on{Ad}(T))^{T\on{-mon}} \subset \ul\Shv((N\backslash G/N)/\on{Ad}(T)) \to
\ul\Shv(G/\on{Ad}(G)),$$
where the second arrow is the functor $\on{CH}$.

\qed[\propref{p:emb char sheaves via spec. fin}]

\ssec{Character sheaves as trace} \label{ss:char shv as Tr}

\sssec{}

Consider again the 1-morphism
$$\bemb^{\on{spec.fin}}:G\mmod^{\on{spec.fin}}\to G\mmod.$$

It is adjointable, and hence induces a 1-morphism
\begin{equation} \label{e:Tr of incl}
\Tr(\on{Id},G\mmod^{\on{spec.fin}})\to \Tr(\on{Id},G\mmod)\simeq \ul\Shv(G/\on{Ad}(G))
\end{equation}
in $\AGCat$.

\sssec{}

We claim:

\begin{thm} \label{t:ch sheaves}
The 1-morphism \eqref{e:Tr of incl} is fully faithful, and its essential image
identifies with $\ul\Shv^{\on{ch}}(G/\on{Ad}(G))\hookrightarrow \ul\Shv(G/\on{Ad}(G))$.
\end{thm}

\begin{rem}

\thmref{t:ch sheaves} is due to D. Ben-Zvi and D.~Nadler, see \cite{BN2}. The proof presented below
is a mild variation on their argument.

\end{rem}

\sssec{}

As an immediate consequence of \thmref{t:ch sheaves} (combined with \propref{p:ch is restr}), we obtain:

\begin{cor} \label{c:Tr id restr}
There exists a canonical equivalence
$$\Tr(\on{Id},G\mmod^{\on{restr}})\simeq \Shv^{\on{ch}}(G/\on{Ad}(G)),$$
whose image under $\bi$ produces the equivalence
$$\Tr(\on{Id},G\mmod^{\on{spec.fin}})\simeq \ul\Shv^{\on{ch}}(G/\on{Ad}(G))$$
of \thmref{t:ch sheaves}.
\end{cor}

\begin{rem}

Note that \corref{c:Tr id restr} gives a different proof of \propref{p:ch is restr}:
we have
\begin{multline*}
\ul\Shv^{\on{ch}}(G/\on{Ad}(G))\simeq \Tr_{\AGCat}(\on{Id},G\mmod^{\on{spec.fin}})\simeq
\Tr_{\AGCat}(\on{Id},\AGCat\underset{\DGCat}\otimes G\mmod^{\on{restr}}) \simeq \\
\simeq \bi(\Tr_{\DGCat}(\on{Id},G\mmod^{\on{restr}}))\simeq \bi(\Shv^{\on{ch}}(G/\on{Ad}(G))).
\end{multline*}

\end{rem}

\sssec{}

The rest of this subsection is devoted to the proof of \thmref{t:ch sheaves}.

\medskip

The fact that \eqref{e:Tr of incl} is fully faithful follows immediately from \corref{c:Tr ff}.
Hence, it remains to identify its essential image.

\sssec{}

We have to identify the essential image of
\begin{multline*}
\ul\CH^{\on{mon}}\underset{\ul\CH^{\on{mon}}\otimes \ul\CH^{\on{mon}}}\otimes \ul\CH^{\on{mon}}\simeq
\Tr(\on{Id},\ul\CH^{\on{mon}}\mmod(\AGCat)) \to \\
\to \Tr(\on{Id},\ul\CH\mmod(\AGCat))\simeq \Tr(\on{Id},G\mmod)\simeq \ul\Shv(G/\on{Ad}(G)).
\end{multline*}

Note, however, that the essential image of
$$\ul\CH^{\on{mon}}\to \ul\CH^{\on{mon}}\underset{\ul\CH^{\on{mon}}\otimes \ul\CH^{\on{mon}}}\otimes \ul\CH^{\on{mon}}$$
value-wise generates the target category, so we have to show that the essential image of
\begin{multline*}
\ul\CH^{\on{mon}}\to \ul\CH^{\on{mon}}\underset{\ul\CH^{\on{mon}}\otimes \ul\CH^{\on{mon}}}\otimes \ul\CH^{\on{mon}}\simeq
\Tr(\on{Id},\ul\CH^{\on{mon}}\mmod(\AGCat)) \to \\
\to \Tr(\on{Id},\ul\CH\mmod(\AGCat))\simeq \Tr(\on{Id},G\mmod)\simeq \ul\Shv(G/\on{Ad}(G))
\end{multline*}
value-wise generates $\ul\Shv^{\on{ch}}(G/\on{Ad}(G))$.

\medskip

Note that we have a commutative diagram
$$
\CD
\ul\Shv(N\backslash G/N)^{\on{mon}}  @>>> \ul\Shv(N\backslash G/N)  \\
@V{=}VV @VV{=}V \\
 \ul\CH^{\on{mon}} @>>>   \ul\CH \\
@VVV @VVV \\
 \Tr(\on{Id},\ul\CH^{\on{mon}}\mmod) @>>>  \Tr(\on{Id},\ul\CH\mmod)  \\
 & & @VVV \\
 & & \Tr(\on{Id},G\mmod) \\
 & & @VVV \\
 & & \ul\Shv(G/\on{Ad}(G)),
\endCD
$$
where the third horizontal arrow is fully faithful.

\medskip

Hence, the required assertion follows from the combination of \corref{c:HCh gen}, the fact that the functor
$$\ul\Shv(N\backslash G/N)  \overset{\Av^{\on{Ad}(T)}_*}\longrightarrow \ul\Shv((N\backslash G/N)/\on{Ad}(T))$$
value-wise generates the essential image and the next statement:
\begin{prop} \label{p:B as class}
The composition
$$\ul\Shv(N\backslash G/N) =  \ul\CH \to \Tr(\on{Id},\ul\CH\mmod(\AGCat)) \simeq \Tr(\on{Id},G\mmod) \simeq \ul\Shv(G/\on{Ad}(G))$$
identifies with
\begin{equation} \label{e:B as class}
\ul\Shv(N\backslash G/N)  \overset{\Av^{\on{Ad}(T)}_*}\longrightarrow \ul\Shv((N\backslash G/N)/\on{Ad}(T)) \overset{\on{CH}}\to
 \ul\Shv(G/\on{Ad}(G)).
 \end{equation}
 \end{prop}

The rest of this subsection is devoted to the proof of \propref{p:B as class}.

\qed[\thmref{t:ch sheaves}]

\sssec{}

For an associative algebra $A$ (in a symmetric monoidal category $\bO$), the map
\begin{equation} \label{e:A map Hoch}
A\to A\underset{A\otimes A}\otimes A\simeq \on{Tr}(\on{Id},A\mod(\bO))
\end{equation}
can be interpreted as follows:

\medskip

Consider the diagram
\begin{equation} \label{e:A diag}
\xy
(0,0)*+{\bO}="A";
(40,0)*+{A\mod(\bO)}="C";
(0,-20)*+{\bO}="B";
(40,-20)*+{A\mod(\bO),}="D";
{\ar@{->}_{A\otimes (-)} "A";"B"};
{\ar@{->}^{A\otimes (-)} "A";"C"};
{\ar@{->}_{A\otimes (-)} "B";"D"};
{\ar@{->}^{\on{Id}} "C";"D"};
{\ar@{=>}^{\alpha} "B";"C"};
\endxy
\end{equation} 
where $\alpha$ is the map of $A$-modules
$$A\otimes \oblv_{A}(A) \to A,$$
given by right multiplication.

\medskip

We regard \eqref{e:A diag} as an instance of \eqref{e:basic diag} (in $\bL(\bO\mmod)$), and the map \eqref{e:A diag}
is obtained by applying to it the trace construction. 

\sssec{}

Hence, the composition
$$\ul\Shv(N\backslash G/N) =  \ul\CH \to \Tr(\on{Id},\ul\CH\mmod(\AGCat)) \simeq \Tr(\on{Id},G\mmod)$$
is given by
$$
\xy
(0,0)*+{\AGCat}="A";
(40,0)*+{G\mmod}="C";
(0,-20)*+{\AGCat}="B";
(40,-20)*+{G\mmod,}="D";
{\ar@{->}_{\ul\Shv(N\backslash G/N)\otimes (-)} "A";"B"};
{\ar@{->}^{\ul\Shv(N\backslash G/N)\otimes (-)} "A";"C"};
{\ar@{->}_{\ul\Shv(N\backslash G/N)\otimes (-)} "B";"D"};
{\ar@{->}^{\on{Id}} "C";"D"};
{\ar@{=>}^{\alpha} "B";"C"};
\endxy
$$
where $\alpha$ is given by
$$\ul\Shv(G/N)\otimes  \ul\Shv(N\backslash G/N)\overset{\on{convolution}}\longrightarrow \ul\Shv(G/N).$$

\medskip

Now the required assertion follows from \propref{p:map of traces stacks}.

\qed[\propref{p:B as class}]

\ssec{Functional equation for the Grothendieck-Springer functor}

In this subsection we will encounter a new feature of the Grothendieck-Springer functor
(in the spectrally finite situation); namely that for a pair of parabolics $P_1$ and $P_2$,
there is a canonical isomorphism
$$\on{GrothSpr}^{\on{ch}}_{P_1}\overset{i_{P_1,P_2}}\longrightarrow \on{GrothSpr}^{\on{ch}}_{P_2}$$
as functors
$$\ul\Shv^{\on{ch}}(M/\on{Ad}(M))\to \ul\Shv^{\on{ch}}(G/\on{Ad}(G)).$$

The question of transitivity of these isomorphisms (for a triple of parabolics) will be discussed
in \secref{s:trans}.

\sssec{}

Let $P$ be a parabolic in $G$. Recall that the (adjointable) 1-morphism
$$\bind_P:M\mmod\to G\mmod$$
induces a map of traces of the identity 1-endomorphisms
$$\Tr(\on{Id},M\mmod)\to \Tr(\on{Id},G\mmod),$$
which identifies with the Grothendieck-Springer functor
$$\on{GrothSpr}_P:\ul\Shv(M/\on{Ad}(M))\to \ul\Shv(G/\on{Ad}(G)).$$

\medskip

Similarly, the functor $\bind_P^{\on{spec.fin}}$ induces a map
$$\Tr(\on{Id},M\mmod^{\on{spec.fin}})\to \Tr(\on{Id},G\mmod^{\on{spec.fin}}),$$
so that we have a commutative diagram
$$
\CD
\Tr(\on{Id},M\mmod) @>{\Tr(\on{Id},\bind_P)}>>  \Tr(\on{Id},G\mmod) \\
@AAA @AAA \\
\Tr(\on{Id},M\mmod^{\on{spec.fin}}) @>{\Tr(\on{Id},\bind^{\on{spec.fin}}_P)}>>  \Tr(\on{Id},G\mmod^{\on{spec.fin}}).
\endCD
$$

Hence, we obtain a functor
$$\on{GrothSpr}^{\on{ch}}_P: \ul\Shv^{\on{ch}}(M/\on{Ad}(M))\to \ul\Shv^{\on{ch}}(G/\on{Ad}(G))$$
so that the diagram
$$
\CD
\ul\Shv(M/\on{Ad}(M)) @>{\on{GrothSpr}_P}>> \ul\Shv(G/\on{Ad}(G)) \\
@AAA @AAA \\
\ul\Shv^{\on{ch}}(M/\on{Ad}(M)) @>>{\on{GrothSpr}^{\on{ch}}_P}> \ul\Shv^{\on{ch}}(G/\on{Ad}(G))
\endCD
$$
commutes.

\medskip

Evaluating on $\on{pt}$ we obtain a functor
$$\on{GrothSpr}^{\on{ch}}_P: \Shv^{\on{ch}}(M/\on{Ad}(M))\to \Shv^{\on{ch}}(G/\on{Ad}(G)).$$

\sssec{}

Let $P_1$ and $P_2$ be a pair of parabolics as in \secref{sss:intertwining functor}.
Recall that the resulting two functors
$$\on{GrothSpr}_{P_1},\on{GrothSpr}_{P_2}: \ul\Shv(M/\on{Ad}(M))\to \ul\Shv(G/\on{Ad}(G))$$
are in general \emph{not} isomorphic.

\medskip

However, we claim:

\begin{thmconstr} \label{t:intertw char}
The (iso)morphism
$$\bind_{P_1}^{\on{spec.fin}} \overset{I^{\on{spec.fin}}_{P_1,P_2}}\longrightarrow \bind_{P_2}^{\on{spec.fin}}$$
induces an isomorphism
$$\on{GrothSpr}^{\on{ch}}_{P_1}\overset{i_{P_1,P_2}}\longrightarrow \on{GrothSpr}^{\on{ch}}_{P_2}$$
as functors
$$\ul\Shv^{\on{ch}}(M/\on{Ad}(M))\to \ul\Shv^{\on{ch}}(G/\on{Ad}(G)).$$
\end{thmconstr}

\begin{proof}

Follows from \secref{sss:higher funct} using \propref{p:intertwiner spec finite}.

\end{proof}

\begin{rem}

In \secref{ss:intw expl} we will describe the isomorphism $i_{P_1,P_2}$ in
(somewhat) explicit terms.

\end{rem}

\sssec{}

Recall that we have a canonical isomorphism
$$\on{GrothSpr}_G\simeq \sigma_G\circ \on{GrothSpr}_G \circ \sigma_M,$$
where $\sigma_G$ and $\sigma_M$ denote the inversion automorphisms of $\ul\Shv(G/\on{Ad}(G))$ and
$\ul\Shv(M/\on{Ad}(M))$, respectively (see \eqref{e:Groth Spr and inversion}).

\medskip

Hence, we obtain an isomorphism
\begin{equation} \label{e:GrothSpr and inv ch}
\on{GrothSpr}^{\on{ch}}_P \simeq \sigma_G\circ  \on{GrothSpr}^{\on{ch}}_P \circ \sigma_M.
\end{equation}

Unwinding, we obtain:

\begin{lem} \label{l:intertwiner and inversion}
The diagram
$$
\CD
\on{GrothSpr}^{\on{ch}}_{P_1}  @>{i_{P_1,P_2}}>> \on{GrothSpr}^{\on{ch}}_{P_2}  \\
@V{\text{\eqref{e:GrothSpr and inv ch}}}VV @VV{\text{\eqref{e:GrothSpr and inv ch}}}V \\
\sigma_G\circ  \on{GrothSpr}^{\on{ch}}_{P_1} \circ \sigma_M @>{i_{P_1,P_2}}>> \sigma_G\circ  \on{GrothSpr}^{\on{ch}}_{P_2} \circ \sigma_M.
\endCD
$$
\end{lem}

\ssec{Functional equation on the regular semi-simple locus}

In this subsection, we will show that when $M=T$ is the Cartan subgroup, the functional
equation boils down to something familiar.

\sssec{} \label{sss:Springer GM}

We will now calculate the isomorphism of \thmref{t:intertw char} is a particular case. We take $P_1=B_1=B$
and $P_2=B_2$ be a Borel subgroup in relative position $w$ with respect to $B$ for some $w\in W$. Note
that we can then rewrite $\on{GrothSpr}^{\on{ch}}_{B_2}$ as
$$\on{GrothSpr}^{\on{ch}}_{B_1}\circ w^{-1},$$
where $w$ denotes the action of the Weyl group on $T$.

\medskip

We start with a 1-dimensional character sheaf $\chi\in \Shv(T)$, which we can regard as an object of
$\Shv(T/\on{Ad}(T))$ as the adjoint action of $T$ on itself is trivial.

\medskip

Consider the objects
\begin{equation} \label{e:GrothSpr and w}
\on{GrothSpr}_{B_1}(\chi) \text{ and } \on{GrothSpr}_{B_2}(\chi)
\end{equation}
of $\Shv^{\on{ch}}(G/\on{Ad}(G))\subset \Shv(G/\on{Ad}(G))$.

\medskip

It is well-known that both lie in $\Shv(G/\on{Ad}(G))^\heartsuit$ and are the Goresky-MacPherson
extensions of their restrictions to the regular semi-simple locus $G^{\on{rs}}/\on{Ad}(G)$.

\begin{lem} \label{l:Springer w}
The restrictions of the objects \eqref{e:GrothSpr and w} to $G^{\on{rs}}/\on{Ad}(G)$ are
canonically isomorphic.
\end{lem}

\begin{proof}

This statement is well-known; we include the proof for completeness.

\medskip

By definition, for a Borel subgroup $B$, the object $\on{GrothSpr}_B(\chi)$ is push-pull of $\chi$ along the diagram
$$
\CD
\widetilde{G}/\on{Ad}(G) @>>> G/\on{Ad}(G) \\
@VVV \\
T/\on{Ad}(T),
\endCD
$$
where $\widetilde{G}\to G$ is the Grothendieck-Springer resolution.

\medskip

Let $\widetilde{G}^{\on{rs}}$ denote the preimage of the open $G^{\on{rs}}\subset G$ in $\widetilde{G}$.
Consider the resulting diagram:
$$
\CD
\widetilde{G}^{\on{rs}}/\on{Ad}(G) @>>> G^{\on{rs}}/\on{Ad}(G)\\
@VVV \\
T/\on{Ad}(T).
\endCD
$$

Now, it is known that the action of $W$ on $T$ can be lifted to one on $\widetilde{G}^{\on{rs}}$, so that the projection
$\widetilde{G}^{\on{rs}}\to G^{\on{rs}}$ is a $W$-torsor.

\medskip

Thus, the action of $w$ induces an isomorphism
$$\on{GrothSpr}_{B_1}(\chi)|_{G^{\on{rs}}/\on{Ad}(G)} \to \on{GrothSpr}_{B_2}(\chi)|_{G^{\on{rs}}/\on{Ad}(G)}.$$

\end{proof}

\sssec{}

We claim:

\begin{thm} \label{t:Springer action}
The isomorphism
$$\on{GrothSpr}_{B_1}(\chi) \overset{i_{B_1,B_2}}\simeq \on{GrothSpr}_{B_2}(\chi)$$
of \thmref{t:intertw char} equals the Goresky-MacPherson extension of the isomorphism
$$\on{GrothSpr}_{B_1}(\chi)|_{G^{\on{rs}}/\on{Ad}(G)} \simeq \on{GrothSpr}_{B_2}(\chi)|_{G^{\on{rs}}/\on{Ad}(G)}$$
of \lemref{l:Springer w}.
\end{thm}

The rest of this subsection is devoted to the proof of  \thmref{t:Springer action}.

\sssec{Interpreting the fibers of the Springer sheaf}

It is sufficient that the two morphisms in question agree at the level of fibers at points $g\in G^{\on{rs}}$.

\medskip

Unwinding the definitions, the !-fiber of $\on{GrothSpr}_{B_1}(\chi)$ at $g\in G$
is the object of $\Vect$ equal to
$$\Tr(g,\ul\Shv(G/N_1)^{T,\chi}),$$
and similarly for $\on{GrothSpr}_{B_2}(\chi)$.

\medskip

Furthermore, the fiber at $g$ of the isomorphism
\begin{equation} \label{e:Springer fibers}
\on{GrothSpr}_{B_1}(\chi) \simeq \on{GrothSpr}_{B_2}(\chi)
\end{equation}
of \thmref{t:intertw char} is given by applying $\Tr(g,-)$ to the
$g$-equivariant 1-isomorphism
\begin{equation} \label{e:intertw chi}
\ul\Shv(G/N_1)^{T,\chi} \overset{\sim}\to \ul\Shv(G/N_2)^{T,\chi},
\end{equation}
given by pull-push along the diagram
\begin{equation} \label{e:B-diag}
\CD
G/N_1\cap N_2 @>>> G/N_2 \\
@VVV \\
G/N_1.
\endCD
\end{equation}

\sssec{Calculating the fibers} \label{sss:Springer fibers}

Explicitly, the object
$$\Tr(g,\ul\Shv(G/N_i)^{T,\chi})\in \Vect$$
is calculated as follows.

\medskip

Consider the scheme
$$(G/B_i)^g:=\{g'\,|\, g'{}^{-1}\cdot g\cdot g'\in B_i\}/B_i.$$

We have a well-defined map
$$q_i:(G/B_i)^g\to T, \quad g'\mapsto \ol{g'{}^{-1}\cdot g\cdot g'},$$
where $b\mapsto \ol{b}$ denotes the projection $B_i\to T$.

\medskip

Then
$$\Tr(g,\ul\Shv(G/N_i)^{T,\chi})\simeq \on{C}^\cdot((G/B_i)^g,q_i^!(\chi)).$$

\sssec{Geometric interpretation}

Consider the diagram
$$
\CD
(G/B_1\cap B_2)^g @>>> (G/B_2)^g @>>{q_2}>  T \\
@VVV \\
(G/B_1)^g \\
@V{q_1}VV \\
T.
\endCD
$$

If $g$ is regular semi-simple, both maps
\begin{equation} \label{e:B-diag fixed}
\CD
(G/B_1\cap B_2)^g @>>> (G/B_2)^g \\
@VVV \\
(G/B_1)^g
\endCD
\end{equation}
are isomorphisms.

\medskip

The resulting identification
\begin{equation} \label{e:ident fibers geom}
\on{C}^\cdot((G/B_1)^g,q_i^!(\chi)) \simeq \on{C}^\cdot((G/B_2)^g,q_i^!(\chi))
\end{equation}
is the fiber at $g$ of the isomorphism
$$\on{GrothSpr}_{B_1}(\chi)|_{G^{\on{rs}}/\on{Ad}(G)} \simeq \on{GrothSpr}_{B_2}(\chi)|_{G^{\on{rs}}/\on{Ad}(G)}$$
of \lemref{l:Springer w}.

\sssec{Comparing the two identifications}

Let $Y_i$ be the preimage in $G/N_i$ of $(G/B_i)^g\subset G/B_i$. Let
$Y_{1,2}$ be the preimage in $G/N_1\cap N_2$ of $(G/B_1\cap B_2)^g\subset G/B_1\cap B_2$.

\medskip

We have a 1-isomorphism
$$\ul\Shv(Y_1)^{T,\chi}\to \ul\Shv(Y_2)^{T,\chi},$$
given by pull-push along
$$
\CD
Y_{1,2} @>>> Y_2 \\
@VVV \\
Y_1,
\endCD
$$
which is compatible with \eqref{e:intertw chi} via the $g$-equivariant 1-morphisms
$$\ul\Shv(Y_i)^{T,\chi}\to \ul\Shv(G/N_i)^{T,\chi}.$$

\medskip

In particular, we have a commutative diagram in $\Vect$
$$
\CD
\Tr(g,\ul\Shv(Y_1)^{T,\chi}) @>>> \Tr(g,\ul\Shv(Y_2)^{T,\chi})  \\
@VVV @VVV \\
\Tr(g,\ul\Shv(G/N_1)^{T,\chi}) @>>> \Tr(g,\ul\Shv(G/N_2)^{T,\chi}),
\endCD
$$
which fits into the diagram
\begin{equation} \label{e:fixed points diag}
\CD
\on{C}^\cdot((G/B_1)^g,q_1^!(\chi)) @>>> \on{C}^\cdot((G/B_1)^g,q_2^!(\chi)) \\
@VVV @VVV \\
\Tr(g,\ul\Shv(Y_1)^{T,\chi})  @>>> \Tr(g,\ul\Shv(Y_2)^{T,\chi}) \\
@VVV @VVV \\
 \Tr(g,\ul\Shv(G/N_1)^{T,\chi})  @>>> \Tr(g,\ul\Shv(G/N_2)^{T,\chi})  \\
@V{\sim}VV @VV{\sim}V \\
 (\on{GrothSpr}_{B_1}(\chi))_g @>>{\text{\eqref{e:Springer fibers}}}>  (\on{GrothSpr}_{B_2}(\chi))_g
\endCD
\end{equation}
in which the composite vertical arrows
$$\on{C}^\cdot((G/B_i)^g,q_i^!(\chi)) \to \Tr(g,\ul\Shv(Y_i)^{T,\chi}) \to \Tr(g,\ul\Shv(G/N_i)^{T,\chi})$$
are the identifications of \secref{sss:Springer fibers}.

\medskip

Hence, it remains to show that the top horizontal arrow in \eqref{e:fixed points diag} is the identification of \eqref{e:ident fibers geom}. However,
the latter is manifest, as the maps in diagram \eqref{e:B-diag fixed} are isomorphisms.

\qed[\thmref{t:Springer action}]

\ssec{Ambidexterity of the Grothendieck-Springer functor}

\sssec{}

Recall that in \thmref{t:jacq mon}, we proved that the 1-morphism
$$\bjacq^{\on{spec.fin}}_P:G\mmod^{\on{spec.fin}}\to M\mmod^{\on{spec.fin}}$$
is 2-adjointable, and its (both left and right) adjoints identify with $\bind_P^{\on{spec.fin}}$.

\medskip

Hence, by \propref{p:adj Tr}, we obtain: an identification between
the functor
\begin{equation} \label{e:Tr jacq spec fin}
\Tr(\on{Id},G\mmod^{\on{spec.fin}}) \overset{\Tr(\on{id},\bjacq^{\on{spec.fin}}_P)}\longrightarrow \Tr(\on{Id},M\mmod^{\on{spec.fin}})
\end{equation}
and the \emph{left} adjoint of
\begin{equation} \label{e:Tr ind spec fin}
\Tr(\on{Id},M\mmod^{\on{spec.fin}}) \overset{\Tr(\on{id},\bind^{\on{spec.fin}}_P)}\longrightarrow \Tr(\on{Id},G\mmod^{\on{spec.fin}}).
\end{equation}

\sssec{}

Recall now that under the identifications
$$\Tr(\on{Id},G\mmod^{\on{spec.fin}})\simeq \ul\Shv^{\on{ch}}(G/\on{Ad}(G)) \text{ and }
\Tr(\on{Id},M\mmod^{\on{spec.fin}})\simeq \ul\Shv^{\on{ch}}(M/\on{Ad}(M)),$$
the functors \eqref{e:Tr jacq spec fin} and \eqref{e:Tr ind spec fin} identify with
$$\fr^{\on{ch}}_P:\ul\Shv^{\on{ch}}(G/\on{Ad}(G)) \to \ul\Shv^{\on{ch}}(M/\on{Ad}(M))$$  and
$$\on{GrothSpr}^{\on{ch}}_P:\ul\Shv^{\on{ch}}(M/\on{Ad}(M)) \to \ul\Shv^{\on{ch}}(G/\on{Ad}(G)),$$
respectively.

\sssec{}

Hence, we obtain:

\begin{cor} \label{c:restr as left adj}
There exists a canonical identification between
$$\fr^{\on{ch}}_P:\ul\Shv^{\on{ch}}(G/\on{Ad}(G)) \to \ul\Shv^{\on{ch}}(M/\on{Ad}(M))$$
and a \emph{left} adjoint of
$$\on{GrothSpr}^{\on{ch}}_P:\ul\Shv^{\on{ch}}(M/\on{Ad}(M)) \to \ul\Shv^{\on{ch}}(G/\on{Ad}(G)).$$
\end{cor}

\begin{rem}
Recall (see \secref{c:ind restr}) that in addition to identification of \corref{c:restr as left adj} we have a
(rather tautological) identification between
$$\fr_P:\ul\Shv(G/\on{Ad}(G)) \to \ul\Shv(M/\on{Ad}(M))$$
and a \emph{right} adjoint of
$$\on{GrothSpr}_P:\ul\Shv(M/\on{Ad}(M)) \to \ul\Shv(G/\on{Ad}(G)).$$

In particular, the $(\on{GrothSpr}_P,\fr_P)$ becomes \emph{ambidexterous} once restricted
to character sheaves.
\end{rem}

\sssec{}

Consider again the usual $(\on{GrothSpr}_P,\fr_P)$-adjunction, and let us conjugate it by Verdier duality.
We obtain an adjunction
\begin{equation} \label{e:Verdier GrothSpr adj}
(\BD_{M/\on{Ad}(M)}\circ \fr_P\circ \BD_{G/\on{Ad}(G)},\BD_{G/\on{Ad}(G)}\circ \on{GrothSpr}_P \circ \BD_{M/\on{Ad}(M)}),
\end{equation}
as functors
$$\Shv(G/\on{Ad}(G))^c\leftrightarrow \Shv(M/\on{Ad}(M))^c.$$

\sssec{}

Note now that the functor $\on{GrothSpr}_P$ satisfies
$$\BD_{G/\on{Ad}(G)}\circ \on{GrothSpr}_P \circ \BD_{M/\on{Ad}(M)}\simeq \on{GrothSpr}_P.$$

Denote
$$\wt\fr_P:=\BD_{M/\on{Ad}(M)}\circ \fr_P\circ \BD_{G/\on{Ad}(G)},$$
and let us ind-extend this functor from $\Shv(G/\on{Ad}(G))^c$ to all of $\Shv(G/\on{Ad}(G))$.

\medskip

Thus, \eqref{e:Verdier GrothSpr adj} can be rewritten as an adjunction
\begin{equation} \label{e:Verdier GrothSpr adj bis}
(\wt\fr_P,\on{GrothSpr}_P),
\end{equation}
as functors
$$\Shv(G/\on{Ad}(G))\leftrightarrow \Shv(M/\on{Ad}(M)).$$

\sssec{}

In particular, restricting \eqref{e:Verdier GrothSpr adj bis} to character sheaves, we obtain an adjunction
\begin{equation} \label{e:Verdier GrothSpr adj ch}
(\wt\fr^{\on{ch}}_P,\on{GrothSpr}^{\on{ch}}_P),
\end{equation}
as functors
$$\Shv^{\on{ch}}(G/\on{Ad}(G))\leftrightarrow \Shv^{\on{ch}}(M/\on{Ad}(M)),$$
where
$$\wt\fr^{\on{ch}}_P:=\wt\fr_P|_{\Shv^{\on{ch}}(G/\on{Ad}(G))}.$$

Hence, combining with \corref{c:restr as left adj}, we obtain:

\begin{cor} \label{c:restr and Verdier}
There is a canonical isomorphism
\begin{equation} \label{e:restr and Verdier}
\wt\fr^{\on{ch}}_P\simeq \fr^{\on{ch}}_P
\end{equation}
as functors
$$\Shv^{\on{ch}}(G/\on{Ad}(G))\to \Shv^{\on{ch}}(M/\on{Ad}(M)).$$
\end{cor}

\begin{rem}

As was explained to us by R.~Bezrukavnikov, one can prove \corref{c:restr and Verdier} directly, without involving
traces on $\AGCat$. We will return to this in \secref{sss:self-dual r expl}.

\end{rem}

\begin{rem}

Note that by construction, the identification \eqref{e:restr and Verdier} is involutive in the sense that the composition
\begin{multline*}
\fr^{\on{ch}}_P\simeq \BD_{M/\on{Ad}(M)}\circ \BD_{M/\on{Ad}(M)}\circ \fr^{\on{ch}}_P
\overset{\text{\eqref{e:restr and Verdier}}}\simeq \BD_{M/\on{Ad}(M)}\circ \fr^{\on{ch}}_P\circ
\BD_{G/\on{Ad}(G)} \overset{\text{\eqref{e:restr and Verdier}}}\simeq \\
\simeq \fr^{\on{ch}}_P\circ  \BD_{G/\on{Ad}(G)}
\circ  \BD_{G/\on{Ad}(G)}\simeq \fr^{\on{ch}}_P
\end{multline*}
is the identity.

\end{rem}

\section{Further properties of the Grothendieck-Springer functor} \label{s:trans}

The main goal of this section is to establish the transitivity property of the functional
equation (this is \thmref{t:intertwiner trans} below).

\medskip

In addition, we will use the Serre functor on $\bind_P$ to give streamlined proofs
for assertions about the Harish-Chandra transform of character sheaves.

\ssec{Transitivity of the functional equation}

\sssec{} \label{sss:three piglets}

Let $P_1,P_2$ and $P_3$ be three parabolic subgroups of $G$ containing a common Levi subgroup $M$. Consider the natural transformations
$$\on{GrothSpr}^{\on{ch}}_{P_1}\overset{i_{P_1,P_2}}\longrightarrow \on{GrothSpr}^{\on{ch}}_{P_2},$$
$$\on{GrothSpr}^{\on{ch}}_{P_2}\overset{i_{P_2,P_3}}\longrightarrow \on{GrothSpr}^{\on{ch}}_{P_3}$$
and
$$\on{GrothSpr}^{\on{ch}}_{P_1}\overset{i_{P_1,P_3}}\longrightarrow \on{GrothSpr}^{\on{ch}}_{P_3}.$$

We will prove the following assertion:

\begin{thm} \label{t:intertwiner trans}
There exists \emph{an} somorphism
\begin{equation} \label{e:intertwiner trans}
i_{P_2,P_3}\circ i_{P_1,P_2}\simeq i_{P_1,P_3}
\end{equation}
as natural transformations $\on{GrothSpr}^{\on{ch}}_{P_1}\to \on{GrothSpr}^{\on{ch}}_{P_3}$.
\end{thm}

The next few subsections are devoted to the proof of this theorem.

\begin{rem}

Our construction of the isomorphism \eqref{e:intertwiner trans} depends on some combinatorial choices.
However, we conjecture that it is actually canonical, i.e., we have a Grothendieck-Springer functor
$$\on{GrothSpr}^{\on{ch}}_{M,G}:\Shv^{\on{ch}}(M/\on{Ad}(M))\to \Shv^{\on{ch}}(G/\on{Ad}(G))$$
that only depends on the pair $M$ and $G$, and moreover it does so in a canonical way.

\medskip

In particular, we conjecture that the functor $\on{GrothSpr}^{\on{ch}}_{M,G}$ is acted on by $\on{Norm}(M)/M$.

\medskip

Note that for $M=T$, this means that the corresponding functor $\on{GrothSpr}^{\on{ch}}_{T,G}$ is acted
on by the Weyl group. (Note that this is \emph{not} in contradiction with \cite[Theorem D]{Gu}, as the latter
talks about all of $\Shv(M/\on{Ad}(M))$ and not $\on{GrothSpr}^{\on{ch}}_{M,G}$.

\end{rem}

\sssec{} \label{sss:add up}

Note that for a triple of parabolics, there is a well-defined notion of their relative positions adding up: this
means that
$$N(P_1)\cap N(P_3)\subset N(P_2).$$

\medskip

For $P_1,P_2,P_3$ as in \secref{sss:three piglets}, if their relative positions add up, then
\begin{equation} \label{e:big intertw compose}
I_{P_1,P_3}\circ I_{P_1,P_2}\simeq I_{P_1,P_3}
\end{equation}
as 2-morphisms $\bind_{P_1}\to \bind_{P_3}$, and the assertion of \thmref{t:intertwiner trans} is
obtained by applying trace to \eqref{e:big intertw compose}.

\sssec{}

We now consider the case opposite to one in \secref{sss:add up}.

\medskip

The key ingredient in the proof of \thmref{t:intertwiner trans} is the following:

\begin{thm} \label{t:involut}
For a pair of parabolics $P_1$ and $P_2$, exists a canonical isomorphism
$$i_{P_2,P_1}\circ i_{P_1,P_2}\simeq \on{id}$$
as endofunctors of $\on{GrothSpr}^{\on{ch}}_{P_1}$.
\end{thm}

\ssec{The implication \thmref{t:involut} \texorpdfstring{$\Rightarrow$}{implies} \thmref{t:intertwiner trans}}

\sssec{}

Let $P'$ and $P''$ be a pair parabolics sharing a Levi subgroup. We shall say that they are \emph{adjacent}
if there exists a parabolic $Q$ containing $P'$ and $P''$, and such that the projections of $P'$ and $P''$
to $L:=Q/N(Q)$ are \emph{sub-maximal opposite} parabolics in $L$.

\sssec{}

We have the following combinatorial statement\footnote{We learnt the proof of \lemref{l:rel Bruhat} from K.~Lin; however, it is possible that
the assertion has been known classically.}:

\begin{lem}  \label{l:rel Bruhat}
For a pair of parabolics $P_1\neq P_2$ sharing a Levi subgroup, let $P'_1$ be another such parabolic adjacent to $P_1$.
Then:

\medskip

\noindent{\em(a)} Exactly one of the following two scenarios occurs:

\begin{enumerate}

\item Either the relative positions $P_1,P'_1,P_2$ add up;

\item Or the relative positions $P'_1,P_1,P_2$ add up.

\end{enumerate}

\medskip

\noindent{\em(b)} There exists $P'_1$ so that scenario \em{(1)} occurs.

\end{lem}

\begin{proof}

We will prove point (a); point (b) is proved along the same lines.

\medskip

Recall that for a pair of parabolic subgroups $Q_1$ and $Q_2$ (with Levi quotients
$L_1$ and $L_2$, respectively), the image of the projection
$$Q_1\cap Q_2/Q_1\cap N(Q_2)\to M_2$$
is a parabolic subgroup in $M_2$, and similarly for $1$ and $2$ swapped.

\medskip

For $P_1$ and $P'_1$ as in the lemma, let $Q$ be a parabolic that contains $P_1$ and $P'_1$ submaximally; let $L$
denote its Levi quotient. Denote by $Q_1$ and $Q'_1$ the images  of
$$P_1\cap Q\to L \leftarrow P'_1\cap Q,$$
respectively. They are two submaximal opposite parabolics with $Q_1\cap Q'_1\simeq M$.

\medskip

Let $Q_2$ denote the image of $Q\cap P_2$ in $L$. Since $M\subset P_2$, we obtain that $Q_2$ contains
(the image of) $M$. Thus, we obtain that $Q_2$ is \emph{yet another} submaximal parabolic in $L$ that contains
(the image of) $M$. This implies that either $Q_2=Q_1$ or $Q_2=Q'_1$.

\medskip

By symmetry, let us assume the latter. We claim that in this case the relative positions $P_1,P'_1,P_2$ add up,
i.e.,
$$N(P_1)\cap N(P_2)\subset N(P'_1).$$

\medskip

Choose a maximal torus $T\subset M$, so that we can consider the root decomposition with respect to it.
Thus, let $\alpha$ be a root in $N(P_1)\cap N(P_2)$. We claim that $\alpha$ is in fact contained in $N(Q)$;
this would imply that $\alpha$ is contained in $N(P'_1)$ as $N(Q)\subset N(P'_1)$.

\medskip

Indeed, if $\alpha\notin N(Q)$, then $\alpha\in N(Q_1)$. Since $\alpha\in N(P_2)$, we also have 
$\alpha \in N(Q_2)$. However, since $Q_1$ and $Q_2$ are opposite parabolics in $L$, there
are no roots common to $N(Q_1)$ and $N(Q_2)$, which is a contradiction.

\end{proof}

\sssec{}

Let us show how \lemref{l:rel Bruhat} and \thmref{t:involut} imply \thmref{t:intertwiner trans}.

\medskip

We argue by induction on the distance $P_1$ and $P_2$, i.e., $\dim(P_1/P_1\cap P_2)=\dim(P_2/P_1\cap P_2)$.

\medskip

The base case is when $P_1=P_2$, in which case there is nothing to prove.

\sssec{}

Next we consider the case when $P_1$ and $P_2$ are adjacent. Then by \lemref{l:rel Bruhat}(a) there are
two scenarios:


%

\medskip

\noindent{(a)} Either the relative positions of $P_1,P_2,P_3$ add up;

\medskip

\noindent{(b)} Or the relative positions of $P_2,P_1,P_3$ add up.

\medskip

In case (a) we are done by \secref{sss:add up}.

\medskip

In case (b), we have
$$i_{P_2,P_3}\circ i_{P_1,P_2} \overset{\text{\secref{sss:add up}}}\simeq
i_{P_1,P_3}\circ i_{P_2,P_1}\circ i_{P_1,P_2} \overset{\text{\thmref{t:involut}}}\simeq i_{P_1,P_3},$$
as required.

\sssec{}

We now perform the induction step, i.e., we will show how to reduce the statement for
a given triple $(P_1,P_2,P_3)$ to another one $(P'_1,P'_2,P'_3)$ for which the distance
between $P'_1$ and $P'_2$ is smaller than that between $P_1$ and $P_2$.

\medskip

We will take $P'_2=P_2$ and $P'_3=P_3$, and we let $P'_1$ be as in \lemref{l:rel Bruhat}(b). Note
that in situation of \secref{sss:add up}, the distances also add up, i.e., the distance between
$P'_1$ and $P_2$ is indeed smaller than the distance between $P_1$ and $P_2$.

\medskip

We have
$$i_{P_2,P_3}\circ i_{P_1,P_2} \overset{\text{\secref{sss:add up}}}\simeq
i_{P_2,P_3}\circ i_{P'_1,P_2}\circ i_{P_1,P'_1}  \overset{\text{induction hypoth.}}\simeq
i_{P'_1,P_3}\circ i_{P_1,P'_1}\overset{\text{adjacent case}}\simeq  i_{P_1,P_3}.$$

\qed[\thmref{t:intertwiner trans}]

\ssec{Explicit description of the functional equation} \label{ss:intw expl}

In this subsection we will give an explicit description of the natural transformation
$$\on{GrothSpr}^{\on{ch}}_{P_1}\overset{i_{P_1,P_2}}\longrightarrow \on{GrothSpr}^{\on{ch}}_{P_2}$$
as functors
$$\Shv^{\on{ch}}(M)\to \Shv^{\on{ch}}(G).$$

\sssec{}

Recall that the functor $\on{GrothSpr}_P$ is given by pull-push along the diagram
$$
\CD
P/\on{Ad}(P) @>{\sfp_P}>> G/\on{Ad}(G) \\
@V{\sfq_P}VV \\
M/\on{Ad}(M).
\endCD
$$

Consider the functor
\begin{equation} \label{e:HCh P}
\on{CH}_P:\ul\Shv((N(P)\backslash G/N(P))/\on{Ad}(M)) \to \ul\Shv(G/\on{Ad}(G)),
\end{equation}
given by pull-push along
\begin{equation} \label{e:HCh P diag}
\CD
G/\on{Ad}(P) @>>> G/\on{Ad}(G) \\
@VVV \\
(N(P)\backslash G/N(P))/\on{Ad}(M).
\endCD
\end{equation}

Let $\iota_P$ denote the functor
\begin{equation} \label{e:iota P}
\ul\Shv(M/\on{Ad}(M))\to \ul\Shv((N(P)\backslash G/N(P))/\on{Ad}(M))
\end{equation}
given by pull-push along
\begin{equation} \label{e:iota diag}
\CD
(N(P)\backslash P/N(P))/\on{Ad}(M) @>>> (N(P)\backslash G/N(P))/\on{Ad}(M) \\
@VVV \\
M/\on{Ad}(M).
\endCD
\end{equation}
We have
$$\on{GrothSpr}_P\simeq \on{CH}_P\circ \iota_P$$

\sssec{}

Consider the category
$$\ul\Shv(N(P)\backslash G/N(P))^{\on{spec.fin}},$$
where the spec.fin condition is taken with respect to the left or equivalently right action of $M$ (or both).

\medskip

Set
$$\ul\Shv^{\on{ch}}((N(P)\backslash G/N(P))/\on{Ad}(M)):=\binv_{\Delta(M)}\left(\ul\Shv(N(P)\backslash G/N(P))^{\on{spec.fin}}\right).$$

The functors $\on{CH}_P$ and $\iota_P$ restrict to functors
$$\on{CH}^{\on{ch}}_P:\ul\Shv^{\on{ch}}(G/\on{Ad}(G))\to \ul\Shv^{\on{ch}}((N(P)\backslash G/N(P))/\on{Ad}(M))$$
and
$$\iota^{\on{ch}}_P:\ul\Shv^{\on{ch}}(M/\on{Ad}(M))\to \ul\Shv^{\on{ch}}((N(P)\backslash G/N(P))/\on{Ad}(M)),$$
respectively, and we have
\begin{equation} \label{e:GrothSpr as compos}
\on{GrothSpr}^{\on{ch}}_P \simeq \on{CH}^{\on{ch}}_P\circ \iota^{\on{ch}}_P.
\end{equation}

\sssec{} \label{e:ip1p2}

With the above notations, applying \eqref{e:2-trace}, we obtain that the isomorphism
$$\on{GrothSpr}^{\on{ch}}_{P_1} \overset{i_{P_1,P_2}}\to \on{GrothSpr}^{\on{ch}}_{P_2}$$
is equal to the composition
\begin{equation} \label{e:intertwiner as compos}
\on{CH}^{\on{ch}}_{P_1}\circ \iota^{\on{ch}}_{P_1} \overset{\sim}\to
\on{CH}^{\on{ch}}_{P_2}\circ ((I^{\on{spec.fin}}_{P_2,P_1})^{-1})\otimes I^{\on{spec.fin}}_{P_1,P_2}) \circ \iota^{\on{ch}}_{P_1} \overset{\sim}\to
\on{CH}^{\on{ch}}_{P_2}\circ \iota^{\on{ch}}_{P_2},
\end{equation}
where the two arrows are described as follows.

\sssec{}

We start with the second arrow in \eqref{e:intertwiner as compos}. We claim that there is a canonical isomorphism
\begin{equation} \label{e:intertwiner as compos 2}
((I^{\on{spec.fin}}_{P_2,P_1})^{-1})\otimes I^{\on{spec.fin}}_{P_1,P_2}) \circ \iota^{\on{ch}}_{P_1} \overset{\sim}\to \iota^{\on{ch}}_{P_2}.
\end{equation}

To construct it, we rewrite the sought for isomorphism as
\begin{equation} \label{e:intertwiner as compos 2 ch}
(\on{Id}\otimes I^{\on{spec.fin}}_{P_1,P_2}) \circ \iota^{\on{ch}}_{P_1} \overset{\sim}\to
(I^{\on{spec.fin}}_{P_2,P_1}\otimes \on{Id}) \circ \iota^{\on{ch}}_{P_2}.
\end{equation}

The latter follows from
\begin{equation} \label{e:intertwiner as compos 2 bis}
(\on{Id}\otimes I_{P_1,P_2}) \circ \iota_{P_1} \overset{\sim}\to
(I_{P_2,P_1}\otimes \on{Id}) \circ \iota_{P_2},
\end{equation}
which is a valid isomorphism, as the two sides are given by pull-push along
\begin{equation} \label{e:P_1 P_2 diag}
\CD
\Bigl(N(P_1)\backslash (N(P_1)\cdot M\cdot N(P_2))/N(P_2)\Bigr)/\on{Ad}(M) @>>> \Bigl(N(P_1)\backslash G/N(P_2)\Bigr)/\on{Ad}(M) \\
@VVV \\
M/\on{Ad}(M).
\endCD
\end{equation}

\sssec{}

We now consider the first arrow in \eqref{e:intertwiner as compos}. We claim that there is a canonical isomorphism
\begin{equation} \label{e:intertwiner as compos 1}
\on{CH}^{\on{ch}}_{P_1} \overset{\sim}\to
\on{CH}^{\on{ch}}_{P_2}\circ ((I^{\on{spec.fin}}_{P_2,P_1})^{-1})\otimes I^{\on{spec.fin}}_{P_1,P_2}).
\end{equation}

To construct it, we rewrite the sought for isomorphism as
\begin{equation} \label{e:intertwiner as compos 1 ch}
\on{CH}^{\on{ch}}_{P_1} \circ (I^{\on{spec.fin}}_{P_2,P_1}\otimes \on{Id})\simeq
\on{CH}^{\on{ch}}_{P_2}\circ (\on{Id}\otimes I^{\on{spec.fin}}_{P_1,P_2}).
\end{equation}

The latter follows from
\begin{equation} \label{e:intertwiner as compos 1 bis}
\on{CH}_{P_1} \circ (I_{P_2,P_1}\otimes \on{Id})\simeq
\on{CH}_{P_2}\circ (\on{Id}\otimes I_{P_1,P_2}),
\end{equation}
which is a valid isomorphism, as the two sides are given by pull-push along
$$
\CD
(G\overset{M\cdot (N(P_2)\cap N(P_1))}\times G)/\on{Ad}(G) @>>> G/\on{Ad}(G) \\
@VVV \\
(G/N(P_2)\overset{M}\times N(P_1)\backslash G)/\on{Ad}(G) \\
@VV{\sim}V \\
(N(P_2)\backslash G/N(P_1))/\on{Ad}(M).
\endCD
$$

\ssec{Proof of \thmref{t:involut}}

This proof will be essentially a tautology modulo one ingredient: the role of the inversion involution on $G$,
which exchanges the roles of $P_1$ and $P_2$.

\sssec{}

Recall that we denote by $\sigma_G$ and $\sigma_M$ the inversion automorphisms of $\ul\Shv(G/\on{Ad}(G))$ and
$\ul\Shv(M/\on{Ad}(M))$, respectively.

\medskip

By \lemref{l:intertwiner and inversion}, we have a commutative diagram
$$
\CD
\on{GrothSpr}_{P_1} @>{i_{P_1,P_2}}>>  \on{GrothSpr}_{P_2} \\
@V{\sim}VV @VV{\sim}V \\
\sigma_G\circ \on{GrothSpr}_{P_1} \circ \sigma_M  @>{i_{P_1,P_2}}>>  \sigma_G\circ \on{GrothSpr}_{P_2} \circ \sigma_M,
\endCD
$$

Thus, the assertion of \thmref{t:involut} is equivalent to the commutativity of the diagram
\begin{equation} \label{e:diagram to prove}
\CD
\on{GrothSpr}^{\on{ch}}_{P_1} @>\sim>> \sigma_G\circ \on{GrothSpr}^{\on{ch}}_{P_1} \circ \sigma_M \\
@V {i_{P_1,P_2}}VV  @AA{i_{P_2,P_1}}A \\
\on{GrothSpr}^{\on{ch}}_{P_2}  @>\sim>>\sigma_G\circ \on{GrothSpr}^{\on{ch}}_{P_2} \circ \sigma_M.
\endCD
\end{equation}

\sssec{}

We denote by $i^\sigma_{P_1,P_2}$ the isomorphism
\begin{equation} \label{e:ip1p2tau}
\on{GrothSpr}^{\on{ch}}_{P_1}\to \on{GrothSpr}^{\on{ch}}_{P_2},
\end{equation}
obtained from isomorphism $i_{P_1,P_2}$ defined in \secref{e:ip1p2} by interchanging the left and the right factors.
Explicitly, it is defined as the composition
\begin{equation*} 
\on{CH}^{\on{ch}}_{P_1}\circ \iota^{\on{ch}}_{P_1} \overset{\sim}\to
\on{CH}^{\on{ch}}_{P_2}\circ (I^{\on{spec.fin}}_{P_1,P_2}\otimes (I^{\on{spec.fin}}_{P_2,P_1})^{-1}) \circ \iota^{\on{ch}}_{P_1} \overset{\sim}\to \on{CH}^{\on{ch}}_{P_2}\circ \iota^{\on{ch}}_{P_2},
\end{equation*}
where the two arrows are deduced from isomorphisms
\begin{equation} \label{e:intertwiner as compos 2 ch tau}
(I^{\on{spec.fin}}_{P_1,P_2}\otimes\on{Id}) \circ \iota^{\on{ch}}_{P_1} \overset{\sim}\to
(\on{Id}\otimes I^{\on{spec.fin}}_{P_2,P_1}) \circ \iota^{\on{ch}}_{P_2}
\end{equation}
and
\begin{equation} \label{e:intertwiner as compos 1 ch tau}
\on{CH}^{\on{ch}}_{P_1} \circ (\on{Id}\otimes I^{\on{spec.fin}}_{P_2,P_1})\simeq
\on{CH}^{\on{ch}}_{P_2}\circ (I^{\on{spec.fin}}_{P_1,P_2}\otimes \on{Id}),
\end{equation}
obtained from the corresponding isomorphisms \eqref{e:intertwiner as compos 2 ch} and \eqref{e:intertwiner as compos 1 ch} by interchanging
the roles of $P_1$ and $P_2$.

\sssec{}

Note that diagram (\ref{e:diagram to prove}) decomposes as

\begin{equation} \label{e:involute to prove 2}
\CD
\on{GrothSpr}^{\on{ch}}_{P_1} @= \on{GrothSpr}^{\on{ch}}_{P_1} @>\sim>> \sigma_G\circ \on{GrothSpr}^{\on{ch}}_{P_1} \circ \sigma_M \\
@V {i_{P_1,P_2}}VV @AA{i^\sigma_{P_2,P_1}}A @AA{i_{P_2,P_1}}A \\
\on{GrothSpr}^{\on{ch}}_{P_2} @= \on{GrothSpr}^{\on{ch}}_{P_2}  @>\sim>>\sigma_G\circ \on{GrothSpr}^{\on{ch}}_{P_2} \circ \sigma_M,
\endCD
\end{equation}
so it remains to show that its both inner squares are commutative.

\medskip

The commutation of the right inner square follows from the fact that the involutions interchange the left and the right factors.
So it remains to show the commutativity of the left inner square.

\sssec{}

By definition, the left inner square of diagram (\ref{e:involute to prove 2}) is

\medskip

\begin{equation} \label{e:involute prefinal diag}
\CD
\on{CH}^{\on{ch}}_{P_1}\circ \iota^{\on{ch}}_{P_1} @>(\ref{e:intertwiner as compos 2 ch}) >> \on{CH}^{\on{ch}}_{P_1}\circ (I^{\on{spec.fin}}_{P_2,P_1}\otimes (I^{\on{spec.fin}}_{P_1,P_2})^{-1}) \circ \iota^{\on{ch}}_{P_2}\\
@V\sim VV @VV\sim V\\
\on{CH}^{\on{ch}}_{P_1}\circ (I^{\on{spec.fin}}_{P_2,P_1}\otimes \on{Id})\circ ((I^{\on{spec.fin}}_{P_2,P_1})^{-1}\otimes \on{Id}) \circ \iota^{\on{ch}}_{P_1} & & \on{CH}^{\on{ch}}_{P_1}\circ (I^{\on{spec.fin}}_{P_1,P_2}\otimes \on{Id})\circ (\on{Id}\otimes(I^{\on{spec.fin}}_{P_1,P_2})^{-1})\circ \iota^{\on{ch}}_{P_2}\\
@V(\ref{e:intertwiner as compos 1 ch})VV @VV(\ref{e:intertwiner as compos 1 ch})V\\
\on{CH}^{\on{ch}}_{P_2}\circ (\on{Id}\otimes I^{\on{spec.fin}}_{P_1,P_2})\circ ((I^{\on{spec.fin}}_{P_2,P_1})^{-1}\otimes \on{Id}) \circ \iota^{\on{ch}}_{P_1}  & & \on{CH}^{\on{ch}}_{P_2}\circ (\on{Id}\otimes I^{\on{spec.fin}}_{P_1,P_2})\circ (\on{Id}\otimes(I^{\on{spec.fin}}_{P_1,P_2})^{-1}) \circ\iota^{\on{ch}}_{P_2}\\
@V\sim VV @VV\sim V\\
\on{CH}^{\on{ch}}_{P_2}\circ ((I^{\on{spec.fin}}_{P_2,P_1})^{-1}\otimes I^{\on{spec.fin}}_{P_1,P_2}) \circ \iota^{\on{ch}}_{P_1} @>(\ref{e:intertwiner as compos 2 ch})>>\on{CH}^{\on{ch}}_{P_2}\circ \iota^{\on{ch}}_{P_2}.
\endCD
\end{equation}

\medskip

Let $\phi_{1,2}$ denote the isomorphism
\begin{equation} \label{e:intertwiner inverse}
((I^{\on{spec.fin}}_{P_2,P_1})^{-1}\otimes\on{Id}) \circ \iota^{\on{ch}}_{P_1} \overset{\sim}\to
(\on{Id}\otimes (I^{\on{spec.fin}}_{P_1,P_2})^{-1}) \circ \iota^{\on{ch}}_{P_2}
\end{equation}
equal to the composition
\begin{multline*}
((I^{\on{spec.fin}}_{P_2,P_1})^{-1}\otimes\on{Id})\circ \iota^{\on{ch}}_{P_1}\simeq
(I^{\on{spec.fin}}_{P_2,P_1}\otimes I^{\on{spec.fin}}_{P_1,P_2})^{-1}\circ (\on{Id}\otimes I^{\on{spec.fin}}_{P_1,P_2})\circ \iota^{\on{ch}}_{P_1}
\overset{(\ref{e:intertwiner as compos 2 ch})}{\simeq} \\
\simeq(I^{\on{spec.fin}}_{P_2,P_1}\otimes I^{\on{spec.fin}}_{P_1,P_2})^{-1}\circ (I^{\on{spec.fin}}_{P_2,P_1}\otimes\on{Id})\circ \iota^{\on{ch}}_{P_2}\simeq (\on{Id}\otimes(I^{\on{spec.fin}}_{P_2,P_1})^{-1})\circ \iota^{\on{ch}}_{P_2}.
\end{multline*}

\medskip

The map $\phi_{1,2}$ provides the two missing horizontal arrows in \eqref{e:involute prefinal diag}. Hence, to prove the commutation of
\eqref{e:involute prefinal diag}, it suffices to show that the resulting three inner squares commute.

\sssec{}

The resulting \emph{middle} inner square in \eqref{e:involute prefinal diag} is
$$
\CD
\on{CH}^{\on{ch}}_{P_1}\circ (I^{\on{spec.fin}}_{P_2,P_1}\otimes \on{Id})\circ ((I^{\on{spec.fin}}_{P_2,P_1})^{-1}\otimes \on{Id}) \circ \iota^{\on{ch}}_{P_1}  @>{\phi_{1,2}}>>
\on{CH}^{\on{ch}}_{P_1}\circ (I^{\on{spec.fin}}_{P_1,P_2}\otimes \on{Id})\circ (\on{Id}\otimes(I^{\on{spec.fin}}_{P_1,P_2})^{-1})\circ \iota^{\on{ch}}_{P_2} \\
@V{\text{\eqref{e:intertwiner as compos 1 ch}}}VV @VV{\text{\eqref{e:intertwiner as compos 1 ch}}}V \\
\on{CH}^{\on{ch}}_{P_2}\circ (\on{Id}\otimes I^{\on{spec.fin}}_{P_1,P_2})\circ ((I^{\on{spec.fin}}_{P_2,P_1})^{-1}\otimes \on{Id}) \circ \iota^{\on{ch}}_{P_1}  @>{\phi_{1,2}}>>
\on{CH}^{\on{ch}}_{P_2}\circ (\on{Id}\otimes I^{\on{spec.fin}}_{P_1,P_2})\circ (\on{Id}\otimes(I^{\on{spec.fin}}_{P_1,P_2})^{-1}) \circ\iota^{\on{ch}}_{P_2},
\endCD
$$
and it commutes tautologically.

\medskip

The top and the bottom inner squares are
$$
\CD
\on{CH}^{\on{ch}}_{P_1}\circ \iota^{\on{ch}}_{P_1} @>{\text{\eqref{e:intertwiner as compos 2 ch}}} >>
\on{CH}^{\on{ch}}_{P_1}\circ (I^{\on{spec.fin}}_{P_2,P_1}\otimes (I^{\on{spec.fin}}_{P_1,P_2})^{-1}) \circ \iota^{\on{ch}}_{P_2}\\
@V{\sim}VV @VV{\sim}V\\
\on{CH}^{\on{ch}}_{P_1}\circ (I^{\on{spec.fin}}_{P_2,P_1}\otimes \on{Id})\circ ((I^{\on{spec.fin}}_{P_2,P_1})^{-1}\otimes \on{Id}) \circ \iota^{\on{ch}}_{P_1} @>{\phi_{1,2}}>>
\on{CH}^{\on{ch}}_{P_1}\circ (I^{\on{spec.fin}}_{P_1,P_2}\otimes \on{Id})\circ (\on{Id}\otimes(I^{\on{spec.fin}}_{P_1,P_2})^{-1})\circ \iota^{\on{ch}}_{P_2}
\endCD
$$
and
$$
\CD
\on{CH}^{\on{ch}}_{P_2}\circ (\on{Id}\otimes I^{\on{spec.fin}}_{P_1,P_2})\circ ((I^{\on{spec.fin}}_{P_2,P_1})^{-1}\otimes \on{Id}) \circ \iota^{\on{ch}}_{P_1} @>{\phi_{1,2}}>>
\on{CH}^{\on{ch}}_{P_2}\circ (\on{Id}\otimes I^{\on{spec.fin}}_{P_1,P_2})\circ (\on{Id}\otimes(I^{\on{spec.fin}}_{P_1,P_2})^{-1}) \circ\iota^{\on{ch}}_{P_2}\\
@V{\sim}VV @VV{\sim}V \\
\on{CH}^{\on{ch}}_{P_2}\circ ((I^{\on{spec.fin}}_{P_2,P_1})^{-1}\otimes I^{\on{spec.fin}}_{P_1,P_2}) \circ \iota^{\on{ch}}_{P_1}
@>{\text{\eqref{e:intertwiner as compos 2 ch}}}>> \on{CH}^{\on{ch}}_{P_2}\circ \iota^{\on{ch}}_{P_2},
\endCD,
$$
respectively, and they commute by the definition of the map $\phi_{1,2}$.

\qed[\thmref{t:involut}]

\ssec{A theorem of [CY] and [GIT]}

In this subsection we will use some of the ideas introduced in \secref{ss:intw expl} to reprove a theorem from \cite{CY}
and \cite{GIT}.

\sssec{}

Let 
$$\on{CH}^{\on{ext}}_P: \ul\Shv(G/N(P)\overset{M}\times N(P)\backslash G)\to \ul\Shv(G)$$
be the functor given by pull-push along the diagram
\begin{equation} \label{e:CH ext diag}
\CD
G\overset{P}\times G @>>> G \\
@VVV \\
G/N(P)\overset{M}\times N(P)\backslash G.
\endCD
\end{equation}

By construction, it is the counit of the $(\bind_P,\bjacq_P)$-adjunction.

\sssec{}

Set
$$\ul\Shv(G/N(P)\overset{M}\times N(P)\backslash G)^{\on{spec.fin}}:=(\bemb^{\on{spec.fin}})^R(\ul\Shv(G/N(P)\overset{M}\times N(P)\backslash G)),$$
where the spec.fin condition refers to either the action of $G$ on one or both sides (the result is the same by \lemref{l:mon Hecke}).

\medskip

Set
$$\on{CH}^{\on{ext,spec.fin}}_P:=\on{CH}^{\on{ext}}_P|_{\ul\Shv(G/N(P)\overset{M}\times N(P)\backslash G)^{\on{spec.fin}}}.$$
This is the counit of the $(\bind^{\on{spec.fin}}_P,\bjacq^{\on{spec.fin}}_P)$-adjunction.

\sssec{}

By \thmref{t:2-Cat Serre} and formula \eqref{e:Serre 2}, we have:
\begin{equation} \label{e:CH ext Serre}
\on{CH}^{\on{ext,spec.fin}}_P \simeq ((\on{CH}^{\on{ext,spec.fin}}_P)^R)^R\circ (\on{Id}\otimes (I^{\on{spec.fin}}_{P^-,P}\circ I^{\on{spec.fin}}_{P,P^-}))[-2\dim(N(P))].
\end{equation}

\sssec{}

Let
$$\on{HC}^{\on{ext}}_P:\ul\Shv(G)\to \ul\Shv(G/N(P)\overset{M}\times N(P)\backslash G)$$
be the functor given by pull-push along the diagram
\begin{equation} \label{e:HC ext diag}
\CD
G\overset{P}\times G @>>> G/N(P)\overset{M}\times N(P)\backslash G \\
@VVV \\
G.
\endCD
\end{equation}

Denote:
$$\on{HC}^{\on{ext,spec.fin}}_P:=\on{HC}^{\on{ext}}_P|_{\ul\Shv(G)^{\on{spec.fin}}}, \quad
\ul\Shv(G)^{\on{spec.fin}}\to \ul\Shv(G/N(P)\overset{M}\times N(P)\backslash G)^{\on{spec.fin}}.$$

\medskip

Note now that the functor $\on{HC}^{\on{ext}}_P$ (resp., $\on{HC}^{\on{ext,\sec.fin}}_P$) 
is the right adjoint of $\on{CH}^{\on{ext}}_P$ (resp., $\on{CH}^{\on{ext,\sec.fin}}_P$), 
\emph{up to a cohomological shift by $2\dim(N(P))$}.

\sssec{}

Let
$$\wt{\on{HC}}^{\on{ext,spec.fin}}_P: \ul\Shv(G)^{\on{spec.fin}}\to \ul\Shv(G/N(P)\overset{M}\times N(P)\backslash G)^{\on{spec.fin}}$$
be the \emph{left} adjoint of $\on{CH}^{\on{ext,spec.fin}}_P$. It is given by \emph{*-pull and !-push} along
the diagram \eqref{e:HC ext diag}.

\medskip

Thus, we obtain that \eqref{e:CH ext Serre} is equivalent to
\begin{cor}
There exists a canonical isomorphism
\begin{equation} \label{e:HC ext Serre}
(\on{Id}\otimes I^{\on{spec.fin}}_{P,P^-})\circ \wt{\on{HC}}^{\on{ext,spec.fin}}_P\simeq
(\on{Id}\otimes (I^{\on{spec.fin}}_{P^-,P})^{-1}))\circ \on{HC}^{\on{ext,spec.fin}}_P
\end{equation}
as of functors
$$\ul\Shv(G)^{\on{spec.fin}}\to \ul\Shv(G/N(P)\overset{M}\times N(P^-)\backslash G)^{\on{spec.fin}}.$$
\end{cor}

\sssec{}

Applying $\binv_{\Delta(G)}$ to $\on{HC}^{\on{ext}}_P$, we obtain a functor
$$\on{HC}_P:\ul\Shv(G/\on{Ad}(G))\to \ul\Shv((N(P)\backslash G/N(P))/\on{Ad}(M)),$$
which is given by pull-push along
\begin{equation} \label{e:HCh P diag tr}
\CD
G/\on{Ad}(P) @>>> (N(P)\backslash G/N(P))/\on{Ad}(M) \\
@VVV \\
G/\on{Ad}(G).
\endCD
\end{equation}

\medskip

Let $\wt{\on{HC}}^{\on{ch}}_P$ be the \emph{left} adjoint of $\on{CH}^{\on{ch}}_P$.
It is given by \emph{*-pull and !-push} along
the diagram \eqref{e:HCh P diag tr}.

\sssec{}

Let
$$\ul\Shv^{\on{ch}}((N(P)\backslash G/N(P))/\on{Ad}(M)):=\binv_G(\ul\Shv(G/N(P)\overset{M}\times N(P)\backslash G)^{\on{spec.fin}}).$$

\medskip

Denote
$$\on{HC}^{\on{ch}}_P:=\on{HC}_P|_{\ul\Shv^{\on{ch}}(G/\on{Ad}(G))}, \quad
\ul\Shv^{\on{ch}}(G/\on{Ad}(G))\to \ul\Shv^{\on{ch}}((N(P)\backslash G/N(P))/\on{Ad}(M)).$$

\sssec{}

Applying $\binv_G$ to both sides of \eqref{e:HC ext Serre}, we obtain:

\begin{cor} \label{c:HC is ambi}
There exists a canonical isomorphism
$$(\on{Id}\otimes I^{\on{spec.fin}}_{P,P^-})\circ \wt{\on{HC}}^{\on{ch}}_P \simeq
(\on{Id}\otimes (I^{\on{spec.fin}}_{P^-,P})^{-1})\circ \on{HC}^{\on{ch}}_P$$
as functors
$$\ul\Shv^{\on{ch}}(G)\to \ul\Shv^{\on{ch}}(N(P)\backslash G/N(P)/\on{Ad}(M)).$$
\end{cor}

\sssec{}

We evaluate the isomorphism of \corref{c:HC is ambi} on $\on{pt}$ and obtain an isomorphism
\begin{equation} \label{e:HC ambi}
(\on{Id}\otimes I^{\on{restr}}_{P,P^-})\circ \wt{\on{HC}}^{\on{ch}}_P \simeq
(\on{Id}\otimes (I^{\on{restr}}_{P^-,P})^{-1})\circ \on{HC}^{\on{ch}}_P
\end{equation}
as functors
$$\Shv^{\on{ch}}(G/\on{Ad}(G))\to \Shv^{\on{ch}}(N(P)\backslash G/N(P)/\on{Ad}(M)).$$

\sssec{}

Note now that we have a canonical identification
$$\BD_{(N(P)\backslash G/N(P))/\on{Ad}(M)}\circ \on{HC}^{\on{ch}}_P \circ \BD_{G/\on{Ad}(G)}\simeq \wt{\on{HC}}^{\on{ch}}_P$$
as functors
$$\Shv(G/\on{Ad}(G))\to \Shv(N(P)\backslash G/N(P)/\on{Ad}(M)).$$

In addition, we have
$$\BD_{(N(P)\backslash G/N(P^-))/\on{Ad}(M)}\circ I^{\on{spec.fin}}_{P,P^-} \circ \BD_{(N(P)\backslash G/N(P))/\on{Ad}(M)} \simeq (I^{\on{spec.fin}}_{P^-,P})^{-1}.$$

\sssec{} \label{sss:counit Serre}

Hence, from \eqref{e:HC ambi}, we obtain:

\begin{cor}  \label{c:HC is ambi bis}
The functor
$$(\on{Id}\otimes (I^{\on{restr}}_{P^-,P})^{-1})\circ \on{HC}^{\on{ch}}_P:
\Shv^{\on{ch}}(G/\on{Ad}(G))\to \Shv^{\on{ch}}(N(P)\backslash G/N(P^-)/\on{Ad}(M))$$
commutes with Verdier duality.
\end{cor}

\begin{rem}
The statement of \corref{c:HC is ambi bis} has been known thanks to \cite{CY} and \cite{GIT} (in a different generality),
where it is proved by a different method.
\end{rem}

\ssec{A zoo of (un)natural isomorphisms}

In this subsection we will revisit some of the isomorphisms established earlier
(and will assert that we do not obtain anything new).

\medskip

This subsection can be regarded as ``bonus material" and can be skipped. 

\sssec{}

Let us think of the functor $\fr^{\on{ch}}_P$ as the right adjoint of the functor $\on{GrothSpr}^{\on{ch}}_{P}$.
By adjunction, the isomorphism
$$\on{GrothSpr}^{\on{ch}}_{P_1} \overset{i_{P_1,P_2}}\to \on{GrothSpr}^{\on{ch}}_{P_2}$$
induces an isomorphism
$$\fr^{\on{ch}}_{P_2}\overset{\sim}\to \fr^{\on{ch}}_{P_1},$$
which we denote by $i^\fr_{P_2,P_1}$.

\medskip

In this subsection we will explain a different way to obtain the isomorphism
$$i^\fr_{P_,P_2}:\fr^{\on{ch}}_{P_1} \overset{\sim}\to \fr^{\on{ch}}_{P_2}$$
in the particular case when $P_1=P$ and $P_2=P^-$ are opposite parabolics.

\sssec{} \label{sss:self-dual r expl}

Let $\jmath_P$ denote the functor given by pull-push along the diagram
\begin{equation} \label{e:iota diag trans}
\CD
(N(P)\backslash P/N(P))/\on{Ad}(M) @>>> M/\on{Ad}(M) \\
@VVV \\
(N(P)\backslash G/N(P))/\on{Ad}(M),
\endCD
\end{equation}
which is the transpose of the diagram \eqref{e:iota diag} that defines $\iota_P$.

\medskip

We have:
\begin{equation} \label{e:r j}
\jmath_P\circ \on{HC}_P\simeq \fr_P.
\end{equation}

\sssec{}

Note that we have
$$\jmath_P\circ (\on{Id}\otimes I_{P^-,P})\simeq \jmath_{P,P^-},$$
where $\jmath_{P,P^-}$ is as in \secref{sss:j p}.

\medskip

Hence, we obtain:
$$\jmath^{\on{ch}}_P\simeq \jmath^{\on{ch}}_{P,P^-}\circ (\on{Id}\otimes (I^{\on{spec.fin}}_{P^-,P})^{-1}).$$

Combining with \eqref{e:r j} we obtain:
\begin{equation} \label{e:r via good j}
\fr^{\on{ch}}_P \simeq \jmath^{\on{ch}}_{P,P^-}\circ (\on{Id}\otimes (I^{\on{spec.fin}}_{P^-,P})^{-1})\circ \on{HC}^{\on{ch}}_P,
\end{equation}

\sssec{}

Note that
$$\BD_{M/\on{Ad}(M)}\circ \jmath_{P,P^-}\circ \BD_{(N(P)\backslash G/N(P^-))/\on{Ad}(M)}\simeq
\jmath_{P,P^-},$$
since $\jmath_{P,P^-}$ is restriction along an open embedding.

\medskip

Combining this with \corref{c:HC is ambi bis} and using \eqref{e:r via good j},
we obtain:\footnote{The idea to deduce \corref{c:restr and Verdier again} from \corref{c:HC is ambi bis} was explained to us by R.~Bezrukavnikov.}

\begin{cor} \label{c:restr and Verdier again}
The functor
$$\fr^{\on{ch}}_P:\Shv^{\on{ch}}(G/\on{Ad}(G))\to \Shv^{\on{ch}}(M/\on{Ad}(M))$$
commutes with Verdier duality.
\end{cor}

However, unwinding, we obtain:

\begin{lem}
The datum of commutation with Verdier duality on $\fr^{\on{ch}}_P$ given by \corref{c:restr and Verdier again}
equals one given by \corref{c:restr and Verdier}.
\end{lem}

\sssec{}

We now recall the following corollary of Braden's theorem (see \cite[Theorem 3.1.6]{DrGa3}):

\begin{lem} \label{l:Braden}
There is a canonical isomorphism
$$\wt\fr_P\simeq \fr_{P^-}$$
as functors
$\Shv(G/\on{Ad}(G))\to \Shv(M/\on{Ad}(M))$.
\end{lem}

\sssec{}

Combining \lemref{l:Braden} with the isomorphism of \corref{c:restr and Verdier again} we obtain an isomorphism:
\begin{equation} \label{e:funct eqn again}
\fr_P^{\on{ch}}\simeq \fr_{P^-}^{\on{ch}}.
\end{equation}

\sssec{}

However, unwinding, we obtain:

\begin{lem}
The isomorphism \eqref{e:funct eqn again} equals the isomorphism $i^\fr_{P,P^-}$.
\end{lem}

\section{Applications to the Deligne-Lusztig theory} \label{s:DL}

In this section we will apply the theory developed in the previous sections and
derive (both old and new) results about the Deligne-Lusztig functor.

\ssec{Trace of Frobenius on restricted representations}

In this subsection, we show that (unlike the case of the identity functor), if we take the trace
of Frobenius on $G\mmod^{\on{spec.fin}}$ (resp., $G\mmod^{\on{restr}}$), we recover all of
$\ul\Rep(G(\BF_q))$ (resp., $\Rep(G(\BF_q))$).

\sssec{}

Consider again the 1-morphism
$$\bemb^{\on{spec.fin}}:G\mmod^{\on{spec.fin}}\to G\mmod.$$

It is adjointable and compatible with the Frobenius automorphism on the two sides.
Hence, it induces a 1-morphism
\begin{equation} \label{e:Tr Frob of incl}
\Tr(\Frob,G\mmod^{\on{spec.fin}})\to \Tr(\Frob,G\mmod)\simeq \ul\Rep(G(\BF_q)).
\end{equation}
in $\AGCat$.

\sssec{}

We claim:

\begin{thm} \label{t:Frob restr}
The 1-morphism \eqref{e:Tr Frob of incl} is an equivalence.
\end{thm}

\begin{rem}

\thmref{t:Frob restr} is due to A.~Eteve. The proof that we give is a reformulation
of the argument in \cite[Theorem 1.3.1]{Ete}.

\end{rem}

\sssec{}

As an immediate corollary of Theorems \ref{t:Frob restr} and \ref{t:restr main}(b), we obtain:

\begin{cor} \label{c:Frob restr}
There exists a canonical equivalence
$$\Tr(\Frob,G\mmod^{\on{restr}})\simeq \Rep(G(\BF_q)),$$
whose image under $\bi$ produces the equivalence of \eqref{e:Tr Frob of incl}.
\end{cor}

The rest of this subsection is devoted to the proof of \thmref{t:Frob restr}.

\sssec{}

The fact that \eqref{e:Tr Frob of incl} is fully faithful follows immediately from \corref{c:Tr ff}.
Hence, it remains to identify its essential image.

\sssec{}

We have to show that the essential image of
\begin{multline} \label{e:HH Frob H}
\ul\CH^{\on{mon}}\underset{\ul\CH^{\on{mon}}\otimes \ul\CH^{\on{mon}},\Frob}\otimes \ul\CH^{\on{mon}}
\simeq \Tr(\Frob,\ul\CH^{\on{mon}}\mmod(\AGCat)) \to \\
\to \Tr(\Frob,\ul\CH\mmod(\AGCat))\simeq \Tr(\Frob,G\mmod)\simeq \ul\Shv(G/\on{Ad}_{\Frob}(G))\simeq \ul\Rep(G(\BF_q))
\end{multline}
generates the target. It is sufficient to show that the precomposition of \eqref{e:HH Frob H} with
$$\ul\Shv(N\backslash G/N)^{\on{mon}}\simeq \ul\CH^{\on{mon}}\to \ul\CH^{\on{mon}}\underset{\ul\CH^{\on{mon}}\otimes \ul\CH^{\on{mon}},\Frob}\otimes \ul\CH^{\on{mon}}$$
generates the target.

\medskip

We note that the resulting functor
\begin{equation} \label{e:HH Frob H 0}
\ul\Shv(N\backslash G/N)^{\on{mon}}\to 
\ul\CH^{\on{mon}}\underset{\ul\CH^{\on{mon}}\otimes \ul\CH^{\on{mon}},\Frob}\otimes \ul\CH^{\on{mon}} \overset{\text{\eqref{e:HH Frob H}}}\longrightarrow 
\Tr(\Frob,G\mmod)
\end{equation} 
factors as
\begin{multline}  \label{e:HH Frob H 1}
\ul\Shv(N\backslash G/N)^{\on{mon}}\to
\ul\Shv(N\backslash G/N)^{\on{mon}}\underset{\ul\Shv(T)^{\on{mon}}_{\on{co}}\otimes \ul\Shv(T)^{\on{mon}}_{\on{co}},\Frob}\otimes \ul\Shv(T)^{\on{mon}}_{\on{co}}
\simeq \\
\simeq \ul\Shv((N\backslash G/N)/\on{Ad}_{\Frob}(T))^{\on{mon}}\to
\Tr(\Frob,G\mmod),
\end{multline}
while the composition of the first two arrows in \eqref{e:HH Frob H 1} identifies with
$$\ul\Shv(N\backslash G/N)^{\on{mon}} \overset{\on{Av}^{\on{Ad}_{\Frob}(T)}_*}\longrightarrow \ul\Shv((N\backslash G/N)/\on{Ad}_{\Frob}(T))^{\on{mon}}.$$

\medskip

Hence, the functor \eqref{e:HH Frob H 0} factors as 
\begin{equation}  \label{e:HH Frob H 3}
\ul\Shv(N\backslash G/N)^{\on{mon}}
\overset{\on{Av}^{\on{Ad}_{\Frob}(T)}_*}\longrightarrow \ul\Shv((N\backslash G/N)/\on{Ad}_{\Frob}(T))^{\on{mon}} \to \Tr(\Frob,G\mmod).
\end{equation} 

\medskip

Moreover, we have a commutative diagram
$$
\CD
\ul\Shv(N\backslash G/N)^{\on{mon}}  @>>> \ul\Shv(N\backslash G/N) \\
@VVV @VVV \\
\ul\Shv((N\backslash G/N)/\on{Ad}_{\Frob}(T))^{\on{mon}} @>>> \ul\Shv((N\backslash G/N)/\on{Ad}_{\Frob}(T)) \\
& & @VVV \\
& & \ul\CH\underset{\ul\CH\otimes \ul\CH,\Frob}\otimes \ul\CH  \\
& & @VV{\sim}V \\
& &  \Tr(\Frob,\ul\CH\mmod) \\
& & @VV{\sim}V \\
& &  \Tr(\Frob,G\mmod)
\endCD
$$

In the above diagram, the two top right vertical arrows generate their respective essential images.
Hence, to prove our assertion, it suffices to prove the following:

\begin{prop} \label{p:ad Frob T}
The functor
$$\ul\Shv((N\backslash G/N)/\on{Ad}_{\Frob}(T))^{\on{mon}} \to \ul\Shv((N\backslash G/N)/\on{Ad}_{\Frob}(T))$$
is an equivalence.
\end{prop}

\qed[\thmref{t:Frob restr}]

\sssec{Proof of \propref{p:ad Frob T}}

The functor in question is fully faithful, so we only have to show that it is essentially surjective.
To do so, we can work on the strata of the Schubert stratification. I.e., we wish to show that
$$\ul\Shv((N\backslash BwB/N)/\on{Ad}_{\Frob}(T))^{\on{mon}} \to
\ul\Shv((N\backslash BwB/N)/\on{Ad}_{\Frob}(T))$$
is equivalence.

\medskip

The above arrow identifies with
\begin{equation} \label{e:ad Frob T}
\ul\Shv(T/\on{Ad}_{w\cdot \Frob}(T))^{\on{mon}} \to \ul\Shv(T/\on{Ad}_{w\cdot \Frob}(T)),
\end{equation}
so we have to show that \eqref{e:ad Frob T} is an equivalence.

\medskip

However, we claim that the proof of Lang's theorem applies, and we have
$$T/\on{Ad}_{w\cdot \Frob}(T) \simeq \on{pt}/T^w(\BF_q),$$
where $T^w$ is the Cartan subgroup equipped with a $\BF_q$-structure twisted by $w$.

\medskip

This makes the assertion manifest, since the (left) action of $T$ on $\on{pt}/T^w(\BF_q)$
is given by the homomorphism
$$T\to \on{pt}/T^w(\BF_q),$$
corresponding to the $T^w(\BF_q)$-cover
$$T\overset{t\mapsto w(t)\cdot \Frob(t^{-1})}\longrightarrow T.$$

\qed[\propref{p:ad Frob T}]

\ssec{Comparing classes}

This subsection is not essential for the main narrative and can be skipped on the first pass.

\sssec{}

Let us consider the functor
$$\Frob:G\mmod\to G\mmod.$$

It induces a commutative diagram
\begin{equation} \label{e:spec fin Frob 1}
\CD
G\mmod @>{\Frob}>> G\mmod \\
@A{\bemb^{\on{spec.fin}}}AA @AA{\bemb^{\on{spec.fin}}}A \\
G\mmod^{\on{spec.fin}} @>{\Frob}>> G\mmod^{\on{spec.fin}}.
\endCD
\end{equation}

Since the horizontal arrows in the above diagram are equivalences, the next diagram also
commutes
\begin{equation} \label{e:spec fin Frob 2}
\CD
G\mmod @>{\Frob}>> G\mmod \\
@V{(\bemb^{\on{spec.fin}})^R}VV @VV{(\bemb^{\on{spec.fin}})^R}V \\
G\mmod^{\on{spec.fin}} @>{\Frob}>> G\mmod^{\on{spec.fin}}.
\endCD
\end{equation}

\sssec{}

Let $\ul\bC$ be an object of $G\mmod$, equipped with a map
$$\alpha:\ul\bC\to \Frob(\ul\bC).$$

\medskip

Denote
$$\ul\bC^{\on{spec.fin}}:=\bemb^{\on{spec.fin}}\circ (\bemb^{\on{spec.fin}})^R(\ul\bC).$$

By \eqref{e:spec fin Frob 1} and \eqref{e:spec fin Frob 2}, the data of $\alpha$ induces a map
$$\alpha^{\on{spec.fin}}:\ul\bC^{\on{spec.fin}}\to \Frob(\ul\bC^{\on{spec.fin}}),$$
so that the diagram
$$
\CD
\ul\bC @>{\alpha}>> \Frob(\ul\bC) \\
@A{\on{emb}^{\on{spec.fin}}}AA @AA{\Frob(\on{emb}^{\on{spec.fin}})}A \\
\ul\bC^{\on{spec.fin}} @>{\alpha^{\on{spec.fin}}}>> \Frob(\ul\bC^{\on{spec.fin}})
\endCD
$$
commutes.

\sssec{}

Assume that $\ul\bC$ is dualizable. By \secref{sss:class of obj}, to $(\ul\bC,\alpha)$,
we can attach an object
$$\on{cl}(\ul\bC,\alpha)\in \Rep(G(\BF_q)).$$

Note that by \lemref{l:adj is inherited}, $\ul\bC^{\on{spec.fin}}$ is also dualizable, so that we can consider
$$\on{cl}(\ul\bC^{\on{spec.fin}},\alpha^{\on{spec.fin}})\in \Rep(G(\BF_q)).$$

Since the map
$$\on{emb}^{\on{spec.fin}}:\ul\bC^{\on{spec.fin}} \to \ul\bC$$ is adjointable, by \secref{sss:higher funct},
we obtain a map
\begin{equation} \label{e:class spec fin}
\on{cl}(\ul\bC^{\on{spec.fin}},\alpha^{\on{spec.fin}})\to \on{cl}(\ul\bC,\alpha)
\end{equation}
in $\Rep(G(\BF_q))$.

\sssec{}

We claim:

\begin{prop}  \label{p:class spec fin}
The map \eqref{e:class spec fin} is an isomorphism.
\end{prop}

\begin{proof}

We can consider the object
$$\on{cl}\left((\bemb^{\on{spec.fin}})^R(\ul\bC),(\bemb^{\on{spec.fin}})^R(\alpha)\right)\in \Tr(\Frob,G\mmod^{\on{spec.fin}})(\on{pt}),$$
and the map \eqref{e:class spec fin} as obtained using the adjoint pair
\begin{equation} \label{e:emb adj}
\Tr(\on{id},\bemb^{\on{spec.fin}}):\Tr(\Frob,G\mmod^{\on{spec.fin}})\rightleftarrows \Tr(\Frob,G\mmod):\Tr(\on{id},(\bemb^{\on{spec.fin}})^R)
\end{equation}
from the identification
$$\on{cl}(\ul\bC^{\on{spec.fin}},\alpha^{\on{spec.fin}})\simeq \Tr(\on{id},\bemb^{\on{spec.fin}})\left(\on{cl}((\bemb^{\on{spec.fin}})^R(\ul\bC),(\bemb^{\on{spec.fin}})^R(\alpha))\right)$$
and the isomorphism
$$
\on{cl}\left((\bemb^{\on{spec.fin}})^R(\ul\bC),(\bemb^{\on{spec.fin}})^R(\alpha)\right)
\simeq  \Tr(\on{id},(\bemb^{\on{spec.fin}})^R)\left(\on{cl}(\ul\bC,\alpha)\right).$$

The assertion of the proposition follows now from the fact that the functors \eqref{e:emb adj}
are mutually inverse equivalences, by \thmref{t:Frob restr}.

\end{proof}

Combining \propref{p:class spec fin} with \corref{c:fiber of class}, we obtain:

\begin{cor} \label{c:Tr on spec fin}
For $\ul\bC$ as above and $g\in G(\BF_q)$, the map
$$\Tr(\alpha\circ g, \ul\bC^{\on{spec.fin}})\to \Tr(\alpha\circ g, \ul\bC)$$
is an isomorphism (in $\Vect$). 
\end{cor} 

\sssec{}

Let $\CY$ be a quasi-compact AG Verdier-compatible algebraic stack acted on by $G$. From \corref{c:Tr on spec fin}, we obtain:

\begin{cor} \label{c:class spec fin}
The map
$$\Tr_{\AGCat}(\Frob^!,\ul\Shv(\CY)^{\on{spec.fin}})\to \Tr_{\AGCat}(\Frob^!,\ul\Shv(\CY))\overset{\on{LT}^{\on{AG}}}\simeq \sFunct(\CY(\BF_q),\ol\BQ_\ell)$$
is an isomorphism, where by a slight abuse of notation, we denote by $\ul\Shv(\CY)^{\on{spec.fin}}$ the object of $\AGCat$, underlying the
same-named object of $G\mmod^{\on{spec.fin}}$.
\end{cor}

\sssec{An application: the Howlett-Lehrer theorem}

We will now use \corref{c:class spec fin} to reprove the following result:

\begin{thm}
Let $P_1$ and $P_2$ be a pair of rational parabolic subgroups of $G$ that share a Levi subgroup. Then the
intertwining map
$$\sFunct((G/N(P_1))(\BF_q),\ol\BQ_\ell)\to \sFunct((G/N(P_2))(\BF_q),\ol\BQ_\ell),$$
given by pull-push along
$$
\CD
(G/N(P_1)\cap N(P_2))(\BF_q) @>>> (G/N(P_2))(\BF_q) \\
@VVV \\
(G/N(P_1))(\BF_q)
\endCD
$$
is an isomorphism.
\end{thm}

\begin{proof}

This is obtained by combining \corref{c:class spec fin} and \propref{p:intertwiner spec finite}.

\end{proof}

\ssec{Trace of Frobenius on character sheaves} \label{ss:Tr Frob char}

In this subsection we will show that (the usual) categorical trace of Frobenius on
$\Shv^{\on{ch}}(G/\on{Ad}(G))$ recovers the space of class functions on $G(\BF_q)$.

\sssec{}

Consider the functor
$$\on{emb}^{\on{ch}}:\ul\Shv^{\on{ch}}(G/\on{Ad}(G))\to \ul\Shv(G/\on{Ad}(G)).$$

It is adjointable and compatible with the action of the Frobenius endofunctor. Hence,
it induces a map (in $\Vect$):

\begin{equation} \label{e:Tr Frob emb ch}
\Tr(\Frob^!,\ul\Shv^{\on{ch}}(G/\on{Ad}(G))) \to \Tr(\Frob^!,\ul\Shv(G/\on{Ad}(G))) \overset{\on{LT}^{\on{AG}}}\simeq \sFunct((G/\on{Ad}(G))(\BF_q),\ol\BQ_\ell).
\end{equation}

\medskip

In this subsection we will prove:

\begin{thm} \label{t:Tr Frob on ch}
The map \eqref{e:Tr Frob emb ch} is an isomorphism.
\end{thm}

\sssec{}

We will now derive from \thmref{e:Tr Frob emb ch} a statement that does \emph{not} involve $\AGCat$,
and which is the counterpart of the Trace Conjecture of \cite{AGKRRV1} for character sheaves:

\begin{cor} \label{c:Tr Frob on ch}
The composition
\begin{multline} \label{e:Tr Frob emb ch bis}
\Tr_{\DGCat}(\Frob^!,\Shv^{\on{ch}}(G/\on{Ad}(G))) \to \Tr_{\DGCat}(\Frob^!,\Shv(G/\on{Ad}(G))) \overset{\on{LT}}\to \\
\to \sFunct((G/\on{Ad}(G))(\BF_q),\ol\BQ_\ell)
\end{multline}
is an isomorphism.
\end{cor}

\begin{proof}[Proof of \corref{c:Tr Frob on ch}]

By \cite[Theorem 5.4.12]{GRV}, the map
$$\Tr_{\DGCat}(\Frob^!,\Shv(G/\on{Ad}(G))) \overset{\on{LT}}\to  \sFunct((G/\on{Ad}(G))(\BF_q),\ol\BQ_\ell)$$
is the same as
$$\Tr_{\DGCat}(\Frob^!,\Shv(G/\on{Ad}(G))) \to \Tr_{\AGCat}(\Frob^!,\ul\Shv(G/\on{Ad}(G)))
\overset{\on{LT}^{\on{AG}}}\simeq \sFunct((G/\on{Ad}(G))(\BF_q),\ol\BQ_\ell).$$

Hence, it suffices to show that the map
\begin{multline*}
\Tr_{\DGCat}(\Frob^!,\Shv^{\on{ch}}(G/\on{Ad}(G))) \simeq
\Tr_{\AGCat}(\Frob^!,\bi(\Shv^{\on{ch}}(G/\on{Ad}(G)))) \to \\
\to  \Tr_{\AGCat}(\Frob^!,\ul\Shv^{\on{ch}}(G/\on{Ad}(G)))
\to \Tr_{\AGCat}(\Frob^!,\ul\Shv(G/\on{Ad}(G)))
\end{multline*}
is an isomorphism.

\medskip

In the above composition, the second arrow is an isomorphism by \propref{p:ch is restr}, while the last arrow
is an isomorphism by \thmref{t:Tr Frob on ch}.

\end{proof}

\sssec{} \label{sss:ident char}

Note that from \corref{c:iterated trace Frob}, we obtain that the diagram
\begin{equation} \label{e:ident char}
\CD
\Tr(\on{Id},\Tr(\Frob,G\mmod^{\on{spec.fin}}))  @>{\text{\corref{c:2 traces}}}>{\sim}>
\Tr(\Tr(\on{id},\Frob),\Tr(\on{Id},G\mmod^{\on{spec.fin}}))  \\
@V{\text{\thmref{t:Frob restr}}}V{\sim}V @V{\sim}V{\text{\thmref{t:ch sheaves}}}V \\
\Tr(\on{Id},\ul\Rep(G(\BF_q))) & & \Tr(\Frob^!,\ul\Shv^{\on{ch}}(G/\on{Ad}(G))) \\
@V{\text{taking characters}}V{\sim}V @V{\sim}V{\text{\thmref{t:Tr Frob on ch}}}V \\
\sFunct(G(\BF_q)/\on{Ad}(G(\BF_q)),\ol\BQ_\ell) @>{\sim}>>  \sFunct((G/\on{Ad}(G))(\BF_q),\ol\BQ_\ell)
\endCD
\end{equation}
commutes, and hence so does the diagram
$$
\CD
\Tr(\on{Id},\Tr(\Frob,G\mmod^{\on{restr}}))  @>{\text{\corref{c:2 traces}}}>{\sim}>
\Tr(\Tr(\on{id},\Frob),\Tr(\on{Id},G\mmod^{\on{restr}}))  \\
@V{\text{\corref{c:Frob restr}}}V{\sim}V @V{\sim}V{\text{\corref{c:Tr id restr}}}V \\
\Tr(\on{Id},\Rep(G(\BF_q))) & & \Tr(\Frob^!,\Shv^{\on{ch}}(G/\on{Ad}(G))) \\
@V{\text{taking characters}}V{\sim}V @V{\sim}V{\text{\corref{c:Tr Frob on ch}}}V \\
\sFunct(G(\BF_q)/\on{Ad}(G(\BF_q)),\ol\BQ_\ell) @>{\sim}>> \sFunct((G/\on{Ad}(G))(\BF_q),\ol\BQ_\ell).
\endCD
$$

\sssec{}

The rest of this subsection is devoted to the proof of \thmref{t:Tr Frob on ch}. Let us apply
\corref{c:2-dual} to the 1-morphism $t$ being
$$\bemb^{\on{spec.fin}}:G\mmod^{\on{spec.fin}}\to G\mmod$$
with $F_1$ and $F_2$ given by $\Frob$.

\medskip

We obtain a commutative diagram
$$
\CD
\Tr(\on{Id},\Tr(\Frob,G\mmod^{\on{spec.fin}})) @>>> \Tr(\on{Id},\Tr(\Frob,G\mmod)) \\
@V{\text{\corref{c:2 traces}}}V{\sim}V  @V{\sim}V{\text{\corref{c:2 traces}}}V \\
\Tr(\Tr(\on{id},\Frob),\Tr(\on{Id},G\mmod^{\on{spec.fin}})) @>>> \Tr(\Tr(\on{id},\Frob),\Tr(\on{Id},G\mmod)),
\endCD
$$
where:

\medskip

\begin{itemize}

\item The top horizontal arrow is obtained by applying $\Tr(\on{Id},-)$ to the functor
$$\Tr(\Frob,G\mmod^{\on{spec.fin}})\to \Tr(\Frob,G\mmod^{\on{spec.fin}})$$
of \eqref{e:Tr Frob of incl}.

\medskip

\item We identify
$$\Tr(\on{Id},G\mmod^{\on{spec.fin}})\simeq \ul\Shv^{\on{ch}}(G/\on{Ad}(G)) \text{ and }
\Tr(\on{Id},G\mmod)\simeq \ul\Shv(G/\on{Ad}(G)),$$
so that $\Tr(\on{id},\Frob)$ in both cases is given by $\Frob^!$; hence the bottom horizontal arrow
is the map \eqref{e:Tr Frob emb ch}.

\end{itemize}

\medskip

Now, the top horizontal arrow is an equivalence by \thmref{t:Frob restr}. Hence, so is the bottom horizontal arrow.

\qed[\thmref{t:Tr Frob on ch}]

\ssec{Trace of Frobenius and induction}

In this subsection we will deduce the key down-to-earth (i.e., numerical) consequence
of the theory we developed: that the Deligne-Lusztig functor \emph{at the level of characters}
equals the trace of Frobenius on the Grothendieck-Springer functor, see \corref{c:ind and Frob restr}
for the precise assertion.

\sssec{}

Recall the 1-morphism
$$\bind_P:M\mmod\to G\mmod.$$

By \secref{sss:ind is comp}, it induces a 1-morphism
$$\bind^{\on{spec.fin}}_P:M\mmod^{\on{spec.fin}}\to G\mmod^{\on{spec.fin}}$$
in $\tAGCat$
and also a 1-morphism
$$\bind^{\on{restr}}_P:M\mmod^{\on{restr}}\to G\mmod^{\on{restr}}$$
in $\tDGCat$.

\sssec{}

In \secref{sss:constr DL}, we have supplied $\bind_P$ with a datum of compatibility
with the endomorphisms $\Frob_M$ and $\Frob_G$, acting on the source and target,
respectively:
$$\bind_P\circ \Frob_M \overset{\delta}\to  \Frob_G\circ \bind_P$$

\medskip

This datum automatically induces a datum of compatibility for $\bind^{\on{spec.fin}}_P$
and also for $\bind^{\on{restr}}_P$:
$$\bind^{\on{spec.fin}}_P\circ \Frob_M \overset{\delta^{\on{spec.fin}}}\to  \Frob_G\circ \bind^{\on{spec.fin}}_P$$
and
$$\bind^{\on{restr}}_P\circ \Frob_M \overset{\delta^{\on{restr}}}\to  \Frob_G\circ \bind^{\on{restr}}_P.$$

\medskip

Hence, we obtain the functor
\begin{equation} \label{e:DL spec fin}
\Tr(\Frob_M,M\mmod^{\on{spec.fin}})\overset{\Tr(\delta^{\on{spec.fin}},\bind^{\on{spec.fin}}_P)}\longrightarrow \Tr(\Frob_G,G\mmod^{\on{spec.fin}}),
\end{equation}
which makes the diagram
\begin{equation} \label{e:DL spec fin diag}
\CD
\Tr(\Frob_M,M\mmod^{\on{spec.fin}}) @>{\Tr(\delta^{\on{spec.fin}},\bind^{\on{spec.fin}}_P)}>>  \Tr(\Frob_G,G\mmod^{\on{spec.fin}}) \\
@V{\text{\thmref{t:Frob restr}}}V{\sim}V @V{\sim}V{\text{\thmref{t:Frob restr}}}V \\
\Tr(\Frob_M,M\mmod) @>{\Tr(\delta,\bind_P)=\text{\eqref{e:DL functor}}}>>  \Tr(\Frob_G,G\mmod) \\
@VV{\sim}V  @V{\sim}VV \\
\Rep(M(\BF_q)) @>{\on{DL}_P}>>\Rep(G(\BF_q))
\endCD
\end{equation}
commute, and similarly for ``restr" instead of ``spec.fin".

\begin{rem}

Since the 1-morphism $\bind^{\on{spec.fin}}_P$ is 2-adjointable, $\Frob_M$ and $\Frob_G$ are equivalences, and
$\delta^{\on{spec.fin}}$ is also an equivalence, we can apply \propref{p:adj Tr}. We obtain that the 1-morphism
$$\bjacq^{\on{spec.fin}}_P:G\mmod^{\on{spec.fin}}\to M\mmod^{\on{spec.fin}},$$
equipped with the natural transformation $((\delta^{\on{spec.fin}})^{\on{BC}})^{-1}$
$$((\delta^{\on{spec.fin}})^{\on{BC}})^{-1}:\bjacq_P\circ \Frob_G\to \Frob_M\circ \bjacq_P$$
induces a functor
\begin{equation} \label{e:adj to DL}
\Rep(G(\BF_q)) \simeq
\Tr(\Frob_G,G\mmod^{\on{spec.fin}})\to \Tr(\Frob_M,M\mmod^{\on{spec.fin}})\simeq \Rep(M(\BF_q)),
\end{equation}
which is right adjoint to $\on{DL}_P$.

\medskip

However, this does not give us anything new: unwinding, we obtain that the functor \eqref{e:adj to DL}
is given by the bimodule
$$\on{C}^\cdot_c(\mathcal{DL}_P),$$
which is dual to the bimodule
$$\on{C}^\cdot(\mathcal{DL}_P,\omega_{\mathcal{DL}_P}),$$
which defines the Deligne-Lusztig functor $\on{DL}_P$.

\end{rem}

\sssec{}

By \propref{p:intertwiner spec finite}, the 2-morphism $\delta^{\on{spec.fin}}$
is actually an isomorphism, and hence is adjointable. Hence, by \secref{sss:higher funct}
it induces a natural transformation (in fact, an isomorphism)
\begin{equation} \label{e:Frob Groth-Spr}
\on{GrothSpr}_P^{\on{spec.fin}}\circ \Frob_M^! \overset{\Tr(\on{id},\delta^{\on{spec.fin}})}\Rightarrow
\Frob_G^! \circ \on{GrothSpr}_P^{\on{spec.fin}}
\end{equation}
as functors
$$\ul\Shv^{\on{ch}}(M/\on{Ad}(M))\to \ul\Shv^{\on{ch}}(G/\on{Ad}(G)).$$

\medskip

By \eqref{e:2-categ trace}, from \eqref{e:Frob Groth-Spr} we obtain a map
\begin{equation} \label{e:Frob Groth-Spr chat}
\Tr(\Frob_M^!,\ul\Shv^{\on{ch}}(M/\on{Ad}(M))) \overset{\Tr(\Tr(\on{id},\delta^{\on{spec.fin}}),\on{GrothSpr}_P^{\on{spec.fin}})}\longrightarrow
\Tr(\Frob_G^!,\ul\Shv^{\on{ch}}(G/\on{Ad}(G))).
\end{equation}

Similarly, we obtain a natural transformation (in fact, an isomorphism)
\begin{equation} \label{e:Frob Groth-Spr restr}
\on{GrothSpr}_P^{\on{restr}}\circ \Frob_M^! \overset{\Tr(\on{id},\delta^{\on{restr}})}\Rightarrow
\Frob_G^! \circ \on{GrothSpr}_P^{\on{restr}}
\end{equation}
as functors
$$\Shv^{\on{ch}}(M/\on{Ad}(M))\to \Shv^{\on{ch}}(G/\on{Ad}(G))$$
and a map
\begin{equation} \label{e:Frob Groth-Spr chat restr}
\Tr(\Frob_M^!,\Shv^{\on{ch}}(M/\on{Ad}(M))) \overset{\Tr(\Tr(\on{id},\delta^{\on{restr}}),\on{GrothSpr}_P^{\on{restr}})}\longrightarrow
\Tr(\Frob_G^!,\Shv^{\on{ch}}(G/\on{Ad}(G))).
\end{equation}

\begin{rem} \label{r:delta expl}

Unwinding, we obtain that the natural transformation $\delta^{\on{spec.fin}}$ is the composition
$$\bind^{\on{spec.fin}}_P\circ \Frob_M \overset{I^{\on{spec.fin}}_{P,P'}}\simeq
\bind^{\on{spec.fin}}_{P'}\circ \Frob_M \simeq \Frob_G\circ \bind^{\on{spec.fin}}_P.$$

From here, we obtain that the natural transformation $\Tr(\on{id},\delta^{\on{spec.fin}})$ in \eqref{e:Frob Groth-Spr} is the composition
$$\on{GrothSpr}_P^{\on{spec.fin}}\circ \Frob_M^! \overset{i_{P,P'}}\simeq
\on{GrothSpr}_{P'}^{\on{spec.fin}}\circ \Frob_M^! \simeq \Frob_G^! \circ \on{GrothSpr}_P^{\on{spec.fin}}.$$

\end{rem}

\sssec{}

Thus, we can apply \corref{c:2-dual} and, taking into account \eqref{e:ident char} and \eqref{e:DL spec fin diag}, we obtain:

\begin{cor} \label{c:ind and Frob spec fin}
The following diagram commutes
$$
\CD
\Tr(\on{Id},\Rep(M(\BF_q)))  @>{\Tr(\on{id},\on{DL}_P)}>> \Tr(\on{Id},\Rep(G(\BF_q)))  \\
@V{\rm{taking\,\, characters}}V{\sim}V @V{\sim}V{\rm{taking\,\, characters}}V  \\
\sFunct((M/\on{Ad}(M))(\BF_q),\ol\BQ_\ell) & &  \sFunct((G/\on{Ad}(G))(\BF_q),\ol\BQ_\ell) \\
@A{\text{\thmref{t:Tr Frob on ch}}}A{\sim}A @A{\sim}A{\text{\thmref{t:Tr Frob on ch}}}A \\
\Tr(\Frob_M^!,\ul\Shv^{\on{ch}}(M/\on{Ad}(M))) @>{\Tr(\Tr(\on{id},\delta^{\on{spec.fin}}),\on{GrothSpr}_P^{\on{spec.fin}})}>>  \Tr(\Frob_G^!,\ul\Shv^{\on{ch}}(G/\on{Ad}(G))).
\endCD
$$
\end{cor}

\medskip

Hence, we also obtain:
\begin{cor} \label{c:ind and Frob restr}
The following diagram commutes
\begin{equation} \label{e:ind and Frob restr}
\CD
\Tr(\on{Id},\Rep(M(\BF_q)))  @>{\Tr(\on{id},\on{DL}_P)}>> \Tr(\on{Id},\Rep(G(\BF_q)))  \\
@V{\rm{taking\,\, characters}}V{\sim}V @V{\sim}V{\rm{taking\,\, characters}}V  \\
\sFunct((M/\on{Ad}(M))(\BF_q),\ol\BQ_\ell) & &  \sFunct((G/\on{Ad}(G))(\BF_q),\ol\BQ_\ell) \\
@A{\text{\corref{c:Tr Frob on ch}}}A{\sim}A @A{\sim}A{\text{\corref{c:Tr Frob on ch}}}A \\
\Tr(\Frob_M^!,\Shv^{\on{ch}}(M/\on{Ad}(M))) @>{\Tr(\Tr(\on{id},\delta^{\on{restr}}),\on{GrothSpr}_P^{\on{restr}})}>>  \Tr(\Frob_G^!,\Shv^{\on{ch}}(G/\on{Ad}(G))).
\endCD
\end{equation}
\end{cor}

\begin{rem}

As was mentioned in the Introduction, a particular case of \corref{c:ind and Frob restr} for $q\gg 0$ was established 
in \cite{Lu3}, and in {\it loc. cit} it was conjectured to hold without that assumption.

\end{rem}

\ssec{What does \corref{c:ind and Frob restr} actually say?}

\sssec{}

Let
$$\on{ch}(\on{DL}_P):\sFunct((M/\on{Ad}(M))(\BF_q),\ol\BQ_\ell)\to \sFunct((G/\on{Ad}(G))(\BF_q),\ol\BQ_\ell)$$
denote the map arising from the \emph{upper} portion of \eqref{e:ind and Frob restr}.

\medskip

This is the map on $\ol\BQ_\ell\underset{\BZ}\otimes K^0(-)$ induced by the Deligne-Lusztig functor
$$\on{DL}_P:\Rep(M(\BF_q))\to \Rep(G(\BF_q)).$$

\sssec{}

Let us denote by $\Tr(\Frob,\on{GrothSpr}_P)$ the map
\begin{equation} \label{e:Tr Frob Groth Spr}
\sFunct((M/\on{Ad}(M))(\BF_q),\ol\BQ_\ell)\to \sFunct((G/\on{Ad}(G))(\BF_q),\ol\BQ_\ell)
\end{equation}
arising from the \emph{lower} portion of \eqref{e:ind and Frob restr}. (Below,
we will characterize the map $\Tr(\Frob,\on{GrothSpr}_P)$ concrete terms, see \eqref{e:concrete}.)

\medskip

\corref{c:ind and Frob restr} says that
\begin{equation} \label{e:char DL}
\on{ch}(\on{DL}_P)=\Tr(\Frob,\on{GrothSpr}_P)
\end{equation}
as maps
$$\sFunct((M/\on{Ad}(M))(\BF_q),\ol\BQ_\ell)\to \sFunct((G/\on{Ad}(G))(\BF_q),\ol\BQ_\ell).$$

\sssec{}

Let $\CY$ be a \emph{Verdier-compatible}\footnote{See \cite[Sect. 4.6.1]{GRV} for what this means.}
quasi-compact algebraic stack and $\phi$ its endomorphism. Let $\CF'$ be a compact object of $\Shv(\CY)$
equipped with a map
$$\alpha': \CF'\to \phi^!(\CF').$$

\medskip

To this datum, we can attach an element
$$\on{cl}(\CF',\alpha')\in \Tr(\phi^!,\Shv(\CY)).$$

\medskip

For a point $y$
$$\on{pt}\overset{i_y}\to \CY^\phi,$$
we obtain an endomorphism $\alpha'_y$ of $i_y^!(\CF')$ and we can consider
\begin{equation} \label{e:! trace}
\Tr(\alpha'_y,i_y^!(\CF'))\in \ol\BQ_\ell.
\end{equation}

Denote $\CF:=\BD(\CF')$. Then $\alpha'$ induces a map
$$\wt\alpha:\CF\to \phi_\blacktriangle(\CF)\simeq \phi_*(\CF),$$
and hence a map
$$\alpha:\phi^*(\CF) \to \CF.$$

To the above data, we can attach an element
$$\on{cl}(\CF,\wt\alpha)\in \Tr(\phi_\blacktriangle,\Shv(\CY)).$$

\medskip

For $y$ as above, the map $\alpha$ induces an endomorphism $\alpha_y$ of $i_y^*(\CF)$, and we can consider
\begin{equation} \label{e:* trace}
\Tr(\alpha_y,i_y^*(\CF))\in \ol\BQ_\ell.
\end{equation}

It is easy to see, however, that \eqref{e:! trace} equals \eqref{e:* trace}
and
\begin{equation} \label{e:weird Frob}
\on{cl}(\CF',\alpha')=\on{cl}(\CF,\wt\alpha)
\end{equation} 
as elements of
$$\Tr(\phi^!,\Shv(\CY))\simeq \Tr((\phi^!)^\vee,\Shv(\CY)^\vee)\simeq \Tr(\phi_\blacktriangle,\Shv(\CY)),$$
see \secref{sss:trace and duality}. 

\sssec{} \label{sss:recall GV}

We will apply this to $\phi$ being the geometric Frobenius of $\CY$. Denote the element
of $\sFunct(\CY(\BF_q),\ol\BQ_\ell)$ given by \eqref{e:! trace} (equialently, \eqref{e:* trace}) by
$$\sfunct(\CF',\alpha')=\sfunct(\CF,\alpha).$$

According to \cite[Theorem 0.8]{GaVa}, the above function equals to the image of
$$\on{cl}(\CF',\alpha')=\on{cl}(\CF,\wt\alpha)$$
under the canonical map
$$\Tr_{\DGCat}(\Frob_\blacktriangle,\Shv(\CY))\to \Tr_{\AGCat}(\Frob_\blacktriangle,\ul\Shv(\CY))\overset{\on{LT}^{\on{AG}}}\simeq
\sFunct(\CY(\BF_q),\ol\BQ_\ell).$$

\sssec{}

Let $\CF'_M$ be a compact object of $\Shv^{\on{ch}}(M/\on{Ad}(M))$, equipped with a map
$$\alpha'_M:\CF'_M\to \Frob_M^!(\CF'_M).$$

Set
$$\CF'_G:=\on{GrothSpr}_P(\CF'_M).$$

By \eqref{e:Frob Groth-Spr}, the data of $\alpha'_M$ induces a data
$$\alpha'_G:\CF'_G\to \Frob_G^!(\CF'_G).$$

\medskip

By \secref{sss:recall GV} we have
\begin{equation} \label{e:concrete}
\Tr(\Frob,\on{GrothSpr}_P)(\sfunct(\CF'_M,\alpha'_M))=\sfunct(\CF'_G,\alpha'_G),
\end{equation}
where $\Tr(\Frob,\on{GrothSpr}_P)$ is the map \eqref{e:Tr Frob Groth Spr}.

\medskip

This is the promised concrete characterization of the map $\Tr(\Frob,\on{GrothSpr}_P)$.

\sssec{}

Thus, from \eqref{e:char DL} we obtain:

\begin{cor} \label{c:char DL}
Let $\rho_M$ be a compact object of $\Rep(M(\BF_q))$ and let $(\CF'_M,\alpha'_M)$ be as above.
Assume that
\begin{equation} \label{e:eq char M}
\on{ch}(\rho_M)=\sfunct(\CF'_M,\alpha'_M)
\end{equation}
as elements of $\sFunct((M/\on{Ad}(M))(\BF_q),\ol\BQ_\ell)$. Then for
$$\rho_G:=\on{DL}_P(\rho_M) \text{ and } (\CF'_G,\alpha'_G) \text{ as in above},$$
we have
$$\on{ch}(\rho_G)=\sfunct(\CF'_G,\alpha'_G)$$
as elements of $\sFunct((G/\on{Ad}(G))(\BF_q),\ol\BQ_\ell)$.
\end{cor}

\ssec{An application: characters of Deligne-Lusztig representations}

In this subsection we will re-derive the (well-known) result that says that characters of the
(usual) Deligne-Lusztig representations are given by pointwise traces of Frobenius on the
Springer sheaf.

\sssec{}

Let $T$ be a rational Cartan subgroup in $G$, and let $\chi$ be a 1-dimensional character sheaf on it,
equipped with a Weil structure
$$\alpha_T:\Frob_T^*(\chi)\simeq\chi.$$

\medskip

Let
$$\chi^0:=\sfunct(\chi,\alpha_T)$$
be the corresponding character on $T(\BF_q)$, cf. \secref{sss:DL ab}, and let
$\ol\BQ_\ell^{\chi^0}$ denote the corresponding 1-dimensional representation of $T(\BF_q)$.

\medskip

Choose a (not necessarily rational) Borel $B$ containing $T$. Consider
$$\on{DL}_B(\ol\BQ_\ell^{\chi^0})\in \Rep(G(\BF_q)),$$
and consider
$$\on{ch}(\on{DL}_B(\ol\BQ_\ell^{\chi^0}))\in \sFunct((G/\on{Ad}(G))(\BF_q),\ol\BQ_\ell).$$

\sssec{}

Consider now the Grothendieck-Springer sheaf
$$\on{GrothSpr}_B(\chi)\in \Shv(G/\on{Ad}(G)).$$

We equip it with the Weil structure
$$\alpha_G:\Frob^*_G(\on{GrothSpr}_B(\chi))\simeq \on{GrothSpr}_B(\chi)$$
as follows.

\sssec{} \label{sss:Weil on Gr Spr}

Since $\on{GrothSpr}_B(\chi)$ (and $\Frob^*_G(\on{GrothSpr}_B(\chi))$) are Goresky-MacPherson extensions
of their respective restrictions to $G^{\on{rs}}/\on{Ad}(G)$, it suffices to specify the data of $\alpha_G$ over
this locus.

\medskip

Denote
$$B':=\Frob_G^{-1}(B).$$

\medskip

A priori, we have
$$\Frob^*_G(\on{GrothSpr}_B(\chi))\simeq \on{GrothSpr}_{B'}(\Frob^*_T(\chi))\overset{\alpha_T}\simeq \on{GrothSpr}_{B'}(\chi).$$

We now identify
$$\on{GrothSpr}_{B'}(\chi)\simeq \on{GrothSpr}_{B}(\chi)$$
using \lemref{l:Springer w}.

\sssec{}

We will now reprove the following result (originally due to Lusztig, \cite{Lus}):

\begin{thm} \label{t:Srpinger chars}
The functions
$$\on{ch}(\on{DL}_B(\ol\BQ_\ell^{\chi^0})) \text{ and } \sfunct(\on{GrothSpr}_B(\chi),\alpha_G)$$
are equal.
\end{thm}

\begin{proof}

Denote $\chi'=\BD(\chi)$, so that
$$\on{GrothSpr}_B(\chi')\simeq \BD(\on{GrothSpr}_B(\chi)).$$

We will apply \corref{c:char DL} to
$$M=T,\,\,P=B,\,\, \CF'_M:=\chi',\,\, \rho_M:=\ol\BQ_\ell^{\chi^0}.$$

Note that \eqref{e:eq char M} holds by construction. Hence, the equality
$$\on{ch}(\on{DL}_B(\ol\BQ_\ell^{\chi^0})) = \sfunct(\on{GrothSpr}_B(\chi'),\alpha'_G)=\sfunct(\on{GrothSpr}_B(\chi),\alpha_G)$$
holds with the following caveat:

\medskip

The datum of
$$\alpha'_G:\on{GrothSpr}_B(\chi') \simeq \Frob_G^!(\on{GrothSpr}_B(\chi'))$$
is constructed by \eqref{e:Frob Groth-Spr} rather than by the procedure in \secref{sss:Weil on Gr Spr}.

\medskip

However, the two agree by \thmref{t:Springer action}.

\end{proof}

\ssec{Independence of Deligne-Lusztig characters of the choice of a parabolic}

\sssec{}

In this subsection we will prove the following result:

\begin{thm} \label{t:indep}
Let $M$ be a Levi subgroup of $G$ defined over $\BF_q$. Then the map
$$\on{ch}(\on{DL}_P):\sFunct((M/\on{Ad}(M))(\BF_q),\ol\BQ_\ell)\to \sFunct((G/\on{Ad}(G))(\BF_q),\ol\BQ_\ell)$$
does \emph{not} depend on the choice of the parabolic $P$.
\end{thm}

The rest of this subsection is devoted to the proof of this theorem.

\sssec{}

By \corref{c:ind and Frob restr}, it suffices to show that the map $\Tr(\Tr(\on{id},\delta^{\on{restr}}),\on{GrothSpr}_P^{\on{restr}})$
does not depend on the choice of the parabolic.

\medskip

Let $P_1$ and $P_2$ be two such parabolics. Consider the intertwining operator
$$i_{P_1,P_2}:\on{GrothSpr}_{P_1}^{\on{restr}}\simeq \on{GrothSpr}_{P_2}^{\on{restr}}.$$

It suffices to show that the diagram
\begin{equation} \label{e:intertw and Frob}
\CD
\on{GrothSpr}_{P_1}^{\on{restr}}\circ \on{Frob}_M^!  @>{i_{P_1,P_2}}>> \on{GrothSpr}_{P_2}^{\on{restr}}\circ \on{Frob}_M^!  \\
@V{\Tr(\on{id},\delta^{\on{restr}}_{P_1})}VV @VV{\Tr(\on{id},\delta^{\on{restr}}_{P_2})}V \\
\Frob_G^! \circ \on{GrothSpr}_{P_1}^{\on{restr}} @>{i_{P_1,P_2}}>> \Frob_G^! \circ \on{GrothSpr}_{P_2}^{\on{restr}}
\endCD
\end{equation}
as functors
$$\Shv^{\on{ch}}(M/\on{Ad}(M))\to \Shv^{\on{ch}}(G/\on{Ad}(G))$$
commutes.

\sssec{}

Using Remark \ref{r:delta expl}, we can rewrite the above diagram as
\begin{equation} \label{e:intertw and Frob 1}
\CD
\on{GrothSpr}_{P_1}^{\on{restr}}\circ \Frob_M^!  @>{i_{P_1,P_2}}>> \on{GrothSpr}_{P_2}^{\on{restr}}\circ \Frob_M^!   \\
@V{i_{P_1,P'_1}}VV @VV{i_{P_2,P'_2}}V \\
\on{GrothSpr}_{P'_1}^{\on{restr}}\circ \Frob_M^!  & & \on{GrothSpr}_{P'_2}^{\on{restr}}\circ \Frob_M^!  \\
@V{\sim}VV @VV{\sim}V \\
\Frob_G^! \circ \on{GrothSpr}_{P_1}^{\on{restr}} @>{i_{P_1,P_2}}>> \Frob_G^! \circ \on{GrothSpr}_{P_2}^{\on{restr}},
\endCD
\end{equation}
where $P'_i=\Frob_G^{-1}(P_i)$.

\medskip

Note that the diagram
$$
\CD
\on{GrothSpr}_{P'_1}^{\on{restr}}\circ \Frob_M^!  @>{i_{P'_1,P'_2}}>> \on{GrothSpr}_{P'_2}^{\on{restr}}\circ \Frob_M^!  \\
@V{\sim}VV @VV{\sim}V \\
\Frob_G^! \circ \on{GrothSpr}_{P_1}^{\on{restr}} @>{i_{P_1,P_2}}>> \Frob_G^! \circ \on{GrothSpr}_{P_2}^{\on{restr}}
\endCD
$$
commutes tautologically.

\medskip

Hence, in order to establish the commutativity of \eqref{e:intertw and Frob 1}, it suffices to show that the diagram
$$
\CD
\on{GrothSpr}_{P_1}^{\on{restr}}\circ \Frob_M^!  @>{i_{P_1,P_2}}>> \on{GrothSpr}_{P_2}^{\on{restr}}\circ \Frob_M^!   \\
@V{i_{P_1,P'_1}}VV @VV{i_{P_2,P'_2}}V \\
\on{GrothSpr}_{P'_1}^{\on{restr}}\circ \Frob_M^!  @>{i_{P'_1,P'_2}}>> \on{GrothSpr}_{P'_2}^{\on{restr}}\circ \Frob_M^!
\endCD
$$
commutes.

\medskip

The latter is equivalent to the commutativity of
\begin{equation} \label{e:intertw and Frob 2}
\CD
\on{GrothSpr}_{P_1}^{\on{restr}} @>{i_{P_1,P_2}}>> \on{GrothSpr}_{P_2}^{\on{restr}}  \\
@V{i_{P_1,P'_1}}VV @VV{i_{P_2,P'_2}}V \\
\on{GrothSpr}_{P'_1}^{\on{restr}}  @>{i_{P'_1,P'_2}}>> \on{GrothSpr}_{P'_2}^{\on{restr}}.
\endCD
\end{equation}

\sssec{}

However, the required assertion follows from \thmref{t:intertwiner trans}: indeed, both circuits in \eqref{e:intertw and Frob 2}
identify with $i_{P_1,P'_2}$.

\qed[\thmref{t:indep}]

\appendix

\section{Sheaves on connected groups} \label{s:shv on gr}

In this section we let $G$ be connected. We will supply proofs for some technical results from \secref{s:mon}.

\ssec{The categories \texorpdfstring{$\Shv(G)^{\on{q-const}}$}{q} and \texorpdfstring{$\Shv(G)^{\on{alm-const}}$}{alm}}

In this subsection we will provide an explicit description of the categories $\Shv(G)^{\on{q-const}}$ and $\Shv(G)^{\on{alm-const}}$.

\sssec{} \label{sss:replace by red}

Note that the categories $\Shv(G)^{\on{q-const}}$ and $\Shv(G)^{\on{alm-const}}$ do not change when we replace
$G$ by its maximal reductive quotient.

\medskip

Hence, for the rest of this subsection we will assume that $G$ is reductive.

\sssec{} \label{sss:replace by isogeny}

Let $p:\wt{G}\to G$ be an isogeny. It is again easy to see that the functor
$$p^!:\Shv(G)\to \Shv(\wt{G})$$
induces equivalences
$$\Shv(G)^{\on{q-const}}\simeq \Shv(\wt{G})^{\on{q-const}} \text{ and } \Shv(G)^{\on{alm-const}}\simeq \Shv(\wt{G})^{\on{alm-const}}.$$

Hence, we can assume that $G$ is a product of a torus and a semi-simple group.

\sssec{}

We claim:

\begin{lem} \label{l:prod by torus}
For a torus $T$ and any (smooth, connected) scheme $X$, the functor
\begin{equation} \label{e:prod by torus}
\Shv(T)^{\on{alm-const}}\otimes \Shv(X)^{\on{q-const}}\overset{\text{\propref{p:q mon torus}}}\simeq 
\Shv(T)^{\on{q-const}}\otimes \Shv(X)^{\on{q-const}}\to \Shv(T\times X)^{\on{q-const}}
\end{equation} 
is an equivalence.
\end{lem}

\begin{proof}

First we claim that \eqref{e:prod by torus} is fully faithful. This follows from the commutative diagram
$$
\CD
\Shv(T)^{\on{alm-const}} \otimes \Shv(X)^{\on{q-const}} @>{\text{\eqref{e:prod by torus}}}>>  \Shv(T\times X)^{\on{q-const}} \\
@VVV  @VVV \\
\Shv(T)^{\on{alm-const}} \otimes \Shv(X) @>>> \Shv(T\times X),
\endCD
$$
in which:

\begin{itemize}

\item The right vertical arrow is fully faithful tautologically;

\item The left vertical arrow is fully faithful because $\Shv(T)^{\on{alm-const}}$ dualizable;

\item The bottom horizontal arrow is fully faithful, since it can be decomposed as
$$\Shv(T)^{\on{alm-const}} \otimes \Shv(X) \to \Shv(T) \otimes \Shv(X) \to \Shv(T\times X)$$
with both arrows fully faithful (the first one since $\Shv(X)$ is dualizable).

\end{itemize} 

\medskip

Thus, it remains to show that the essential image of \eqref{e:prod by torus} generates the target. It suffices to show that the
pullback functor 
$$\pi^*: \Shv(X)^{\on{q-const}} \to  \Shv(T\times X)^{\on{q-const}}$$
generates the target.

\medskip

It is easy to see that the functor
$$\pi_*:\Shv(T\times X)\to \Shv(X)$$ sends 
$\Shv(T\times X)^{\on{q-const}}\to  \Shv(X)^{\on{q-const}}$.

\medskip

Hence, it suffices to show that $\pi_*$ is conservative on $\Shv(T\times X)^{\on{q-const}}$. Taking the !-fiber
at some/any point $x\in X$, we reduce the assertion to one about $T$ itself, in which case it follows from 
\propref{p:q mon torus}.

\end{proof}

\sssec{} \label{sss:replace by semisimple}

Thus, using \lemref{l:prod by torus}, and given the description of
$$\Shv(T)^{\on{alm-const}}\simeq \Shv(T)^{\on{q-const}}$$
in \secref{sss:alm const torus expl}, we can focus on the case when $G$ is semi-simple.

\sssec{}

First, as in \secref{sss:alm const torus expl}, we have:
$$\Shv(G)^{\on{alm-const}}\simeq \on{C}^\cdot(G)\mod,$$
where the functor $\leftarrow$ sends
$$\on{C}^\cdot(G)\in \on{C}^\cdot(G)\mod$$
to
$$\omega_G\in \Shv(G)^{\on{alm-const}}.$$

\medskip

Since $G$ is semi-simple, we have
$$\on{C}^\cdot(G) \simeq \Sym(W),$$
where $W$ is a finite-dimensional cohomologivally graded vector space concentrated in odd degrees $\geq 3$.

\medskip

The fiber of !-fiber at $1\in G$ identifies with the functor
\begin{equation} \label{e:fiber G}
(-)\underset{\Sym(W)}\otimes \ol\BQ_\ell, \quad \Sym(W)\mod\to \Vect.
\end{equation}

\sssec{} \label{sss:Koszul}

Koszul duality identifies
$$\Sym(W)\mod \simeq \Sym(W^\vee[-1])\mod_{\{0\}},$$
where the subscript $\{0\}$ corresponds to the condition that the generators of $W^\vee[-1]$ act locally nilpotently.

\medskip

In terms of this equivalence, the functor \eqref{e:fiber G} corresponds to the functor
$$\Sym(W^\vee[-1])\mod_{\{0\}} \to \Sym(W^\vee[-1])\mod \to  \Vect,$$
where the second arrow is the natural forgetful functor.

\medskip

In particular, the t-structure on $\Shv(G)^{\on{alm-const}}$ corresponds to the t-structure on the category
$\Sym(W^\vee[-1])\mod_{\{0\}}$, induced by the natural t-structure on $\Sym(W^\vee[-1])\mod$.

\begin{rem}

We should think of the vector space $W^\vee[-1]$ as the abelian Lie algebra that governs the rational
homotopy type of $G$.

\end{rem}

\sssec{}   \label{sss:descr q const ss}

Note now that as in \secref{sss:q-const T}, the category $\Shv(G)^{\on{q-const}}$ identifies with the
left-completion of $\Shv(G)^{\on{alm-const}}$ with respect to its t-structure.

\medskip

From here, we obtain:
$$\Shv(G)^{\on{q-const}}\simeq \Sym(W^\vee[-1])\mod.$$

\sssec{} \label{sss:q vs ind}

Note that the above description of $\Shv(G)^{\on{q-const}}$ implies that it is compactly generated.
In particular, this implies \lemref{l:q-cnst comp gen}.

\medskip

Note, however,
that the natural embedding
$$\Shv(G)^{\on{q-const}}\to \Shv(G)$$
does \emph{not} preserve compactness. Indeed, it sends the compact generator
$$\Sym(W^\vee[-1])\in \Sym(W^\vee[-1])\mod$$
to an unbounded complex, denoted
$$\on{q-const}_G\in \Shv(G)^{\on{q-const}},$$
whose fiber at $1\in G$ is $\Sym(W^\vee[-1])$.

\medskip

It is also
makes it manifest that the embedding
$$\Shv(G)^{\on{alm-const}}\subset \Shv(G)^{\on{q-const}}$$
is \emph{not} an equivalence.

\ssec{The left adjoint for semi-simple groups}

In this subsection we let $G$ be semi-simple.

\sssec{}

We will prove the following assertion:

\begin{prop} \label{p:left adj ss}
The embedding
$$\on{emb}^{\on{q-const}}:\Shv(G)^{\on{q-const}}\to \Shv(G)$$
admits a left adjoint.
\end{prop}

The rest of the subsection is devoted to the proof of the proposition.

\sssec{}

Note that the Koszul duality in \secref{sss:Koszul} extends to an \emph{equivalence}
$$\Sym(W)\mod^{\on{ren}}\to \Sym(W^\vee[-1])\mod, \quad M \mapsto \CHom_{\Sym(W)}(\ol\BQ_\ell,M),$$
where $\Sym(W)\mod^{\on{ren}}$ is a renormalized version of the category $\Sym(W)\mod$, in which we declare
finite-dimensional modules as compacts.

\medskip

We will show that in terms of this equivalence, the left adjoint of $\on{emb}^{\on{q-const}}$ is given by
$$\CF\mapsto \on{C}^\cdot_c(G,\CF),$$
where $\on{C}^\cdot_c(G,\CF)$ is viewed naturally as a module over $\on{C}^\cdot(G)$.

\sssec{}

To prove this, we need to establish an equivalence
\begin{equation} \label{e:cohomology with compact supports}
\CHom_{\Sym(W)\mod}(\on{C}^\cdot_c(G,\CF),\ol\BQ_\ell) \simeq \CHom_{\Shv(G)}(\CF,\on{q-const}_G).
\end{equation}
as $\Sym(W^\vee[-1])$-modules, where the $\Sym(W^\vee[-1])$-action on the left-hand side is via
$$\Sym(W^\vee[-1])\simeq \CHom_{\Sym(W)\mod}(\ol\BQ_\ell,\ol\BQ_\ell).$$

\sssec{}

Let us describe the object $\on{q-const}_G\in \Shv(G)^{\on{q-const}}$ explicitly. Write
$$\ol\BQ_\ell\simeq \underset{i}{\on{colim}}\, \sP_i,$$
where $\sP_i$ are perfect modules over $\on{C}^\cdot(G)$, and such that for every $n$, the system
$$i\mapsto \tau^{\leq n}(\sP_i)$$
stabilizes.

\medskip

Then
$$\on{q-const}_G \simeq \underset{i}{\on{lim}}\, \sP_i^\vee\underset{\on{C}^\cdot(G)}\otimes \omega_G,$$
where we note that for every $n$ the system
$$i\mapsto \tau^{\geq -n}\left(\sP_i^\vee\underset{\on{C}^\cdot(G)}\otimes \omega_G\right)$$
stabilizes.

\sssec{}

Thus, we rewrite the right-hand in \eqref{e:cohomology with compact supports} as
\begin{multline*}
\underset{i}{\on{lim}}\, \CHom_{\Shv(G)}(\CF,\sP_i^\vee\underset{\on{C}^\cdot(G)}\otimes \omega_G)\simeq
\underset{i}{\on{lim}}\, \sP_i^\vee\underset{\on{C}^\cdot(G)}\otimes \CHom_{\Shv(G)}(\CF,\omega_G) \simeq \\
\simeq \underset{i}{\on{lim}}\, \sP_i^\vee\underset{\on{C}^\cdot(G)}\otimes \CHom_{\Vect}(\on{C}^\cdot_c(G,\CF),\ol\BQ_\ell)\simeq
\underset{i}{\on{lim}}\,  \CHom_{\on{C}^\cdot(G)\mod}(\sP_i,\CHom_{\Vect}(\on{C}^\cdot_c(G,\CF),\ol\BQ_\ell))\simeq \\
\simeq \CHom_{\on{C}^\cdot(G)\mod}(\ol\BQ_\ell,\CHom_{\Vect}(\on{C}^\cdot_c(G,\CF),\ol\BQ_\ell))\simeq
\CHom_{\Vect}(\ol\BQ_\ell\underset{\on{C}^\cdot(G)}\otimes \on{C}^\cdot_c(G,\CF),\ol\BQ_\ell)\simeq \\
\simeq \CHom_{\on{C}^\cdot(G)}(\on{C}^\cdot_c(G,\CF),\ol\BQ_\ell),
\end{multline*}
which is the left-hand in \eqref{e:cohomology with compact supports}.

\qed[\propref{p:left adj ss}]

\ssec{Proof of \lemref{l:alm const}} \label{ss:alm const}

\sssec{}

As in \secref{sss:replace by red}, we can assume that $G_1$ and $G_2$ are reductive. Furthermore, as in
Sects. \ref{sss:replace by isogeny} and \ref{sss:replace by semisimple}, we can assume that both $G_1$ and $G_2$
are semi-simple.

\sssec{}

We now use the description of $\Shv(G_i)^{\on{q-const}}$ as $\Sym(W_i^\vee[-1])\mod$ in \secref{sss:descr q const ss}.
We note that the corresponding vector space $W$ for $G:=G_1\times G_2$ equals $W_1\oplus W_2$. This follows
from the fact that
$$\on{C}^\cdot(G_1)\otimes \on{C}^\cdot(G_2)\simeq \on{C}^\cdot(G_1\times G_2).$$

\medskip

Under the above identifications, the functor
$$\Shv(G_1)^{\on{q-const}}\otimes \Shv(G_2)^{\on{q-const}}\to \Shv(G_1\times G_2)^{\on{q-const}}$$
corresponds to
$$\Sym(W_1^\vee[-1])\mod \otimes \Sym(W_1^\vee[-1])\mod \to \Sym((W_1\oplus W_2)^\vee[-1])\mod,$$
which is manifestly an equivalence.

\qed[\lemref{l:alm const}]

\ssec{Proof of \propref{p:ff mon}} \label{ss:ff mon}

%
%
%

\sssec{}

Let $G\overset{p}\to G_{\on{red}}$ be the maximal reductive quotient of $G$. We have a
commutative diagram
$$
\CD
\bC\otimes \Shv(G)^{\on{q-mon}} @>>>  \bC\otimes \Shv(G) \\
@A{\sim}A{\on{Id}\otimes p^*}A  @AA{\on{Id}\otimes p^*}A \\
\bC\otimes \Shv(G_{\on{red}})^{\on{q-mon}} @>>> \bC\otimes \Shv(G_{\on{red}}),
\endCD
$$
in which the right vertical arrow is fully faithful (since $p^*$ is fully faithfful and admits a
right adjoint).

\medskip

This allows us to assume that $G$ is reductive.

\sssec{} \label{sss:covers}

Set
$$\wt{G}:=G'\times Z^0(G) \text{ and } \Gamma:=G'\cap Z^0(G),$$
where $G'$ is the derived group and $Z^0(G)$ is the connected center of $G$.

\medskip

We have
$$\bC\otimes \Shv(G)^{\on{q-mon}} \simeq (\bC\otimes \Shv(\wt{G})^{\on{q-mon}})^\Gamma \text{ and }
\bC\otimes \Shv(G)\simeq (\bC\otimes \Shv(\wt{G}))^\Gamma.$$

This reduces the assertion to the case when $G$ is a product of
a semi-simple group and a torus.

\sssec{}

Thus, let $G$ be semi-simple and $T$ a torus. We factor the functor
$$\bC\otimes  \Shv(T\times G)^{\on{q-mon}} \to \bC\otimes \Shv(T\times G)$$
as
\begin{multline*}
\bC\otimes \Shv(T\times G)^{\on{q-mon}}\simeq \bC\otimes \Shv(T)^{\on{q-mon}}\otimes \Shv(G)^{\on{q-mon}} \simeq
\bC\otimes \Shv(T)^{\on{mon}}\otimes \Shv(G)^{\on{q-mon}} \to \\
\to \bC\otimes \Shv(T)\otimes \Shv(G)^{\on{q-mon}} \to \bC\otimes \Shv(T)\otimes \Shv(G) \to \bC\otimes \Shv(T\times G),
\end{multline*}
where:

\begin{itemize}

\item The third arrow is fully faithful since $\Shv(T)^{\on{mon}}\to \Shv(T)$ is fully faithful and admits a right adjoint;

\medskip

\item The fifth arrow is fully faithful since the functor $\Shv(T)\otimes \Shv(G) \to \Shv(T\times G)$
is fully faithful and admits a right adjoint.

\end{itemize}

\medskip

Thus, it remains to show that the fourth arrow is fully faithful.
This reduces to the assertion of the proposition to the case when $G$ is semi-simple.

\sssec{}

It is easy to reduce the statement to the case when $G$ is simply-connected (the assertion survives taking a finite
cover as in \secref{sss:covers}).

\medskip

When $G$ is simply connected, we have $\Shv(G)^{\on{q-mon}}=\Shv(G)^{\on{q-const}}$, so it is enough to show that
$$\Shv(G)^{\on{q-const}}\otimes \bC\to \Shv(G)\otimes \bC$$
is fully faithful.

\medskip

However, this is the case since
$$\Shv(G)^{\on{q-const}}\to \Shv(G)$$
admits a \emph{left} adjoint, see \propref{p:left adj ss}.

\qed[\propref{p:ff mon}]

\ssec{Proof of \thmref{t:decomp}} \label{ss:decomp}

The proof presented below was explained to us by S.~Raskin.

\sssec{}

Let $\CF$ be an object of $\Shv(G)^{\on{mult-decomp}}$. We claim that in this case its
perverse cohomologies sheaves $h^i(\CF)$ also belong to $\Shv(G)^{\on{mult-decomp}}$.

\medskip

Indeed, this follows from the fact that the functor $\on{mult}^!$ is t-exact up to a cohomological
shift, and the functor
$$\Shv(G)\otimes \Shv(G) \overset{\boxtimes}\to \Shv(G\times G)$$
is t-exact, so its essential image is preserved by the cohomological truncations.

\sssec{}

Similarly, if $\CF\in \Shv(G)^{\on{mult-decomp}}$ is perverse, then its irreducible subquotients
also belong to $\Shv(G)^{\on{mult-decomp}}$.

\medskip

Hence, we can assume that $\CF$ is irreducible.

\sssec{}

The (cohomologically shifted) functor $\on{mult}^!$ sends irreducible perverse sheaves to irreducible perverse sheaves. Irreducible objects
of $(\Shv(G)\otimes \Shv(G))^\heartsuit$ have the form
$$\CF_1\boxtimes \CF_2,$$

Hence,
\begin{equation} \label{e:split LS}
\on{mult}^!(\CF)\simeq \CF_1\boxtimes \CF_2.
\end{equation}

Let $U\subset G$ (resp., $U_1$, $U_2$) be the maximal open subset over which $\CF$ (resp., $\CF_1$, $\CF_2$)
is lisse.

\medskip

We obtain that
$$\on{mult}^{-1}(U)=U_1\times U_2.$$

However, it is easy to see that this forces $U=U_1=U_2=G$. Hence, $\CF$ (and hence $\CF_1$ and $\CF_2$)
are irreducible local systems.

\sssec{}

Restricting \eqref{e:split LS} to $1_G\times G$, we obtain
$$\CF\simeq (\CF_1)_{1_G}\otimes \CF_2.$$

Since $\CF$ is irreducible, we obtain that $\CF_1$ is one-dimensional, and $\CF\simeq \CF_2$.

\medskip

Symmetrically, we obtain that $\CF_2$ is one-dimensional, and $\CF\simeq \CF_1$.

\sssec{}

Thus, $\CF$ is a one-dimensional local system, and we have an isomorphism
\begin{equation} \label{e:split LS bis}
\on{mult}^!(\CF)\simeq \CF\boxtimes \CF.
\end{equation}

I.e., $\CF$ is an abelian character sheaf on $G$.

\qed[\thmref{t:decomp}]

\ssec{Another proof of \thmref{t:decomp} for \texorpdfstring{$G=\BG_a$}{Ga}}

\sssec{}

In this subsection we will present another proof of \thmref{t:decomp}, specific to
$G=\BG_a$. This proof will have a feature that it shows that the functor
$$\bC\otimes \Shv(\BG_a)^{\on{q-mon}}\to \bC\otimes \Shv(\BG_a)^{\on{mult-decomp}}$$
is an equivalence for \emph{any} $\bC\in \DGCat$.

\medskip

Hence, by the same logic as in the proof of \thmref{t:mon cat rep}, this would
imply the assertion of \conjref{c:mon cat rep} for $\BG_a$.

%
%
%
%
%
%
%

\sssec{}

Let $X$ be an algebraic variety. Let
$$\Shv(X)^{\on{diag-decomp}}\subset \Shv(X)$$
be the full subcategory consisting of objects $\CF$, for which
$$(\Delta_X)_*(\CF)\in \Shv(X)\otimes \Shv(X).$$

We will prove:

\begin{thm} \label{t:diag decomp}
The subcategory $\Shv(X)^{\on{diag-decomp}}$ equals the full subcategory
$\Shv^{\on{ponct}}(X)\subset \Shv(X)$
generated under colimits by the objects $\delta_x$ for closed points $x\in X$.
\end{thm}

\sssec{}

Note that the statement of \thmref{t:decomp} for $\BG_a$ is equivalent to that of
\thmref{t:diag decomp} via Fourier-Deligne transform.

\medskip

The rest of this subsection is devoted to the proof of \thmref{t:diag decomp}.

\sssec{} \label{sss:diag decomp fib}

For a closed point
$$i:\Spec(k) \to X,$$
consider the canonical map
\begin{equation} \label{e:! to *}
i^!(\CF)\to i^!\circ i_*\circ i^*(\CF) \simeq i^*(\CF).
\end{equation}

\sssec{} \label{sss:! to * zero}

Note that the map \eqref{e:! to *} is zero when $x$ is a smooth point of (a positive-dimensional) $X$ and $\CF$ is lisse near $x$
(i.e., the (perverse) cohomologies of $\CF$
are colimits of perverse sheaves that are lisse on a neighborhood of $x$).

\medskip

Indeed, if $x$ is a smooth point, the map \eqref{e:! to *} is zero for $\CF=\on{const}_X$. Now, we have a commutative square
$$
\CD
i^!(\on{const}_X)\otimes i^*(\CF) @>>> i^!(\CF) \\
@V{\text{\eqref{e:! to *}}\otimes \on{id}}VV @VV{\text{\eqref{e:! to *}}}V \\
i^*(\on{const}_X)\otimes i^*(\CF) @>{\sim}>> i^*(\CF),
\endCD
$$
where the top horizontal arrow is an isomorphism for $\CF$ lisse near $x$.

\sssec{}

We claim that if $\CF\in \Shv(X)^{\on{diag-decomp}}$, then the map \eqref{e:! to *}
is an isomorphism. Indeed, we can interpret \eqref{e:! to *} as the natural transformation
\begin{equation} \label{e:two restrictions}
i^* \circ (\on{id}\times i)^! \to i^!\circ (i\times \on{id})^*,
\end{equation}
corresponding to the diagram
$$
\CD
X\times \on{pt} @>{\on{id}\times i}>> X \times X \\
@A{i}AA @AA{i\times \on{id}}A \\
\on{pt} @>{i}>> \on{pt}\times X
\endCD
$$
evaluated on $(\Delta_X)_*(\CF)$.

\medskip

However, \eqref{e:two restrictions} is an isomorphism on the essential image of $\Shv(X)\otimes \Shv(X)$.

\sssec{}

Let $\CF\notin \Shv^{\on{ponct}}(X)$. Let $Y\subset X$ be the subvariety of maximal dimension such that the
*-fiber of $\CF$ at the generic point of $Y$ is non-zero. Up to replacing $X$ by $Y$ and $\CF$ by its restriction to
$Y$, we can assume that $\CF$ does not vanish at the generic point of (a positive-dimensional) $X$.

\medskip

Let $k'$ denote the algebraic closure of the field of fractions of $X$. Let $X'$, $\CF'$ denote the
corresponding base-changed objects. Note that
$$\CF'\in \Shv(X')^{\on{diag-decomp}}.$$

\medskip

Let
$$i':\Spec(k')\to X'$$
denote the canonical point. However, it is easy to see to the map \eqref{e:! to *}
$$(i')^!(\CF')\to (i')^*(\CF')$$
is zero for any $\CF'$ that comes by base change from $X$ (see \secref{sss:! to * zero}).

\medskip

Combining with \secref{sss:diag decomp fib},
we obtain that $(i')^*(\CF')=0$, which is a contradiction.

\qed

\end{document}